\documentclass[aos]{imsart}

\usepackage{amsmath,amssymb,amsfonts,bbm,verbatim,multirow,color,xcolor,url}
\usepackage{epic,eepic,psfrag,epsfig}
\usepackage[sf,bf,SF,footnotesize]{subfigure}
\usepackage{graphicx}
\usepackage{amsthm,breakcites}
\usepackage{amscd}
\usepackage{epsfig}
\usepackage{natbib}
\usepackage{enumerate}
\usepackage{array}

\usepackage{xr-hyper}
\RequirePackage[colorlinks,citecolor={blue!60!black},urlcolor={blue!70!black},linkcolor={red!60!black},breaklinks,hypertexnames=false]{hyperref}

\makeatletter
\newcommand*{\addFileDependency}[1]{
  \typeout{(#1)}
  \@addtofilelist{#1}
  \IfFileExists{#1}{}{\typeout{No file #1.}}
}
\makeatother

\startlocaldefs

\newtheorem{eg}{Example}
\newtheorem{thm}{Theorem}
\newtheorem{prop}[thm]{Proposition}
\newtheorem{lemma}[thm]{Lemma}
\newtheorem{cor}[thm]{Corollary}

\DeclareMathOperator*{\argmin}{argmin}
\DeclareMathOperator*{\sargmin}{sargmin}

\newcommand{\vertiii}[1]{{\left\vert\kern-0.25ex\left\vert\kern-0.25ex\left\vert #1 
    \right\vert\kern-0.25ex\right\vert\kern-0.25ex\right\vert}}
\def\hat{\widehat}

\endlocaldefs

\begin{document}

\begin{frontmatter}

\title{Efficient functional estimation and the super-oracle phenomenon}
\runtitle{Efficient functional estimation}

\begin{aug}
\author[A]{\fnms{Thomas B.}~\snm{Berrett}\ead[label=e1]{tom.berrett@warwick.ac.uk}}
\and
\author[B]{\fnms{Richard J.}~\snm{Samworth}\ead[label=e2]{r.samworth@statslab.cam.ac.uk}}
\address[A]{Department of Statistics, University of Warwick\printead[presep={,\ }]{e1}}

\address[B]{Statistical Laboratory, University of Cambridge\printead[presep={,\ }]{e2}}
\end{aug}

\begin{abstract}
We consider the estimation of two-sample integral functionals, of the type that occur naturally, for example, when the object of interest is a divergence between unknown probability densities.  Our first main result is that, in wide generality, a weighted nearest neighbour estimator is efficient, in the sense of achieving the local asymptotic minimax lower bound.  Moreover, we also prove a corresponding central limit theorem, which facilitates the construction of asymptotically valid confidence intervals for the functional, having asymptotically minimal width.  One interesting consequence of our results is the discovery that, for certain functionals, the worst-case performance of our estimator may improve on that of the natural `oracle' estimator, which itself can be optimal in the related problem where the data consist of the values of the unknown densities at the observations.  
\end{abstract}



\end{frontmatter}

\section{Introduction}

This paper concerns the estimation of two-sample density functionals of the form
\begin{equation}
  \label{Eq:Functional}
  T = T(f,g) := \int_{\mathcal{X}} f(x)\phi\bigl(f(x),g(x)\bigr) \, dx,
\end{equation}
where $\mathcal{X} := \{x \in \mathbb{R}^d:f(x) > 0,g(x) > 0\}$, based on independent $d$-dimensional random vectors $X_1,\ldots,X_m,Y_1,\ldots,Y_n$, where $X_1,\ldots,X_m$ have density $f$ and $Y_1,\ldots,Y_n$ have density $g$.  The interest in the estimation of such functionals arises from many applications: for instance, many divergences such as the Kullback--Leibler divergence, total variation and Hellinger distances (or more generally, all $\varphi$-divergences) are of this form.  The estimation of such divergences is important for two-sample testing \citep{WornowizkiFried2016}, registration problems in image analysis \citep{HMMG2002} and generative adversarial networks \citep{NCT2016}, to name just a few examples.  Of course, we can regard the problem of estimation of one-sample density functionals
\begin{equation}
  \label{Eq:OneSampleFunctional}
  H(f) := \int_{\{x:f(x) > 0\}} f(x)\psi\bigl(f(x)\bigr) \, dx,
\end{equation}
which include Shannon and R\'enyi entropies, as a special case.

Motivated by these applications, the estimation of the two-sample functional~\eqref{Eq:Functional} (or closely related quantities) has received considerable attention in the literature recently \citep[e.g.][]{KKPW2014,Kandasamy15,SinghPoczos2016,SSP2018,MSGH2018}.  Naturally, the one-sample version of the problem, and special cases of it, have been highly-studied subjects over several decades \citep[e.g.][]{KozachenkoLeonenko1987,BickelRitov1988,BirgeMassart1995,Laurent1996,Beirlant:97,LPS2008,leonenko2010statistical,BiauDevroye2015,HJWW2018,BSY2018}.  It turns out that many functionals of interest involve functions $\phi$ in~\eqref{Eq:Functional} that are non-smooth as their arguments approach zero, or functions $\psi$ in~\eqref{Eq:OneSampleFunctional} that are non-smooth as their argument vanishes.  For instance, for the Shannon entropy, $\psi(y) = -\log y$, while the R\'enyi entropy of order $\kappa$ is essentially equivalent to $\psi(y) = y^{\kappa-1}$, which is non-smooth as $y \rightarrow 0$ when $\kappa \in (0,1)$.  To avoid problems caused by this lack of smoothness, many of the aforementioned authors assume that the density $f$ is bounded away from zero on its (compact) support.  In that case, \emph{efficient} estimators can sometimes be obtained; to give just one example, when $f$ is also $s$-H\"older smooth on $\{x:f(x) > 0\}$ with $s > d/4$, \citet{Laurent1996} obtained a Shannon entropy estimator $\hat{H}_m$ satisfying
\begin{equation}
\label{Eq:Laurent}
  m\mathbb{E}\bigl[\{\hat{H}_m - H(f)\}^2\bigr] \rightarrow \int_{\{x:f(x) > 0\}} f \log^2 f - H(f)^2.
\end{equation}
The limit in~\eqref{Eq:Laurent} is the asymptotic rescaled mean squared error of the oracle estimator $H_m^* := -m^{-1}\sum_{i=1}^m \log f(X_i)$, and is optimal in a local asymptotic minimax sense \citep{IbragimovKhasminskii1991,Laurent1996}.

However, the assumption that the density $f$ is bounded away from zero on its support is made purely for mathematical convenience; it assumes away the essential difficulty of the problem caused by the non-smoothness and rules out many standard densities of common interest.  In the related problem of density estimation, it is known that, depending on the loss function and the smoothness of the densities considered, optimal rates of convergence can be very different when densities with unbounded support are allowed \citep{DJKP1996,JuditskyLambertLacroix2004,GoldenshlugerLepski2014}. 

It is therefore of great interest to understand the ways in which low density regions interact with the potential non-smoothness of the functional to determine the behaviour of estimators.  Previous works in this direction have tended to focus on specific functionals and on rates of convergence \citep[e.g.][]{Tsybakov:96,HJWW2018}.  By contrast, in this work our aim is to provide a class of  estimators that are efficient for a wide spectrum of functionals.  Our estimators will be deterministically weighted versions of preliminary estimators based on nearest neighbour distances.  To set the scene, for integers $k_X \in \{1,\ldots,m-1\}$ and $k_Y \in \{1,\ldots,n\}$, write $\rho_{(k_X),i,X}$ for the (Euclidean) distance between $X_i$ and its $k_X$th nearest neighbour in the sample $X_1,\ldots,X_{i-1},X_{i+1},\ldots,X_m$, and write $\rho_{(k_Y),i,Y}$ for the distance between $X_i$ and its $k_Y$th nearest neighbour in the sample $Y_1,\ldots,Y_n$.  The starting point for the construction of our estimators is the approximation
\[
  f(X_i)V_d\rho_{(k_X),i,X}^d \approx k_X/m,
\]
where $V_d := \pi^{d/2}/\Gamma(1+d/2)$ denotes the $d$-dimensional Lebesgue measure of the unit Euclidean ball in $\mathbb{R}^d$; this arises by comparing the proportion of points in a ball of radius $\rho_{(k_X),i,X}$ about $X_i$ with a local constant approximation to the probability content of the same ball.  This motivates the initial estimator
\begin{equation}
  \label{Eq:Unweighted}
  \widetilde{T}_{m,n} = \widetilde{T}_{m,n,k_X,k_Y} := \frac{1}{m}\sum_{i=1}^m \phi\biggl(\frac{k_X}{mV_d\rho_{(k_X),i,X}^d},\frac{k_Y}{nV_d\rho_{(k_Y),i,Y}^d}\biggr).
\end{equation}
Restricting attention for simplicity of exposition to the one-sample analogue $\widetilde{T}_m = \widetilde{T}_{m,k}$ of~\eqref{Eq:Unweighted} that simply replaces $\phi(\cdot,\cdot\cdot)$ with $\psi(\cdot)$ and $k_X$ with $k$, it has long been known in the special case of the Shannon entropy functional that one should debias $\widetilde{T}_m$ by replacing~$k$ with $e^{\Psi(k)}$, where $\Psi(\cdot)$ denotes the digamma function \citep{KozachenkoLeonenko1987}.  This amounts to adding $\log k - \Psi(k)$ to the original estimator.  \citet{Ganguly18} argued that for general two-sample functionals, the estimator~\eqref{Eq:Unweighted} can be debiased to leading order via an implicit inverse Laplace transform, and showed that this has an explicit expression in certain examples.  It turns out, however, that even the remaining bias is large enough to preclude efficient estimation when $d \geq 4$, and this motivates us to consider weighted linear combinations of estimators of the form~\eqref{Eq:Unweighted} over different choices of $k_X$ and $k_Y$, where the weights are chosen to cancel sufficient terms in the bias expansion.  A subtle question concerns the issue of whether to apply our weights to the original estimators~\eqref{Eq:Unweighted} or their debiased versions.  We address this by using fractional calculus techniques to provide an explicit expression for the leading order remaining bias of the debiased estimators.  We conclude that, in general, the gain from the fact that fewer non-zero weights are required to obtain an efficient estimator when applying these weights to the debiased estimator is outweighed by the added complication of the resulting estimator.  However, in special cases such as the Kullback--Leibler and R\'enyi divergences, where the correct explicit debiasing terms are available, the weighting scheme simplifies and we advocate applying the weights to the debiased estimator.

Returning to the general case, our final estimators $\hat{T}_{m,n}$ are based on weighted averages of estimators of the form $\widetilde{T}_{m,n,k_X,k_Y}$ for different choices of $k_X$ and $k_Y$; such estimators are attractive because they generalise easily to multivariate cases (unlike, for example, estimators based on sample spacings), and because they are straightforward to compute.  Our first main result (Theorem~\ref{Thm:Main} in Section~\ref{Sec:Main}), reveals that the dominant asymptotic contribution to the squared error risk of $\hat{T}_{m,n}$ is of the form $v_1/m + v_2/n$ as $m,n \rightarrow \infty$, uniformly over appropriate classes of densities $f,g$, functions $\phi$ and choices of weights, for certain variance functionals $v_1 = v_1(f,g)$ and $v_2 = v_2(f,g)$ given in~\eqref{Eq:V1V2} below.  Theorem~\ref{Thm:LAM} in Section~\ref{Sec:LAM} complements this by establishing that $v_1$ and $v_2$ are optimal in a local asymptotic minimax sense.  We therefore conclude that, under the conditions of these results, the estimators $\hat{T}_{m,n}$ are efficient.

In addition to studying the efficiency of our estimators $\hat{T}_{m,n}$, it is also highly desirable to be able to derive their asymptotic distributions; such a result could be used, for instance, to obtain an asymptotically valid confidence interval for $T$.  Despite the fact that the summands in our estimator are dependent, for the special case of the one-sample Shannon entropy functional, it is straightforward to derive the asymptotic normality of the weighted nearest neighbour estimator, as it is well approximated by the efficient, `oracle' estimator $-m^{-1}\sum_{i=1}^m \log f(X_i)$.  However, for general functionals, the natural oracle estimator may not be efficient, as explained in the next paragraph; this means that deriving the asymptotic distribution of $\hat{T}_{m,n}$ in such cases remains a significant challenge.  In our second main result (Theorem~\ref{Thm:Normality} in Section~\ref{Sec:Main}), we show how the problem can be reexpressed in a form where we can apply the central limit theorem of \citet{Baldi1989} for dependent random variables for which the degrees of the nodes in the pairwise dependency graph are controlled.  Thus, the estimators $\hat{T}_{m,n}$ are indeed asymptotically normal under appropriate conditions.

As a byproduct of our efficiency analysis, we uncover a curious phenomenon that can occur for certain functionals; for ease of exposition here, we focus on the R\'enyi-type functional
\[
  H_\kappa := \int_{\mathbb{R}^d} f(x)^{\kappa} \, dx,
\]
with $\kappa \in (1/2,1)$.  Given access to $f(X_1),\ldots,f(X_m)$, the natural oracle estimator in this setting is
\[
  H_m^* := \frac{1}{m}\sum_{i=1}^m f(X_i)^{\kappa-1}.
\]
Indeed, Proposition~\ref{Thm:SuperOracleLAM} reveals that this oracle estimator can be optimal in a local asymptotic minimax sense for the oracle problem where the practitioner seeks to estimate a one-sample functional such as $H_\kappa$ based on $f(X_1),\ldots,f(X_m)$.  Nevertheless, surprisingly, we find that there exists an estimator $\hat{H}_m$ and general classes $\mathcal{F}$ of densities for which
\begin{equation}
  \label{Eq:SuperOracle}
  \lim_{m \rightarrow \infty} \sup_{f \in \mathcal{F}} \frac{\mathbb{E}_f\bigl\{(\hat{H}_m - H_\kappa)^2\}}{\mathbb{E}_f\bigl\{(H_m^* - H_\kappa)^2\}} = \kappa^2 < 1.
  \end{equation}
We refer to this as the \emph{super-oracle phenomenon}.  It is important to note that this is very different from the phenomenon of superefficiency, as occurs with, e.g., the Hodges estimator \citep[][Example~6.2.5]{LehmannCasella1998}.  There, in the case of scalar parameter estimation, asymptotic improvement in mean squared error risk is possible at a set of fixed parameter values, which form a Lebesgue null set \citep{LeCam1953,vanderVaart1997}.  Moreover, and more importantly from our perspective, the superefficient asymptotic behaviour is necessarily accompanied by worse finite-sample performance in a neighbourhood of points of superefficiency, so that any apparent improvement is really an artefact of the pointwise asymptotic regime considered.  By contrast, in~\eqref{Eq:SuperOracle}, the supremum is taken inside the limit, so that the super-oracle improvement for large $m$ can be considered as genuine.

The remainder of the paper is organised as follows: in Section~\ref{Sec:Main}, we present our main results on the asymptotic squared error risk and asymptotic normality of our general two-sample functional estimators.  Section~\ref{Sec:Bias} is devoted to understanding the bias of these estimators and a discussion of the potential benefits of debiasing them before computing our weighted averages, while Section~\ref{Sec:Variance} considers their variance properties.  In Section~\ref{Sec:SuperOracle}, we describe the super-oracle phenomenon in greater detail, and in Section~\ref{Sec:LAM} we present a local asymptotic minimax lower bound that illustrates the asymptotic optimality of our estimators and justifies referring to them as efficient.  Our main theoretical arguments are given in the supplementary material, as well as various auxiliary results and bounds on remainder terms. 

We end this section by introducing some notation used throughout the paper.  For $m \in \mathbb{N}_0$, we write $[m] := \{0,1,\ldots,m\}$.  If $A$ is a vector, matrix or array, we write $\|A\|$ for its Euclidean vectorised norm.  For $x \in \mathbb{R}^d$ and $r \geq 0$, let $B_x(r) := \{y \in \mathbb{R}^d: \|y-x\| \leq r\}$ denote the closed Euclidean ball or radius $r$ about $x$. 
For vectors $a$ and $b$ of the same dimension, we write $a \circ b$ for their Hadamard product.  If $Z$ is a random variable, we write $\mathcal{L}(Z)$ for its law.  We write $\mathcal{Z} := (0,\infty)^2$.  For a smooth function $\phi:\mathcal{Z} \rightarrow \mathbb{R}$, $\mathbf{z} = (u,v) \in \mathcal{Z}$ and $j,l \in \mathbb{N}$, we write $\phi_{jl}(\mathbf{z}) := \frac{\partial^{j+l} \phi}{\partial u^j \partial v^l}$.  We also use multi-index notation for derivatives, so that, for a sufficiently smooth density $f^*$ on $\mathbb{R}^d$, $x = (x_1,\ldots,x_d)^T \in \mathbb{R}^d$, $t \in \mathbb{N}$ and a multi-index $\boldsymbol{\alpha} = (\alpha_1,\ldots,\alpha_d) \in \mathbb{N}_0^d$ with $|\boldsymbol{\alpha}| := \sum_{j=1}^d \alpha_j = t$, we write $\partial^{\boldsymbol{\alpha}} f^* := \frac{\partial^t f^*}{\partial x_1^{\alpha_1}\ldots \partial x_d^{\alpha_d}}$.  For $\alpha > 0$ and a density $f^*$ on $\mathbb{R}^d$, we write $\mu_\alpha(f^*) := \int_{\mathbb{R}^d} \|x\|^\alpha f^*(x) \, dx$ and $\|f^*\|_\infty := \sup_{x \in \mathbb{R}^d} f^*(x)$.  For $r \in [0,\infty)$ and $x \in \mathbb{R}^d$, we also define $h_{x,f^*}(r) := \int_{B_x(r)} f^*(y) \, dy$ and, for $s \in [0,1)$, let $h_{x,f^*}^{-1}(s) := \inf\{r \geq 0:h_{x,f^*}(r) \geq s\}$.  Recall that, for $a, b > 0$, the beta function is defined by $\mathrm{B}_{a,b} := \int_0^1 t^{a-1}(1-t)^{b-1} \, dt$ and define also the corresponding density $\mathrm{B}_{a,b}(s) := s^{a-1}(1-s)^{b-1}/\mathrm{B}_{a,b}$ for $s \in (0,1)$.  

\section{Main results}
\label{Sec:Main}

Let $X_1,\ldots,X_m,$ and $Y_1,\ldots,Y_n$ be independent $d$-dimensional random vectors, with $X_1,\ldots,X_m$ having density $f$ and with $Y_1,\ldots,Y_n$ having density $g$, both with respect to Lebesgue measure on~$\mathbb{R}^d$.  We consider the estimation of the functional $T(f,g)$ in~\eqref{Eq:Functional}.

Before we can state our main theorems on the asymptotic risk and normality of our functional estimators, we need some preparatory work.  This will consist of definitions of the classes of functionals and densities over which our results will hold, the definitions of our weighted nearest neighbour estimators and the corresponding classes of allowable weights, as well as various parameters that will play a role in the statements of our results.

Starting with our classes of functionals, we impose a condition on the function $\phi$ in~\eqref{Eq:Functional}.  It will be convenient to introduce the shorthand $x_\wedge := x \wedge 1$ and $x_\vee := x \vee 1$ for $x \geq 0$.  Let $\Xi := \mathbb{R}^2 \times (\mathbb{N} \setminus \{1\}) \times (1,\infty)$, and for $\xi = (\kappa_1,\kappa_2,\beta^*,L) \in \Xi$, let $\Phi \equiv \Phi(\xi)$ denote the class of functions $\phi: \mathcal{Z} \rightarrow \mathbb{R}$ for which the partial derivatives $\phi_{\ell_1\ell_2}$ exist for all $\ell_1,\ell_2 \in \mathbb{N}_0$ with $\ell_1+\ell_2 \leq \beta^*$ and satisfy
\[
	\bigl| u^{\ell_1} v^{\ell_2} \phi_{\ell_1\ell_2}(u,v) \bigr| \leq Lu_\wedge^{\kappa_1} u_\vee^L v_\wedge^{\kappa_2} v_\vee^L. 
      \]
for all $(u,v) \in \mathcal{Z}$.  This is a growth condition on $\phi$ and its partial derivatives of order up to~$\beta^*$.  The pre-multiplier $u^{\ell_1} v^{\ell_2}$ allows us to control discrepancies of $\phi$ under relative, as opposed to absolute, changes in its arguments.  Moreover, the right-hand side of the bound affords additional flexibility regarding the level of regularity required for both small and large values of these first and second arguments, controlled by the parameters $\kappa_1, \kappa_2$ and~$L$.  This latter aspect will allow us to include functionals such as the Kullback--Leibler and R\'enyi divergences, for which the corresponding $\phi$ is non-smooth as the densities approach zero; see Examples~\ref{Ex:KL} and~\ref{Ex:Renyi} below.
More generally, for the $\varphi$-divergence functional with $\phi(u,v) = \varphi(v/u)$, it is straightforward to express this condition in terms of a condition on $\varphi$.  

For our classes of densities, fix $\beta > 0$, a density $f$ on $\mathbb{R}^d$, and $x \in \mathbb{R}^d$ with $f(x) > 0$ such that $f$ is $\underline{\beta} :=\lceil \beta \rceil -1$-times differentiable at $x$.  Write $f^{(t)}(x) \in (\mathbb{R}^d)^{\otimes t}$ for the $t^{\mathrm{th}}$ derivative array of $f$ at $x$ for $t \in [\underline{\beta}]$, so that $f^{(t)}_{j_1\ldots j_t}(x) := \frac{\partial^t f}{\partial x_{j_1}\ldots \partial x_{j_t}}(x)$ for $(j_1,\ldots,j_t) \in \{1,\ldots,d\}^t$.  Now define 
\begin{align*}
  M_{f,\beta}(x):= \inf\Biggl\{ M \geq 1  :& \max_{t \in [\underline{\beta}]} \biggl( \frac{\|f^{(t)}(x)\|}{f(x)} \biggr)^{1/t} \\
  &\bigvee \sup_{\substack{y,z \in B_x(1/M), \\ y \neq z}} \biggl(\frac{\|f^{(\underline{\beta})}(z)-f^{(\underline{\beta})}(y)\|}{f(x)\|z-y\|^{\beta-\underline{\beta}}} \biggr)^{1/\beta} \leq M \Biggr\};
\end{align*}
otherwise, we set $M_{f,\beta}(x) := \infty$.  The quantity $M_{f,\beta}(x)$ measures the smoothness of derivatives of $f$ in neighbourhoods of $x$, relative to $f(x)$ itself, but does not require $f$ to be smooth everywhere.  For instance, if $f$ is the uniform density on the unit ball $B_0(1)$, then $M_{f,\beta}(x) = 1/(1-\|x\|)$ for $\|x\| < 1$.  Now, for $\theta=(\alpha,\beta,\lambda,C) \in (0,\infty)^4$, and writing $\mathcal{F}_d$ for the class of densities on $\mathbb{R}^d$, let
\[
	\mathcal{G}_{d,\theta} := \biggl\{ f \in \mathcal{F}_d : \mu_\alpha(f) \leq C, \|f\|_\infty \leq C, \int_{\{x:f(x) > 0\}} f(x) \Bigl\{\frac{M_{f, \beta}(x)^d}{f(x)}\Bigr\}^\lambda \,dx  \leq C \biggr\}.
      \]
Thus, in addition to requiring a moment assumption and a bounded density, the classes $\mathcal{G}_{d,\theta}$ also impose an integrability condition on our local measure of smoothness; to understand this condition, we note that in constructing a nearest-neighbour based estimate of $f(x)$, the crucial quantity that controls the bias is the function
\[
s \mapsto \inf\biggl\{r \geq 0:\int_{B_x(r)} f(y) \, dy \geq s\biggr\} =: h_{x,f}^{-1}(s)
\]
on $(0,1)$.  If $f$ is constant in a neighbourhood of $x$ with $f(x) > 0$, then $h_{x,f}^{-1}(s)^d = \frac{s}{V_df(x)}$ for small $s > 0$.  More generally, the error of the approximation of $h_{x,f}^{-1}(s)^d$ by this linear function of $s$ (together with higher-order Taylor expansion terms) is controlled by $\frac{M_{f,\beta}(\cdot)^d}{f(\cdot)}$; see Lemma~\ref{Lemma:hxinvbounds} for a formal statement.  This explains why we ask for a condition on an appropriate norm of $\frac{M_{f,\beta}(\cdot)^d}{f(\cdot)}$ in our classes.  It is an attractive feature that the assumption comes in an integral form, as opposed to requiring a boundedness condition on $M_{f, \beta}(x)$, for instance.  This integrability condition is our primary tool for avoiding the assumption that the density is bounded away from zero on its support (see the discussion in the Introduction).  While \citet{Tsybakov:96} and \citet{BSY2018} made first steps in this direction in the context of Shannon entropy estimation, the former of these works, which focused on the case $d=1$, required a strictly positive density on the whole real line; the latter relaxed this condition a little, but made extremely stringent requirements on the behaviour of the density $f$ in neighbourhoods of points $x_0 \in \mathbb{R}^d$ with $f(x_0) = 0$.  In particular, no $\mathrm{Beta}(a,a)$ density was allowed, for any $a > 0$, and the only densities having points $x_0$ with $f(x_0) = 0$ that were shown to belong to their classes involved all derivatives also vanishing at $x_0$.  By contrast, Proposition~\ref{Prop:Beta} below shows that a multivariate spherically symmetric generalisation of a $\mathrm{Beta}(a,b)$ density belongs to $\mathcal{G}_{d,\theta}$ for suitable $\theta \in (0,\infty)^4$, provided only that $a,b \geq 1$ (though in fact the requirements of our Theorem~\ref{Thm:Main} on efficiency would actually also need $b > d-1$ for this family).  
\begin{prop}
\label{Prop:Beta}
Fix $a,b \in [1,\infty)$, and let $f$ denote the density on $\mathbb{R}^d$ given by
\[
  f(x) = C_{d,a,b}\|x\|^{a-1}(1-\|x\|)^{b-1}\mathbbm{1}_{\{\|x\| \leq 1\}},
\]
where $C_{d,a,b} := \frac{\Gamma(a+b+d-1)}{dV_d \Gamma(a+d-1) \Gamma(b)}$.  Then for any $\alpha, \beta>0$ and any $\lambda \in \bigl(0,b/(b+d-1)\bigr)$, there exists $C_0 > 0$, depending only on $\alpha, \beta$ and $\lambda$, such that $f \in \mathcal{G}_{d,(\alpha,\beta,\lambda,C)}$ for any $C \geq C_0$.
\end{prop}
From Proposition~\ref{Prop:Beta} we also see that discontinuous densities may also belong to $\mathcal{G}_{d,\theta}$ for suitable $\theta \in (0,\infty)^4$; in particular, the $U[-1,1]$ density belongs to $\mathcal{G}_{1,(\alpha,\beta,\lambda,C)}$ for any $\alpha,\beta > 0$, $\lambda \in (0,1)$ and $C \geq 1/(1-\lambda)$.  We also remark that, similar to \citet{BSY2018}, all Gaussian densities belong to $\mathcal{G}_{d,\theta}$ for any $\alpha,\beta>0, \lambda \in (0,1)$ and sufficiently large $C > 0$, and multivariate-$t$ densities with $\nu$ degrees of freedom belong to $\mathcal{G}_{d,\theta}$ for any $\alpha \in (0,\nu)$, any $\beta>0$, $\lambda \in \bigl(0,\nu/(\nu+d)\bigr)$ and $C > 0$ sufficiently large.  

To define our main class of densities, then, for $\Theta = (0,\infty)^5$ and $\vartheta = (\alpha,\beta,\lambda_1,\lambda_2,C) \in \Theta$, let $M_\beta(x) \equiv M_{f,g,\beta}(x) := M_{f,\beta}(x) \vee M_{g,\beta}(x)$ and set
\begin{align*}
  \mathcal{F}_{d,\vartheta} &:= \biggl\{(f,g) \in \mathcal{G}_{d,(\alpha,\beta,\lambda_1,C)} \times \mathcal{F}_d : \int_{\mathcal{X}} f(x) \biggl[\Bigl\{ \frac{M_{\beta}(x)^d}{f(x)} \Bigr\}^{\lambda_1} + \Bigl\{ \frac{M_{\beta}(x)^d}{g(x)} \Bigr\}^{\lambda_2}\biggr] \,dx \leq C, \\
  &\hspace{60pt}\mu_{1/C}(g) \leq C, \|g\|_\infty \leq C,  \int_{\mathcal{X}} f(x)^{2+2 \kappa_1-1/C} g(x)^{2 \kappa_2 -1-1/C}   \,dx \leq C \biggr\}.
\end{align*}
Note that $\mathcal{F}_{d,\vartheta}$ also depends on $\xi$ through $\kappa_1$ and $\kappa_2$, i.e.~on the functional we wish to estimate, though we suppress this in our notation.  To understand the final integrability condition in $\mathcal{F}_{d,\vartheta}$, we first note that the efficient variance $v_2$, defined in~\eqref{Eq:V1V2} below, can be bounded above as follows:
\[
	v_2 = \mathrm{Var} \bigl( f(Y_1)\phi_{01}(f(Y_1),g(Y_1))\bigr) \leq L^2 C^{4L+2(|\kappa_1|+|\kappa_2|)} \int_\mathcal{X} f(x)^{2+2 \kappa_1} g(x)^{2 \kappa_2 -1} \, dx,
\]
for $C \geq 1$.  Thus, for large values of $C$, our condition is only slightly stronger than assuming that $v_2$ is bounded. This slight strengthening of that assumption is made so that the integral over $\mathcal{X}$ in $v_2$ can be approximated by integrals over large subsets of $\mathcal{X}$, uniformly over $(f,g) \in \mathcal{F}_{d, \vartheta}$.

We now introduce the class of weights that we consider for our estimators.  To this end, for $k, I \in \mathbb{N}$ and $c \in (0,1)$, define
\begin{align}
  \label{Eq:Weights}
	\mathcal{W}_{I,c}^{(k)}:= \biggl\{ w =& (w_1,\ldots,w_k) \in \mathbb{R}^{k} : \sum_{j=1}^{k} w_{j} =1 \text{ and }  w_{j}=0 \text{ for } j<ck, \|w\|_1 \leq 1/c,\nonumber \\
	&\sum_{j=1}^k j^{\frac{2\ell}{d}-i} w_j = 0 \text{ for } (\ell,i) \in \bigl( \bigl[\lceil d/2 \rceil -1 \bigr] \times [I] \bigr)  \setminus \{(0,0)\} \biggr\}.
\end{align}
Fixing $\xi = (\kappa_1,\kappa_2,\beta^*,L) \in \Xi$ and $c \in (0,1)$, and for $w_X \in \mathcal{W}_{\lceil (\beta^*-1)/2 \rceil,c}^{(k_X)}$ and $w_Y \in \mathcal{W}_{\lceil (\beta^*-1)/2 \rceil,c}^{(k_Y)}$, we can now define our weighted functional estimators as 
\begin{equation}
\label{Eq:WeightedEstimator}
	\hat{T}_{m,n} \equiv \hat{T}_{m,n}^{w_X,w_Y} := \sum_{j_X=1}^{k_X} \sum_{j_Y=1}^{k_Y} w_{X,j_X} w_{Y,j_Y} \widetilde{T}_{m,n,j_X,j_Y}.
\end{equation}
Note that the constraint on the support of $w_X$ ensures that all component indices with non-zero weight are of the same order as $k_X$, with the corresponding property also holding for $w_Y$.  Once this is satisfied, and given appropriate choices of $k_X,k_Y$, the remaining  constraints in~\eqref{Eq:Weights} will ensure that the bias of $\hat{T}_{m,n}$ is asymptotically negligible.

It is convenient to use the shorthand $\phi_x:=\phi\bigl(f(x),g(x)\bigr)$, as well as $(f \phi_{10})_x := f(x) \phi_{10}\bigl(f(x),g(x)\bigr)$ and $(f \phi_{01})_x := f(x) \phi_{01}\bigl(f(x),g(x)\bigr)$ for $x \in \mathbb{R}^d$. Our result on the asymptotic risk of $\hat{T}_{m,n}$ will be expressed in terms of
\begin{equation}
\label{Eq:V1V2}
	v_1 = v_1(f,g) := \mathrm{Var}\bigl( \phi_{X_1} + (f \phi_{10})_{X_1} \bigr)  \! \quad \text{and} \! \quad v_2 = v_2(f,g) := \mathrm{Var}\bigl((f\phi_{01})_{Y_1}\bigr).
\end{equation}
Fixing $d \in \mathbb{N}$, $\vartheta = (\alpha,\beta,\lambda_1,\lambda_2,C) \in \Theta $ and $\xi = (\kappa_1,\kappa_2,\beta^*,L) \in \Xi$, we will moreover impose requirements on various derived parameters.  In particular, writing $\kappa_i^-:=\max(-\kappa_i,0)$ for $i=1,2$, it will also be convenient to define
\begin{align}
\label{Eq:Zeta}
	\zeta& := \frac{\kappa_1^-}{\lambda_1} + \frac{\kappa_2^-}{\lambda_2} + \frac{d(\kappa_1^-+ \kappa_2^-)}{\alpha} \\ 
	\tau_i&:=1 -  \max \biggl( \frac{d}{2\beta} , \frac{d}{2(2 \wedge \beta) + d} ,  \frac{d}{2(2\wedge\beta) \beta^*} ,\frac{1}{2 (\lambda_i \wedge 1)(1-\zeta)} \biggr), \quad i=1,2. \nonumber 
\end{align}
Finally, then, we are in a position to state our first main result, on the asymptotic squared error risk of $\hat{T}_{m,n}$:
\begin{thm}
  \label{Thm:Main}
  Fix $d \in \mathbb{N}$, fix $\vartheta = (\alpha,\beta,\lambda_1,\lambda_2,C) \in \Theta$ and fix $\xi = (\kappa_1,\kappa_2,\beta^*,L) \in \Xi$.  Assume that $\zeta<1/2$ and that $\min(\tau_1,\tau_2)>1/\beta^*$.  Let $(k_{X}^{\mathrm{L}})$, $(k_{Y}^{\mathrm{L}})$, $(k_{X}^{\mathrm{U}})$ and $(k_{Y}^{\mathrm{U}})$ be deterministic sequences of positive integers that satisfying $\min(k_{X}^{\mathrm{L}}m^{-1/\beta^*}, k_{Y}^{\mathrm{L}} n^{-1/\beta^*} ) \rightarrow \infty$ and $\max(k_{X}^{\mathrm{U}} m^{-(\tau_1-\epsilon)}, k_{Y}^{\mathrm{U}} n^{-(\tau_2 - \epsilon)}) \rightarrow 0$ for some $\epsilon>0$.  Then for each $c \in (0,1)$, each $w_X = w_{X}^{(k_X)} \in \mathcal{W}_{\lceil (\beta^*-1)/2 \rceil,c}^{(k_X)}$ and each $w_Y = w_{Y}^{(k_Y)} \in \mathcal{W}_{\lceil (\beta^*-1)/2 \rceil,c}^{(k_Y)}$, we have
      \[
        \sup_{\phi \in \Phi(\xi)}  \sup_{(f,g) \in \mathcal{F}_{d,\vartheta}}  \max_{\substack{k_X \in \{k_X^{\mathrm{L}},\ldots, k_X^{\mathrm{U}}\} \\ k_Y \in \{k_Y^{\mathrm{L}},\ldots, k_Y^{\mathrm{U}}\}}} \biggl|\mathbb{E}_{f,g}\bigl\{(\hat{T}_{m,n} - T)^2\bigr\} - \frac{v_1}{m} - \frac{v_2}{n}\biggr| = o\biggl(\frac{1}{m} + \frac{1}{n}\biggr)
      \]
as $m,n \rightarrow \infty$.
\end{thm}
In Proposition~\ref{Prop:ImprovedBias} in Section~\ref{Sec:ImprovedBiasSec}, we will improve Theorem~\ref{Thm:Main} by showing that when $\beta \in (0,1]$, the same conclusion holds when we replace the term $d/(2\beta)$ in the definitions of $\tau_1, \tau_2$ in~\eqref{Eq:Zeta} with $d/(4\beta)$.  This allows us to weaken the smoothness requirement on our densities for the estimators $\hat{T}_{m,n}$ to be efficient.  In particular, we only need $\beta > d/4$ instead of $\beta > d/2$, when $d \in \{1,2,3\}$ and when $\beta^*$ may be taken to be arbitrarily large, which is the case in several examples of interest, as illustrated below.

Theorem~\ref{Thm:Main} follows immediately from combining Proposition~\ref{Thm:SimpleBias} in Section~\ref{Sec:Bias} with Proposition~\ref{Prop:Variance} in Section~\ref{Sec:Variance}, which elucidate the asymptotic bias and variance of $\hat{T}_{m,n}$ respectively.  We therefore defer a description of the main ideas of our proofs until after the statements of these results, and first illustrate Theorem~\ref{Thm:Main} via several examples.

\begin{eg}
\label{Ex:KL}
Consider the Kullback--Leibler divergence, for which we may take $\phi(u,v) = \log(u/v)$. For any $\epsilon \in (0,1/2)$, any $\beta^* \geq 2$, and any $L> (\beta^*-1) !$, we have that $\phi \in \Phi(-\epsilon,-\epsilon,\beta^*,L)$. Thus, for any $d \in \mathbb{N}$ and $\vartheta=(\alpha,\beta,\lambda_1,\lambda_2,C) \in \Theta$ such that $\beta > d/2$ and $\min(\lambda_1,\lambda_2) > 1/2$, Theorem~\ref{Thm:Main} tells us that we can find sequences $(k_X),(k_Y),(w_X),(w_Y)$ such that
\[
	\sup_{(f,g) \in \mathcal{F}_{d,\vartheta} } \biggl| \mathbb{E}_{f,g} \{ ( \hat{T}_{m,n} -T)^2 \} - \frac{1}{m}\mathrm{Var}_f\log\biggl(\frac{f(X_1)}{g(X_1)}\biggr) - \frac{1}{n}\mathrm{Var}_g \biggl(\frac{f(Y_1)}{g(Y_1)}\biggr) \biggr| = o\biggl( \frac{1}{m} + \frac{1}{n} \biggr).
\]
If $f$ and $g$ are spherically symmetric beta densities as in Proposition~\ref{Prop:Beta} with parameters $(a_f,b_f)$ and $(a_g,b_g)$ respectively, then we see from the proof of Proposition~\ref{Prop:Beta} that we have $M_\beta(x) \leq A/\{\|x\|(1-\|x\|)\}$, where $A > 0$ depends only on $d, a_f, b_f, a_g$ and $b_g$.  Thus $(f,g) \in \mathcal{F}_{d, \vartheta}$ for sufficiently large $C > 0$ whenever
\begin{align*}
\lambda_1 &\in \biggl(0,\frac{b_f}{b_f+d-1}\biggr), \quad \lambda_2 \in \biggl(0, \min\biggl\{\frac{a_f+d-1}{a_g+d-1},\frac{b_f}{b_g+d-1}\biggr\}\biggr), \\ 
2a_f&-a_g + d-1>0 \quad and \quad 2b_f-b_g>0.
\end{align*}
It follows from simplifying the condition $\min(\lambda_1,\lambda_2) >1/2$ that we have efficiency whenever $\beta>d/2$ and 
\[
    \min \biggl( \frac{b_f}{b_f+d-1}, \frac{a_f+d-1}{a_g+d-1}, \frac{b_f}{b_g+d-1} \biggr) > \frac{1}{2}.
\]
As mentioned above, in Section~\ref{Sec:ImprovedBiasSec} we will see that here, as in Examples~\ref{Ex:Renyi} and~\ref{Ex:L2} below, we can weaken the first of these conditions to $\beta > d/4$ whenever $d \in \{1,2,3\}$.
\end{eg}

\begin{eg}
\label{Ex:Renyi}
For $ \kappa \in (1/2,3/2)$, consider the $\kappa$-R\'enyi divergence, for which we may take $\phi(u,v) = (u/v)^{\kappa-1}$. For any $\beta^* \geq 2$ and $L \geq (\beta^*)!$ we have $\phi \in \Phi(\kappa-1,1-\kappa,\beta^*,L)$. Let $d \in \mathbb{N}$ and $\vartheta = (\alpha,\beta,\lambda_1,\lambda_2,C) \in \Theta$ be such that $\beta > d/2$, such that
\[
	\zeta = \frac{(\kappa-1)_-}{\lambda_1} + \frac{(1-\kappa)_-}{\lambda_2} + \frac{d|1-\kappa|}{\alpha} < \frac{1}{2},
\]
and such that $\min(\lambda_1,\lambda_2) > 1/\{2(1-\zeta)\}$. Then, by Theorem~\ref{Thm:Main}, we can find sequences $(k_X),(k_Y),(w_X),(w_Y)$ such that
\begin{align*}
  \sup_{(f,g) \in \mathcal{F}_{d,\vartheta} } \biggl| \mathbb{E}_{f,g} \{ ( \hat{T}_{m,n} -T)^2 \} - \frac{\kappa^2}{m}\mathrm{Var}_f\biggl(\frac{f(X_1)^{\kappa-1}}{g(X_1)^{\kappa-1}}\biggr) - \frac{(\kappa-1)^2}{n}\mathrm{Var}_g&\biggl(\frac{f(Y_1)^{\kappa}}{g(Y_1)^{\kappa}}\biggr) \biggr| \\
  &= o\biggl( \frac{1}{m} + \frac{1}{n} \biggr).
\end{align*}
As in Example~\ref{Ex:KL}, we simplify these conditions for spherically symmetric beta distributions, but here we restrict attention to $d=1$ and $\beta > 1/4$ for simplicity.  When $\kappa \in (1,3/2)$ we have efficiency when $\min(a_f/a_g,b_f/b_g) > \kappa-1/2$, and when $\kappa \in (1/2,1)$ we have efficiency when $\min(a_f/a_g,b_f/b_g) > 1/(2\kappa)$.
\end{eg}

\begin{eg}
\label{Ex:L2}
Suppose we would like to estimate $\int_{\mathbb{R}^d} \{ f(x) - g(x) \}^2 \,dx = \int_{\mathbb{R}^d} f(x)^2 \,dx + \int_{\mathbb{R}^d} g(x)^2 \,dx - 2 \int_{\mathbb{R}^d} f(x) g(x) \,dx$. We may estimate each of these terms separately using one- or two-sample estimators as appropriate. Then, by Theorem~\ref{Thm:Main} and a corresponding one-sample version, we can achieve a mean squared error of $O(1/m+1/n)$ uniformly over classes of densities $(f,g)$ such that $\|f\|_\infty,\|g\|_\infty \leq C$, such that $\mu_{1/C}(f), \mu_{1/C}(g) \leq C$, such that
\[
	\int_{\mathbb{R}^d} f(x)^{1-\lambda_1} M_\beta(x)^{d \lambda_1} \,dx \leq C , \quad \int_{\mathbb{R}^d} g(x)^{1-\lambda_2} M_\beta(x)^{d \lambda_2} \,dx \leq C,
\]
and such that
\begin{equation}
\label{Eq:OneWayRound}
	 \int_{\{x:g(x) > 0\}} f(x) \biggl\{ \frac{M_\beta(x)^d}{g(x)} \biggr\}^{\lambda_3} \,dx \leq C,
\end{equation}
for any $C>0$, for any $\beta>d/2$ and for any $\lambda_1,\lambda_2,\lambda_3 > 1/2$.  It may be the case that $f$ has heavier tails than $g$, so that~\eqref{Eq:OneWayRound} holds with the roles of $f$ and $g$ reversed.  In that case, we can obtain the same order of mean squared error by reversing the roles of the two samples in our estimator.  
\end{eg}


To study the asymptotic normality of $\hat{T}_{m,n}$, we impose a stronger condition on the pair $(f,g)$: for $\vartheta = (\alpha,\beta,\lambda_1,\lambda_2,C) \in \Theta$, let
\begin{align}
  \label{Eq:Fdtilde}
  \widetilde{\mathcal{F}}_{d,\vartheta}:= \biggl\{&(f,g) \in \mathcal{F}_{d,\vartheta} : \min(v_1,v_2) \geq 1/C, \nonumber \\
&\max_{p=3,4} \max \biggl( \int_\mathcal{X} f(x)^{1+p\kappa_1} g(x)^{p\kappa_2} \,dx, \int_\mathcal{X} g(y)^{1+p(\kappa_2-1)} f(y)^{p + p\kappa_1} \,dy \biggr) \leq C\biggr\}.
\end{align}


To explain the lower bounds on $v_1$ and $v_2$ in~\eqref{Eq:Fdtilde}, consider the setting in which $\phi(u,v)=\varphi(v/u)$, as is the case with $\varphi$-divergences.  Then, writing $W:=g(X_1)/f(X_1)$ and $Z:=g(Y_1)/f(Y_1)$, we have that
\[
	v_1 = \mathrm{Var} \bigl( \varphi(W) - W \varphi'(W) \bigr) \quad \text{ and } \quad v_2= \mathrm{Var} \bigl( \varphi'(Z) \bigr).
\]
Now, if $f=g$ then we have $v_1=v_2=0$, and it is possible that estimators will converge to~$T$ at a faster rate than $m^{-1/2}+n^{-1/2}$ (with a potentially non-normal limiting distribution).  Thus, in order to state uniform results on the asymptotic normality of $\hat{T}_{m,n}$, we work over a class of densities for which $v_1$ and $v_2$ are bounded below.

The bounds on the integrals in~\eqref{Eq:Fdtilde} arise from considering the influence functions given by $\mathrm{IF}_1(x) := \phi_x + (f \phi_{10})_x$ and $\mathrm{IF}_2(y) := (f \phi_{01})_y$. Our conditions on $\phi$ imply that $|\mathrm{IF}_1(x)| \leq 2LC^{2L+|\kappa_1|+|\kappa_2|} f(x)^{\kappa_1} g(x)^{\kappa_2}$ and $|\mathrm{IF}_2(y)| \leq L C^{2L+|\kappa_1|+|\kappa_2|} f(y)^{\kappa_1+1} g(y)^{\kappa_2-1}$. Under our assumptions we can therefore obtain bounds on $\mathbb{E}\{ |\mathrm{IF}_1(X_1)|^p\}$ and $\mathbb{E}\{ |\mathrm{IF}_2(Y_1)|^p\}$ for $p=3,4$.  This is helpful for the application of the central limit theorem of \citet{Baldi1989}.

For two random variables $X$ and $Y$ with distribution functions $F$ and $G$ (where for later convenience we allow $X$ and $Y$ to take values in the extended real line), let
\[
	d_{\mathrm{K}}\bigl(\mathcal{L}(X),\mathcal{L}(Y)\bigr):= \sup_{t \in \mathbb{R}} | F(t) - G(t) |
\]
denote the Kolmogorov distance between the distributions of $X$ and $Y$.
\begin{thm}
\label{Thm:Normality}
Suppose that the conditions of Theorem~\ref{Thm:Main} hold. 
If $(k_X^\mathrm{U})^4 \log^8 m = o(m)$ and $(k_Y^\mathrm{U})^4 \log^8 n = o(n)$, then
\[
	 \sup_{\phi \in \Phi(\xi)} \sup_{(f,g) \in \widetilde{\mathcal{F}}_{d,\vartheta}} \max_{\substack{k_X \in \{k_X^{\mathrm{L}},\ldots, k_X^{\mathrm{U}}\} \\ k_Y \in \{k_Y^{\mathrm{L}},\ldots, k_Y^{\mathrm{U}}\}}} d_\mathrm{K} \biggl( \mathcal{L}\biggl(\frac{\hat{T}_{m,n} - T}{\{v_1/m + v_2/n \}^{1/2}}\biggr), N(0,1) \biggr) \rightarrow 0
        \]
        as $m,n \rightarrow \infty$.
      \end{thm}

The proof of Theorem~\ref{Thm:Normality} relies on a Poissonisation argument.  By this, we mean that we initially consider the related problem where instead of observing samples $X_1,\ldots,X_m$ and $Y_1,\ldots,Y_n$ of fixed size, we first sample $M \sim \mathrm{Poi}(m)$ and $N \sim \mathrm{Poi}(n)$, and, conditional on~$M$ and $N$, observe two independent samples $X_1,\ldots,X_M \stackrel{\mathrm{iid}}{\sim} f$ and $Y_1,\ldots,Y_N \stackrel{\mathrm{iid}}{\sim} g$.  The main reason for doing this is because in this model, appropriately truncated nearest neighbour distances of $X_i$ and $X_j$ are independent provided that $X_i$ and $X_j$ are sufficiently far apart.  One of the key ideas of the proof is the observation that, after Poissonisation and nearest neighbour distance truncation, we can construct a careful partition of $\mathbb{R}^d$ into Voronoi cells, such that the probability content of each cell is roughly the same and decays with the sample size, and yet each cell has only a small number of other cells that are close to it (Proposition~\ref{Prop:Partition}).  By decomposing our estimator into contributions from each cell of the partition, we therefore obtain a sum of terms with a sparse dependency graph, which enables us to apply the central limit theorem of \citet{Baldi1989}.

Another key aspect of the proof of Theorem~\ref{Thm:Normality} is an approximation of our unweighted nearest neighbour functional estimators by a sum of two terms, each of which only depends on one of the samples.  To describe this decomposition, we write $\rho_{(k),i,\ell}$ for the $k$th nearest neighbour distance of $X_i$ among the sample $X_1,\ldots,X_\ell$ whenever $\ell \geq \max(k+1,i)$. We will also write $\rho_{(k),\ell}(x)$ for the $k$th nearest neighbour distance of $x$ among the sample $Y_1,\ldots,Y_\ell$ whenever $\ell \geq k$.  Now define the random variables
\begin{align}
\label{Eq:SemiOracle}
	T_m^{(1)} &:= \frac{1}{m} \sum_{i=1}^m \phi \biggl( \frac{k_X}{mV_d \rho_{(k_X),i,m}^d}, g(X_i) \biggr) \nonumber \\
	T_n^{(2)} &:= \int_{\mathcal{X}} f(x) \phi\biggl( f(x), \frac{k_Y}{n V_d \rho_{(k_Y),n}(x)^d} \biggr) \,dx
\end{align}
We can think of $T_m^{(1)}$ and $T_n^{(2)}$ as semi-oracle estimators, where in the first case the sample size $n$ from density $g$ is infinite, and in the second case, the sample size $m$ from density $f$ is infinite.  In particular, the crucial point is that $T_m^{(1)}$ depends only on $X_1,\ldots,X_m$ and $T_n^{(2)}$ depends only on $Y_1,\ldots,Y_n$.  In fact, our proof reveals the interesting observation that under our conditions,
      \[
        \widetilde{T}_{m,n} - \mathbb{E}(\widetilde{T}_{m,n}) = T_m^{(1)} - \mathbb{E}(T_m^{(1)}) + T_n^{(2)} - \mathbb{E}(T_n^{(2)}) + o_p(m^{-1/2} + n^{-1/2}).
      \]
The main advantage of this decomposition is that it allows us to establish the asymptotic normality of $\widetilde{T}_{m,n}$ by considering $T_m^{(1)}$ and $T_n^{(2)}$ separately.  A further benefit is that it facilitates control of the Poissonisation error more easily than would otherwise be the case, as we now explain.  Let $M \sim \mathrm{Poi}(m)$ and $N \sim \mathrm{Poi}(n)$ be independent (and independent of the data), and, when $M \geq (k_X+1) \log (em)$ and $N \geq k_Y \log(en)$, define
\begin{align*}
	T_m^{(1),\mathrm{p}} &:= \frac{1}{m} \sum_{i=1}^M \phi \biggl( \frac{k_X}{m V_d \rho_{(k_X),i,M}^d}, g(X_i) \biggr) - \Bigl( \frac{M}{m} -1 \Bigr) \int_\mathcal{X} f(x)\{ \phi_x +(f \phi_{10})_x \} \, dx \\
	T_n^{(2),\mathrm{p}} &:=\int_\mathcal{X} f(x) \phi \biggl( f(x), \frac{k_Y}{n V_d \rho_{(k_Y),N}(x)^d} \biggr) \, dx - \Bigl( \frac{N}{n} -1 \Bigr) \int_\mathcal{X} f(x) (g \phi_{01})_x \,dx.
\end{align*}
If $M <(k_X+1) \log (em)$, say $T_m^{(1),\mathrm{p}}:=0$, and similarly if $N <k_Y \log (en)$, say $T_n^{(2),\mathrm{p}}:=0$. The following result bounds the mean squared difference of these approximations.
\begin{prop}
\label{Prop:Poisson}
Assume that the conditions of Theorem~\ref{Thm:Main} hold and additionally assume that $k_X^\mathrm{U}=o(m^{1/4})$ and $k_Y^\mathrm{U}=o(n^{1/6})$. Then
\[
\sup_{\phi \in \Phi(\xi)}  \sup_{(f,g) \in \mathcal{F}_{d,\vartheta}}  \max_{ k_X \in \{k_X^{\mathrm{L}},\ldots, k_X^{\mathrm{U}}\}} \mathbb{E} \bigl\{ (T_m^{(1)} - T_m^{(1),\mathrm{p}})^2 \bigr\} = o(1/m)
\]
and
\[
\sup_{\phi \in \Phi(\xi)}  \sup_{(f,g) \in \mathcal{F}_{d,\vartheta}}  \max_{ k_Y \in \{k_Y^{\mathrm{L}},\ldots, k_Y^{\mathrm{U}}\}} \mathbb{E} \bigl\{ (T_n^{(2)} - T_n^{(2),\mathrm{p}})^2 \bigr\} = o(1/n).
\]
\end{prop}      

Theorem~\ref{Thm:Normality} also facilitates the construction of asymptotically valid confidence intervals of asymptotically minimal width, provided we can find consistent estimators of $v_1$ and $v_2$.  To describe our methodology here, it is convenient to introduce the shorthand 
\begin{equation}
  \label{Eq:fhatghat}
	\hat{f}_{(k_X),i}:=\frac{k_X}{mV_d \rho_{(k_X),i,X}^d} \quad \text{and} \quad \hat{g}_{(k_Y),i}:=\frac{k_Y}{nV_d \rho_{(k_Y),i,Y}^d}
\end{equation}
for $i \in \{1,\ldots,m\}, k_X \in \{1,\ldots,m-1\}$ and $k_Y \in \{1,\ldots,n\}$.  Further, define
\begin{align*}
	\hat{V}_{m,n}^{(1),1} &:= \frac{1}{m} \sum_{i=1}^m \min \Bigl[ \bigl\{ \phi \bigl( \hat{f}_{(k_X),i} ,\hat{g}_{(k_Y),i} \bigr) +\hat{f}_{(k_X),i} \phi_{10} \bigl(\hat{f}_{(k_X),i},\hat{g}_{(k_Y),i} \bigr) \bigr\}^2, \log m , \log n \Bigr] \\
	\hat{V}_{m,n}^{(1),2} &:= \widetilde{T}_{m,n} + \frac{1}{m} \sum_{i=1}^m \hat{f}_{(k_X),i} \phi_{10} \bigl(\hat{f}_{(k_X),i},\hat{g}_{(k_Y),i} \bigr) \\
	\hat{V}_{m,n}^{(2),1} &:= \frac{1}{m} \sum_{i=1}^m \min \Bigl\{ \hat{f}_{(k_X),i}\hat{g}_{(k_Y),i} \phi_{01} \bigl( \hat{f}_{(k_X),i},\hat{g}_{(k_Y),i} \bigr)^2, \log m, \log n \Bigr\} \\
	\hat{V}_{m,n}^{(2),2} &:= \frac{1}{m} \sum_{i=1}^m \hat{g}_{(k_Y),i} \phi_{01} \bigl( \hat{f}_{(k_X),i},\hat{g}_{(k_Y),i} \bigr),
\end{align*}
as well as $\hat{V}_{m,n}^{(1)}:= \max\{\hat{V}_{m,n}^{(1),1} - ( \hat{V}_{m,n}^{(1),2} )^2,0\}$ and $\hat{V}_{m,n}^{(2)}:= \max\{\hat{V}_{m,n}^{(2),1} - ( \hat{V}_{m,n}^{(2),2} )^2,0\}$.  It turns out that $\hat{V}_{m,n}^{(1)}$ and $\hat{V}_{m,n}^{(2)}$ satisfy the consistency property that we seek, so, writing $z_q$ for the $(1-q)$th quantile of the standard normal distribution, $\hat{v}_{m,n}:= \hat{V}_{m,n}^{(1)}/m + \hat{V}_{m,n}^{(2)}/n$ and
\[
	I_{m,n,q}:= \bigl[ \hat{T}_{m,n} - z_{q/2} \hat{v}_{m,n}^{1/2} \ , \ \hat{T}_{m,n} + z_{q/2} \hat{v}_{m,n}^{1/2} \bigr],
\]
we have the following result.
\begin{thm}
\label{Thm:ConfidenceIntervals}
Suppose that the conditions of Theorem~\ref{Thm:Normality} hold. Then
\[
	\sup_{\phi \in \Phi(\xi)} \sup_{(f,g) \in \widetilde{\mathcal{F}}_{d,\vartheta}}  \max_{\substack{k_X \in \{k_X^{\mathrm{L}},\ldots, k_X^{\mathrm{U}}\} \\ k_Y \in \{k_Y^{\mathrm{L}},\ldots, k_Y^{\mathrm{U}}\}}} d_\mathrm{K} \biggl( \mathcal{L}\biggl(\frac{\hat{T}_{m,n} - T}{\hat{v}_{m,n}^{1/2}}\biggr), N(0,1) \biggr) \rightarrow 0
        \]
        as $m,n \rightarrow \infty$. In particular, 
\[
	\sup_{q \in (0,1)}  \sup_{\phi \in \Phi(\xi)} \sup_{(f,g) \in \widetilde{\mathcal{F}}_{d,\vartheta}}  \max_{\substack{k_X \in \{k_X^{\mathrm{L}},\ldots, k_X^{\mathrm{U}}\} \\ k_Y \in \{k_Y^{\mathrm{L}},\ldots, k_Y^{\mathrm{U}}\}}} \Bigl| \mathbb{P} \bigl( I_{m,n,q} \ni T(f,g) \bigr) - (1-q) \biggr| \rightarrow 0
\]
as $m,n \rightarrow \infty$.
\end{thm}

\section{Bias}
\label{Sec:Bias}

\subsection{Bias of the naive estimator}
Here we state a result on the bias of the estimator~\eqref{Eq:Unweighted}.  It is in fact an immediate consequence of a more general statement, given as Proposition~\ref{Thm:GeneralBias}, which considers a wider range of choices of $k_X$ and $k_Y$.
\begin{prop}
  \label{Thm:SimpleBias}
  \sloppy{Fix $d \in \mathbb{N}$, $\vartheta = (\alpha,\beta,\lambda_1,\lambda_2,C) \in \Theta$ and $\xi = (\kappa_1,\kappa_2,\beta^*,L) \in \Xi$.  Assume that $\zeta<1/2$ and that $\min(\tau_1,\tau_2)>1/\beta^*$. Suppose further that $\min(k_X^{\mathrm{L}} m^{-1/\beta^*}, k_Y^{\mathrm{L}} n^{-1/\beta^*} ) \rightarrow \infty$ and that there exists $\epsilon>0$ with $\max(k_X^{\mathrm{U}} m^{-\tau_1+\epsilon}, k_Y^{\mathrm{U}} n^{-\tau_2 + \epsilon}) \rightarrow 0$.  Then for each $i_1,i_2 \in \bigl[\lceil d/2 \rceil -1 \bigr]$ and $j_1,j_2 \in \mathbb{N}_0$ such that $j_1+j_2 \leq \lceil (\beta^*-1)/2 \rceil$, we can find coefficients $\lambda_{i_1i_2j_1j_2} \equiv \lambda_{i_1i_2j_1j_2}(d,f,g,\phi)$, with the properties that $\lambda_{0,0,0,0}=T(f,g)$, that}
\[
   \sup_{\phi \in \Phi(\xi)} \sup_{(f,g) \in \mathcal{F}_{d,\vartheta}}  |\lambda_{i_1i_2j_1j_2}| < \infty,
\]
and that
\begin{align*}
  \Biggl| \mathbb{E}_{f,g}(\widetilde{T}_{m,n}) - \sum_{i_1,i_2=0}^{\lceil d/2 \rceil-1} \sum_{j_1,j_2=0}^{\infty} \mathbbm{1}_{\{j_1+j_2 \leq \lceil (\beta^*-1)/2 \rceil\}} \frac{\lambda_{i_1i_2j_1j_2}}{k_X^{j_1} k_Y^{j_2}} \Bigl( \frac{k_X}{m} \Bigr)^{\frac{2i_1}{d}} &\Bigl( \frac{k_Y}{n} \Bigr)^{\frac{2i_2}{d}}  \Biggr| \\
  &= o(m^{-1/2}+n^{-1/2})
\end{align*}
as $m,n \rightarrow \infty$, uniformly for $\phi \in \Phi(\xi), (f,g) \in \mathcal{F}_{d,\vartheta}, k_X \in \{k_X^{\mathrm{L}},\ldots, k_X^{\mathrm{U}} \}$ and $k_Y \in \{k_Y^{\mathrm{L}}, \ldots, k_Y^{\mathrm{U}}\}$.
\end{prop}
Proposition~\ref{Thm:SimpleBias} provides conditions on the classes of densities and functionals under which we can give a uniform asymptotic expansion of the bias of $\widetilde{T}_{m,n}$, up to terms of negligible order.  This expansion also holds uniformly over a range of values of $k_X$ and $k_Y$, which can be chosen adaptively (i.e.~without knowledge of the parameters of the underlying densities) to satisfy the conditions of the theorem, e.g.~by setting $k_X = m^{1/\beta^*} \log m$ and $k_Y = n^{1/\beta^*} \log n$.

As revealed by Corollary~\ref{Cor:WeightedBias} below, Proposition~\ref{Thm:SimpleBias} allows us to form weighted versions of the estimators $\widetilde{T}_{m,n,k_X,k_Y}$, for different choices of $k_X$ and $k_Y$, so as to cancel the dominant terms in the expression for the bias of the naive estimator.  Indeed, it was this result that motivated our choice of the class of weights that we consider in Theorem~\ref{Thm:Main}.
\begin{cor}
\label{Cor:WeightedBias}
Suppose that the conditions of Proposition~\ref{Thm:SimpleBias} hold.  Then for each $c \in (0,1)$, each $w_X = w_{X}^{(k_X)} \in \mathcal{W}_{\lceil (\beta^*-1)/2 \rceil,c}^{(k_X)}$ and each $w_Y = w_{Y}^{(k_Y)} \in \mathcal{W}_{\lceil (\beta^*-1)/2 \rceil,c}^{(k_Y)}$, we have
\[
	\sup_{\phi \in \Phi(\xi)} \sup_{(f,g) \in \mathcal{F}_{d,\vartheta}}   \sup_{\substack{k_X \in \{k_X^{\mathrm{L}},\ldots, k_X^{\mathrm{U}}\}\\k_Y \in \{k_Y^{\mathrm{L}}, \ldots, k_Y^{\mathrm{U}}\}}}\Bigl| \mathbb{E}_{f,g}(\hat{T}_{m,n}^{w_X,w_Y}) - T(f,g) \Bigr| = o(m^{-1/2}+n^{-1/2})
      \]
      as $m,n \rightarrow \infty$.
\end{cor}
In order to gain intuition about the level of smoothness of the functional required in Corollary~\ref{Cor:WeightedBias}, it is helpful to consider the following (favourable) case: if our assumptions hold for all $\alpha,\beta,\lambda_2 > 0$ and all $\lambda_1 < 1$, then it suffices that $\kappa_1>-1/2$ and that $\beta^* > \max\bigl\{2,1+d/4, \frac{2(1-\kappa_1^-)}{1-2\kappa_1^-} \bigr\}$.

The key idea of our bias proofs is a truncation argument that partitions $\mathcal{X}$ as $\mathcal{X}_{m,n} \cup (\mathcal{X} \setminus \mathcal{X}_{m,n})$, where
\[
\mathcal{X}_{m,n} := \biggl\{x \in \mathcal{X}: \frac{f(x)}{M_\beta(x)^d} \geq \frac{k_X \log m}{m},\frac{g(x)}{M_\beta(x)^d} \geq \frac{k_Y \log n}{n}\biggr\}.
\]
Further, by Lemma~\ref{Lemma:15over7}, we have that $f$ and $g$ are uniformly well-approximated in a relative sense, over balls of an appropriate radius, by their values at the centres of these balls; more precisely, for every $\vartheta \in \Theta$, and writing $A := (16d)^{1/(\beta -\underline{\beta})}$ and $r_0(x) := 1/\bigl\{AM_\beta(x)\bigr\}$,
\[
\sup_{(f,g) \in \mathcal{F}_{d,\vartheta}} \sup_{y \in B_x(r_0(x))} \biggl|\frac{f(y)}{f(x)} - 1\biggr| \bigvee \biggl|\frac{g(y)}{g(x)} - 1\biggr| \leq \frac{1}{2}.
\]
In particular, this means that
\begin{equation}
  \label{Eq:hxfhxg}
\inf_{x \in \mathcal{X}_{m,n}} h_{x,f}\bigl(r_0(x)\bigr) \geq \frac{V_dk_X\log m}{2A^dm} \quad \text{and} \quad \inf_{x \in \mathcal{X}_{m,n}} h_{x,g}\bigl(r_0(x)\bigr) \geq \frac{V_dk_Y\log n}{2A^dn}
\end{equation}
whenever $(f,g) \in \mathcal{F}_{d,\vartheta}$.  Thus for each $x \in \mathcal{X}_{m,n}$, it is the case that with high probability, the $k_X$ nearest neighbours of $x$ among $X_1,\ldots,X_m$, as well as the $k_Y$ nearest neighbours of $x$ among $Y_1,\ldots,Y_n$, lie in $B_x\bigl(r_0(x)\bigr)$.  Moreover, the functions $h_{x,f}(\cdot)$ and $h_{x,g}(\cdot)$ can be approximated by Taylor expansions on $[0,r_0(x)]$, which yield corresponding expansions for their respective inverses.  Since $h_{X_i,f}(\rho_{(k),i,X})|X_i \sim \mathrm{Beta}(k,m-k)$ and $h_{X_i,g}(\rho_{(k),i,Y})|X_i \sim \mathrm{Beta}(k,n+1-k)$, these facts, in combination with~\eqref{Eq:hxfhxg}, allow us to deduce a stochastic expansion for $\rho_{(k),i,X}$ and $\rho_{(k),i,Y}$ in terms of powers of the relevant beta random variables.  The contribution to the bias from the region $\mathcal{X}_{m,n}$ can then be computed by a Taylor expansion of $\phi$ and using exact formulae for moments of beta random variables.  For $x \in \mathcal{X} \setminus \mathcal{X}_{m,n}$, we have no guarantees about the proximity of the $k_X$ nearest neighbours of $x$ among $X_1,\ldots,X_m$, nor the $k_Y$ nearest neighbours of $x$ among $Y_1,\ldots,Y_n$; however,
\[
\mathbb{P}(X_1 \in \mathcal{X} \setminus \mathcal{X}_{m,n}) \leq C\biggl\{\Bigl(\frac{k_X \log m}{m}\Bigr)^{\lambda_1} \bigvee \Bigl(\frac{k_Y \log n}{n}\Bigr)^{\lambda_2}\biggr\}, 
\]
so the integrability conditions in our classes $\mathcal{F}_{d,\vartheta}$ allow us to control the contribution to the bias from this region.

\subsection{Tighter control of the bias when \texorpdfstring{$\beta \leq 1$}{beta <= 1}}
\label{Sec:ImprovedBiasSec}

Our general bias result in Proposition~\ref{Thm:GeneralBias} has remainder terms of the order $(k_X/m)^{\beta/d}$ and $(k_Y/n)^{\beta/d}$ in the expansion, and leads naturally to the condition $\beta>d/2$ for efficiency.  A requirement of this level of smoothness for a parametric rate of convergence (albeit with smoothness measured in different ways) also appears in several other related works on functional estimation, including \citet{leonenko2010statistical}, \citet{Kandasamy15} and \citet{SinghPoczos2016}. However, other results show that $d/4$ smoothness (often in the case $d=1$ or while also requiring this smoothness to be at most $1$) may suffice for certain functionals without singularities \citep{BickelRitov1988,BirgeMassart1995,Laurent1996,gine2008simple,leonenko2010statistical}.  The purpose of Proposition~\ref{Prop:ImprovedBias} below, then, is to demonstrate that when $\beta \in (0,1]$, it is possible to tighten our bias bounds to have terms of the order $(k_X/m)^{2\beta/d}$ and $(k_Y/n)^{2\beta/d}$, so that we only require $\beta > d/4$ for efficiency.
\begin{prop}
\label{Prop:ImprovedBias}
Fix $d \in \mathbb{N}$, $\vartheta = (\alpha,\beta,\lambda_1,\lambda_2,C) \in \Theta$ with $\beta \in (0,1]$ and $\xi = (\kappa_1,\kappa_2,\beta_1^*,\beta_2^*,L) \in \Xi$. Let $k_X^{\mathrm{L}} \leq k_X^{\mathrm{U}},k_Y^{\mathrm{L}} \leq k_Y^{\mathrm{U}}$ be deterministic sequences of positive integers such that $k_X^{\mathrm{L}} / \log m \rightarrow \infty$, $k_Y^{\mathrm{L}}/ \log n \rightarrow \infty$, $k_X^{\mathrm{U}} = O(m^{1-\epsilon})$ and $k_Y^{\mathrm{U}} = O(n^{1-\epsilon})$ for some $\epsilon>0$. Suppose that $\zeta<1$. Then for each $j_1 \in \bigl[\lceil (\beta^*-1)/2 \rceil\bigr]$ and $j_2 \in \bigl[\lceil (\beta^*-1)/2 \rceil\bigr]$, we can find $\lambda_{j_1j_2} \equiv \lambda_{j_1j_2}(d,f,g,\phi)$, with the properties that $\lambda_{0,0}=T(f,g)$,
\[
   \sup_{\phi \in \Phi(\xi)} \sup_{(f,g) \in \mathcal{F}_{d,\vartheta}}  |\lambda_{j_1j_2}| < \infty,
\]
and that, for every $\epsilon > 0$,
\begin{align}
\label{Eq:ImprovedBiasError}
	\sup_{\phi \in \Phi(\xi)}  \sup_{(f,g) \in \mathcal{F}_{d,\vartheta}}  & \Biggl| \mathbb{E}_{f,g}(\widetilde{T}_{m,n}) - \sum_{j_1,j_2=0}^{\infty} \mathbbm{1}_{\{j_1+j_2 \leq \lceil (\beta^*-1)/2 \rceil\}} \frac{\lambda_{j_1j_2}}{k_X^{j_1} k_Y^{j_2}} \Biggr| \nonumber \\
	&\hspace{20pt} = O \biggl( \max \biggl\{ k_X^{-\beta^*/2 }, \Bigl( \frac{k_X}{m} \Bigr)^{2\beta/d}, \Bigl( \frac{k_X}{m} \Bigr)^{ \lambda_1(1-\zeta) -\epsilon}, k_Y^{-\beta^*/2 },  \nonumber \\
  &\hspace{90pt} \Bigl( \frac{k_Y}{n} \Bigr)^{2\beta/d}, \Bigl( \frac{k_Y}{n} \Bigr)^{ \lambda_2(1-\zeta) -\epsilon}, 1/m, 1/n \biggr\} \biggr),
\end{align}
as $m,n \rightarrow \infty$, uniformly for $k_X \in \{k_X^{\mathrm{L}},\ldots, k_X^{\mathrm{U}}\}$ and $k_Y \in \{k_Y^{\mathrm{L}}, \ldots, k_Y^{\mathrm{U}}\}$.
\end{prop}

The proof of Proposition~\ref{Prop:ImprovedBias} is given in Section~\ref{Sec:ImprovedBias}.  The interest in the result arises because it reveals that the bias of nearest-neighbour functional estimators is of smaller order than that of the corresponding density estimators, at least when $\beta \leq 1$ and when the function~$\phi$ is smooth away from its singularities.  This reduced bias is due to the fact that the nearest-neighbour density estimate biases at different values of $x \in \mathcal{X}$ cancel to leading order when we integrate over $\mathcal{X}$.  While similar phenomena have been observed for kernel-based density estimates in the context of the estimation of quadratic functionals \citep{gine2008simple,leonenko2010statistical}, we are not aware of corresponding results for nearest-neighbour methods or non-smooth functionals.

An immediate corollary of Proposition~\ref{Prop:ImprovedBias} is that the conclusions of Theorems~\ref{Thm:Main} and~\ref{Thm:Normality} hold with the $d/(2\beta)$ term in the definitions of $\tau_1$ and $\tau_2$ in~\eqref{Eq:Zeta} replaced with $d/(4\beta)$, provided $\beta \leq 1$.  In particular, in this case it suffices to have $\beta > d/4$ in Examples~\ref{Ex:KL},~\ref{Ex:Renyi} and~\ref{Ex:L2}.


\subsection{Bias of an alternative debiased estimator}

As mentioned in the introduction, building on the original debiasing idea of \citet{KozachenkoLeonenko1987}, \citet{Ganguly18} proposed a debiasing technique for the naive estimator $\widetilde{T}_{m,n}$ of a general two-sample functional.  The initial goal of this subsection is to use fractional calculus techniques to give an informal study of the remaining bias of these resulting estimators, with a view to addressing the question of whether to apply our weighting scheme to the naive estimator~\eqref{Eq:Unweighted} or that of \citet{Ganguly18}.

For simplicity we will focus on the one-sample setting in~\eqref{Eq:OneSampleFunctional}, though all of the calculations have analogues in the two-sample setting.  Suppose that there exists a sequence of differentiable functions $(\psi_k)$ for which
\begin{equation}
\label{Eq:OneSampleLaplaceTransform}
	\psi(u) = \int_0^\infty e^{-s} \frac{s^{k-1}}{\Gamma(k)} \psi_k \Bigl( \frac{ku}{s} \Bigr) \,ds
\end{equation}
for all $u \in (0,\infty)$; examples in the cases of Shannon and R\'enyi entropies will be given below. We will consider the debiased estimator of $H(f)$ given by
\[
	\widetilde{H}_m := \frac{1}{m} \sum_{i=1}^m \psi_k \bigl( \hat{f}_{(k),i} \bigr).
\]
Write $\mathcal{X}_f := \{x:f(x) > 0\}$.  Then, under regularity conditions on $f$ and $\psi_k$, since $m\mathrm{Beta}(k,m-k)$ can be approximated by a $\Gamma(k,1)$ random variable, we have that
\begin{align}
  \label{Eq:DominantBias}
	&\mathbb{E} \widetilde{H}_m = \int_{\mathcal{X}_f} f(x) \int_0^1 \psi_k \Bigl( \frac{k}{mV_d h_{x,f}^{-1}(s)^d} \Bigr) \mathrm{B}_{k,m-k}(s) \,ds \,dx \nonumber \\
	& \! \approx \! \int_{\mathcal{X}_f} \! \! \! f(x) \! \int_0^1 \! \biggl\{ \! \psi_k \Bigl( \frac{kf(x)}{ms} \Bigr) \! - \! \frac{V_df(x) h_{x,f}^{-1}(s)^d -s }{ms^2 /\{kf(x)\}} \psi_k' \Bigl( \frac{kf(x)}{ms} \Bigr) \biggr\} \mathrm{B}_{k,m-k}(s) \,ds \,dx \nonumber \\
	& \! \approx \! \int_{\mathcal{X}_f} \! \! \! f(x) \! \int_0^\infty \! \biggl\{ \! \psi_k \Bigl( \frac{kf(x)}{t} \Bigr) \! + \! \frac{k t^{\frac{2}{d}-1} \Delta f(x) }{2(d+2)\{V_d n f(x)\}^{\frac{2}{d}}} \psi_k' \Bigl( \frac{kf(x)}{t} \Bigr) \biggr\} \frac{e^{-t}t^{k-1}}{\Gamma(k)} \,dt \,dx \nonumber \\
	& \!= H(f) + \frac{1}{2(d+2)(V_dn)^{\frac{2}{d}}} \int_{\mathcal{X}_f} \frac{\Delta f(x)}{f(x)^{\frac{2}{d}-1}} \int_0^\infty \frac{e^{-t} t^{k+2/d-2}}{\Gamma(k-1)} \psi_k' \Bigl( \frac{kf(x)}{t} \Bigr)  \,dt \,dx.
\end{align}
In order to understand the behaviour of the dominant bias term on the right-hand side of~\eqref{Eq:DominantBias}, for $\alpha \in [0,1)$ define the operator $D^\alpha$ by
\[
	(D^\alpha g)(u) := -\frac{1}{\Gamma(1-\alpha)} \int_u^\infty \frac{g'(s)}{(s-u)^\alpha} \,ds.
\]
This is closely related to the Caputo fractional derivative \citep[][Section~2.4]{KST2006}. Then, with $g(s)=e^{-\lambda s}$ for some $\lambda \in (0,\infty)$, we have that
\[
	(D^\alpha g)(u) = \frac{1}{\Gamma(1-\alpha)} \int_u ^\infty \frac{\lambda e^{-\lambda s}}{(s-u)^{\alpha}} \,ds = \lambda^\alpha e^{-\lambda u} = \lambda^\alpha g(u).
\]
From~\eqref{Eq:OneSampleLaplaceTransform} we can see that
\begin{equation}
  \label{Eq:FracCalc}
	\frac{\Gamma(k-1)}{u^{k-1}} \psi'(u) = u^{-(k-1)} \int_0^\infty e^{-t} t^{k-2} \psi_k' \Bigl(\frac{ku}{t} \Bigr) \,dt = \int_0^\infty e^{-su} s^{k-2} \psi_k' \Bigl( \frac{k}{s} \Bigr) \,ds.
\end{equation}
When $d \geq 3$, we can apply the operator $D^{2/d}$ to both sides of~\eqref{Eq:FracCalc} to simplify the inner integral in our expression for the dominant bias term in~\eqref{Eq:DominantBias} as follows: 
\begin{align}
  \label{Eq:FracBias}
	&\frac{1}{\Gamma(k-1)} \int_0^\infty  e^{-t} t^{k+\frac{2}{d}-2} \psi_k' \Bigl( \frac{ku}{t} \Bigr)  \,dt = \frac{u^{k+2/d-1}}{\Gamma(k-1)} \int_0^\infty e^{-su} s^{k+\frac{2}{d}-2} \psi_k' \Bigl( \frac{k}{s} \Bigr) \,ds \nonumber \\
	& \hspace{8pt}= - \frac{u^{k+2/d-1}}{\Gamma(1-2/d)} \int_u^\infty \frac{\frac{d}{ds} (\psi'(s)/s^{k-1} )}{(s-u)^{2/d}} \,ds \nonumber \\
	& \hspace{8pt}= \frac{u^{k+2/d-1}}{\Gamma(1-2/d)} \int_u^\infty \frac{(k-1) s^{-k} \psi'(s) - s^{1-k} \psi''(s)}{(s-u)^{2/d}} \,ds \nonumber \\
	& \hspace{8pt}= \frac{\Gamma(k+2/d-1)}{\Gamma(k-1)} \int_0^1 \mathrm{B}_{1-2/d,k+2/d-1}(s) \Bigl\{ \psi' \Bigl( \frac{u}{1-s} \Bigr)- \frac{u}{(k-1)(1-s)} \psi'' \Bigl( \frac{u}{1-s} \Bigr)\Bigr\} \,ds.
\end{align}
For Shannon and R\'enyi entropies, both $\psi'$ and $\psi''$ are constant multiples of functions $g$ with the property that $g(xy)=g(x)g(y)$ for any $x,y \in (0,\infty)$. In these cases, the leading order bias separates into a coefficient depending only on $d$, $n$ and $f$ and a factor that is a function of $k$. Using weights, this leading order bias may be cancelled out, and it can be seen that, when $f$ is sufficiently regular, the next term is of order $k^{4/d}/n^{4/d}$. However, the only continuous functions $g$ with this property are $g(x)=x^a$ for some $a \in \mathbb{R}$ \citep[e.g.][(4.3.7), p.~86]{Dieudonne69}. If the term in braces in~\eqref{Eq:FracBias} is separable for all values of $k$ then both $u \mapsto \psi'(u)$ and $u \mapsto u\psi''(u)$ must be separable individually, and so $\psi'(u) \propto u^a$ for some $a \in \mathbb{R}$.  Thus the Shannon and R\'enyi entropies are the only functionals with this property. In general, all that can be said is that this term in the bias can be expanded as a series of the form $\frac{k^{2/d}}{n^{2/d}}(c_0+c_1/k+c_2/k^{2} +\ldots)$. For larger values of $d$, to cancel out sufficient bias that the resulting estimator is efficient, the weighting scheme is then only marginally simpler than the weighting scheme for the naive estimator, and the analysis is significantly more complicated.

Despite the general conclusion of our discussion in the previous paragraph, returning to the two-sample functional setting, we now show that in the special case of the Kullback--Leibler and R\'enyi divergence functionals, the debiasing scheme described above significantly simplifies the weighting scheme, while facilitating the same conclusions regarding efficiency.  To this end, for the Kullback--Leibler divergence, we define the following class of weight vectors:
\begin{align*}
\mathcal{W}_{c}^{(k), \mathrm{KL}}:= \biggl\{ w = (w_1,\ldots&,w_k) \in \mathbb{R}^{k} : \sum_{j=1}^{k} w_{j} =1 \text{ and }  w_{j}=0 \text{ for } j<ck, \|w\|_1 \leq 1/c,\nonumber \\
                                                                                       &\sum_{j=1}^k \frac{\Gamma(j+2 \ell/d)}{\Gamma(j)} w_j = 0 \text{ for } \ell \in \bigl[\lceil d/2 \rceil -1 \bigr] \setminus \{0\}  \biggr\}.
\end{align*}                                                                                        
The analogue of the Kozachenko--Leonenko debiased estimator is
\begin{align*}
	\widetilde{D}_{m,n} &:= \frac{1}{m} \sum_{i=1}^m \log \biggl( \frac{e^{\Psi(k_X)}}{m\rho_{(k_X),i,X}^d} \frac{n \rho_{(k_Y),i,Y}^d}{e^{\Psi(k_Y)}} \biggr) = \widetilde{T}_{m,n} + \Psi(k_X)- \log k_X -\Psi(k_Y) + \log k_Y
\end{align*}
\citep{Ganguly18}.  If the weighted estimator $\hat{D}_{m,n}^{w_X,w_Y}$ is then formed as in~\eqref{Eq:WeightedEstimator} then the following theorem elucidates its asymptotic bias.  Since this result uses very similar (in fact, somewhat simpler) arguments to those in~Proposition~S1 in the online supplement \citep{BerrettSamworth2018}, its proof, together with that of Proposition~\ref{Thm:RenyiBias} below, is omitted for brevity.
\begin{prop}
\label{Thm:KLBias}
Fix $d \in \mathbb{N}$, let $\vartheta = (\alpha,\beta,\lambda_1,\lambda_2,C) \in \Theta$ and let $\phi(u,v)=\log(u/v)$.  Assume that 
\[
	\tau_1 = 1- \max \biggl( \frac{d}{2 \beta}, \frac{1}{2 \lambda_1}  \biggr) > 0 \quad \text{and} \quad \tau_2 = 1- \max \biggl( \frac{d}{2 \beta}, \frac{1}{2 \lambda_2} \biggr) > 0,
\]
and that there exists $\epsilon>0$ such that $\max(k_{X}^{\mathrm{U}} m^{-\tau_1+\epsilon}, k_{Y}^{\mathrm{U}} n^{-\tau_2 + \epsilon}) \rightarrow 0$.  Then for each $c \in (0,1)$, each $w_X = w_{X}^{(k_X)} \in \mathcal{W}_{c}^{(k_X),\mathrm{KL}}$ and each $w_Y = w_{Y}^{(k_Y)} \in \mathcal{W}_{c}^{(k_Y),\mathrm{KL}}$, we have
\[
	\sup_{(f,g) \in \mathcal{F}_{d,\vartheta}} \sup_{\substack{k_X \in \{1,\ldots, k_X^{\mathrm{U}}\}\\k_Y \in \{1, \ldots, k_Y^{\mathrm{U}}\}}}\Bigl| \mathbb{E}_{f,g}(\hat{D}_{m,n}^{w_X,w_Y}) - T(f,g) \Bigr| = o(m^{-1/2}+n^{-1/2})
      \]
      as $m,n \rightarrow \infty$.
\end{prop}
Since $\widetilde{D}_{m,n}$ is simply a deterministic translation of $\widetilde{T}_{m,n}$, our variance results in Section~\ref{Sec:Variance} continue to hold, so the corresponding efficiency result for $\hat{D}_{m,n}^{w_X,w_Y}$ is immediate.

When estimating the R\'enyi integral $\int_{\mathcal{X}} f^\kappa g^{-(\kappa-1)}$, for $b \in \mathbb{R}$ and $c > 0$, we define
\begin{align*}
  \mathcal{W}_{b,c}^{(k),\mathrm{R}}:= \biggl\{ w = (w_1,&\ldots,w_k) \in \mathbb{R}^{k} : \sum_{j=1}^{k} w_{j} =1 \text{ and }  w_{j}=0 \text{ for } j<ck, \|w\|_1 \leq 1/c,\nonumber \\
	&\sum_{j=1}^k \frac{\Gamma(j-b+2\ell/d)}{\Gamma(j-b)} w_j = 0 \text{ for } \ell \in \bigl[\lceil d/2 \rceil -1 \bigr] \setminus \{0\} \biggr\}.
\end{align*}
The corresponding debiased estimator is
\begin{align*}
  \check{D}_{m,n} &:= \frac{1}{m} \sum_{i=1}^m \frac{\Gamma(k_X) \Gamma(k_Y)}{\Gamma(k_X - \kappa+1) \Gamma(k_Y + \kappa-1)} \biggl( \frac{n \rho_{(k_Y),i,Y}^d}{m \rho_{(k_X),i,X}^d} \biggr)^{\kappa-1} \\
  &\phantom{:}= \frac{k_X^{1-\kappa} \Gamma(k_X) k_Y^{\kappa-1} \Gamma(k_Y)}{\Gamma(k_X - \kappa+1) \Gamma(k_Y + \kappa-1)} \widetilde{T}_{m,n}
\end{align*}
\citep{Ganguly18}.  If the weighted estimator $\hat{D}_{m,n}^{w_X,w_Y}$ is again formed as in~\eqref{Eq:WeightedEstimator} then the following result provides the corresponding bias guarantee.
\begin{prop}
\label{Thm:RenyiBias}
Fix $d \in \mathbb{N}$, let $\vartheta = (\alpha,\beta,\lambda_1,\lambda_2,C) \in \Theta$ and let $\phi(u,v)=(u/v)^{\kappa-1}$ for some $\kappa \in (1/2,\infty)$.  With $\zeta$ as defined as in~\eqref{Eq:Zeta}, $ \kappa_1=- \kappa_2 =\kappa-1 $,
\[
	\tau_1 = 1- \max \biggl( \frac{d}{2 \beta}, \frac{1}{2\lambda_1(1- \zeta)} \biggr) \quad  \! \text{and} \! \quad \tau_2 = 1- \max \biggl( \frac{d}{2 \beta}, \frac{1}{2 \lambda_2(1-\zeta)} \biggr),
\]
assume that $\zeta<1/2$ and $\min(\tau_1,\tau_2)>0$. Suppose further that there exists $\epsilon>0$ such that $\max(k_{X}^{\mathrm{U}} m^{-\tau_1+\epsilon}, k_{Y}^{\mathrm{U}} n^{-\tau_2 + \epsilon}) \rightarrow 0$.  Then for each $c \in (0,1)$, each $w_X = w_{X}^{(k_X)} \in \mathcal{W}_{\kappa-1,c}^{(k_X),\mathrm{R}}$ and each $w_Y = w_{Y}^{(k_Y)} \in \mathcal{W}_{1-\kappa,c}^{(k_Y),\mathrm{R}}$, we have
\[
	\sup_{(f,g) \in \mathcal{F}_{d,\vartheta}} \sup_{\substack{k_X \in \{1,\ldots, k_X^{\mathrm{U}}\}\\k_Y \in \{1, \ldots, k_Y^{\mathrm{U}}\}}}\Bigl| \mathbb{E}_{f,g}(\hat{D}_{m,n}^{w_X,w_Y}) - T(f,g) \Bigr| = o(m^{-1/2}+n^{-1/2})
      \]
      as $m,n \rightarrow \infty$.
\end{prop}
In this case, with $k_X^{\mathrm{L}}$ and $k_Y^{\mathrm{L}}$ defined as in Theorem~\ref{Thm:Main}, we have
\[
\frac{\check{D}_{m,n}}{\widetilde{T}_{m,n}} - 1=  \frac{  k_X^{1-\kappa} \Gamma(k_X) k_Y^{\kappa-1} \Gamma(k_Y) }{\Gamma(k_X - \kappa+1) \Gamma(k_Y + \kappa-1)} -1  \rightarrow 0
\]
uniformly for $k_X \geq k_X^\mathrm{L}$ and $k_Y \geq k_Y^\mathrm{L}$, so we can again deduce an efficiency result for $\hat{D}_{m,n}^{w_X,w_Y}$.

\section{Variance}
\label{Sec:Variance}

The following result provides the main asymptotic variance expansion for our weighted estimators. Write $\tau_i'=1-\max\{ \frac{d}{d+2(2\wedge\beta)}, \frac{1}{2(\lambda_i \wedge 1)(1-\zeta)} \}$ for $i=1,2$.
\begin{prop}
\label{Prop:Variance}
Fix $d \in \mathbb{N}$, $\vartheta = (\alpha,\beta,\lambda_1,\lambda_2,C) \in \Theta$ and $\xi = (\kappa_1,\kappa_2,\beta^*,L) \in \Xi$ such that $ \zeta<1/2, \tau_1'>0, \tau_2' >0$.  Let $(k_{X}^{\mathrm{L}})$, $(k_{Y}^{\mathrm{L}})$, $(k_{X}^{\mathrm{U}})$ and $(k_{Y}^{\mathrm{U}})$ be deterministic sequences of positive integers satisfying $\min(k_{X}^{\mathrm{L}}/ \log^5 m, k_{Y}^{\mathrm{L}}/ \log^5 n) \rightarrow \infty$ and $\max(k_{X}^{\mathrm{U}} m^{-(\tau_1'-\epsilon)}, k_{Y}^{\mathrm{U}} n^{-(\tau_2' - \epsilon)}) \rightarrow 0$ for some $\epsilon>0$.  Then for each $c \in (0,1)$, each $w_X = w_{X}^{(k_X)} \in \mathcal{W}_{\lceil (\beta^*-1)/2 \rceil,c}^{(k_X)}$ and each $w_Y = w_{Y}^{(k_Y)} \in \mathcal{W}_{\lceil (\beta^*-1)/2 \rceil,c}^{(k_Y)}$, we have
      \[
      \sup_{\phi \in \Phi(\xi)}  \sup_{(f,g) \in \mathcal{F}_{d,\vartheta}}  \max_{\substack{k_X \in \{k_X^{\mathrm{L}},\ldots, k_X^{\mathrm{U}}\} \\ k_Y \in \{k_Y^{\mathrm{L}},\ldots, k_Y^{\mathrm{U}}\}}}\biggl| \mathrm{Var}_{f,g}(\hat{T}_{m,n}^{w_X,w_Y}) - \frac{v_1}{m} - \frac{v_2}{n}\biggr| = o\biggl(\frac{1}{m} + \frac{1}{n}\biggr)
      \]
as $m,n \rightarrow \infty$.
\end{prop}

The proof of Proposition~\ref{Prop:Variance} is significantly more complicated that those of the bias proofs in Section~\ref{Sec:Bias}, primarily owing to the need to consider the joint distribution of nearest neighbour distances around two different points, $X_1$ and $X_2$, say.  These have an intricate dependence structure because, for instance, $X_1$ may be one of the five nearest neighbours of $X_2$, but not vice-versa.  To describe our main strategy for approximating $\mathrm{Var}_{f,g}(\hat{T}_{m,n}^{w_X,w_Y})$, we write $\hat{T}_{m,n}^{w_X,w_Y} =: m^{-1}\sum_{i=1}^m \hat{T}_{m,n}^{(i)}$ as shorthand, so that 
\begin{equation}
  \label{Eq:VarianceDecomp}
  \mathrm{Var}_{f,g}(\hat{T}_{m,n}^{w_X,w_Y}) = \frac{1}{m}\mathrm{Var}_{f,g}(\hat{T}_{m,n}^{(1)}) + \frac{m-1}{m} \mathrm{Cov}_{f,g}(\hat{T}_{m,n}^{(1)},\hat{T}_{m,n}^{(2)}).
\end{equation}
Using similar techniques to those employed in Section~\ref{Sec:Bias}, it can be shown that
\[
  \mathrm{Var}_{f,g}(\hat{T}_{m,n}^{(1)}) \rightarrow \mathrm{Var}_f \, \phi_{X_1}.
\]
For the covariance term in~\eqref{Eq:VarianceDecomp}, we first condition on $X_1$ and $X_2$.  It turns out that this term can be further decomposed into a sum of two terms, representing the contributions from the events on which $X_1$ and $X_2$ either share or do not share nearest neighbours.  Observe that if
\begin{align*}
  \|X_1 - X_2\| > \biggl\{\frac{k_X}{mV_d}\biggl(1 + \frac{\log^{1/2} m}{k_X^{1/2}}\biggr)\biggr\}^{1/d} \bigl\{f(X_1)^{-1/d} + f(X_2)^{-1/d}\bigr\} =: R(X_1,X_2),
\end{align*}
say, then, with high probability, $X_1$ and $X_2$ do not share any of their $k_X$ nearest neighbours among $X_3,\ldots,X_m$.  This means that the random vector $\bigl(h_{X_1}(\rho_{(k_X),1,X}),h_{X_2}(\rho_{(k_X),2,X}),1 - h_{X_1}(\rho_{(k_X),1,X}) - h_{X_2}(\rho_{(k_X),2,X})\bigr)$ has approximately the same distribution as $(Z_1,Z_2,Z_3)$, say, where $(Z_1,Z_2,Z_3) \sim \mathrm{Dirichlet}(k_X,k_X,m-2k_X-1)$.  Writing $\|\cdot\|_{\mathrm{TV}}$ for the total variation norm on signed measures, we can then exploit the facts that
\[
  \bigl\|\mathcal{L}(Z_1,Z_2) - \mathrm{Beta}(k_X,m-k_X) \otimes \mathrm{Beta}(k_X,m-k_X)\bigr\|_{\mathrm{TV}} = O(k_X/m)
\]
and
\begin{equation}
  \label{Eq:kXratio}
  \frac{\hat{f}_{(k_X),1}}{f(X_1)} = 1 + O_p(k_X^{-1/2})
  \end{equation}
  to show that the contribution to the covariance from this region is $O(1/m)$ (where in fact we also determine the leading constant).  On the other hand,
  \[
    \mathbb{P}\bigl[\{\|X_1 - X_2\| \leq R(X_1,X_2)\} \cap \{X_1 \in \mathcal{X}_{m,n}\}\bigr] = O(k_X/m),
  \]
  and this, together with~\eqref{Eq:kXratio} again, allows us to demonstrate that the contribution to the covariance from this region due to the nearest neighbour distances among $X_3,\ldots,X_m$ is also $O(1/m)$ (with a different leading constant).  The terms arising from the nearest neighbour distances of $Y_1,\ldots,Y_n$ from $X_1$ and $X_2$ can be handled similarly, and their contributions can be shown to be $O(1/n)$.  Combining these dominant terms results in the expansion
  \begin{align*}
    \mathrm{Cov}_{f,g}(\hat{T}_{m,n}^{(1)},\hat{T}_{m,n}^{(2)}) = \frac{2}{m} \mathrm{Cov}_{f}\bigl(\phi_{X_1},(f\phi_{10})_{X_1}\bigr) + \frac{1}{m}\mathrm{Var}_{f}\bigl((f\phi_{10})_{X_1}\bigr) + \frac{v_2}{n} + o\biggl(\frac{1}{m}\!+\!\frac{1}{n}\biggr),
  \end{align*}
  and the conclusion follows.
  
  \section{The super-oracle phenomenon}
\label{Sec:SuperOracle}
  
In this section, we consider an alternative estimation problem, where we are still interested in the functional $T(f,g)$ in~\eqref{Eq:Functional}, but where instead of observing data $X_1,\ldots,X_m,Y_1,\ldots,Y_n$ as before, we instead observe $f(X_1),\ldots,f(X_m),g(X_1),\ldots,g(X_m)$.  Although this latter framework should be considered as an `oracle' version of the problem, because typically $f(X_1),\ldots,f(X_m)$ and $g(X_1),\ldots,g(X_m)$ are unknown, it is nevertheless instructive to compare the performance of our efficient estimator $\hat{T}_{m,n}$ with that of the estimator
  \[
    T_{m}^* := \frac{1}{m}\sum_{i=1}^m \phi\bigl(f(X_i),g(X_i) \bigr)
  \]
in the new problem.  The estimator $T_m^*$ is unbiased, and moreover, $m^{1/2}(T_{m}^* - T) \stackrel{d}{\rightarrow} N(0,\sigma^2)$, where $\sigma^2 = \sigma^2(f,g) := \mathrm{Var}_f \, \phi\bigl(f(X_1),g(X_1)\bigr)$.  In fact, as we now show, $T_m^*$ can be the optimal estimator, in a local asymptotic minimax sense, of $T$ in our oracle problem.  Our aim here is not to seek maximal generality, but instead to give a simple class of examples for which $T_m^*$ has this optimality property.  

For simplicity of exposition, we will focus on the one-sample functional~\eqref{Eq:OneSampleFunctional} with $\psi(u)=u^{-(1-\kappa)}$ for some $\kappa \in (1/2,1)$.  Thus, we consider estimation of the R\'enyi functional
\[
	H(f) = \int_0^\infty f(x) \psi \bigl( f(x) \bigr) \,dx = \int_0^\infty f(x)^\kappa \,dx,
\]
based on the observations $f(X_1),\ldots,f(X_m)$. Moreover, we take $\mathcal{X} = [0,\infty)$, and assume that $f(x) = e^{-P(x)}$ for some convex, strictly increasing polynomial $P:[0,\infty) \rightarrow \mathbb{R}$.  Define the function $h: [0,\infty) \rightarrow \mathbb{R}$ by
\begin{equation}
\label{Eq:h}
	h(x) := \frac{f'(x)}{f(x)} \int_0^x \bigl\{ \psi \bigl( f(y) \bigr) - H(f) \bigr\} \,dy = -P'(x) \int_0^x \{ f(y)^{-(1-\kappa)} - H(f) \} \,dy.
\end{equation}
Now, for $t \in [0,\infty)$, define $f_t : [0,\infty) \rightarrow \mathbb{R}$ by
\[
	f_t(x) := \{1-th(x) \} f(x);
\]
in the proof of Proposition~\ref{Thm:SuperOracleLAM} below, we will see that $f_t$ is a bounded probability density for sufficiently small $t \geq 0$.  Moreover $f_0 = f$, and we will see that $\{f_t:t \in [0,\infty)\}$ constitutes a least favourable sub-model in this problem.  

Recall that $(H_m)$ is called an \emph{estimator sequence} if $H_m: \mathbb{R}^{m\times d} \rightarrow \mathbb{R}$ is a measurable function for each $m \in \mathbb{N}$.  We are now in a position to state a local asymptotic minimax lower bound that reveals the optimality of the one-sample version of $T_m^*$ in this context. 
\begin{prop}
\label{Thm:SuperOracleLAM}
Writing $\mathcal{I}$ for the set of all finite subsets of $[0,\infty)$, for any estimator sequence $(\tilde{H}_m)$ we have that
\[
	\sup_{I \in \mathcal{I}} \liminf_{m \rightarrow \infty} \max_{t \in I} m\mathbb{E}_{f_{t/m^{1/2}}} \bigl[ \bigl\{H_m - H(f_{t/m^{1/2}}) \bigr\}^2 \bigr] \geq \mathrm{Var}_f \, \psi \bigl( f(X_1) \bigr). 
\]
Moreover, fixing $\alpha, \beta > 0$ and $\lambda \in (0,1)$, there exist $t_0 > 0$, depending only on $\kappa \in (1/2,1)$ and $f$ as defined above, and $C = C(\alpha,\beta,\lambda,\kappa,f) > 0$ such that $f_t \in \mathcal{G}_{1,\theta}$ for $t \in [0,t_0]$, where $\theta = (\alpha,\beta,\lambda,C)$.
\end{prop}
Specialising the estimator $T_m^*$ to this one-sample problem, we see that $T_m^*$ is efficient in the sense of \citet[][Chapter~25]{vanderVaart1997}, and hence optimal in this local asymptotic minimax sense.

  The following result, which is an immediate consequence of Theorem~\ref{Thm:Main}, compares the asymptotic worst-case squared error risks of $\widehat{T}_{m,n}$ (in the original problem with data $X_1,\ldots,X_m,Y_1,\ldots,Y_n$) and $T_m^*$ (in the oracle problem with data $f(X_1),\ldots,f(X_m)$ and $g(X_1),\ldots,g(X_m)$).  We first define a slight modification of the class $\mathcal{F}_{d,\vartheta}$, by setting
  \begin{equation}
  \label{Eq:Fstar}
    \mathcal{F}_{d,\vartheta}^* := \bigl\{(f,g) \in \mathcal{F}_{d,\vartheta} : \min(v_1,v_2) \geq 1/C\bigr\}.
  \end{equation}
  \begin{thm}
  \label{Thm:SuperOracle}
    Assume the conditions of Theorem~\ref{Thm:Main}.  Then
      \[
        \sup_{\phi \in \Phi(\xi)} \sup_{(f,g) \in \mathcal{F}_{d,\vartheta}^*} \max_{\substack{k_X \in \{k_X^{\mathrm{L}},\ldots, k_X^{\mathrm{U}}\} \\ k_Y \in \{k_Y^{\mathrm{L}},\ldots, k_Y^{\mathrm{U}}\}}} \frac{\mathbb{E}_{f,g}\bigl\{(\hat{T}_{m,n} - T)^2\bigr\}}{\mathbb{E}_{f}\bigl\{(T_{m}^* - T)^2\bigr\}} \cdot \frac{\sigma^2/m}{v_1/m + v_2/n} \rightarrow 1
      \]
as $m,n \rightarrow \infty$.
  \end{thm}
  To understand the implications of this theorem, consider the case where $n$ is at least of the same order as $m$, so that $A := \limsup_{n \rightarrow \infty} m/n \in [0,\infty)$.  If $\sigma^2/(v_1 + Av_2) > 1$, then the worst-case risk of $\hat{T}_{m,n}$ is asymptotically better than that of $T_m^*$, and we have an illustration of the super-oracle phenomenon.  The one-sample functional~\eqref{Eq:OneSampleFunctional} corresponds to $A = 0$, and the arguments above reveal that for the R\'enyi-type functional $\int_{\mathbb{R}^d} f(x)^\kappa \, dx$ with $\kappa \in (1/2,1)$, the efficient variance in the original problem is strictly smaller than that in the oracle problem since $\sigma^2 \equiv \sigma^2(f) = \mathrm{Var}_f \bigl(f(X_1)^{\kappa-1}\bigr)$ and $v_1 = \kappa^2 \sigma^2$ (note that $\sup_{f \in \mathcal{F}_{d,\vartheta}^*} \sigma^2(f) < \infty$ whenever $\lambda_1 > 2 - 2\kappa$).  In general, the phenomenon occurs if and only if
  \[
    2 \mathrm{Cov}_f\bigl(\phi_{X_1},(f\phi_{10})_{X_1}\bigr) < - \mathrm{Var}_f(f\phi_{10})_{X_1} - Av_2.
    \]
One of the surprising aspects of the super-oracle phenomenon is the fact that the estimator $\hat{T}_{m,n}$ is constructed so as to mimic $T_m^*$, by estimating $f(X_1),\ldots,f(X_m)$ and $g(X_1),\ldots,g(X_m)$, but can in some cases outperform $T_m^*$ itself.  
  
    \section{A local asymptotic minimax lower bound}
    \label{Sec:LAM}

    Before we can state our local asymptotic minimax result we require some further assumptions on the function $\phi$.  
For $\xi = (\kappa_1,\kappa_2,\beta^*,L) \in \Xi$ let $\tilde{\Phi}(\xi)$ denote the subset of $\Phi(\xi)$ consisting of those $\phi$ for which
\begin{enumerate}[(i)]
\item for all $\mathbf{z}=(u,v) \in \mathcal{Z}$ and $\ell_1 \in [\beta^*]$ we have
\[
	\max_{\ell_2 \in [\beta^*-\ell_1]} \frac{ u^{\ell_1} v^{\ell_2} | \phi_{\ell_1 \ell_2} ( \mathbf{z} ) | }{| \phi(\mathbf{z}) + u \phi_{10} ( \mathbf{z})| \vee 1} \bigvee \max_{ \ell_2 \in [\beta^*-\ell_1] \setminus \{0\}} \frac{ u^{\ell_1+1} v^{\ell_2-1} | \phi_{\ell_1 \ell_2} ( \mathbf{z} ) | }{ \bigl(u | \phi_{01}(\mathbf{z})| \bigr) \vee 1}  \leq L;
\]
	\item for all $\boldsymbol{\epsilon}=(\epsilon_1,\epsilon_2) \in (-1/L,1/L)^2$, $\mathbf{z}=(u,v) \in \mathcal{Z}$, and $\ell_1, \ell_2 \in \mathbb{N}_0$ with $\ell_1+\ell_2 \leq \beta^*-1$, we have
\begin{align*}
  &\frac{u^{\ell_1} v^{\ell_2} \bigl|\phi_{\ell_1 \ell_2}( \mathbf{z} \! + \! \boldsymbol{\epsilon}) \! -\! \phi_{\ell_1 \ell_2}( \mathbf{z}) \bigr|}{ | \phi(\mathbf{z}) + u \phi_{10} ( \mathbf{z})| \vee 1 } \leq L \Bigl( \Bigl|\frac{\epsilon_1}{u_1}\Bigr|^{(\beta^*-\ell_1) \wedge 1} + \Bigl|\frac{\epsilon_2}{u_2}\Bigr|^{(\beta^*-\ell_2) \wedge 1} \Bigr); \\
  &\frac{u^{\ell_1+1} v^{\ell_2-1} \bigl|\phi_{\ell_1 \ell_2}( \mathbf{z} \!+\! \boldsymbol{\epsilon}) \!-\! \phi_{\ell_1 \ell_2}( \mathbf{z}) \bigr|}{\bigl( u | \phi_{01}(\mathbf{z})| \bigr) \vee 1 } \leq L \Bigl( \Bigl|\frac{\epsilon_1}{u_1}\Bigr|^{(\beta^*-\ell_1) \wedge 1} \! \! + \Bigl|\frac{\epsilon_2}{u_2}\Bigr|^{(\beta^*-\ell_2) \wedge 1} \Bigr) \ \ \text{when $\ell_2 \geq 1$}.
\end{align*}
\end{enumerate}
To understand these conditions it is instructive to consider the case of $\varphi-$divergences, for which $\phi(u,v)=\varphi(v/u)$ for some function $\varphi$. Here, (i) reduces to requiring that 
\[
	\sup_{w>0} \biggl\{\max_{\ell \in [\beta^*]} \frac{ w^\ell | \varphi^{(\ell)} (w) | }{| \varphi(w) -w\varphi'(w)| \vee 1}, \max_{ \ell \in [\beta^*] \setminus \{0\}} \frac{ w^{\ell-1} | \varphi^{(\ell)}(w) | }{| \varphi'(w) | \vee 1} \biggr\} < \infty,
\]
and a similar reduction holds for (ii). This is satisfied for the Kullback--Leibler divergence and all R\'enyi divergences.  Moreover, when $\phi(u,v) = v$, we have $\phi \in \tilde{\Phi}(0,0,\beta^*,1+1/\beta^*)$ for every $\beta^* > 0$.

Now fix $(f,g) \in \mathcal{F}_d^2$ and $\phi: \mathcal{Z} \rightarrow \mathbb{R}$ and define the functions
\begin{align*}
	h_1(x) &:= \phi_x + (f\phi_{10})_x - \mathbb{E} \{\phi_{X_1} + (f\phi_{10})_{X_1} \} \\
	h_2(x) &:= (f\phi_{01})_x - \mathbb{E}\{ (f \phi_{01})_{Y_1} \}.
\end{align*}
This enables us to define, for each $t=(t_1,t_2) \in \mathbb{R}^2$, the densities
\[
	f_{t_1}(x) := c_1(t_1) K\bigl(t_1h_1(x)\bigr) f(x)  \quad \text{and} \quad g_{t_2}(x) := c_2(t_2) K\bigl(t_2 h_2(x)\bigr) g(x),
\]
where $K(t) := 1/2 + 1/(1+e^{-4t})$ and $c_1( \cdot),c_2(\cdot)$ are normalising constants.  Our choice of $K$ is made so that $K(0)=K'(0)=1$, that $K$ is smooth, and that $K$ is bounded above and below by positive constants. Now, for each $t=(t_1,t_2) \in \mathbb{R}^2$ we define the sequence of probability measures $(P_{n,t})$ on $\mathbb{R}^{(m+n) \times d}$ so that $P_{n,t}$ has density $f_{m^{-1/2} t_1}^{ \otimes m} \otimes g_{n^{-1/2} t_2}^{ \otimes n}$ (here we think of $m$ as a function of $n$).  It turns out that the family $\{P_{n,t}:t \in \mathbb{R}^2\}$ constitutes a least favourable parametric sub-model for this estimation problem.  For an arbitrary probability measure $P$ on $\mathbb{R}^{(m+n)\times d}$, we write $\mathbb{E}_P$ to denote expectation over $(X_1,\ldots,X_m,Y_1,\ldots,Y_n)^T \sim P$. 

We can now state our local asymptotic minimax lower bound, and the consequent optimality property of our estimators $\hat{T}_{m,n}$.
\begin{thm}
\label{Thm:LAM}
Fix $d \in \mathbb{N}$, $\vartheta=(\alpha,\beta,\lambda_1,\lambda_2,C) \in \Theta$ and $\xi = (\kappa_1,\kappa_2,\beta^*,L) \in \Xi$. Let $m = m_n$ be any sequence of positive integers such that $m \rightarrow \infty$ and $m / n \rightarrow A$ for some $A \in [0,\infty]$, let $(f,g) \in \mathcal{F}_{d,\vartheta}$, let $\phi \in \tilde{\Phi}(\xi)$ and let $\mathcal{I}$ denote the set of finite subsets of $\mathbb{R}^2$.
\begin{enumerate}[(i)]
\item For any estimator sequence $(T_{m,n})$, we have that
\[
	\sup_{I \in \mathcal{I}} \liminf_{n \rightarrow \infty} \max_{t = (t_1,t_2) \in I} n\mathbb{E}_{P_{n,t}} \Bigl[ \bigl\{T_{m,n} - T(f_{m^{-1/2}t_1},g_{n^{-1/2}t_2})\bigr\}^2\Bigr] \geq \frac{1}{A}v_1(f,g) + v_2(f,g) .
\]
\item There exists $t_0=t_0(d,\vartheta,\xi) \in (0,1]$ such that, for any $t_1,t_2 \in (-t_0,t_0)$, we have $(f_{t_1},g_{t_2}) \in \mathcal{F}_{d,\tilde{\vartheta}}$, where $\tilde{\vartheta}=(\alpha,\tilde{\beta}, \lambda_1,\lambda_2, C/t_0)$ and $\tilde{\beta}:=\min\{\beta, (1 \wedge \beta)( \beta^* -1)\}$.  In particular, when the conditions of Theorem~\ref{Thm:Main} hold and $\tilde{\beta} = \beta$, the estimators $\hat{T}_{m,n}$ in~\eqref{Eq:WeightedEstimator} satisfy
\[
\sup_{I \in \mathcal{I}} \limsup_{n \rightarrow \infty} \max_{t = (t_1,t_2) \in I} n\mathbb{E}_{P_{n,t}} \Bigl[ \bigl\{\hat{T}_{m,n} - T(f_{m^{-1/2}t_1},g_{n^{-1/2}t_2})\bigr\}^2\Bigr] = \frac{1}{A}v_1(f,g) + v_2(f,g).
\]
\end{enumerate}
\end{thm}
Recall, for example, that for both the Kullback--Leibler divergence and all R\'enyi-type divergences, we can take $\beta^*$ large enough that $\tilde{\beta} = \beta$.  In these and other cases for which the conditions hold, then, the local asymptotic minimax bounds in Theorem~\ref{Thm:LAM} justify the claim that suitably chosen versions of our weighted nearest neighbour estimator~\eqref{Eq:WeightedEstimator} are efficient over these classes of densities and functionals.  

We conclude with a few extensions of Theorem~\ref{Thm:LAM}.  The condition $\tilde{\beta} = \beta$ can be weakened to $\tilde{\beta} > d/2$ (or in fact $\tilde{\beta} > d/4$ when $d \in \{1,2,3\}$), at the expense of slightly stronger conditions on the tuning parameters in the definition of $\hat{T}_{m,n}$.  Theorem~\ref{Thm:LAM}(i) implies a (non-local) minimax lower bound over the classes $\mathcal{F}_{d,\vartheta}^* \subseteq \mathcal{F}_{d,\vartheta}$ from~\eqref{Eq:Fstar}, and this matches the upper bound in Theorem~\ref{Thm:Main} over $\mathcal{F}_{d,\vartheta}$.  Theorem~\ref{Thm:LAM}(i) may also be extended to broader classes of loss functions, namely those that have closed, convex, symmetric sub-level sets; see \citet[][Theorem~3.11.5]{vanderVaartWellner1996} for details.   Finally, Theorem~\ref{Thm:Normality} allows us to extend Theorem~\ref{Thm:Main}, and consequently Theorem~\ref{Thm:LAM}(ii), to $L_q$-losses with $q \in (0,2)$.  The combination of these results implies that our estimators are asymptotically optimal in a local asymptotic minimax sense for these $L_q$ losses too; we omit formal statements for brevity.

\medskip

\noindent \textbf{Acknowledgements}: The authors are very grateful to the anonymous reviewers for their constructive comments, which helped to improve the paper. 

\begin{funding}
The first author was supported by Engineering and Physical Sciences Reseach Council (EPSRC) New Investigator Award EP/W016117/1.

 The second author was supported in part by EPSRC Programme grant EP/N031938/1, EPSRC Fellowship EP/P031447/1 and European Research Council Advanced grant 101019498.
\end{funding}

\begin{supplement}
\stitle{Supplementary material for `Efficient functional estimation and the super-oracle phenomenon'}
\sdescription{Proofs of results from the main text and auxiliary results.}
\end{supplement}

\newpage

\title{Supplementary material to `Efficient functional estimation and the super-oracle phenomenon'}

This is the supplementary material to \citet{BS2019Main}, hereafter referred to as the main text.

\medskip

Throughout our proofs we will use the notation
\[
	u_{x,s}:= \frac{k_X}{mV_d h_{x,f}^{-1}(s)^d} \quad \text{and} \quad v_{x,t}:= \frac{k_Y}{nV_d h_{x,g}^{-1}(t)^d}.
      \]
      for $x \in \mathcal{X}$ and $s,t \in (0,1)$.  Moreover, since many of our error terms will depend on $k_X$, $k_Y$, $f$, $g$ and $\phi$ (as well as $q$, in Theorem~\ref{Thm:ConfidenceIntervals}), we adopt the convention, without further comment, that all of these error bounds hold uniformly over the relevant sets as claimed in the statements of the results.  In addition, when we write $a \lesssim b$, we mean that there exists $C > 0$, depending only on the parameters $d, \vartheta$ and $\xi$ of the problem, such that $a \leq Cb$.  It will be convenient throughout to assume that $m,n \geq 3$.

\subsection{Proof  of Proposition~\ref{Prop:Beta}}

\begin{proof}[Proof of Proposition~\ref{Prop:Beta}]
  First, we have that $\mu_\alpha(f) \leq 1$ and $\|f\|_\infty \leq C_{d,a,b}$, and it remains to bound the function $M_{f,\beta}( \cdot)$ for each $\beta > 0$.  Writing $g(r) := C_{d,a,b}r^{a-1}(1-r)^{b-1}\mathbbm{1}_{\{r \leq 1\}}$, so that $f(x)=g(\|x\|)$ we may see by induction that
\[
	\sup_{r \in (0,2/3)} r^{-(a-t-1)} |g^{(t)}(r)| < \infty \quad \text{and} \quad \sup_{r \in (1/3,1)} r^{-(b-t-1)} |g^{(t)}(r)| < \infty
\]
for any $t \in \mathbb{N}$. Moreover, for any $t \in \mathbb{N}$ and multi-index $\boldsymbol{\alpha} = (\alpha_1,\ldots,\alpha_d) \in \mathbb{N}_0^d$ with $|\boldsymbol{\alpha}| = t$, we have that
\[
	\sup_{x \in B_0(2/3)} \|x\|^{t-1} \bigl|\partial^{\boldsymbol{\alpha}} \|x\| \bigr| < \infty \quad \text{and} \quad \sup_{x \in B_0(1) \setminus B_0(1/3)} \bigl|\partial^{\boldsymbol{\alpha}} \|x\| \bigr| < \infty.
\]
Using these facts we have that 
\[
	\bigl| \partial^{\boldsymbol{\alpha}}f(x)\bigr| \lesssim \frac{f(x)}{\|x\|^t (1-\|x\|)^t}
\]
for any $t \in \mathbb{N}$.  Now, writing $\underline{\beta} := \lceil \beta \rceil - 1$ and fixing $\boldsymbol{\alpha} = (\alpha_1,\ldots,\alpha_d) \in \mathbb{N}_0^d$ with $|\boldsymbol{\alpha}| = \underline{\beta}$, if $y,z \in B_x\bigl( \|x\|(1-\|x\|)/8\bigr)$ then we have for any some $w$ on the line segment between $x$ and $y$ that
\begin{align*}
	\bigl|\partial^{\boldsymbol{\alpha}} f(z) - \partial^{\boldsymbol{\alpha}} f(y)\bigr| &\leq d^{1/2} \|z-y\|\vertiii{f^{(\underline{\beta}+1)}(w)} \\
	& \lesssim \|z-y\| \| w \|^{a-1-(\underline{\beta}+1)} (1-\|w\|)^{b-1-(\underline{\beta}+1)} \\
	& \lesssim f(y) \|z-y\|^{\beta-\underline{\beta}} \frac{\|z-y\|^{\underline{\beta}+1-\beta}}{ \|x\|^{\underline{\beta}+1}(1-\|x\|)^{\underline{\beta}+1}} \lesssim \frac{ f(y) \|z-y\|^{\beta-\underline{\beta}} }{ \|x\|^\beta (1-\|x\|)^\beta}.
\end{align*}
It follows that $M_{f,\beta}(x) \lesssim 1/\{\|x\|(1-\|x\|)\}$. Therefore, for any $\lambda \in \bigl(0,b/(b+d-1)\bigr)$, we have
\begin{align*}
	\int_{B_0(1)} f(x) \biggl\{ & \frac{M_{f,\beta}(x)^d}{f(x)} \biggr\}^\lambda \,dx \\
	&\lesssim \int_{B_0(1)} \|x\|^{a-1}(1-\|x\|)^{b-1} \biggl\{ \frac{1}{\|x\|^{a+d-1}(1-\|x\|)^{b+d-1}} \biggr\}^\lambda \,dx \\
	& = dV_d \int_0^1 r^{a+d-2 - \lambda(a+d-1)} (1-r)^{b-1 - \lambda(b+d-1)} \,dr < \infty,
\end{align*}
as claimed.
\end{proof}

\subsection{Proof of Proposition~\ref{Thm:SimpleBias} on asymptotic bias}

The following general result on the bias of the naive estimator $\widetilde{T}_{m,n}$ yields Proposition~\ref{Thm:SimpleBias} as an immediate consequence.
\begin{prop}
\label{Thm:GeneralBias}
Fix $d \in \mathbb{N}$, $\vartheta = (\alpha,\beta,\lambda_1,\lambda_2,C) \in \Theta$ and $\xi = (\kappa_1,\kappa_2,\beta^*,L) \in \Xi$. Let $k_X^{\mathrm{L}} \leq k_X^{\mathrm{U}},k_Y^{\mathrm{L}} \leq k_Y^{\mathrm{U}}$ be deterministic sequences of positive integers such that $k_X^{\mathrm{L}} / \log m \rightarrow \infty$, $k_Y^{\mathrm{L}}/ \log n \rightarrow \infty$, $k_X^{\mathrm{U}} = O(m^{1-\epsilon})$ and $k_Y^{\mathrm{U}} = O(n^{1-\epsilon})$ for some $\epsilon>0$. Suppose that $\zeta<1$. Then for each $i_1,i_2 \in \bigl[\lceil d/2 \rceil -1\bigr]$ and $j_1, j_2 \in \mathbb{N}_0$ with $j_1+j_2 \leq \lceil (\beta^*-1)/2 \rceil$, we can find $\lambda_{i_1i_2j_1j_2} \equiv \lambda_{i_1i_2j_1j_2}(d,f,g,\phi)$, with the properties that $\lambda_{0,0,0,0}=T(f,g)$,
\[
   \sup_{\phi \in \Phi(\xi)} \sup_{(f,g) \in \mathcal{F}_{d,\vartheta}}  |\lambda_{i_1i_2j_1j_2}| < \infty,
\]
and that, for every $\epsilon > 0$,
\begin{align*}
	\sup_{\phi \in \Phi(\xi)}  \sup_{(f,g) \in \mathcal{F}_{d,\vartheta}}  & \Biggl| \mathbb{E}_{f,g}(\widetilde{T}_{m,n}) - \! \sum_{i_1,i_2=0}^{\lceil d/2 \rceil-1} \sum_{j_1,j_2=0}^\infty \! \mathbbm{1}_{\{j_1+j_2 \leq \lceil (\beta^*-1)/2 \rceil\}} \frac{\lambda_{i_1i_2j_1j_2}}{k_X^{j_1} k_Y^{j_2}} \Bigl( \frac{k_X}{m} \Bigr)^{\frac{2i_1}{d}} \Bigl( \frac{k_Y}{n} \Bigr)^{\frac{2i_2}{d}} \Biggr| \\
	&\hspace{20pt} = O \biggl( \max \biggl\{ k_X^{-\beta^*/2 }, \Bigl( \frac{k_X}{m} \Bigr)^{\frac{2 \wedge \beta}{d} \beta^*}, \Bigl( \frac{k_X}{m} \Bigr)^{\beta/d}, \Bigl( \frac{k_X}{m} \Bigr)^{ \lambda_1(1-\zeta) -\epsilon}, k_Y^{-\beta^*/2 },  \\
  &\hspace{90pt}\Bigl( \frac{k_Y}{n} \Bigr)^{\frac{2 \wedge \beta}{d} \beta^*} , \Bigl( \frac{k_Y}{n} \Bigr)^{\beta/d}, \Bigl( \frac{k_Y}{n} \Bigr)^{ \lambda_2(1-\zeta) -\epsilon}, 1/m, 1/n \biggr\} \biggr),
\end{align*}
as $m,n \rightarrow \infty$, uniformly for $k_X \in \{k_X^{\mathrm{L}},\ldots, k_X^{\mathrm{U}}\}$ and $k_Y \in \{k_Y^{\mathrm{L}}, \ldots, k_Y^{\mathrm{U}}\}$.
\end{prop}
\begin{proof}
Define
\[
	a_{m,X}^\pm:= 0 \vee \frac{k_X}{m} \biggl(1 \pm \frac{3\log^{1/2}m}{k_X^{1/2}}\biggr) \wedge 1, \quad a_{n,Y}^\pm:= 0 \vee \frac{k_Y}{n} \biggl( 1 \pm \frac{3\log^{1/2}n}{k_Y^{1/2}}\biggr) \wedge 1,
\]
let $\mathcal{I}_{m,X} := [a_{m,X}^-,a_{m,X}^+]$, $\mathcal{I}_{n,Y} := [a_{n,Y}^-,a_{n,Y}^+]$, and set
\[
	 \mathcal{X}_{m,n}:=\biggl\{x \in \mathcal{X} : \frac{f(x)}{M_{\beta}(x)^{d}} \geq \frac{k_X \log m}{m}, \frac{g(x)}{M_\beta(x)^{d}} \geq \frac{k_Y \log n}{n} \biggr\}.
\]
To begin our bias calculation, we recall the definitions of $\hat{f}_{(k_X),i}$ and $\hat{g}_{(k_Y),i}$ from~\eqref{Eq:fhatghat}.  Observe that, conditionally on $X_1$, we have $h_{X_1,f}(\|X_j-X_1\|) \sim U[0,1]$ for $j \in \{2,\ldots,n\}$, and it follows that
\[
  \bigl(\hat{f}_{(k_X),1},\hat{g}_{(k_Y),1}\bigr)\bigm| X_1 \stackrel{d}{=} \biggl(\frac{k_X}{mV_dh_{X_1,f}^{-1}(B_1)^d},\frac{k_Y}{nV_dh_{X_1,g}^{-1}(B_2)^d}\biggr) \biggm| X_1,
\]
where $B_1 \sim \mathrm{Beta}(k_X,m-k_X)$ and $B_2 \sim \mathrm{Beta}(k_Y,n+1-k_Y)$ are independent.  Moreover, we may write, for example,
\begin{equation}
\label{Eq:uxsdecomp}
	\frac{u_{x,s}}{f(x)}-1 = \frac{k_X}{ms}-1 + \frac{s}{V_df(x)h_{x,f}^{-1}(s)^d}-1 + \Bigl(  \frac{k_X}{ms}-1 \Bigr) \Bigl(  \frac{s}{V_df(x)h_{x,f}^{-1}(s)^d}-1 \Bigr),
\end{equation}
and use Lemma~\ref{Lemma:hxinvbounds} to expand $V_df(x)h_{x,f}^{-1}(s)^d/s$ in powers of $s^{2/d}$.  Since the $\mathrm{Beta}(k,n-k)$ distribution concentrates around its mean at rate $k^{-1/2}$ in an approximately symmetric way, we will also see later that for every $a \in \mathbb{R}$, we have an asymptotic expansion of the form
\[
	\Bigl(\frac{n}{k} \Bigr)^a \int_0^1 s^a \Bigl( \frac{k}{ns} -1 \Bigr)^j \mathrm{B}_{k,n-k}(s) \,ds = c_1 k^{-\lceil j /2 \rceil} + c_2 k^{- \lceil j/2 \rceil -1 } + \ldots + O(1/n),
\]
provided that $k=k_n \rightarrow \infty$ and $k/n \rightarrow 0$ as $n \rightarrow \infty$. These facts mean that for remainder terms $R_1,\ldots,R_4$ to be bounded below and functions $c_{i_1i_2j_1j_2}(x)$ to be specified later we may write
\begin{align}
\label{Eq:BiasExpansion}
	\mathbb{E} &\widetilde{T}_{m,n} = \int_\mathcal{X} f(x) \int_0^1 \int_0^1 \phi ( u_{x,s}, v_{x,t}) \mathrm{B}_{k_X,m-k_X}(s) \mathrm{B}_{k_Y,n+1-k_Y}(t) \,ds \,dt \,dx \nonumber \\
	&  = \int_{\mathcal{X}_{m,n}}  f(x)   \int_{\mathcal{I}_{m,X}}  \int_{\mathcal{I}_{n,Y}}  \phi ( u_{x,s}, v_{x,t} )  \mathrm{B}_{k_X,m-k_X}(s) \mathrm{B}_{k_Y,n+1-k_Y}(t) \,ds \,dt \,dx + R_1 \nonumber \\
	& = \sum_{\ell_1,\ell_2=0}^\infty \frac{\mathbbm{1}_{\{\ell_1+\ell_2 \leq \beta^*-1\}}}{\ell_1! \ell_2!} \int_{\mathcal{X}_{m,n}} \! \!  f(x)     \int_{\mathcal{I}_{m,X}} \int_{\mathcal{I}_{n,Y}} \biggl( \frac{u_{x,s}}{f(x)}-1 \biggr)^{\ell_1} \biggl( \frac{v_{x,t}}{g(x)} -1 \biggr)^{\ell_2} f(x)^{\ell_1}g(x)^{\ell_2} \nonumber \\
	& \hspace{25pt} \times  \phi_{\ell_1\ell_2}\bigl(f(x),g(x)\bigr) \mathrm{B}_{k_X,m-k_X}(s) \mathrm{B}_{k_Y,n+1-k_Y}(t) \,ds \,dt \,dx + R_1 + R_2 \\
	&  = \sum_{i_1,i_2=0}^{\lceil d/2 \rceil -1} \sum_{j_1,j_2=0}^\infty \mathbbm{1}_{\{j_1+j_2 \leq \beta^*-1\}} \int_{\mathcal{X}_{m,n}} \! f(x) c_{i_1i_2j_1j_2}(x) \int_{\mathcal{I}_{m,X}} \! \int_{\mathcal{I}_{n,Y}}  \Bigl( \frac{k_X}{ms} -1 \Bigr)^{j_1} \Bigl( \frac{k_Y}{nt} -1 \Bigr)^{j_2}\nonumber  \\
	& \hspace{50pt} \times s^{\frac{2i_1}{d}} t^{\frac{2i_2}{d}} \mathrm{B}_{k_X,m-k_X}(s) \mathrm{B}_{k_Y,n+1-k_Y}(t) \,ds \,dt \,dx + R_1 + R_2 + R_3 \nonumber \\
	& = \sum_{i_1,i_2=0}^{\lceil d/2 \rceil -1} \sum_{j_1,j_2=0}^\infty \mathbbm{1}_{\{j_1+j_2 \leq \lceil (\beta^*-1)/2 \rceil\}} \frac{\lambda_{i_1i_2j_1j_2}}{k_X^{j_1} k_Y^{j_2}} \Bigl( \frac{k_X}{m} \Bigr)^{\frac{2i_1}{d}} \Bigl( \frac{k_Y}{n} \Bigr)^{\frac{2i_2}{d}} + R_1+R_2+R_3+R_4.
\end{align}
It now remains to bound each of the remainder terms.

\bigskip

\emph{To bound $R_1$:} Since we are assuming that $\zeta<1$, we may apply Lemma~\ref{Lemma:Acomplement} to see that
\begin{align}
  \label{Eq:R11}
  \int_\mathcal{X} f(x) \int_0^1 \int_0^1 &\bigl( 1 - \mathbbm{1}_{\{s \in \mathcal{I}_{m,X} \}}\mathbbm{1}_{\{t \in \mathcal{I}_{n,Y} \}} \bigr) \phi ( u_{x,s}, v_{x,t}) \nonumber \\
  &\times \mathrm{B}_{k_X,m-k_X}(s) \mathrm{B}_{k_Y,n+1-k_Y}(t) \,ds \,dt \,dx = o(m^{-4}+n^{-4}).
\end{align}
\sloppy{When $s \in \mathcal{I}_{m,X}$ and $k_X^\mathrm{L} \geq 36 \log m$, we have by Lemma~\ref{Lemma:hxinvbounds2} that $u_{x,s} \leq \frac{C k_X}{m a_{m,X}^-} \leq 2C$, and, similarly, when $t \in \mathcal{I}_{n,Y}$ and $k_Y^\mathrm{L} \geq 36 \log n$, we have $v_{x,t} \leq 2C$. Thus, when $s \in \mathcal{I}_{m,X}, t \in \mathcal{I}_{n,Y}$ and $\min( k_X^\mathrm{L} / \log m , k_Y^\mathrm{L} / \log n) \geq 36$, we may use the fact that $|u_{x,s}^{\ell_1} v_{x,t}^{\ell_2} \phi_{\ell_1,\ell_2}(u_{x,s},v_{x,t}) | \leq L (2C)^{2L+|\kappa_1|+|\kappa_2|} u_{x,s}^{\kappa_1} v_{x,t}^{\kappa_2}$ for all $\ell_1,\ell_2 \in \mathbb{N}_0$ such that $\ell_1+\ell_2 \leq \beta^*-1$.}

In the following we consider the decomposition $\mathcal{X}_{m,n}^c = \mathcal{X}_{m,f}^c \cup \mathcal{X}_{n,g}^c$, where $\mathcal{X}_{m,f} := \{x : f(x)M_\beta(x)^{-d} \geq k_X \log m /m\}$ and $\mathcal{X}_{n,g} := \{x: g(x)M_\beta(x)^{-d} \geq k_Y \log n /n\}$. Using Lemma~\ref{Lemma:GeneralisedHolder} and Lemma~\ref{Lemma:hxinvbounds2} we have that
\begin{align}
\label{Eq:ftailerror}
	\int_{\mathcal{X}_{m,f}^c} & f(x)  \int_{\mathcal{I}_{m,X}} \int_{\mathcal{I}_{n,Y}} \phi ( u_{x,s}, v_{x,t} )  \mathrm{B}_{k_X,m-k_X}(s) \mathrm{B}_{k_Y,n+1-k_Y}(t) \,ds \,dt \,dx \nonumber \\
	& \lesssim \int_{\mathcal{X}_{m,f}^c} f(x)^{1-\kappa_1^-} g(x)^{-\kappa_2^-} M_\beta(x)^{d (\kappa_1^-+\kappa_2^-)} (1+\|x\|)^{d (\kappa_1^-+\kappa_2^-)} \,dx\nonumber \\
	& \leq \inf_{ a >0} \Bigl( \frac{k_X \log m}{m} \Bigr)^a \int_\mathcal{X} f(x) \frac{M_\beta(x)^{d(a+\kappa_1^-+\kappa_2^-)}}{f(x)^{a+\kappa_1^-}g(x)^{\kappa_2^-}} (1+\|x\|)^{d (\kappa_1^-+\kappa_2^-)} \, dx \nonumber \\
	& = O \biggl( \biggl( \frac{k_X}{m} \biggr)^{\lambda_1(1-\zeta)- \epsilon} \biggr)
\end{align}
for every $\epsilon>0$. With a similar bound over $\mathcal{X}_{n,g}^c$ we conclude that
\begin{align}
  \label{Eq:R12}
	\int_{\mathcal{X}_{m,n}^c} f(x)  \int_{\mathcal{I}_{m,X}} \int_{\mathcal{I}_{n,Y}}& \phi ( u_{x,s}, v_{x,t})  \mathrm{B}_{k_X,m-k_X}(s) \mathrm{B}_{k_Y,n+1-k_Y}(t) \,ds \,dt \,dx \nonumber \\
	&= O \biggl( \max \biggl\{ \biggl( \frac{k_X}{m} \biggr)^{\lambda_1(1-\zeta)- \epsilon}, \biggl( \frac{k_Y}{n} \biggr)^{\lambda_2(1-\zeta)- \epsilon} \biggr\} \biggr)
\end{align}
for every $\epsilon>0$.  From~\eqref{Eq:R11} and~\eqref{Eq:R12}, we deduce that
\begin{equation}
\label{Eq:R1Bound}
  R_1 = O \biggl( \max \biggl\{ \biggl( \frac{k_X}{m} \biggr)^{\lambda_1(1-\zeta)- \epsilon}, \biggl( \frac{k_Y}{n} \biggr)^{\lambda_2(1-\zeta)- \epsilon},\frac{1}{m^4},\frac{1}{n^4} \biggr\} \biggr).
\end{equation}

\bigskip

\emph{To bound $R_2$:} We first observe that, by~\eqref{Eq:uxsdecomp} and Lemma~\ref{Lemma:hxinvbounds}, we have that
\begin{align}
\label{Eq:epsilonn}
	\epsilon_{m,n}:= \sup_{x \in \mathcal{X}_{m,f}} \sup_{s \in \mathcal{I}_{m,X}} \Bigl| \frac{u_{x,s}}{f(x)} -1 \Bigr| \bigvee \sup_{x \in \mathcal{X}_{n,g}} \sup_{t \in \mathcal{I}_{n,Y}} \Bigl| \frac{v_{x,t}}{g(x)} -1 \Bigr| =o(1).
\end{align}
Now, for $t \in [0,1]$ we have that $h(t):=t-\log(1+t) \geq t^2/4$. Thus, letting $B \sim \mathrm{Beta}(k,n-k)$, whenever $\frac{3 \alpha^{1/2} \log^{1/2}n}{k^{1/2}} \leq 1$ and $ \frac{k^{1/2}+ 3\alpha^{1/2} \log^{1/2} n}{n^{1/2}} \leq 2^{1/2}-1$ we may integrate the Beta tail bound in Lemma~\ref{Lemma:BetaTailBounds} to see that
\begin{align}
\label{Eq:CentralMoments}
	\int_0^1 \Bigl| \frac{ns}{k} -&1 \Bigr|^\alpha  \mathrm{B}_{k,n-k}(s) \,ds = \alpha \int_0^{n/k} y^{\alpha-1} \mathbb{P} \Bigl( \Bigl| B - \frac{k}{n} \Bigr| \geq \frac{ky}{n} \Bigr) \,dy \nonumber \\
	& \leq 2 \alpha k^{-\alpha/2} \int_0^{k^{-1/2}n} u^{\alpha-1} \Bigl\{ \exp \Bigl( - k h \Bigl( \frac{n^{1/2}k^{-1/2}u}{n^{1/2}+k^{1/2}+u} \Bigr) \Bigr) \nonumber \\
	& \hspace{140pt} + \exp \Bigl( - n h \Bigl( \frac{u}{n^{1/2}+k^{1/2}+u} \Bigr) \Bigr) \Bigr\} \,du \nonumber \\
	& \leq 4 \alpha k^{-\alpha/2} \int_0^{3 \alpha^{1/2} \log^{1/2} n} u^{\alpha-1} e^{-u^2/8} \,du + \frac{4 n^{\alpha}}{k^{\alpha}} \exp \Bigl( - \frac{9 \alpha \log n}{8} \Bigr) \nonumber \\
	& \leq \frac{2^{3(\alpha-1)/2} \alpha \Gamma(\alpha/2)}{k^{\alpha/2}} + \frac{4}{k^{\alpha}}.
\end{align}
Next, by Lemma~\ref{Lemma:GeneralisedHolder}, we have for any $\tau \geq 0$ that
\begin{align}
\label{Eq:fbodyerror}
	\Bigl( \frac{k_X}{m} \Bigr)^\tau \int_{\mathcal{X}_{m,n}} & f(x)^{1-\kappa_1^-}g(x)^{-\kappa_2^-} \Bigl\{ \frac{M_\beta(x)^d}{f(x)} \Bigr\}^\tau \,dx \nonumber \\
	& \leq \inf_{a>0} \Bigl( \frac{k_X}{m} \Bigr)^{\tau-a} \int_\mathcal{X} f(x) \Bigl\{ \frac{M_\beta(x)^d}{f(x)} \Bigr\}^{\tau+\kappa_1^- -a} \Bigl\{ \frac{M_\beta(x)^d}{g(x)} \Bigr\}^{\kappa_2^-} \,dx \nonumber \\
	& = O \biggl( \max \biggl\{ \Bigl( \frac{k_X}{m} \Bigr)^\tau, \Bigl( \frac{k_X}{m} \Bigr)^{\lambda_1(1-\zeta) - \epsilon} \biggr\} \biggr)
\end{align}
for all $\epsilon>0$. Analogously, 
\begin{equation}
\label{Eq:gbodyerror}
	\int_{\mathcal{X}_{m,n}} f(x)^{1-\kappa_1^-}g(x)^{-\kappa_2^-} \Bigl\{ \frac{M_\beta(x)^d}{g(x)} \Bigr\}^\tau \,dx = O \biggl( \max \biggl\{ 1, \Bigl( \frac{k_Y}{n} \Bigr)^{\lambda_2(1-\zeta) - \tau- \epsilon} \biggr\} \biggr)
\end{equation}
for any $\tau \geq 0$ and $\epsilon>0$. Now, since $\phi \in \Phi$ and by~\eqref{Eq:uxsdecomp}, \eqref{Eq:epsilonn}, \eqref{Eq:CentralMoments}, \eqref{Eq:fbodyerror}, \eqref{Eq:gbodyerror} and Lemmas~\ref{Prop:FunctionalClasses}(ii) and~\ref{Lemma:hxinvbounds} we have, when $m,n$ are sufficiently large that $\epsilon_{m,n}<1/2$,
\begin{align}
\label{Eq:R2Bound}
	&|R_2| \lesssim  L  \int_{\mathcal{X}_{m,n}} f(x)^{1+\kappa_1} g(x)^{\kappa_2}  \int_{\mathcal{I}_{m,X}} \int_{\mathcal{I}_{n,Y}} \biggl\{ \biggl| \frac{u_{x,s}}{f(x)} -1 \biggr|^{\beta^*} + \biggl| \frac{v_{x,t}}{g(x)} -1 \biggr|^{\beta^*} \biggr\} \nonumber \\
	& \hspace{150pt} \times \mathrm{B}_{k_X,m-k_X}(s) \mathrm{B}_{k_Y,n+1-k_Y}(t)  \,ds \,dt \,dx \nonumber \\
	& \lesssim \int_{\mathcal{X}_{m,n}} f(x)^{1+\kappa_1}g(x)^{\kappa_2} \int_{\mathcal{I}_{m,X}} \int_{\mathcal{I}_{n,Y}} \mathrm{B}_{k_X,m-k_X}(s) \mathrm{B}_{k_Y,n+1-k_Y}(t)  \biggl[ \Bigl| \frac{k_X}{ms} -1 \Bigr|^{\beta^*} \nonumber \\
	& \hspace{50pt} +  \Bigl| \frac{k_Y}{nt} -1 \Bigr|^{\beta^*}  + \Bigl\{ \frac{sM_\beta(x)^d}{f(x)} \Bigr\}^{\frac{2 \wedge \beta}{d} \beta^*}+ \Bigl\{ \frac{tM_\beta(x)^d}{g(x)} \Bigr\}^{\frac{2 \wedge \beta}{d} \beta^*} \biggr] \,ds \,dt \,dx \nonumber \\
	&= O \biggl( \max \biggl\{ k_X^{-\beta^*/2},  k_Y^{-\beta^*/2}, \Bigl( \frac{k_X}{m} \Bigr)^{\frac{2\wedge \beta}{d}\beta^*} ,  \Bigl( \frac{k_X}{m} \Bigr)^{\lambda_1(1-\zeta)- \epsilon},  \Bigl( \frac{k_Y}{n} \Bigr)^{\frac{2 \wedge \beta}{d} \beta^*} \Bigl( \frac{k_Y}{n} \Bigr)^{\lambda_2(1-\zeta)- \epsilon} \biggr\} \biggr).
\end{align}

\bigskip

\emph{To bound $R_3$:} By \eqref{Eq:uxsdecomp} and Lemma~\ref{Lemma:hxinvbounds}, when $\ell_1>0$ we have expansions of the form
\[
	\biggl| \biggl( \frac{u_{x,s}}{f(x)} -1 \biggr)^{\ell_1} - \sum_{i=0}^{\lceil d/2 \rceil -1} \sum_{j=0}^{\ell_1} b_{i,j}(x) s^{2i/d} \Bigl( \frac{k}{ms} - 1 \Bigr)^j \biggr| \lesssim \Bigl\{ \frac{sM_\beta(x)^d}{f(x)} \Bigr\}^{ \frac{\beta}{d} \wedge 1},
\]
with $|b_{i,j}(x)| \lesssim \{M_\beta(x)^d/f(x)\}^{2i/d}$ and $b_{0,0}=0$. A similar expansion can also be written for $(v_{x,t}/g(x)-1)^{\ell_2}$.  Using these two expansions it can be seen that $c_{i_1i_2j_1j_2}$ can be chosen in~\eqref{Eq:BiasExpansion} with $|c_{i_1i_2j_1j_2}(x)| \lesssim f(x)^{\kappa_1}g(x)^{\kappa_2} \{M_\beta(x)^d/f(x)\}^{2i_1/d} \{M_\beta(x)^d/g(x)\}^{2i_2/d}$, with $c_{0,0,0,0}(x)=\phi(f(x),g(x))$,  and, using~\eqref{Eq:fbodyerror} and~\eqref{Eq:gbodyerror}, with
\begin{align*}
	|R_3| &\lesssim \int_{\mathcal{X}_{m,n}} f(x)^{1+\kappa_1}g(x)^{\kappa_2} \biggl\{ \Bigl( \frac{k_XM_\beta(x)^d}{mf(x)} \Bigr)^{ \frac{\beta}{d} \wedge 1} + \Bigl( \frac{k_Y M_\beta(x)^d}{ng(x)} \Bigr)^{ \frac{\beta}{d} \wedge 1} \biggr\} \,dx \\
	&=O \biggl( \max \biggl\{ \Bigl( \frac{k_X}{m} \Bigr)^{\frac{\beta}{d} \wedge 1 } ,  \Bigl( \frac{k_X}{m} \Bigr)^{ \lambda_1 (1-\zeta)- \epsilon}, \Bigl( \frac{k_Y}{n} \Bigr)^{ \frac{\beta}{d} \wedge 1 }, \Bigl( \frac{k_Y}{n} \Bigr)^{\lambda_2(1-\zeta)- \epsilon} \biggr\} \biggr).
\end{align*}

\bigskip

\emph{To bound $R_4$:} Whenever $a \in \mathbb{R}$ is fixed, we have an asymptotic series of the form
\begin{equation}
\label{Eq:BetaMomentExpansion}
	\frac{\Gamma(m+a)}{k_X^a\Gamma(m)} \int_0^1 s^a \mathrm{B}_{k_X,m-k_X}(s) \,ds = \frac{\Gamma(k_X+a)}{\Gamma(k_X) k_X^a}= 1 + c_1/k_X + c_2/ k_X^{2} + \ldots.
\end{equation}
On the other hand, arguing similarly to~\eqref{Eq:CentralMoments}, for fixed $j \in \mathbb{N}$ we have the bound
\begin{align}
\label{Eq:CentralBound}
	\Bigl( \frac{m}{k_X} \Bigr)^a & \int_0^1 s^a  \Bigl| \frac{k_X}{ms} -1 \Bigr|^j \mathrm{B}_{k_X,m-k_X}(s) \,ds \nonumber \\
	& \leq \frac{2^{j-1}m^a \Gamma(k_X + a - j) \Gamma(m)}{k_X^a \Gamma(k_X) \Gamma(m+a-j)} \biggl\{  \Bigl| \frac{k_X+a-j}{m+a-j} - \frac{k_X}{m} \Bigr|^j  \nonumber \\
	& \hspace{100pt} +\int_0^1 \Bigl| s - \frac{k_X+a-j}{m+a-j} \Bigr|^j \mathrm{B}_{k_X+a-j,m-k_X}(s) \,ds \biggr\} \nonumber \\
	& = O(k_X^{-j/2}).
\end{align}
Moreover, by Lemma~\ref{Lemma:BetaTailBounds}, letting $B \sim \mathrm{Beta}(k_X+a-j,m-k_X)$ we have that
\begin{align}
\label{Eq:TailMoment}
	&\Bigl( \frac{m}{k_X} \Bigr)^a  \int_{[0,1] \setminus \mathcal{I}_{m,X}} s^a \Bigl| \frac{k_X}{ms} -1 \Bigr|^j \mathrm{B}_{k_X,m-k_X}(s) \,ds \nonumber \\
	& \lesssim \int_{[0,1] \setminus \mathcal{I}_{m,X}} \Bigl| \frac{ms}{k_X} -1 \Bigr|^j \mathrm{B}_{k_X+a-j,m-k_X}(s) \,ds \nonumber \\
	&  \leq \mathbb{P} \biggl( \Bigl| \frac{mB}{k_X} -1 \Bigr| \geq \frac{3 \log^{1/2} m}{k_X^{1/2}} \biggr) + \Bigl( \frac{m}{k_X} \Bigr)^j \mathbb{P} \biggl(  \Bigl| \frac{mB}{k_X} -1 \Bigr| \geq 1 \biggr) = o(m^{-4}).
\end{align}
With the similar expression in terms of $k_Y$ and $n$, we now conclude from~\eqref{Eq:BetaMomentExpansion}, \eqref{Eq:CentralBound} and~\eqref{Eq:TailMoment} that we have an asymptotic expansion of the form
\begin{align}
  \label{Eq:Series}
	& \int_{\mathcal{I}_{m,X}} \int_{\mathcal{I}_{n,Y}}  s^\frac{2i_1}{d} t^\frac{2i_2}{d} \Bigl( \frac{k_X}{ms} -1 \Bigr)^{j_1} \Bigl( \frac{k_Y}{nt} -1 \Bigr)^{j_2} \mathrm{B}_{k_X,m-k_X}(s) \mathrm{B}_{k_Y,n+1-k_Y}(t) \,ds \,dt \nonumber \\
	& = \Bigl( \frac{k_X}{m} \Bigr)^\frac{2i_1}{d} \Bigl( \frac{k_Y}{n} \Bigr)^{\frac{2i_2}{d}}\biggl\{ \sum_{r=\lceil j_1/2 \rceil}^\infty \! \! c_rk_X^{-r} + O(1/m) \biggr\} \biggl\{ \sum_{r=\lceil j_2/2 \rceil}^\infty \! \! d_rk_Y^{-r} + O(1/n) \biggr\}.
\end{align}
Now for fixed $i_1,i_2 \in [\lceil d/2 \rceil - 1]$ with $\frac{\kappa_1^-+2i_1/d}{\lambda_1}+ \frac{\kappa_2^- +  2i_2/d}{\lambda_2}  \geq 1$, we have by Lemma~\ref{Lemma:GeneralisedHolder} that
\begin{align}
  \label{Eq:cxbound}
	&\int_{\mathcal{X}_{m,n}} f(x) |c_{i_1i_2j_1j_2}(x)| \,dx  \lesssim \int_{\mathcal{X}_{m,n}}f(x) \frac{ M_\beta(x)^{2i_1+2i_2}}{f(x)^{\kappa_1^-+2i_1/d} g(x)^{\kappa_2^-+2i_2/d}} \,dx \nonumber \\
	& \leq \min \biggl\{ \inf_{a >0} \Bigl( \frac{k_X }{m} \Bigr)^{-a} \int_{\mathcal{X}}f(x) \Bigl\{ \frac{M_\beta(x)^d}{f(x)} \Bigr\}^{\kappa_1^- + \frac{2i_1}{d}-a}  \Bigl\{ \frac{M_\beta(x)^d}{g(x)} \Bigr\}^{\kappa_2^-+\frac{2i_2}{d}} \,dx , \nonumber \\
	& \hspace{50pt}  \inf_{a >0} \Bigl( \frac{k_Y }{n} \Bigr)^{-a} \int_{\mathcal{X}}f(x) \Bigl\{ \frac{M_\beta(x)^d}{f(x)} \Bigr\}^{\kappa_1^- + \frac{2i_1}{d}}  \Bigl\{ \frac{M_\beta(x)^d}{g(x)} \Bigr\}^{\kappa_2^-+\frac{2i_2}{d}-a} \,dx  \biggr\} \nonumber \\
	& =  O \biggl( \min \biggl\{ \Bigl( \frac{k_X}{m} \Bigr)^{\lambda_1(1-\zeta - \frac{2i_2}{d \lambda_2}) -2i_1/d-\epsilon}, \Bigl( \frac{k_Y}{n} \Bigr)^{\lambda_2(1-\zeta - \frac{2i_1}{d \lambda_1}) -2i_2/d -\epsilon} \biggr\} \biggr) \nonumber \\
	&=   O \biggl( \Bigl( \frac{k_X}{m} \Bigr)^{-\frac{2i_1}{d}} \Bigl( \frac{k_Y}{n} \Bigr)^{-\frac{2i_2}{d}}\max \biggl\{ \Bigl( \frac{k_X}{m} \Bigr)^{\lambda_1(1-\zeta) - \epsilon} , \Bigl( \frac{k_Y}{n} \Bigr)^{\lambda_2(1-\zeta)-\epsilon} \biggr\} \biggr),
\end{align}
for all $\epsilon>0$, where the final inequality can be established by considering the cases $( \frac{k_X}{m} )^{\lambda_1} \geq (\frac{k_Y}{n} )^{\lambda_2}$ and $( \frac{k_X}{m} )^{\lambda_1} < (\frac{k_Y}{n} )^{\lambda_2}$ separately. For such $i_1,i_2$ we set $\lambda_{i_1i_2j_1j_2}=0$ for all $j_1,j_2$. When, instead, $ \frac{\kappa_1^-+2i_1/d}{\lambda_1}+ \frac{\kappa_2^- + 2i_2/d}{\lambda_2} < 1$, we again consider these two cases separately, use the decomposition $\mathcal{X}_{m,n}^c =  (\mathcal{X}_{m,f}^c \cap \mathcal{X}_{n,g}^c) \cup (\mathcal{X}_{m,f}^c \cap \mathcal{X}_{n,g}) \cup ( \mathcal{X}_{m,f} \cap \mathcal{X}_{n,g}^c)$ and apply Lemma~\ref{Lemma:GeneralisedHolder} to write
\begin{align}
  \label{Eq:cxbound2}
	& \int_{\mathcal{X}_{m,n}^c} f(x) |c_{i_1i_2j_1j_2}(x)| \,dx  \lesssim \int_{\mathcal{X}_{m,f}^c \cap \mathcal{X}_{n,g}^c} f(x)\frac{ M_\beta(x)^{2i_1+2i_2}}{f(x)^{\kappa_1^-+2i_1/d} g(x)^{\kappa_2^-+2i_2/d}} \,dx  \,dx \nonumber \\
	& \hspace{50pt} + O \biggl( \Bigl( \frac{k_X}{m} \Bigr)^{-\frac{2i_1}{d}} \Bigl( \frac{k_Y}{n} \Bigr)^{-\frac{2i_2}{d}}   \biggl\{  \Bigl( \frac{k_X}{m} \Bigr)^{\lambda_1(1-\zeta) - \epsilon} \vee \Bigl( \frac{k_Y}{n} \Bigr)^{\lambda_2(1-\zeta)-\epsilon} \biggr\} \biggr) \nonumber \\ 
	& \leq \min \biggl\{ \inf_{a>0} \Bigl( \frac{k_X \log m}{m} \Bigr)^a \int_\mathcal{X} f(x) \Bigl\{ \frac{M_\beta (x)^d}{f(x)} \Bigr\}^{\kappa_1^- +\frac{2i_1}{d}+a}  \Bigl\{ \frac{M_\beta (x)^d}{g(x)} \Bigr\}^{\kappa_2^- + \frac{2i_2}{d}} \, dx , \nonumber \\
	& \hspace{45pt}  \inf_{a>0} \Bigl( \frac{k_Y \log n}{n} \Bigr)^a \int_\mathcal{X} f(x) \Bigl\{ \frac{M_\beta (x)^d}{f(x)} \Bigr\}^{\kappa_1^- +\frac{2i_1}{d}}  \Bigl\{ \frac{M_\beta (x)^d}{g(x)} \Bigr\}^{\kappa_2^- + \frac{2i_2}{d}+a} \,dx \biggr\} \nonumber \\
	& \hspace{45pt} + O \biggl(\Bigl( \frac{k_X}{m} \Bigr)^{-\frac{2i_1}{d}} \Bigl( \frac{k_Y}{n} \Bigr)^{-\frac{2i_2}{d}}  \biggl\{  \Bigl( \frac{k_X}{m} \Bigr)^{\lambda_1(1-\zeta) - \epsilon} \vee \Bigl( \frac{k_Y}{n} \Bigr)^{\lambda_2(1-\zeta)-\epsilon} \biggr\} \biggr) \nonumber \\
	&= O \biggl( \Bigl( \frac{k_X}{m} \Bigr)^{-\frac{2i_1}{d}} \Bigl( \frac{k_Y}{n} \Bigr)^{-\frac{2i_2}{d}} \biggl\{  \Bigl( \frac{k_X}{m} \Bigr)^{\lambda_1(1-\zeta) - \epsilon} \vee \Bigl( \frac{k_Y}{n} \Bigr)^{\lambda_2(1-\zeta)-\epsilon} \biggr\} \biggr)
\end{align}
for all $\epsilon>0$. It follows from~\eqref{Eq:Series},~\eqref{Eq:cxbound} and~\eqref{Eq:cxbound2} that
\begin{align*}
	R_4 = O \biggl( \max \biggl\{ k_X^{-\lceil (\beta^*-1)/2 \rceil -1}, k_Y^{-\lceil (\beta^*-1)/2 \rceil -1}, &\Bigl( \frac{k_X}{m} \Bigr)^{\lambda_1(1-\zeta)- \epsilon},  \Bigl( \frac{k_Y}{n} \Bigr)^{\lambda_2(1-\zeta)- \epsilon} \! , \frac{1}{m}, \frac{1}{n} \biggr\} \biggr),
\end{align*}
and this concludes the proof.
\end{proof}

\subsection{Proof of Proposition~\ref{Prop:ImprovedBias} on improved bias bounds}
\label{Sec:ImprovedBias}

\begin{proof}[Proof of Proposition~\ref{Prop:ImprovedBias}]

From~\eqref{Eq:BiasExpansion} in the proof of Proposition~\ref{Thm:GeneralBias}, we may write
\begin{align*}
    \mathbb{E} &\widetilde{T}_{m,n} \\
    &=  \sum_{\ell_1,\ell_2=0}^\infty  \frac{\mathbbm{1}_{\{\ell_1+\ell_2 \leq \beta^*-1\}}}{\ell_1! \ell_2!} \int_{\mathcal{X}_{m,n}} \! \!  f(x)     \int_{\mathcal{I}_{m,X}} \int_{\mathcal{I}_{n,Y}} \biggl( \frac{u_{x,s}}{f(x)}-1 \biggr)^{\ell_1} \biggl( \frac{v_{x,t}}{g(x)} -1 \biggr)^{\ell_2} f(x)^{\ell_1}g(x)^{\ell_2} \nonumber \\
	& \hspace{50pt} \times  \phi_{\ell_1\ell_2}\bigl(f(x),g(x)\bigr) \mathrm{B}_{k_X,m-k_X}(s) \mathrm{B}_{k_Y,n+1-k_Y}(t) \,ds \,dt \,dx + R_1 + R_2 \nonumber
\end{align*}
where $R_1$ and $R_2$ satisfy the bounds~\eqref{Eq:R1Bound} and~\eqref{Eq:R2Bound} respectively.
We may expand $(u_{x,s}/f(x)-1)^{\ell_1}$  using~\eqref{Eq:uxsdecomp}, and also expand and $(v_{x,t}/g(x)-1)^{\ell_2}$ analogously. Any term including $s/\{V_d f(x)h_{x,f}^{-1}(s)^d\}-1$ and $t/\{V_d g(x)h_{x,g}^{-1}(t)^d\}-1$ to a combined power greater than one can be bounded by
\[
O \biggl( \max \biggl\{ \Bigl( \frac{k_X}{m} \Bigr)^{2\beta/d}, \Bigl( \frac{k_X}{m} \Bigr)^{ \lambda_1(1-\zeta) -\epsilon}, \Bigl( \frac{k_Y}{n} \Bigr)^{2\beta/d}, \Bigl( \frac{k_Y}{n} \Bigr)^{ \lambda_2(1-\zeta) -\epsilon} \biggr\} \biggr)
\]
by Lemma~\ref{Lemma:hxinvbounds}, so can absorbed into the error term in~\eqref{Eq:ImprovedBiasError}. The key difference with the proof of Proposition~\ref{Thm:GeneralBias} is that for $s \in \mathcal{I}_{m,X}$ we now write
\begin{align}
\label{Eq:ImprovedDecomposition}
    &\int_{\mathcal{X}_{m,n}} f(x)^2  \phi_{10}\bigl(f(x),g(x)\bigr) \bigl\{ s - V_d f(x)h_{x,f}^{-1}(s)^d \bigr\} \, dx \nonumber \\
    &=\int_{\mathcal{X}_{m,n}} \int_{\mathcal{X}_{m,n}} f(x)^2 \phi_{10}\bigl(f(x),g(x)\bigr) \mathbbm{1}_{\{\|x-y\| \leq h_{x,f}^{-1}(s)\}} \{f(y)-f(x)\} \,dy\,dx + \tilde{R}_1(s),
    \end{align}
where
\[
\int_{\mathcal{I}_{m,X}} \frac{1}{s} \mathrm{B}_{k_X,m-k_X}(s) |\tilde{R}_1(s)| \, ds = O \biggl( \Bigl( \frac{k_X}{m} \Bigr)^{ \lambda_1(1-\zeta) -\epsilon}  \biggr).
\]
Moreover, by Fubini's theorem,
\begin{align}
\label{Eq:ImprovedDecomposition2}
&\int_{\mathcal{X}_{m,n}} \int_{\mathcal{X}_{m,n}} f(x)^2 \phi_{10}\bigl(f(x),g(x)\bigr) \mathbbm{1}_{\{\|x-y\| \leq h_{x,f}^{-1}(s)\}} \{f(y)-f(x)\} \,dy\,dx \nonumber \\
    & = \frac{1}{2}\int_{\mathcal{X}_{m,n}} \int_{\mathcal{X}_{m,n}} \biggl[ f(x)^2 \phi_{10}\bigl(f(x),g(x)\bigr) \mathbbm{1}_{\{\|x-y\| \leq h_{x,f}^{-1}(s)\}} \{f(y)-f(x)\} \nonumber \\
    &\hspace{100pt} + f(y)^2 \phi_{10}\bigl(f(y),g(y)\bigr) \mathbbm{1}_{\{\|x-y\| \leq h_{y,f}^{-1}(s)\}} \{f(x)-f(y)\} \biggr] \,dy \,dx \nonumber \\
    & = \frac{1}{2}\int_{\mathcal{X}_{m,n}} \int_{\mathcal{X}_{m,n}} \biggl[ \mathbbm{1}_{\{\|x-y\| \leq h_{x,f}^{-1}(s)\}} \{f(y)-f(x)\} \bigl\{ f(x)^2 \phi_{10}\bigl(f(x),g(x)\bigr) \nonumber \\
    &\hspace{270pt}- f(y)^2 \phi_{10}\bigl(f(y),g(y)\bigr) \bigr\} \nonumber \\
    & \hspace{20pt} + \{f(y)-f(x)\} f(y)^2 \phi_{10}\bigl(f(y),g(y)\bigr) \bigl(  \mathbbm{1}_{\{\|x-y\| \leq h_{y,f}^{-1}(s)\}} - \mathbbm{1}_{\{\|x-y\| \leq h_{x,f}^{-1}(s)\}} \bigr) \biggr] \,dy \,dx.
\end{align}
Using Lemmas~\ref{Prop:FunctionalClasses}(i),~\ref{Lemma:15over7} and~\ref{Lemma:hxinvbounds}, and arguing as around~\eqref{Eq:MConstant2}, for $s\in\mathcal{I}_{m,X}$, $x,y \in \mathcal{X}_{m,n}$ with $\|x-y\| \leq h_{x,f}^{-1}(s)$ and $m,n$ sufficiently large, we have
\begin{align}
\label{Eq:ImprovedTerm1}
     &\biggl|\int_{\mathcal{X}_{m,n}} \int_{\mathcal{X}_{m,n}} \mathbbm{1}_{\{\|x-y\| \leq h_{x,f}^{-1}(s)\}} \{f(y)-f(x)\} \bigl\{ f(x)^2 \phi_{10}\bigl(f(x),g(x)\bigr) \nonumber \\
     &\hspace{250pt}- f(y)^2 \phi_{10}\bigl(f(y),g(y)\bigr) \bigr\} \,dy \,dx \biggr| \nonumber \\
     & \lesssim \int_{\mathcal{X}_{m,n}} \! \int_{\mathcal{X}_{m,n}} \! \! \! \! \mathbbm{1}_{\{\|x-y\| \leq h_{x,f}^{-1}(s)\}} f(x)^{1+\kappa_1}g(x)^{\kappa_2} |f(y)\!-\!f(x)| \biggl\{ \Bigl|\frac{f(y)}{f(x)}\!-\!1\Bigr| \!+\! \Bigl|\frac{g(y)}{g(x)}\!-\!1\Bigr|\biggr\} \,dy\,dx \nonumber \\
     & \lesssim \int_{\mathcal{X}_{m,n}} \! \int_{\mathcal{X}_{m,n}} \! \! \! \! \mathbbm{1}_{\{\|x-y\| \leq h_{x,f}^{-1}(s)\}} f(x)^{2+\kappa_1} g(x)^{\kappa_2} \{h_{x,f}^{-1}(s) M_\beta(x)\}^{2\beta} \,dy \,dx \nonumber \\
     & \lesssim \int_{\mathcal{X}_{m,n}} f(x)^{2+\kappa_1} g(x)^{\kappa_2} h_{x,f}^{-1}(s)^d \{h_{x,f}^{-1}(s) M_\beta(x)\}^{2\beta} \,dx \nonumber \\
     & \lesssim s \int_{\mathcal{X}_{m,n}}f(x)^{1+\kappa_1} g(x)^{\kappa_2} \biggl\{ \frac{s M_\beta(x)^d}{f(x)} \biggr\}^{2\beta/d} \,dx.
\end{align}
Now, similarly, by Lemma~\ref{Lemma:hxinvbounds} we have
\[
    \max \biggl\{ \biggl| \frac{V_d f(x) h_{x,f}^{-1}(s)^d}{s} - 1 \biggr|, \biggl| \frac{V_d f(y) h_{y,f}^{-1}(s)^d}{s} - 1 \biggr| \biggr\} \lesssim \biggl\{ \frac{s M_\beta(x)^d}{f(x)} \biggr\}^{\beta/d}.
\]
It follows that there exist $C,C' > 0$, depending only on $\vartheta$, such that
\begin{align*}
    V_d f(x) h_{y,f}^{-1}(s)^d &\geq \frac{f(x)}{f(y)}\biggl[ s - C s\biggl\{ \frac{sM_\beta(x)^d}{f(x)} \biggr\}^{\beta/d} \biggr] \\
    & \geq \frac{f(x)}{f(y)}\biggl[ V_df(x) h_{x,f}^{-1}(s)^d - 2Cs \biggl\{\frac{s M_\beta(x)^d}{f(x)} \biggr\}^{\beta/d} \biggr] \\
    & \geq V_df(x) h_{x,f}^{-1}(s)^d - 2C's \biggl\{\frac{s M_\beta(x)^d}{f(x)} \biggr\}^{\beta/d},
\end{align*}
where the final bound is from Lemma~\ref{Lemma:15over7}.  Hence,
\begin{align}
\label{Eq:ImprovedTerm2}
    &\biggl| \int_{\mathcal{X}_{m,n}} \! \int_{\mathcal{X}_{m,n}} \! \! \! \{f(y)\!-\!f(x)\} f(y)^2 \phi_{10}\bigl(f(y),g(y)\bigr) \bigl(  \mathbbm{1}_{\{\|x-y\| \leq h_{y,f}^{-1}(s)\}} \!-\! \mathbbm{1}_{\{\|x-y\| \leq h_{x,f}^{-1}(s)\}} \bigr) \,dy \,dx \biggr| \nonumber \\
    & \lesssim \int_{\mathcal{X}_{m,n}} \int_{\mathcal{X}_{m,n}} f(x)^{1+\kappa_1}g(x)^{\kappa_2} |f(y)-f(x)| \mathbbm{1}_{\{ h_{y,f}^{-1}(s) \leq \|x-y\| \leq h_{x,f}^{-1}(s)\}} \,dy \,dx \nonumber \\
    & \lesssim \int_{\mathcal{X}_{m,n}} f(x)^{2+\kappa_1}g(x)^{\kappa_2} M_\beta(x)^{\beta} h_{x,f}^{-1}(s)^\beta \frac{s}{f(x)} \biggl\{ \frac{s M_\beta(x)^d}{f(x)} \biggr\}^{\beta/d} \,dx \nonumber \\
    & \lesssim s \int_{\mathcal{X}_{m,n}} f(x)^{1+\kappa_1} g(x)^{\kappa_2}\biggl\{ \frac{s M_\beta(x)^d}{f(x)} \biggr\}^{2 \beta/d} \,dx.
\end{align}
It now follows from~\eqref{Eq:ImprovedDecomposition},~\eqref{Eq:ImprovedDecomposition2},~\eqref{Eq:ImprovedTerm1} and~\eqref{Eq:ImprovedTerm2} that
\begin{align*}
    &\biggl| \int_{\mathcal{I}_{m,X}} \int_{\mathcal{I}_{n,Y}}  \mathrm{B}_{k_X,m-k_X}(s) \mathrm{B}_{k_Y,n+1-k_Y}(t) \int_{\mathcal{X}_{m,n}} f(x)^2 \phi_{10}\bigl(f(x),g(x)\bigr) \nonumber \\
    &\hspace{250pt}\biggl\{ \frac{V_df(x)h_{x,f}^{-1}(s)^d}{s} - 1 \biggr\} \,dx \,dt \,ds \biggr| \\
    & \lesssim \int_{\mathcal{I}_{m,X}} \int_{\mathcal{I}_{n,Y}}  \! \! \mathrm{B}_{k_X,m-k_X}(s) \mathrm{B}_{k_Y,n+1-k_Y}(t) \int_{\mathcal{X}_{m,n}} f(x)^{1+\kappa_1} g(x)^{\kappa_2} \biggl\{ \frac{s M_\beta(x)^d}{f(x)} \biggr\}^{2 \beta/d} \,dx \,dt \,ds \nonumber \\
    &\hspace{310pt}+ \biggl( \frac{k_X}{m} \biggr)^{\lambda_1(1-\zeta)-\epsilon}\\
    & = O \biggl( \max \biggl\{ \biggl( \frac{k_X}{m} \biggr)^{2\beta/d}, \biggl( \frac{k_X}{m} \biggr)^{\lambda_1(1-\zeta)-\epsilon} \biggr\} \biggr).
\end{align*}
The other terms in the expansion can be dealt with in the same way, and we conclude that
\begin{align*}
    &\mathbb{E} \widetilde{T}_{m,n} =\int_{\mathcal{I}_{m,X}} \int_{\mathcal{I}_{n,Y}} \mathrm{B}_{k_X,m-k_X}(s) \mathrm{B}_{k_Y,n+1-k_Y}(t) \\
    &\times \int_\mathcal{X} \sum_{\ell_1=0}^{\beta^*-1} \sum_{\ell_2=0}^{\beta^*-1-\ell_1} \frac{f(x)^{1+\ell_1} g(x)^{\ell_2} \phi_{\ell_1 \ell_2}\bigl(f(x),g(x)\bigr)}{\ell_1! \ell_2!}\biggl( \frac{k_X}{ms} - 1 \biggr)^{\ell_1} \biggl( \frac{k_Y}{nt} - 1 \biggr)^{\ell_2} \,dx \,dt \,ds \\
     & + O\biggl( \max\biggl\{ \Bigl( \frac{k_X}{m} \Bigr)^{\lambda_1(1-\zeta)-\epsilon}, \Bigl( \frac{k_Y}{n} \Bigr)^{\lambda_2(1-\zeta)-\epsilon}, \Bigl( \frac{k_X}{m} \Bigr)^{2\beta/d}, \Bigl( \frac{k_Y}{n} \Bigr)^{2\beta/d}, k_X^{-\beta_1^*/2}, k_Y^{-\beta_2^*/2} \biggr\} \biggr).
\end{align*}
The result therefore follows from~\eqref{Eq:Series}.
\end{proof}

\subsection{Proof of Proposition~\ref{Prop:Variance} on asymptotic variance}
  
\begin{proof}[Proof of Proposition~\ref{Prop:Variance}]
We initially consider the unweighted estimator $\widetilde{T}_{m,n}$, deferring the extension to the weighted estimator $\widehat{T}_{m,n}^{w_X,w_Y}$ to the end of the proof.  We start by writing
\begin{align}
\label{Eq:VarDecomp}
	\mathrm{Var} (\widetilde{T}_{m,n})  = \frac{1}{m} \mathrm{Var} \ \phi & \bigl( \hat{f}_{(k_X),1} ,\hat{g}_{(k_Y),1} \bigr) \nonumber \\
	& + \Bigl( 1- \frac{1}{m} \Bigr) \mathrm{Cov}\Bigl( \phi \bigl(\hat{f}_{(k_X),1},\hat{g}_{(k_Y),1} \bigr) , \phi \bigl(\hat{f}_{(k_X),2},\hat{g}_{(k_Y),2} \bigr) \Bigr).
\end{align}
Taking $\mathcal{X}_{m,n}, \mathcal{I}_{m,X}$ and $\mathcal{I}_{n,Y}$ as defined in the proof of Proposition~\ref{Thm:GeneralBias}, and letting $S_1$, $S_2$ and $S_3$ be error terms, we now write
\begin{align*}
	&\mathbb{E} \bigl\{ \phi \bigl( \hat{f}_{(k_X),1} ,\hat{g}_{(k_Y),1} \bigr)^2 \bigr\} \nonumber \\
	&=\int_{\mathcal{X}_{m,n}} \! \! f(x)\int_{\mathcal{I}_{m,X}} \int_{\mathcal{I}_{n,Y}} \! \! \phi ( u_{x,s} , v_{x,t} )^2 \mathrm{B}_{k_X,m-k_X}(s) \mathrm{B}_{k_Y,n+1-k_Y}(t) \,ds \,dt \,dx + S_1  \nonumber \\
	&=\int_{\mathcal{X}_{m,n}} \! \! f(x) \int_{\mathcal{I}_{m,X}} \int_{\mathcal{I}_{n,Y}} \! \! \phi \Bigl( \frac{k_X f(x)}{m s} , \frac{k_Y g(x)}{n t} \Bigr)^2 \nonumber \\
	& \hspace{130pt} \times  \mathrm{B}_{k_X,m-k_X}(s) \mathrm{B}_{k_Y,n+1-k_Y}(t) \,ds \,dt \,dx + S_1 + S_2\nonumber \\
	&= \mathbb{E}\bigl\{(\phi_{X_1})^2\bigr\} + \sum_{j=1}^3 S_j.
\end{align*}
We show in Section~\ref{Sec:RemainderTerms} that
\begin{align*}
	&\sum_{j=1}^3 S_j = O \biggl( \max \biggl\{ \Bigl( \frac{k_X}{m} \Bigr)^{\lambda_1(1-2\zeta)- \epsilon}, \Bigl( \frac{k_Y}{n} \Bigr)^{\lambda_2(1-2\zeta)- \epsilon} ,\Bigl(\frac{k_X}{m} \Bigr)^{\frac{2 \wedge \beta}{d}}, \Bigl(\frac{k_Y}{n} \Bigr)^{\frac{2 \wedge \beta}{d}},  k_X^{-\frac{1}{2}} , k_Y^{-\frac{1}{2}} \biggr\} \biggr)
\end{align*}
for every $\epsilon > 0$.  Using Proposition~\ref{Thm:GeneralBias} we can now see that
\begin{equation}
  \label{Eq:Diagonal}
	\biggl| \frac{1}{m} \mathrm{Var} \ \phi \bigl( \hat{f}_{(k_X),1} ,\hat{g}_{(k_Y),1}  \bigr) - \frac{\mathrm{Var}(\phi_{X_1})}{m} \biggr| = o \biggl( \frac{1}{m} \biggr).
\end{equation}

We now turn to the second term in~\eqref{Eq:VarDecomp}. Let $F_{m,n,x,y}: [0,1]^4 \rightarrow [0,1]$ denote the conditional distribution function of
\begin{align*}
	&\bigl( h_{x,f}(\rho_{(k_X),1,X}), h_{y,f}(\rho_{(k_X),2,X}), h_{x,g}(\rho_{(k_Y),1,Y}), h_{y,g}(\rho_{(k_Y),2,Y}) \bigl) | X_1=x, X_2=y.
\end{align*}
Moreover, for $s_1,s_2,t_1,t_2 \in [0,1]$ such that $s_1+s_2 \leq 1$ and $t_1+t_2 \leq 1$ define
\begin{align*}
  G_{m}^{(1)}(s_1,s_2) &:= \int_0^{s_1} \int_0^{s_2} \mathrm{B}_{k_X,k_X,m-2k_X-1}(u_1,u_2) \,du_1 \,du_2 \\
  G_{n}^{(2)}(t_1,t_2) &:= \int_0^{t_1} \int_0^{t_2} \mathrm{B}_{k_Y,k_Y,n-2k_Y+1}(v_1,v_2) \,dv_1 \,dv_2 \\
	G_{m,n}(s_1,s_2,t_1,t_2)&:=G_{m}^{(1)}(s_1,s_2) G_{n}^{(2)}(t_1,t_2),
\end{align*}
so that we have $F_{m,n,x,y}(s_1,s_2,t_1,t_2)=G_{m,n}(s_1,s_2,t_1,t_2)$ for $s_1,s_2,t_1,t_2,x$ and $y$ such that $\|x-y\| > \max\bigl(h_{x,f}^{-1}(s_1)+h_{y,f}^{-1}(s_2), h_{x,g}^{-1}(t_1)+h_{y,g}^{-1}(t_2)\bigr)$. We will also use the shorthand $h(s_1,s_2,t_1,t_2) := \phi ( u_{x,s_1} , v_{x,t_1} ) \phi ( u_{y,s_2} , v_{y,t_2} )$ and
\begin{align*}
    &H_{m}^{(1)}(s_1,s_2) := G_{m}^{(1)}(s_1,s_2) - \int_0^{s_1}\int_0^{s_2} \mathrm{B}_{k_X,m-k_X}(u_1) \mathrm{B}_{k_X,m-k_X}(u_2) \, du_1 \, du_2 \\
  &H_{n}^{(2)}(t_1,t_2) := G_{n}^{(2)}(t_1,t_2) - \int_0^{t_1}\int_0^{t_2} \mathrm{B}_{k_Y,n+1-k_Y}(v_1) \mathrm{B}_{k_Y,n+1-k_Y}(v_2) \, dv_1 \, dv_2 \\
  &H_{m,n}(s_1,s_2,t_1,t_2):= H_{m}^{(1)}(s_1,s_2)G_{n}^{(2)}(t_1,t_2) +  G_{m}^{(1)}(s_1,s_2)H_{n}^{(2)}(t_1,t_2)  \\
	& \hspace{220pt} -  H_{m}^{(1)}(s_1,s_2)H_{n}^{(2)}(t_1,t_2).
\end{align*}
With this newly-defined notation, we now have
\begin{align}
\label{Eq:CovDecomp}
	\mathrm{Cov}\Bigl( \phi &\bigl( \hat{f}_{(k_X),1} ,\hat{g}_{(k_Y),1} \bigr) , \phi \bigl( \hat{f}_{(k_X),2} ,\hat{g}_{(k_Y),2}  \bigr) \Bigr)\nonumber  \\
	&= \int_{\mathcal{X} \times \mathcal{X}} \!\!\! f(x)f(y) \int_{[0,1]^4} h(s_1,s_2,t_1,t_2) \bigl\{dF_{m,n,x,y}(s_1,s_2,t_1,t_2)\nonumber \\
	& \hspace{100pt} - d(H_m^{(1)}-G_m^{(1)})(s_1,s_2) d(H_n^{(2)}-G_n^{(2)})(t_1,t_2) \bigr\} \,dx \,dy \nonumber \\
	&= \int_{\mathcal{X} \times \mathcal{X}} \!\!\!f(x)f(y) \int_{\mathcal{I}_{m,X}^2} \int_{\mathcal{I}_{n,Y}^2} h(s_1,s_2,t_1,t_2) \bigl\{ d(F_{m,n,x,y} - G_{m,n})(s_1,s_2,t_1,t_2)\nonumber  \\
	& \hspace{100pt} + dH_{m,n}(s_1,s_2,t_1,t_2) \bigr\} \,dx \,dy + o(m^{-2}+n^{-2}) ,
\end{align}
where the bound on the final term follows from the fact that $\zeta<1/2$, Lemma~\ref{Lemma:Acomplement} and Cauchy--Schwarz. We first study the second term in this expansion. The intuition behind the following expansion is that, when $X_1$ and $X_2$ do not share nearest neighbours, the dependence between $(\hat{f}_{(k_X),1}, \hat{g}_{(k_Y),1})$ and $(\hat{f}_{(k_X),2}, \hat{g}_{(k_Y),2})$ is relatively weak, and we may expand the functions $\phi, h_{x,f}^{-1},h_{x,g}^{-1}$ as in the proof of Proposition~\ref{Thm:GeneralBias} and approximate integrals. We therefore make use of the shorthand
\begin{align*}
	h^{(1)}(s_1,s_2,t_1,t_2)&:=\Bigl\{ \phi \bigl(f(x), v_{x,t_1} \bigr) + \Bigl( \frac{k_X}{ms_1}-1 \Bigr) f(x) \phi_{10}\bigl( f(x), v_{x,t_1} \bigr) \Bigr\} \\
	& \hspace{80pt} \times \Bigl\{ \phi \bigl(f(y), v_{y,t_2} \bigr) + \Bigl( \frac{k_X}{ms_2}-1 \Bigr) f(y) \phi_{10}\bigl( f(y), v_{y,t_2} \bigr) \Bigr\} \\
	h^{(2)}(s_1,s_2,t_1,t_2)&:=\Bigl\{ \phi \bigl( u_{x,s_1}, g(x) \bigr) + \Bigl( \frac{k_Y}{nt_1}-1 \Bigr) g(x) \phi_{01}\bigl( u_{x,s_1}, g(x) \bigr) \Bigr\} \\
	&\hspace{80pt} \times \Bigl\{ \phi \bigl(u_{y,s_2}, g(y) \bigr) + \Bigl( \frac{k_Y}{nt_2}-1 \Bigr) g(y) \phi_{01}\bigl( u_{y,t_2},g(y) \bigr) \Bigr\}
\end{align*}
for linearised versions of $h$.  We also write, for example,
\[
  (h \, dH_m^{(1)} \, dG_{n}^{(2)})(s_1,s_2,t_1,t_2) := h(s_1,s_2,t_1,t_2) \, dH_{m}^{(1)}(s_1,s_2) \, dG_{n}^{(1)}(t_1,t_2).
\]
Writing $T_1$, $T_2$ and $T_3$ for error terms, we therefore have
\begin{align}
  \label{Eq:Ts}
	&\int_{\mathcal{X}^2} f(x)f(y)\int_{\mathcal{I}_{m,X}^2} \int_{\mathcal{I}_{n,Y}^2} (h \, dH_{m,n})(s_1,s_2,t_1,t_2)  \,dx \,dy \nonumber \\
	&= \int_{\mathcal{X}_{m,f}^2} f(x)f(y)  \int_{\mathcal{I}_{m,X}^2} \int_{\mathcal{I}_{n,Y}^2} (h \, dH_m^{(1)} \, dG_{n}^{(2)})(s_1,s_2,t_1,t_2) \,dx \,dy + T_1 \nonumber \\
	& \hspace{50pt} +  \int_{\mathcal{X}_{n,g}^2} f(x)f(y) \int_{\mathcal{I}_{m,X}^2} \int_{\mathcal{I}_{n,Y}^2} (h \, d(G_m^{(1)}-H_m^{(1)}) \, dH_n^{(2)})(s_1,s_2,t_1,t_2)  \,dx \,dy   \nonumber \\
	& = \int_{\mathcal{X}_{m,f}^2} f(x)f(y)  \int_{\mathcal{I}_{m,X}^2} \int_{\mathcal{I}_{n,Y}^2} (h^{(1)}\, dH_m^{(1)} \, dG_{n}^{(2)})(s_1,s_2,t_1,t_2) \,dx \,dy + T_1 + T_2 \nonumber \\
	& \hspace{50pt} + \int_{\mathcal{X}_{n,g}^2} f(x)f(y) \int_{\mathcal{I}_{m,X}^2} \int_{\mathcal{I}_{n,Y}^2} (h^{(2)} d(G_m^{(1)}-H_m^{(1)}) dH_n^{(2)})(s_1,s_2,t_1,t_2) \,dx \,dy \nonumber \\
	& = -\frac{1}{m}  \int_{\mathcal{X}_{m,f}^2} f(x)f(y) \int_{\mathcal{I}_{n,Y}^2} \Bigl\{ 2f(x) \phi_{10} \bigl(f(x), v_{x,t_1} \bigr) \phi \bigl(f(y),v_{y,t_2} \bigr) \nonumber \\
	& \hspace{90pt} + f(x) \phi_{10} \bigl(f(x), v_{x,t_1} \bigr)f(y) \phi_{10} \bigl(f(y), v_{y,t_2} \bigr) \Bigr\} dG_{n}^{(2)}(t_1,t_2) \,dx \,dy \nonumber \\
	& \hspace{50pt} - \frac{1}{n} \int_{\mathcal{X}_{n,g}^2} f(x)f(y)  \int_{\mathcal{I}_{m,X}^2} g(x) \phi_{01} \bigl(u_{x,s_1},g(x) \bigr) g(y) \phi_{01} \bigl( u_{y,s_2}, g(y) \bigr) \nonumber \\
	& \hspace{90pt} \times d(G_m^{(1)}-H_m^{(1)})(s_1,s_2) \,dx \,dy +T_1+T_2+T_3  \nonumber \\
	&= -\frac{2}{m} \mathbb{E}\bigl\{(f \phi_{10})_{X_1}\bigr\} \mathbb{E}(\phi_{X_1}) - \frac{1}{m} \bigl\{\mathbb{E}(f \phi_{10})_{X_1}\bigr\}^2 - \frac{1}{n} \bigl\{\mathbb{E}(g \phi_{01})_{X_1}\bigr\}^2  \nonumber \\
	& \hspace{200pt} + T_1 + T_2 +T_3 + o(1/m+1/n),
\end{align}
where the bound on the final term follows from~\eqref{Eq:epsilonn}, Lemma~\ref{Prop:FunctionalClasses}(i), Lemma~\ref{Lemma:BetaTailBounds} and tail bounds similar to~\eqref{Eq:ftailerror}. We show in Section~\ref{Sec:RemainderTerms} that
\begin{align}
\label{Eq:TBound}
	\sum_{j=1}^3 T_j = O \biggl(& \max \biggl\{ \Bigl( \frac{k_X}{m} \Bigr)^{1+\lambda_1(1- \zeta)-\epsilon} , \Bigl( \frac{k_Y}{n} \Bigr)^{1+\lambda_2(1-\zeta)-\epsilon} , \frac{k_X^{\frac{1}{2}+\frac{2 \wedge \beta}{d}}}{m^{1+\frac{2 \wedge \beta}{d}}},  \frac{k_Y^{\frac{1}{2} +\frac{2 \wedge \beta}{d}}}{n^{1+\frac{2 \wedge \beta}{d}}},  \nonumber \\
	&\Bigl( \frac{k_X}{m} \Bigr)^{1+\frac{2(2 \wedge \beta)}{d}},  \Bigl( \frac{k_Y}{n} \Bigr)^{1+\frac{2(2 \wedge \beta)}{d}}, \frac{\log m}{m k_X^{\frac{1}{2}}} , \frac{\log n}{n k_Y^{\frac{1}{2}}} \biggr\} \biggr) + o(1/m+1/n).
\end{align}

We now consider the contribution of the first term in~\eqref{Eq:CovDecomp}. In Section~\ref{Sec:RemainderTerms}, we show that
\begin{align}
  \label{Eq:U0}
	U_0 &:= \int_{\mathcal{X} \times \mathcal{X}_{m,n}^c}  f(x)f(y) \int_{\mathcal{I}_{m,X}^2} \int_{\mathcal{I}_{n,Y}^2} h(s_1,s_2,t_1,t_2) \nonumber \\
	& \hspace{150pt} \times d(F_{m,n,x,y}-G_{m,n})(s_1,s_2,t_1,t_2) \, dx \,dy \nonumber \\
	&\phantom{:}= O \biggl( \max \biggl\{ \Bigl(\frac{k_X}{m} \Bigr)^{2\lambda_1(1-\zeta)- \epsilon}, \Bigl(\frac{k_Y}{n} \Bigr)^{2\lambda_2(1-\zeta) - \epsilon} \biggr\} \biggr),
\end{align}
so that we may restrict attention to $x \in \mathcal{X}_{m,n}$, in which case $F_{m,n,x,y}-G_{m,n}$ is only non-zero when $x$ and $y$ are close and we may approximate $f(y) \approx f(x)$ and $g(y) \approx g(x)$.  Let $p_\cap^{(1)}:=\int_{B_x(h_{x,f}^{-1}(s_1) )\cap B_y(h_{y,f}^{-1}(s_2))} f(w) \,dw$ and $p_\cap^{(2)}:=\int_{B_x(h_{x,g}^{-1}(t_1) )\cap B_y(h_{y,g}^{-1}(t_2))} g(w) \,dw$, and let
\begin{align}
  \label{Eq:Multivector}
	(N_1^{(1)},N_2^{(1)},N_3^{(1)},N_4^{(1)}) &\sim \! \mathrm{Multi} \bigl(m \!- \!2; s_1-p_\cap^{(1)}; s_2 - p_\cap^{(1)}, p_\cap^{(1)},1-s_1-s_2 + p_\cap^{(1)} ) \\
	(N_1^{(2)},N_2^{(2)},N_3^{(2)},N_4^{(2)}) &\sim \! \mathrm{Multi} \bigl(n; t_1-p_\cap^{(2)}; t_2 - p_\cap^{(2)},p_\cap^{(2)}, 1-t_1-t_2 + p_\cap^{(2)} ). \nonumber
\end{align}
Now set
\begin{align}
  \label{Eq:MultinomialExpression}
  F_{m,x,y}^{(1)}(s_1,s_2) &:= \mathbb{P}(N_1^{(1)}+N_3^{(1)} \geq k_X - \mathbbm{1}_{\{\|x-y\| \leq h_{x,f}^{-1}(s_1)\}} , \nonumber \\
	& \hspace{75pt} N_2^{(1)}+N_3^{(1)} \geq k_X - \mathbbm{1}_{\{\|x-y\| \leq h_{y,f}^{-1}(s_2)\}} ) \nonumber \\
  F_{n,x,y}^{(2)}(t_1,t_2) &:= \mathbb{P}(N_1^{(2)}+N_3^{(2)} \geq k_Y , N_2^{(2)}+N_3^{(2)} \geq k_Y  ),
 \end{align}
so that $F_{m,n,x,y}(s_1,s_2,t_1,t_2) = F_{m,x,y}^{(1)}(s_1,s_2)F_{n,x,y}^{(2)}(t_1,t_2)$.
We use the decomposition
\begin{align}
\label{Eq:IBPDecomp}
	& F_{m,n,x,y}-G_{m,n} = F_{m,x,y}^{(1)}F_{n,x,y}^{(2)}-G_{m}^{(1)}G_{n}^{(2)} \nonumber \\
	& \hspace{10pt} = (F_{m,x,y}^{(1)} \!-\!G_{m}^{(1)})(F_{n,x,y}^{(2)}\!-\!G_{n}^{(2)}) + (F_{m,x,y}^{(1)}\!-\!G_{m}^{(1)})G_{n}^{(2)} + G_{m}^{(1)}(F_{n,x,y}^{(2)}\!-\!G_{n}^{(2)}),
\end{align}
so that each term is of product form and involves at least one of the marginal errors.  We will see that the first term is asymptotically negligible, while the second and third terms can be studied through the normal approximation given in Lemma~\ref{Lemma:NormalApproximation}. For a general distribution function $F$, for $a_{-} \leq a_{+}$ and for a smooth $h:[a_{-},a_{+}]^2 \rightarrow \mathbb{R}$ with first partial derivatives $h_{10}$, $h_{10}$ and mixed second partial derivative $h_{11}$, we will use the formula
\begin{align}
\label{Eq:IBPformula}
& \int_{[a_{-},a_{+}]^2} (h \,dF)(u,v) - \int_{a_{-}}^{a_{+}} \int_{a_{-}}^{a_{+}} (h_{11}F(u,v)) \,du \,dv  \nonumber \\
&= \! \int_{a_{-}}^{a_{+}} \! \bigl[ (h_{10}F)(u,a_{-})-(h_{10}F)(u,a_{+}) \bigr] du +  (hF)(a_{-},a_{-}) -(hF)(a_{+},a_{-}) \nonumber \\
	& \hspace{20pt} + \int_{a_{-}}^{a_{+}}\! \bigl[ (h_{01}F)(a_{-},v)-(h_{01}F)(a_{+},v) \bigr] dv+(hF)(a_{+},a_{+}) -(hF)(a_{-},a_{+}).
\end{align}
We now deal with each of the three terms on the right-hand side of~\eqref{Eq:IBPDecomp} in turn, starting with $F=F^{(1)}F^{(2)}=(F_{m,x,y}^{(1)}-G_{m}^{(1)})(F_{n,x,y}^{(2)}-G_{n}^{(2)})$. For remainder terms $U_1$, $U_2$ and $U_3$ to be bounded later, we write
\begin{align}
\label{Eq:IBP1}
	&\int_{\mathcal{X} \times \mathcal{X}_{m,n}} \!\!\!\! f(x)f(y) \int_{\mathcal{I}_{m,X}^2} \int_{\mathcal{I}_{n,Y}^2} (h \, dF)(s_1,s_2,t_1,t_2) \, dx \,dy \nonumber \\
	& = \int_{\mathcal{X} \times \mathcal{X}_{m,n}} \!\!\!\! f(x)f(y) \int_{\mathcal{I}_{m,X}^2} \int_{\mathcal{I}_{n,Y}^2} \bigl(h_{0011} \, (dF^{(1)}) \, F^{(2)}\bigr)(s_1,s_2,t_1,t_2) \,dt_1 \,dt_2 \,dx \,dy  +U_1 \nonumber \\
	& = \int_{\mathcal{X} \times \mathcal{X}_{m,n}} \!\!\!\! f(x)f(y) \int_{\mathcal{I}_{n,Y}^2} F^{(2)}(t_1,t_2) \biggl\{\int_{\mathcal{I}_{m,X}^2} (h_{1111} F^{(1)})(s_1,s_2,t_1,t_2)  \, ds_1 \,ds_2\nonumber \\
	&\hspace{10pt} - \int_{\mathcal{I}_{m,X}} (h_{1011}F^{(1)})(s_1,a_{m,X}^+,t_1,t_2) \, ds_1  \nonumber \\
	& \hspace{50pt} - \int_{\mathcal{I}_{m,X}} (h_{0111}F^{(1)})(a_{m,X}^+,s_2,t_1,t_2) \, ds_2 \biggr\} \,dt_1 \,dt_2 \,dx \,dy +U_1+U_2 \nonumber\\
	&= \sum_{j=1}^3 U_j. 
\end{align}
We show in Section~\ref{Sec:RemainderTerms} that
\begin{equation}
\label{Eq:UBounds1}
	\sum_{j=1}^3 U_j = O \biggl( \max \biggl\{ \frac{1}{m^{2}}, \frac{1}{n^{2}}, \frac{ \log^2 m }{mk_X}, \frac{\log^2 n}{nk_Y} \biggr\} \biggr). 
\end{equation}
We next consider $F=F^{(1)}F^{(2)}=(F_{m,x,y}^{(1)}-G_{m}^{(1)})G_{n}^{(2)}$, and recall from Lemma~\ref{Lemma:NormalApproximation} that $\alpha_z =  \mu_d \bigl( B_0(1) \cap B_z(1) \bigr)/V_d$, that
\[
	\Sigma= \begin{pmatrix} 1 & \alpha_z \\ \alpha_z & 1 \end{pmatrix},
\]
and the definitions of the normal distribution functions $\Phi_{I_2}$ and $\Phi_{\Sigma}$.  Then, for remainder terms $U_4,U_5,U_6$ to be bounded below, we use the change of variables $y=x+(\frac{k_X}{mV_d f(x)})^{1/d}z$ and the approximation $\frac{\partial}{\partial s} \phi(u_{x,s},v_{x,t}) \approx f(x) \phi_{10}(f(x),g(x))/s$  to write
\begin{align}
\label{Eq:IBP2}
	&\int_{\mathcal{X} \times \mathcal{X}_{m,n}} f(x)f(y) \int_{\mathcal{I}_{m,X}^2} \int_{\mathcal{I}_{n,Y}^2} (h \, dF)(s_1,s_2,t_1,t_2) \, dx \,dy \nonumber \\
	& \hspace{40pt} = \int_{\mathcal{X} \times \mathcal{X}_{m,n}} f(x)f(y) \int_{\mathcal{I}_{n,Y}^2} \biggl\{ \int_{\mathcal{I}_{m,X}^2} (h_{1100} \, F^{(1)})(s_1,s_2,t_1,t_2) \,ds_1 \,ds_2 \nonumber \\
	& \hspace{100pt} -  \int_{\mathcal{I}_{m,X}}  (h_{1000}F^{(1)})(s_1,a_{m,X}^+,t_1,t_2) \, ds_1 \nonumber \\
	& \hspace{100pt} -  \int_{\mathcal{I}_{m,X}}  (h_{0100}F^{(1)})(a_{m,X}^+,s_2,t_1,t_2) \, ds_2 \biggr\} dG_n^{(2)}(t_1,t_2) \,dx \,dy + U_4 \nonumber \\
	& \hspace{40pt}= \frac{1}{m V_d} \int_{\mathcal{X}_{m,n}} f(x) \int_{\mathbb{R}^d}\biggl\{ (f \phi_{10})_x^2 \int_{\mathbb{R}^2} (\Phi_{\Sigma}-\Phi_{I_2})(u_1,u_2) \,du_1 \,du_2  \nonumber \\
	& \hspace{200pt} + 2 (f \phi_{10})_x \phi_x \mathbbm{1}_{\{\|z\| \leq 1 \}} \biggr\} \,dz \,dx + U_4 + U_5 \nonumber \\
	&\hspace{40pt}= \frac{1}{m} \mathbb{E}\bigl\{ (f \phi_{10})_{X_1}^2 \bigr\} + \frac{2}{m} \mathbb{E}\bigl\{ (f\phi_{10})_{X_1} \phi_{X_1}\bigr\} + \sum_{j=4}^6 U_j.
\end{align}
We show in Section~\ref{Sec:RemainderTerms} that
\begin{align}
\label{Eq:UBounds2}
	\sum_{j=4}^6 U_j = O \biggl( \frac{1}{m} \max \biggl\{ \frac{\log^{\frac{5}{2}} m}{k_X^{1/2}}, \frac{\log^{\frac{1}{2}} n}{k_Y^{1/2}}, \log^2m &\Bigl(\frac{k_X}{m} \Bigr)^{\frac{1 \wedge \beta}{d}}, \Bigl( \frac{k_X}{m} \Bigr)^{ \lambda_1( 1-2\zeta)- \epsilon}, \nonumber \\
	& \Bigl( \frac{k_Y}{n} \Bigr)^{ \lambda_2( 1-2\zeta)- \epsilon}, \Bigl(\frac{k_Y}{n} \Bigr)^{\frac{2 \wedge \beta}{d}}\biggr\} \biggr)
\end{align}
for every $\epsilon>0$. The final term in \eqref{Eq:IBPDecomp} can be approximated by writing $F=F^{(1)}F^{(2)}=G_{m}^{(1)}(F_{n,x,y}^{(2)}-G_{n}^{(2)})$, using the changes of variables $y=x+(\frac{k_Y}{nV_dg(x)})^{1/d}z, t_i=(k_Y+k_Yv_i)/n$ for $i=1,2$ and using the approximation $\frac{\partial}{\partial t} \phi(u_{x,s},v_{x,t}) \approx g(x) \phi_{01}(f(x),g(x))/t$ to write
\begin{align}
\label{Eq:IBP3}
	&\int_{\mathcal{X} \times \mathcal{X}_{m,n}} f(x)f(y) \int_{\mathcal{I}_{m,X}^2} \int_{\mathcal{I}_{n,Y}^2} (h \, dF)(s_1,s_2,t_1,t_2) \, dx \,dy \nonumber \\
	& = \int_{\mathcal{X} \times \mathcal{X}_{m,n}} \!\!\!\! f(x) f(y) \!\! \int_{\mathcal{I}_{m,X}^2}\! \int_{\mathcal{I}_{n,Y}^2}  \!\!\!\!(h_{0011} \, dG_{m}^{(1)} F^{(2)})(s_1,s_2,t_1,t_2) \,dt_1 \,dt_2 \,dx \,dy + U_7 \nonumber \\
	& = \frac{1}{nV_d} \int_{\mathcal{X}_{m,n}} g(x) (f \phi_{01})_x^2 \int_{\mathbb{R}^d}\int_{\mathbb{R}^2} (\Phi_{\Sigma}-\Phi_{I_2})(v_1,v_2) \,dv_1 \,dv_2 \,dz \,dx + U_7 + U_8 \nonumber \\
	& = \frac{1}{n} \int_{\mathcal{X}_{m,n}} g(x) (f \phi_{01})_x^2 \,dx + U_7 +U_8.
\end{align}
Let $\epsilon_0=\epsilon_0(\lambda_1,\lambda_2,\kappa_1,\kappa_2,C) \in (0,\lambda_1 \wedge \lambda_2)$ be sufficiently small that
\[
	\frac{2+2\kappa_1 - \epsilon/(\lambda_1 \wedge \lambda_2)}{1-\epsilon_0/(\lambda_1 \wedge \lambda_2)} > 2 + 2\kappa_1 - 1/C \quad \text{and} \quad \frac{2 \kappa_2 -1}{1-\epsilon_0/(\lambda_1 \wedge \lambda_2)} > 2\kappa_2 -1 -1/C.
\]
Then, by H\"older's inequality, we have that
\begin{align}
\label{Eq:Holderv2}
	\sup_{(f,g) \in \mathcal{F}_{d,\vartheta}} & \int_\mathcal{X} f(x)^{2+2\kappa_1} g(x)^{2 \kappa_2-1}  \biggl[ \Bigl\{ \frac{M_\beta(x)^d}{f(x)} \Bigr\}^{\epsilon_0} + \Bigl\{ \frac{M_\beta(x)^d}{g(x)} \Bigr\}^{\epsilon_0} \biggr] \,dx \nonumber \\
	& \leq 2 \sup_{(f,g) \in \mathcal{F}_{d,\vartheta}}  \max_{i=1,2} C^{\epsilon_0/\lambda_i} \biggl[ \int_\mathcal{X} f(x)^\frac{2 + 2\kappa_1 - \epsilon_0/\lambda_i}{1-\epsilon_0/\lambda_i} g(x)^\frac{2 \kappa_2 -1}{1-\epsilon_0/\lambda_i} \,dx \biggr]^{1-\epsilon_0/\lambda_i} < \infty.
\end{align}
It follows that
\begin{align}
\label{Eq:UseIntCondition}
	&\int_{\mathcal{X}_{m,n}^c} g(x) (f \phi_{01})_x^2 \,dx \lesssim \int_{\mathcal{X}_{m,n}^c}  f(x)^{2+2 \kappa_1} g(x)^{-1+2\kappa_2} \,dx \nonumber \\
	&\leq \int_{\mathcal{X}} f(x)^{2 + 2\kappa_1} g(x)^{2 \kappa_2- 1} \biggl[\Bigl\{ \frac{ k_X \log mM_{\beta}(x)^d}{m f(x)} \Bigr\}^{\epsilon_0} + \Bigl\{ \frac{ k_Y \log n M_{\beta}(x)^d}{n g(x)} \Bigr\}^{\epsilon_0}\biggr] \,dx \nonumber \\
	& = O \biggl( \max \biggl\{ \Bigl( \frac{k_X \log m}{m} \Bigr)^{\epsilon_0},  \Bigl( \frac{k_Y \log n}{n} \Bigr)^{\epsilon_0} \biggr\} \biggr)
\end{align}
We show in Section~\ref{Sec:RemainderTerms} that
\begin{align}
  \label{Eq:UBounds3}
	U_7+U_8 = O \biggl( \frac{1}{n} \max\biggl\{&\frac{\log^{5/2}n}{k_Y^{1/2}} , \log^2 n \Bigl( \frac{k_Y}{n} \Bigr)^{(1 \wedge \beta)/d}, \Bigl( \frac{k_Y}{n} \Bigr)^{\epsilon_0/2},   \nonumber \\
	& \hspace{50pt} \frac{\log^{1/2}m}{k_X^{1/2}}, \Bigl( \frac{k_X}{m} \Bigr)^{(2 \wedge \beta)/d},  \Bigl( \frac{k_X}{m} \Bigr)^{\epsilon_0/2}  \biggr\} \biggr).
\end{align}
It now follows from \eqref{Eq:VarDecomp}, \eqref{Eq:Diagonal}, \eqref{Eq:CovDecomp}, \eqref{Eq:Ts}, \eqref{Eq:TBound}, \eqref{Eq:U0}, \eqref{Eq:IBP1}, \eqref{Eq:UBounds1}, \eqref{Eq:IBP2}, \eqref{Eq:UBounds2}, \eqref{Eq:IBP3}, \eqref{Eq:UseIntCondition} and~\eqref{Eq:UBounds3} that
\begin{align*}
  \mathrm{Var}(\widetilde{T}_{m,n}) &= \frac{1}{m} \Bigl[ \mathrm{Var}(\phi_{X_1}) - 2 \mathbb{E}\bigl\{(f \phi_{10})_{X_1}\bigr\} \mathbb{E}(\phi_{X_1}) - \bigl\{\mathbb{E}(f \phi_{10})_{X_1}\bigr\}^2 + \mathbb{E}\bigl\{ (f \phi_{10})_{X_1}^2 \bigr\} \\
  &\hspace{17pt}+ 2 \mathbb{E}\bigl\{ (f\phi_{10} \phi)_{X_1}\bigr\} \Bigr] + \frac{1}{n}\Bigl[\mathbb{E}\bigl\{(f \phi_{01})_{Y_1}^2\bigr\} - \bigl\{\mathbb{E}(g \phi_{01})_{X_1}\bigr\}^2\Bigr] + o(1/m + 1/n) \\
	&= \frac{v_1}{m} + \frac{v_2}{n} + o(1/m + 1/n).
\end{align*}
For the general, weighted case, we rely on the decomposition
\begin{align}
  \label{Eq:WeightedVarDecomp}
  \mathrm{Var}(\widehat{T}_{m,n}) = \sum_{j_X,\ell_X=1}^{k_X} \sum_{j_Y,\ell_Y=1}^{k_Y} & w_{X,j_X}w_{X,\ell_X}w_{Y,j_Y}w_{Y,\ell_Y} \nonumber \\
  &\times\biggl\{\frac{1}{m}\mathrm{Cov}\bigl(\phi(\hat{f}_{(j_X),1},\hat{g}_{(j_Y),1}) , \phi(\hat{f}_{(\ell_X),1},\hat{g}_{(\ell_Y),1})\bigr)  \nonumber \\
  &+\Bigl(1-\frac{1}{m}\Bigr)\mathrm{Cov}\bigl(\phi(\hat{f}_{(j_X),1},\hat{g}_{(j_Y),1}) , \phi(\hat{f}_{(\ell_X),2},\hat{g}_{(\ell_Y),2})\biggr\}.
\end{align}
Now, for example, when $\ell_X > j_X$, we have
\begin{align*}
  \bigl(h_{x,f}(\rho_{(j_X),1,X}),h_{x,f}(\rho_{(\ell_X),1,X}),1 \!-\! h_{x,f}(\rho_{(\ell_X),1,X})\bigr)|X_1=x\sim \mathrm{Dir}(j_X,\ell_X-j_X,m-\ell_X),
\end{align*}
and it may therefore be deduced similarly to the arguments leading to~\eqref{Eq:Diagonal} that
\begin{align}
  \label{Eq:WeightedOn}
  \max_{\substack{j_X,\ell_X:w_{X,j_X},w_{X,\ell_X} \neq 0 \\ j_Y,\ell_Y:w_{Y,j_Y},w_{Y,\ell_Y} \neq 0}} \Bigl|\mathrm{Cov}\bigl(\phi(\hat{f}_{(j_X),1},\hat{g}_{(j_Y),1}) , \phi(\hat{f}_{(\ell_X),1},&\hat{g}_{(\ell_Y),1})\bigr)- \mathrm{Var}(\phi_{X_1})\Bigr| =o(1).
\end{align}
The second term on the right-hand side of~\eqref{Eq:WeightedVarDecomp} is handled using relatively small modifications of the arguments used to study the covariance term in~\eqref{Eq:VarDecomp}.  These modifications are required to account for the fact that the $k_X$ that appears twice in the covariance term in~\eqref{Eq:VarDecomp} is now replaced with $j_X$ and $\ell_X$ (with similar changes to $k_Y$).  Thus, for instance, the joint conditional distribution function of
\begin{align*}
\bigl( h_{x,f}(\rho_{(j_X),1,X}), h_{y,f}(\rho_{(\ell_X),2,X}), &h_{x,g}(\rho_{(j_Y),1,Y}), h_{y,g}(\rho_{(\ell_Y),2,Y}) \bigl) | X_1=x, X_2=y,
\end{align*}
is now given by
\begin{align*}
  &F_{m,n,x,y}(s_1,s_2,t_1,t_2) \\
  &= \mathbb{P}\bigl(N_1^{(1)}\!\!+\!N_3^{(1)}\! \geq\! j_X \!-\! \mathbbm{1}_{\{\|x-y\| \leq h_{x,f}^{-1}(s_1)\}} , N_2^{(1)}\!\!+\!N_3^{(1)}\! \geq\! \ell_X \!-\! \mathbbm{1}_{\{\|x-y\| \leq h_{y,f}^{-1}(s_2)\}} \bigr) \nonumber \\
	& \hspace{150pt} \times \mathbb{P}(N_1^{(2)}+N_3^{(2)} \geq j_Y , N_2^{(2)}+N_3^{(2)} \geq \ell_Y  ).
 \end{align*}
 Following the arguments through reveals that
 \begin{align}
   \label{Eq:WeightedOff}
   \max_{\substack{j_X,\ell_X:w_{X,j_X},w_{X,\ell_X} \neq 0 \\ j_Y,\ell_Y:w_{Y,j_Y},w_{Y,\ell_Y} \neq 0}} \Bigl|\mathrm{Cov}\bigl(\phi(\hat{f}_{(j_X),1},&\hat{g}_{(j_Y),1}) , \phi(\hat{f}_{(\ell_X),2},\hat{g}_{(\ell_Y),2})\bigr)- \frac{v_1 - \mathrm{Var}(\phi_{X_1})}{m} - \frac{v_2}{n}\Bigr| \nonumber \\
	& \hspace{150pt} = o(1/m + 1/n).
 \end{align}
 Finally, then, we can deduce from~\eqref{Eq:WeightedVarDecomp},~\eqref{Eq:WeightedOn} and~\eqref{Eq:WeightedOff}, and using our hypotheses on $\|w_X\|_1$ and $\|w_Y\|_1$, that
 \[
   \mathrm{Var}(\widehat{T}_{m,n}) - \frac{v_1}{m} - \frac{v_2}{n} = o\biggl(\Bigl(\frac{1}{m} + \frac{1}{n}\Bigr)\|w_X\|_1^2\|w_Y\|_1^2\biggr) =  o\Bigl(\frac{1}{m} + \frac{1}{n}\Bigr),
 \]
 as required.                                                                     
\end{proof}

\subsection{Proofs of Theorems~\ref{Thm:Normality} and~\ref{Thm:ConfidenceIntervals} on asymptotic normality and confidence intervals}

Since the proof of Theorem~\ref{Thm:Normality} depends on Proposition~\ref{Prop:Poisson}, we prove Proposition~\ref{Prop:Poisson} first.
\begin{proof}[Proof of Proposition~\ref{Prop:Poisson}]
Where it does not cause confusion, we will supress suffices to write $k$ instead of $k_X$ or $k_Y$. For any $\ell \geq \max(k+1,i)$ we use the shorthand
\[
	\hat{f}_{(k),i,\ell} := \frac{k}{ \ell V_d \rho_{(k),i,\ell}^d},
\]
and we write $\phi_x^g(\cdot) := \phi\bigl(\cdot, g(x)\bigr)$. We will first study the difference $T_m^{(1)} - T_m^{(1),\mathrm{p}}$ by bounding its first and second conditional moments given $M$.  On the event that $|m/M-1| \leq 1/L$, when $m \geq (1+1/L)(1+k) \log(em)$, we have that
\begin{align*}
	\mathbb{E}&\{ T_m^{(1)} - T_m^{(1),\mathrm{p}} | M \} \\
	&= \mathbb{E} T_m^{(1)} - \frac{M}{m} \mathbb{E}( T_M^{(1)} | M ) + \frac{M}{m} \mathbb{E} \biggl\{ \phi_{X_1}^g \bigl( \hat{f}_{(k),1,M} \bigr) - \phi_{X_1}^g \biggl( \frac{M}{m} \hat{f}_{(k),1,M} \biggr) \biggm| M \biggr\} \\
	& \hspace{30pt} + \Bigl( \frac{M}{m} -1 \Bigr) \int_\mathcal{X} f(x) \{\phi_x + (f \phi_{10})_x\} \,dx \\
	& = ( \mathbb{E} T_m^{(1)} - T) - \frac{M}{m} \{ \mathbb{E} (T_M^{(1)} |M) - T \} +\Bigl( \frac{M}{m} -1 \Bigr) \int_\mathcal{X} f(x) (f \phi_{10})_x \,dx \\
	& \hspace{30pt} + \frac{M}{m} \int_\mathcal{X} f(x) \int_0^1 \Bigl\{ \phi_x^g \Bigl( \frac{m}{M} u_{x,s} \Bigr) - \phi_x^g \bigl( u_{x,s} \bigr) \Bigr\} \mathrm{B}_{k,M-k}(s) \,ds \,dx \\
	& = \mathbb{E} T_m^{(1)} - \mathbb{E}( T_M^{(1)} | M) + \Bigl( \frac{M}{m} -1 \Bigr) \int_\mathcal{X} f(x) (f \phi_{10})_x \,dx + o\biggl( m^{-1/2} + \Bigl| \frac{M}{m} -1 \Bigr| \biggr) \\
	& \hspace{30pt} + \frac{M}{m} \int_{\mathcal{X}_{m,f}} f(x) \int_{\mathcal{I}_{m,X}} \Bigl\{ \phi_x^g \Bigl( \frac{m}{M} u_{x,s} \Bigr) - \phi_x^g ( u_{x,s}) \Bigr\} \mathrm{B}_{k,M-k}(s) \,ds \,dx \\
	& = \mathbb{E} T_m^{(1)} - \mathbb{E}( T_M^{(1)} | M) + \Bigl( \frac{M}{m} -1 \Bigr) \int_\mathcal{X} f(x) (f \phi_{10})_x \,dx + o\biggl( m^{-1/2} + \Bigl| \frac{M}{m} -1 \Bigr| \biggr) \\
	& \hspace{30pt} + \Bigl( 1 - \frac{M}{m} \Bigr) \int_{\mathcal{X}_{m,f}} f(x) \int_{\mathcal{I}_{m,X}} u_{x,s} \phi_{10}\bigl(u_{x,s} ,g(x)\bigr) \mathrm{B}_{k,M-k}(s) \,ds \,dx \\
	& = \mathbb{E} T_m^{(1)} - \mathbb{E}( T_M^{(1)} | M) + o\biggl( m^{-1/2} + \Bigl| \frac{M}{m} -1 \Bigr| \biggr).
\end{align*}
It now follows from the one-sample ($n = \infty$) version of Proposition~\ref{Thm:GeneralBias} and the fact that, for $a>0$, we have $(k/m)^a - (k/M)^a = o(|M/m-1|)$, that on the event that $|m/M-1| \leq 1/L$ we have
\[
	\mathbb{E}\{ T_m^{(1)} - T_m^{(1),\mathrm{p}} | M \} = o\biggl( m^{-1/2} + \Bigl| \frac{M}{m} -1 \Bigr| \biggr).
\]

We now bound the conditional variance of $T_m^{(1)} - T_m^{(1),\mathrm{p}}$ on the event $A_m := \{|M/m -1| \leq 1/\log (em)\}$. We first see that, when $m \geq (k+1)\log(em)/\{1-1/\log (em)\}$, we have
\begin{align*}
	\mathrm{Var} &\biggl\{ \frac{1}{m} \sum_{i=1}^M \phi_{X_i}^g \biggl( \frac{M}{m} \hat{f}_{(k),i,M} \biggr) - \frac{1}{m} \sum_{i=1}^m \phi_{X_i}^g \biggl( \frac{M}{m} \hat{f}_{(k),i,M} \biggr) \biggm| M \biggr\} \\
	& = \frac{|M-m|}{m^2} \mathrm{Var} \biggl\{ \phi_{X_1}^g \biggl( \frac{M}{m} \hat{f}_{(k),1,M} \biggr) \biggm| M \biggr\} \\
	&\hspace{0.5cm} + \frac{|M-m|(|M-m|-1)}{m^2} \mathrm{Cov} \biggl\{ \phi_{X_1}^g \biggl( \frac{M}{m} \hat{f}_{(k),1,M} \biggr), \phi_{X_2}^g \biggl( \frac{M}{m} \hat{f}_{(k),2,M} \biggr) \biggm| M\biggr\} \\
	& = \frac{(M-m)^2}{m^2} \mathrm{Cov}( \phi_{X_1}, \phi_{X_2} ) + O\biggl( \frac{|M-m|}{m^2} \biggr) + o \biggl( \frac{(M-m)^2}{m^2} \biggr) \\
	& =O\biggl( \frac{|M-m|}{m^2} \biggr) + o \biggl( \frac{(M-m)^2}{m^2} \biggr).
\end{align*}
To bound the conditional variance of $T_m^{(1)} - T_m^{(1),\mathrm{p}}$, it now suffices to bound $\mathrm{Var}(D_m | M)$, where
\[
	D_m:= \frac{1}{m} \sum_{i=1}^m \biggl\{ \phi_{X_i}^g \biggl( \frac{k}{mV_d \rho_{(k),i,m}^d} \biggr) - \phi_{X_i}^g \biggl( \frac{k}{mV_d \rho_{(k),i,M}^d} \biggr) \biggr\}.
\]
To proceed, we will now use the Efron--Stein inequality; see, for example, \citet[][Theorem~3.1]{BGM2013}. Given $M$, the random variable $D_m=D_m(X_1,\ldots,X_M)$ is a function of the independent random variables $X_1,\ldots,X_M$; letting $X_1',\ldots,X_M'$ denote an independent copy of these random variables, for $j=1,\ldots,M$, write $D_m^{(j)}:=D_m(X_1,\ldots,X_{j-1},X_j',X_{j+1},\ldots,X_M)$ for the random variable calculated by replacing $X_j$ in $D_m$ by $X_j'$. Similarly define $\rho_{(k),i,\ell}^{(j)}$. The Efron--Stein inequality gives that
\[
	\mathrm{Var}(D_m |M) \leq \frac{1}{2} \sum_{j=1}^M \mathbb{E}\bigl\{ (D_m - D_m^{(j)})^2 |M \bigr\}.
\]
For now, we will work on the event $\{M>m\}$. Observe that for $i=1,\ldots ,m$ and $j=m+1,\ldots,M$ we have $\rho_{(k),i,M}^{(j)} = \rho_{(k),i,M}$ unless either $X_j$ is one of the $k$ nearest neighbours of $X_i$ in the sample $X_1,\ldots,X_M$ or $X_j'$ is one of the $k$ nearest neighbours of $X_i$ in the sample $X_1,\ldots,X_{j-1}, X_j', X_{j+1}, \ldots, X_M$. For $j=m+1,\ldots,M$ we have, by arguments similar to those in the proof of Proposition~\ref{Prop:Variance}, using the fact that $\mathbb{P}(\|X_j-X_1\| \leq \rho_{(k),1,M} | M) = k/(M-1)$ and splitting up into the cases $X_1 \in \mathcal{X}_{m,f}$ and $X_1 \not\in \mathcal{X}_{m,f}$, that
\begin{align*}
	&\mathbb{E}\{ (D_m - D_m^{(j)})^2 |M \} \\
	&\leq 4 \mathbb{E} \biggl[ \biggl\{ \frac{1}{m} \sum_{i=1}^m \mathbbm{1}_{\{\|X_j -X_i\| \leq \rho_{(k),i,M}\}} \biggl( \phi_{X_i}^g \biggl( \frac{k}{mV_d \rho_{(k),i,M}^d} \biggr) - \phi_{X_i} \biggr) \biggr\}^2 \biggm| M \biggr] \\
	& = 4(1-1/m) \mathbb{E} \biggl[ \mathbbm{1}_{\{\|X_j-X_1\| \leq \rho_{(k),1,M} \}} \mathbbm{1}_{\{\|X_j-X_2\| \leq \rho_{(k),2,M} \}} \biggl\{ \phi_{X_1}^g \biggl( \frac{k}{mV_d \rho_{(k),1,M}^d} \biggr) - \phi_{X_1} \biggr\} \\
	& \hspace{150pt} \times \biggl\{ \phi_{X_2}^g \biggl( \frac{k}{mV_d \rho_{(k),2,M}^d} \biggr) - \phi_{X_2} \biggr\} \biggm| M \biggr] \\
	& \hspace{50pt} + \frac{4}{m} \mathbb{E} \biggl[ \mathbbm{1}_{\{\|X_j-X_1\| \leq \rho_{(k),1,M} \}} \biggl\{ \phi_{X_1}^g \biggl( \frac{k}{mV_d \rho_{(k),1,M}^d} \biggr) - \phi_{X_1} \biggr\}^2 \biggm| M \biggr] \\
	& \lesssim \frac{k}{m} \biggl( \frac{k}{m} + \frac{1}{m} \biggr) \int_{\mathcal{X}_{m,f}} f(x)^{1+2 \kappa_1} g(x)^{2 \kappa_2} \max \biggl[ \frac{(M-m)^2}{m^2} ,  \frac{\log m}{k} , \biggl\{ \frac{k M_\beta(x)^d}{m f(x) } \biggr\}^\frac{2(2 \wedge \beta)}{d}\biggr] \,dx  \\ 
	& \hspace{50pt} + \frac{k}{m} \biggl\{ \Bigl( \frac{k}{m} \Bigr)^{2 \lambda_1(1-\zeta) - \epsilon} + \frac{1}{m}\Bigl( \frac{k}{m} \Bigr)^{ \lambda_1(1-\zeta) - \epsilon} \biggr\} + o(m^{-2}) \\
	& = O \biggl( \frac{k}{m} \max\biggl\{ \frac{k(M-m)^2}{m^3},  \frac{\log m}{m},  \Bigl( \frac{k}{m} \Bigr)^{1+ \frac{2(2\wedge \beta)}{d}}, \Bigl( \frac{k}{m} \Bigr)^{2\lambda_1(1-\zeta) - \epsilon} \biggr\} \biggr) \\
	&= o \biggl( \max\biggl\{ \frac{(M-m)^2}{m^{5/2}}, \frac{1}{m^{3/2}} \biggr\} \biggr). 
\end{align*}
Now for $j=1,\ldots,m$ we have
\begin{align*}
	\mathbb{E}\{ & (D_m - D_m^{(j)})^2 |M \} = \mathbb{E}\{ (D_m - D_m^{(1)})^2 |M \} \\
	& \leq 2 \mathbb{E} \biggl[ \biggl\{ \frac{1}{m} \sum_{i=2}^m \biggl( \phi_{X_i}^g \biggl( \frac{k}{mV_d \rho_{(k),i,m}^d} \biggr) - \phi_{X_i}^g \biggl( \frac{k}{mV_d (\rho_{(k),i,m}^{(1)})^d} \biggr)  \\
	& \hspace{50pt}- \phi_{X_i}^g \biggl( \frac{k}{mV_d \rho_{(k),i,M}^d} \biggr)  + \phi_{X_i}^g \biggl( \frac{k}{mV_d (\rho_{(k),i,M}^{(1)})^d} \biggr) \biggr) \biggr\}^2 \biggm| M \biggr] + o(m^{-2}).
\end{align*}
Write $\rho_i^{(-1)}$ for the $k$th nearest neighbour distance of $X_i$ in the sample $X_2,X_3,\ldots,X_M$. The $i$th term in the above sum is equal to zero unless $\{X_{m+1}, \ldots, X_M\} \cap B_{X_i}(\rho_i^{(-1)}) \neq \emptyset$ and either $X_1 \in B_{X_i}(\rho_i^{(-1)})$ or $X_1' \in B_{X_i}(\rho_i^{(-1)})$. Thus, by similar arguments to those used in the proof of Proposition~\ref{Prop:Variance}, splitting up into the cases $X_2 \in \mathcal{X}_{m,f}$ and $X_2 \not\in \mathcal{X}_{m,f}$, we have
\begin{align*}
	&\mathbb{E}\{ (D_m - D_m^{(j)})^2 |M \} \\
	& \lesssim \biggl| \mathbb{E} \biggl[ \mathbbm{1}_{\{\|X_2-X_1\| \leq \rho_2^{(-1)}, \|X_3-X_1\| \leq \rho_3^{(-1)} \}} \biggl\{ \phi_{X_2}^g \bigl( \hat{f}_{(k),2,m} \bigr) - \phi_{X_2}^g \biggl( \frac{M}{m} \hat{f}_{(k),2,M} \biggr) \biggr\} \\
	& \hspace{150pt} \times \biggl\{ \phi_{X_3}^g \bigl( \hat{f}_{(k),3,m} \bigr) - \phi_{X_3}^g \biggl( \frac{M}{m} \hat{f}_{(k),3,M} \biggr) \biggr\} \biggm| M \biggr]  \biggr| \\
	&  \hspace{30pt} +  \frac{1}{m} \mathbb{E} \biggl[ \mathbbm{1}_{\{\|X_1 - X_2 \| \leq \rho_2^{(-1)}\}} \biggl\{ \phi_{X_2}^g \bigl( \hat{f}_{(k),2,m} \bigr) - \phi_{X_2}^g \biggl( \frac{M}{m} \hat{f}_{(k),2,M} \biggr) \biggr\}^2 \biggm| M \biggr] + o(m^{-2}) \\
	& \lesssim |M-m| \biggl\{ \Bigl( \frac{k}{m} \Bigr)^3 + \frac{1}{m} \Bigl( \frac{k}{m} \Bigr)^2 \biggr\} \int_{\mathcal{X}_{m,f}} f(x)^{1+2\kappa_1} g(x)^{2 \kappa_2} \\
	& \hspace{150pt} \times \max \biggl\{\frac{(M-m)^2}{m^2},  \frac{\log m}{k}, \biggl( \frac{k M_\beta(x)^d}{m f(x)} \biggr)^{\frac{2(2 \wedge \beta)}{d}} \biggr\}  \,dx \\
	& \hspace{50pt} + |M-m| \Bigl( \frac{k}{m} \Bigr)^2 \biggl\{ \Bigl( \frac{k}{m} \Bigr)^{2 \lambda_1(1-\zeta) - \epsilon} + \frac{1}{m} \Bigl( \frac{k}{m} \Bigr)^{ \lambda_1(1-\zeta) - \epsilon} \biggr\} +o(m^{-2}) \\
	& \lesssim |M-m| \Bigl( \frac{k}{m} \Bigr)^2 \max\biggl\{ \frac{k (M-m)^2}{m^3}, \frac{\log m}{m}, \Bigl( \frac{k}{m} \Bigr)^{1+\frac{2(2\wedge\beta)}{d}}, \Bigl( \frac{k}{m} \bigr)^{2 \lambda_1(1-\zeta) -\epsilon} \biggr\} + o(m^{-2}) \\
	& = o \biggl( \max\biggl\{ \frac{1}{m^2}, \frac{|M-m|^3}{m^{7/2}} \biggr\} \biggr).
\end{align*}
It now follows by the Efron--Stein inequality that, on the event $A_m$, we have
\[
	\mathrm{Var} \{ T_m^{(1)} - T_m^{(1),\mathrm{p}}  |M \} = o \biggl( \max\biggl\{ \frac{|M-m|^3}{m^{5/2}}, \frac{1}{m} \biggr\} \biggr).
\]

We now bound the contribution from the event $A_m^c$. We will use the fact that for $x \geq 0$ we have
\[
	\mathbb{P} \Bigl( \Bigl| \frac{M}{m} -1 \Bigr| \geq x \Bigr) \leq 2 \exp \Bigl( - \frac{m x^2}{2(1+x)} \Bigr).
\]
It follows from this that
\[
	\mathbb{P} (A_m^c) \leq 2 \exp \Bigl( - \frac{m}{4 \log^2(em)} \Bigr). 
\]
Moreover, we have for any $a \geq 1$ that 
\begin{align*}
	\mathbb{E} \Bigl[ \Bigl( \frac{M}{m} \Bigr)^a \mathbbm{1}_{A_m^c} \Bigr] &\leq \int_0^\infty \mathbb{P} \biggl( \Bigl| \frac{M}{m} -1 \Bigr| \geq \max\Bigl\{ \frac{1}{\log(em)}, x^{1/a}-1 \Bigr\} \biggr) \,dx \\
	& \leq 2^a \mathbb{P} (A_m^c) + 2a \int_2^\infty y^{a-1} e^{ - my/8} \,dy \\
	& \leq 2^a \mathbb{P} (A_m^c) + 2a(a-1) \log(16/a) \int_2^\infty e^{-(m-1/2)y/8 } \, dy\\
	& \leq 2^{a+1} \exp \Bigl( - \frac{m}{4 \log^2(em)} \Bigr) + \frac{2a(a-1) \log(16/a)}{m-1/2} e^{-\frac{2m-1}{8}}.
\end{align*}
It now follows using Lemma~\ref{Lemma:hxinvbounds2} that, when $\log(em) > 2d \kappa_1^{-}/\alpha$, we have
\begin{align}
\label{Eq:RarePoisson}
	&\mathbb{E} \{ (T_m^{(1)} - T_m^{(1),\mathrm{p}})^2 \mathbbm{1}_{A_m^c} \} \nonumber \\
	& \leq 3 \mathbb{P}(A_m^c) \mathbb{E} \{(T_m^{(1)})^2\} +3 \mathbb{E} \biggl( \Bigl| \frac{M}{m} -1 \Bigr| \mathbbm{1}_{A_m^c} \biggr) \biggl| \int_\mathcal{X} f(x) \{\phi_x + (f \phi_{10})_x \} \,dx \biggr| \nonumber \\
	& \hspace{100pt} + 3 \mathbb{E} \biggl\{ \frac{M}{m}\phi_{X_1}^g \biggl( \frac{k}{mV_d \rho_{(k),1,M}^d} \biggr)^2 \mathbbm{1}_{A_m^c \cap \{M \geq (k+1) \log (em) \} } \biggr\} \nonumber  \\
	& \lesssim \mathbb{E} \biggl[  \frac{M}{m} \mathbbm{1}_{A_m^c \cap \{M \geq (k+1) \log(em) \}} \int_\mathcal{X} f(x)g(x)^{2\kappa_2} \int_0^1 \biggl\{ \frac{k}{mV_d h_{x,f}^{-1}(s)^d} \biggr\}^{2\kappa_1} \mathrm{B}_{k,M-k}(s) \,ds \,dx \biggr]  \nonumber \\
	& \hspace{150pt}  \mathbb{P}(A_m^c) + \mathbb{E} \biggl( \frac{M}{m} \mathbbm{1}_{A_m^c} \biggr)\nonumber \\
	& \lesssim \mathbb{E} \biggl[ \frac{M}{m} \mathbbm{1}_{A_m^c} \int_\mathcal{X} f(x)g(x)^{2 \kappa_2} \max \Bigl\{ 1 + \|x\|^{2d \kappa_1^-}, \Bigl( \frac{M}{m} \Bigr)^{2 \kappa_1^+} \Bigr\} \,dx \biggr]  + \mathbb{P}(A_m^c) + \mathbb{E} \biggl( \frac{M}{m} \mathbbm{1}_{A_m^c} \biggr) \nonumber \\
	& \lesssim \mathbb{P}(A_m^c) + \mathbb{E} \biggl\{ \biggl(\frac{M}{m}\biggl)^{1+2\kappa_1^+} \mathbbm{1}_{A_m^c} \biggr\} = o(1/m).
\end{align}
Hence,
\begin{align*}
  \mathbb{E}\{ (T_m^{(1)} - T_m^{(1), \mathrm{p}})^2\} &= \mathbb{E} \bigl[ \mathrm{Var}\bigl(T_m^{(1)} - T_m^{(1), \mathrm{p}}\mathbbm{1}_{A_m} \bigm|M\bigr)\mathbbm{1}_{A_m} \\
  &\hspace{1cm}+ \bigl\{ \mathbb{E} \bigl(T_m^{(1)} - T_m^{(1), \mathrm{p}} \bigm|M\bigr) \bigr\}^2\mathbbm{1}_{A_m} \bigr] + \mathbb{E} \bigl\{ (T_m^{(1)} - T_m^{(1),\mathrm{p}})^2 \mathbbm{1}_{A_m^c} \bigr\} \\
	& = o \biggl( \mathbb{E} \max \biggl\{\frac{|M-m|^3}{m^{5/2}} , \frac{1}{m}  \biggr\} \biggr) = o(1/m),
\end{align*}
as required.

We now turn our attention to $T_n^{(2)} - T_n^{(2),\mathrm{p}}$, for which similar arguments apply. We write $\phi_x^f(\cdot) := \phi( f(x), \cdot)$. We have, on the event $|n/N-1| \leq 1/L$ and when $n \geq (1+1/L)k\log(en)$, that
\begin{align*}
	\mathbb{E}& \bigl\{ T_n^{(2)} - T_n^{(2),\mathrm{p}} | N \bigr\} = \mathbb{E} \phi_{X_1}^f \biggl( \frac{k}{n V_d \rho_{(k),1,n}^d} \biggr) - \mathbb{E}\biggl\{ \phi_{X_1}^f \biggl( \frac{k}{n V_d \rho_{(k),1,N}^d} \biggr) \biggm| N \biggr\} \\
	& \hspace{150pt} + \Bigl( \frac{N}{n} -1 \Bigr) \int_\mathcal{X} f(x) (g \phi_{01})_x \,dx \\
	& = \mathbb{E} \biggl\{ \phi_{X_1}^f \biggl( \frac{k}{N V_d \rho_{(k),1,N}^d} \biggr) - \phi_{X_1}^f \biggl( \frac{k}{n V_d \rho_{(k),1,N}^d} \biggr) \biggm| N \biggr\}  \\
	& \hspace{100pt} + \Bigl( \frac{N}{n} -1 \Bigr) \int_\mathcal{X} f(x) (g \phi_{01})_x \,dx + o\biggl(n^{-1/2} + \Bigl| \frac{N}{n} -1 \Bigr|\biggr)\\
	& = o \biggl( n^{-1/2} + \Bigl| \frac{N}{n} -1 \Bigr| \biggr).
\end{align*}
To bound the conditional variance of $T_n^{(2)} - T_n^{(2),\mathrm{p}}$ on the event that $|N/n-1| \leq 1/\log(en)$, we again appeal to the Efron--Stein inequality.  Similar to before, for $\ell \geq k$ and $x \in \mathcal{X}$, we define
\[
  \hat{g}_{(k),\ell}(x) := \frac{k}{\ell V_d \rho_{(k),\ell}^d(x)}.
\]
We redefine
\[
	D_n := \int_\mathcal{X} f(x) \biggl\{ \phi_x^f \bigl(\hat{g}_{(k),n}(x)\bigr) - \phi_x^f \biggl( \frac{N}{n}\hat{g}_{(k),N}(x)\biggr) \biggr\} \,dx.
\]
Similarly to above, letting $Y_1',Y_2',\ldots$ be independent copies of $Y_1,Y_2,\ldots$, for $j \in [N]$ write $D_n^{(j)}$ for the value of $D_n$ when it is computed on $Y_1,\ldots,Y_{j-1},Y_j',Y_{j+1},\ldots,Y_N$ instead of $Y_1,\ldots,Y_N$.  On the event $\{N > n\}$, for $j=n+1,\ldots,N$, splitting up into the cases $X_1 \in \mathcal{X}_{n,g}$ and $X_1 \not\in \mathcal{X}_{n,g}$, we have
\begin{align*}
	\mathbb{E}\bigl\{ (D_n - D_n^{(j)})^2 |& N \bigr\} \leq 4 \mathbb{E} \biggl[ \biggl\{ \int_{\{x: \|Y_j-x\| \leq \rho_{(k),N}(x) \}} f(x) \phi_x^f \biggl( \frac{N}{n}\hat{g}_{(k),N}(x) \biggr) \,dx \biggr\}^2 \biggm| N \biggr] \\
	& = 4 \mathbb{E} \biggl[ \mathbbm{1}_{\{ \|Y_j - X_1\| \leq \rho_{(k),N}(X_1)\}}  \phi_{X_1}^f \biggl( \frac{N}{n}\hat{g}_{(k),N}(X_1)\biggr)\\
	& \hspace{80pt} \times \mathbbm{1}_{\{ \|Y_j - X_2\| \leq \rho_{(k),N}(X_2)\}} \phi_{X_2}^f \biggl( \frac{N}{n}\hat{g}_{(k),N}(X_2)\biggr)  \biggm| N \biggr] \\
	& \lesssim \Bigl( \frac{k}{n} \Bigr)^2 \int_{\mathcal{X}_{n,g}} f(x)^{2+2\kappa_1} g(x)^{2 \kappa_2 -1} \,dx  + \Bigl(\frac{k}{n} \Bigr)^{1+2 \lambda_2(1-\zeta) -\epsilon}= o \bigl(n^{-3/2} \bigr).
\end{align*}
On the other hand, for $j\in [n]$ and on the same event $\{N > n\}$, we have
\begin{align*}
	\mathbb{E}&\bigl\{ (D_n - D_n^{(j)})^2 | N \bigr\} \\
	& \leq 4 \mathbb{E} \biggl[ \biggl\{ \int_{\{x: \|Y_j-x\| \leq \rho_{(k),N}(x) \}} f(x) \biggl(\phi_x^f \bigl( \hat{g}_{(k),n}(x)\bigr) - \phi_x^f \biggl( \frac{N}{n}\hat{g}_{(k),N}(x)\biggr) \biggr) \,dx \biggr\}^2 \biggm| N \biggr] \\
	& \lesssim (N-n) \Bigl( \frac{k}{n} \Bigr)^3 \int_{\mathcal{X}_{n,g}} f(x)^{2+2 \kappa_1} g(x)^{2 \kappa_2-1} \,dx + (N-n) \Bigl( \frac{k}{n} \Bigr)^{2+ 2 \lambda_2(1-\zeta) - \epsilon} \\
	& = o \biggl( \frac{|N-n|}{n^{5/2}} \biggr).
\end{align*}
On the event $\{N < n\}$, the same final bound holds, and it follows by the Efron--Stein inequality that, on the event that $|N/n-1| \leq 1/\log(en)$, we have
\[
	\mathrm{Var}\bigl( T_n^{(2)} - T_n^{(2),\mathrm{p}} \bigm| N\bigr) = o \biggl( \frac{|N-n|}{n^{3/2}} \biggr).
\]
Now, similarly to~\eqref{Eq:RarePoisson}, redefining $A_n:=\{|N/n-1| \leq 1/\log(en)\}$ we have
\begin{align*}
	&\mathbb{E}\bigl\{ (T_n^{(2)} - T_n^{(2),\mathrm{p}})^2 \mathbbm{1}_{A_n^c} \bigr\} \lesssim \mathbb{P}(A_n^c) + \mathbb{E} \biggl\{ \Bigl(\frac{N}{n} \Bigr)^{1+ 2 \kappa_2^+}  \mathbbm{1}_{A_n^c} \biggr\} = o(1/n),
\end{align*}
and the result follows.
\end{proof}
Our second preparatory result provides a convenient partition of (minor modifications of) $\mathcal{X}_{m,f}$ and $\mathcal{X}_{n,g}$ so that, under the Poisson sampling scheme, the $k$-nearest neighbour distances of points in distant pieces are roughly independent.  
\begin{prop}
\label{Prop:Partition}
Let $f \in \mathcal{F}_d$ be $\underline{\beta}:=(\lceil \beta \rceil -1)$-times differentiable. Then there exists $n_0=n_0(d,\beta)$ such that, for all $n\geq n_0$ and $k \in[3,n/\log n)$, we can find a partition $\{C_j : j \in 1,\ldots,V_n\}$ of $\mathcal{X}_{n}:=\{x: f(x)/ M_{f,\beta}(x)^{d} \geq (k/n) \log^2 n \}$ and points $\{x_j : j =1,\ldots,V_n\}$ in $\widetilde{\mathcal{X}}_{n}:= \{x: f(x)/ M_{f,\beta}(x)^{d} \geq (k/n) \log^{7/4} n \}$ satisfying the following properties for each $j=1,\ldots,V_n$:
\begin{enumerate}[(i)]
	\item we have $C_j \subseteq B_{x_j} \Bigl( 3 \bigl( \frac{k \log n}{n V_d f(x_j)} \bigr)^{1/d} \Bigr)$;
	\item we have $\Bigl|\Bigl\{ j'=1,\ldots,V_n : \! \mathrm{dist}(C_j,C_{j'}) \leq 4  \bigl( \frac{k}{nV_df(x_j)} \bigr)^{1/d} \Bigr\} \Bigr| \leq 2^{2+4d} \log n $.
\end{enumerate}
\end{prop}

\begin{proof}[Proof of Proposition~\ref{Prop:Partition}]
  Let $\{x_j : j =1,\ldots,V_n \}$ be a Poisson process on $\widetilde{\mathcal{X}}_{n}$ with intensity function $nf(\cdot)/k$, and let $P$ denote the corresponding Poisson random measure.  Writing $\sargmin(S)$ for the smallest element of an ordered set $\argmin(S)$, we may partition $\mathcal{X}_{n}$ into the associated (random) Voronoi cells $\{C_j : j = 1,\ldots,V_n\}$, where $C_j:=\{ x \in \mathcal{X}_{n} : \sargmin_{j'=1,\ldots,V_n} \|x-x_{j'}\| = j\}$. We proceed by showing that, for $n$ and $k$ sufficiently large, there is an event of positive probability on which $\{C_j: j =1,\ldots, V_n\}$ and $\{x_j : j =1,\ldots, V_n\}$ satisfy (i) and (ii), and we therefore deduce the existence of such a partition. First, let $z_1,\ldots,z_N \in \mathcal{X}_{n}$ be such that 
\[
	\|z_i-z_j\| \geq h_{z_i,f}^{-1}(k/n) + h_{z_j,f}^{-1}(k/n) =: r(z_i,z_j)
\]
for all $i \neq j$, and such that $\sup_{x \in \mathcal{X}_{n}} \min_{j=1,\ldots,N} \|x-z_j\|/r(x,z_j) < 1$. (We can construct this set inductively: first, choose $z_1 \in \mathcal{X}_n$ arbitrarily.  If the second condition is not satisfied once $z_1,\ldots,z_N$ have been defined, then there exists $x \in \mathcal{X}_{n}$ such that $\|x-z_j\| \geq r(x,z_j)$ for all $j=1,\ldots, N$ and we can set $z_{N+1}:=x$.) For all $i \neq j$, the intersection $B_{z_i}\bigl(h_{z_i,f}^{-1}(k/n) \bigr) \cap B_{z_j}\bigl( h_{z_j,f}^{-1}(k/n)\bigr)$ has Lebesgue measure zero and thus
\[
	1 \geq \sum_{j=1}^N h_{z_j,f} \Bigl( h_{z_j,f}^{-1}(k/n) \Bigr)= \frac{N k}{n}.
\]
In particular, $N \leq n/k$. 

We now show that if $x \in \mathcal{X}_n$ is such that $\|x-z\| < r(x,z)$ for some $z \in \{z_1,\ldots,z_N\} \subseteq \mathcal{X}_n$ then $f(x) \approx f(z)$. Suppose initially that $r_2 := \{M_{f,\beta}(z)^d \log n\}^{-1/d} \leq \|x-z\| < r(x,z)$. Then, writing $r_1:=\|x-z\| -r_2/2$, writing $\bar{x}$ for the point on the line segment between $x$ and $z$ such that $\|\bar{x}-z\|=r_2$ and writing $I(s):=\int_0^s \mathrm{B}_{(d+1)/2,1/2}(t) \,dt$, we have by Lemma~\ref{Lemma:15over7} that, for $n \geq n_0(d,\beta)$ sufficiently large,
\begin{align*}
	& \int_{B_x(r_1)} f(w) \,dw \geq \int_{B_x(r_1)\cap B_z(r_2)} f(w) \,dw \geq \frac{f(z)}{2} \mu_d \bigl( B_x(r_1) \cap B_z(r_2) \bigr) \\
	& \hspace{14pt} \geq \frac{f(z)}{2} \mu_d \bigl( B_{\bar{x}}(r_2/2) \cap B_z(r_2) \bigr) = \frac{V_df(z)}{2} \Bigl\{ \Bigl( \frac{r_2}{2} \Bigr)^d I(15/16) + r_2^d I(15/64) \Bigr\} \\
	& \hspace{14pt} \geq \frac{V_d}{2^{d+1}} I(15/16) \frac{k \log n}{n} \geq	\frac{k}{n}.
\end{align*}
It follows, by Lemma~\ref{Lemma:hxinvbounds} and the fact that $z \in \mathcal{X}_n$, that there exists $n_1 = n_1(d,\beta) \geq n_0$, such that for $n \geq n_1$, 
\begin{align*}
	\|x-z\| \leq r_1 + h_{z,f}^{-1}(k/n) \leq r_1 + 2 \Bigl( \frac{k}{nV_d f(z)} \Bigr)^{1/d} \leq r_1 + \frac{r_2}{4} = \|x-z\| - \frac{r_2}{4},
\end{align*}
which is a contradiction. Thus, for $n \geq n_1$ we have that $\|x-z\| \leq r_2$. In particular, by Lemma~\ref{Lemma:15over7}, for $x,z \in \mathcal{X}_n$ with $\|x-z\| < r(x,z)$, and for $n \geq n_1$, we have that
\begin{equation}
\label{Eq:PackingApprox}
	\Bigl| \frac{f(x)}{f(z)} -1 \Bigr| \leq \frac{2}{ \log^{(1 \wedge \beta)/d}n}.
\end{equation}
To establish (i), first we define the event
\[
  \Omega_0 := \bigcap_{j =1}^N \biggl\{P \Bigl\{B_{z_j} \Bigl( h_{z_j,f}^{-1}(k \log n/n)  \Bigr) \Bigr\} \geq 1\biggr\}.
\]
By Lemmas~\ref{Lemma:hxinvbounds} and~\ref{Lemma:15over7} and very similar arguments to those leading up to~\eqref{Eq:Xntilde}, there exists $n_2=n_2(d,\beta) \geq n_1$ such that $B_{z_j} \bigl( h_{z_j,f}^{-1}(k \log n /n) \bigr) \subseteq \widetilde{\mathcal{X}}_n$ for all $n \geq n_2$ and $j =1, \ldots, V_n$. Then, for $n \geq n_2$ we have that 
\[
	\mathbb{P}(\Omega_0^c) \leq N \exp \Bigl( - \frac{n}{k} \frac{k \log n}{n} \Bigr) \leq \frac{1}{k}.
\]
Let $j \in \{1,\ldots,V_n\}$ be given, and suppose that $x \in C_j$. Let $z$ be in our covering set such that $ \|x-z\| < r(x,z)$ and, on the event $\Omega_0$, let $j' \in \{1,\ldots,V_n\}$ be such that $\|x_{j'} - z\| \leq h_{z,f}^{-1}(k \log n /n)$.  By~\eqref{Eq:PackingApprox}, Lemma~\ref{Lemma:hxinvbounds} and Lemma~\ref{Lemma:15over7}, there exists $n_3=n_3(d,\beta) \geq n_2$ such that, for $n \geq n_3$, we have that $h_{z,f}^{-1}(k \log n /n) \leq \frac{3}{2} (\frac{k \log n}{n V_d f(z)} )^{1/d}$ and hence that
\begin{align*}
	\|x_{j '}-x \| \leq \|x_{j'}-z\| + \|z-x\| &< h_{z,f}^{-1}(k \log n/n) + h_{z,f}^{-1}(k/n) + h_{x,f}^{-1}(k/n) \\
	& \leq 2 \Bigl( \frac{k \log n}{n V_d f(x_{j'})} \Bigr)^{1/d}.
\end{align*}
If $j'=j$ then we are done, so suppose instead that $\|x-x_j \| \leq \|x-x_{j'}\|$. Then
\[
	\|x_j - x_{j'}\| \leq 2 \|x-x_{j'}\| \leq 4 \Bigl( \frac{k \log n}{n V_d f(x_{j'})} \Bigr)^{1/d}
\]
so we can use Lemma~\ref{Lemma:15over7} to argue that $f(x_j) \approx f(x_{j'})$. In particular, there exists $n_4=n_4(d,\beta) \geq n_3$ such that, for $n \geq n_4$ we have that
\[
	\|x-x_j\| \leq \|x-x_{j'}\| \leq 2 \Bigl( \frac{k \log n}{n V_d f(x_{j'})} \Bigr)^{1/d} \leq 3 \Bigl( \frac{k \log n}{n V_d f(x_{j})} \Bigr)^{1/d}.
\]
So, for $n \geq n_4$, we have that (i) holds on $\Omega_0$.

Now, by Lemma~\ref{Lemma:15over7}, there exists $n_5=n_5(d,\beta) \geq n_4$ such that, for $n \geq n_5$ we have that $\frac{n}{k}h_{z_j,f}\bigl(16 (\frac{k \log n}{nV_d f(z_j)})^{1/d}\bigr) \leq 2^{1+4d}\log n$ for all $j \in \{1,\ldots,N\}$, and hence, by Bennett's inequality, that the event
\[
\Omega_1 := \bigcap_{j =1}^N \biggl\{P \Bigl\{ B_{z_j} \Bigl(16 \Bigl(\frac{k \log n}{n V_d f(z_j)} \Bigr)^{1/d} \Bigr) \Bigr\} \leq 2^{2+4d} \log n \biggr\}
\]
satisfies
\[
	\mathbb{P}(\Omega_1^c) \leq N \exp \biggl( - \frac{(2^{2+4d} \log n - 2^{1+4d} \log n )^2}{2^{3+4d} \log n } \biggr) \leq \frac{n}{k} \exp( - 2^{4d-1} \log n) \leq \frac{1}{k}.
\]
Now, on $\Omega_0$, if $\mathrm{dist}(C_j,C_{j'})  \leq 4  \bigl( \frac{k}{nV_df(x_j)} \bigr)^{1/d}$ then we must have
\begin{align}
\label{Eq:NearCells}
	\|x_j-x_{j'}\| \leq 4\Bigl( \frac{k }{n V_d f(x_j)} \Bigr)^{1/d} + 3\Bigl( \frac{k \log n}{n V_d f(x_j)} \Bigr)^{1/d} +3 \Bigl( \frac{k \log n}{n V_d f(x_{j'})} \Bigr)^{1/d}.
\end{align}
Using Lemma~\ref{Lemma:hxinvbounds}, there exists $n_6 = n_6(d,\beta) \geq n_5$ such that $\|x_j-x_{j'}\| \leq 6h_{x_j,f}^{-1}(k \log n /n) + 6h_{x_{j'},f}^{-1}(k \log n /n)$ for $n \geq n_6$ and hence, by a very similar argument to that leading up to~\eqref{Eq:PackingApprox}, we have that $|f(x_{j'})/f(x_j)-1| \leq 2 \log^{-(1 \wedge \beta)/(2d)} n$ for $n \geq n_6$. Thus, writing $z_j^*$ for an element of our covering set with $\|x_j-z_j^*\| < r(x_j,z_j^*)$, there exists $n_7=n_7(d,\beta) \geq n_6$ such that, on $\Omega_0 \cap \Omega_1$, for all $n \geq n_7$ we have that
\begin{align*}
	\Bigl| \Bigl\{ j' \in V_n : \mathrm{dist}(C_j,C_{j'}) \leq 4 & \Bigl( \frac{k}{nV_df(x_j)} \Bigr)^{1/d} \Bigr\} \Bigr| \\
	&\leq \Bigl| \Bigl\{ j' \in V_n : \|x_j-x_{j'} \| \leq 8  \Bigl( \frac{k \log n}{nV_df(x_j)} \Bigr)^{1/d} \Bigr\} \Bigr| \\
	& \leq \Bigl| \Bigl\{ j' \in V_n : \|z_j^*-x_{j'} \| \leq 16  \Bigl( \frac{ k\log n}{nV_df(x_j)} \Bigr)^{1/d} \Bigr\} \Bigr| \\
	& \leq 2^{2+4d} \log n
\end{align*}
for all $j \in V_n$. This establishes that, for $n \geq n_7$, with probability at least $1-2/k$ we have that both (i) and (ii) hold. Thus, since $k \geq 3$, there is a positive probability of both (i) and (ii) holding simultaneously and we can deduce the existence of the required partition.
\end{proof}

\begin{proof}[Proof of Theorem~\ref{Thm:Normality}]
We start by linearising our unweighted estimator. Consider
\begin{align*}
	E_{m,n} &:= \frac{1}{m} \sum_{i=1}^m \Bigl\{ \phi \bigl( \hat{f}_{(k_X),i} ,\hat{g}_{(k_Y),i}  \bigr) - \phi \bigl( \hat{f}_{(k_X),i}, g(X_i) \bigr) - \phi \bigl(f(X_i), \hat{g}_{(k_Y),i} \bigr) + \phi_{X_i} \Bigr\} \\
	&\phantom{:}=\frac{1}{m} \sum_{i=1}^m \phi^* \bigl( \hat{f}_{(k_X),i} ,\hat{g}_{(k_Y),i}  \bigr)
\end{align*}
\sloppy{with $\phi^*(u,v) := \phi(u,v) - \phi\bigl(u,g(x)\bigr) - \phi\bigl(f(x),v\bigr) + \phi\bigl(f(x),g(x) \bigr)$. This is of the same form as the estimators we have already considered, and we have $\phi^*(f(x),g(x)) \equiv \phi_{10}^*(f(x),g(x)) \equiv \phi_{01}^*(f(x),g(x)) \equiv 0$. Therefore, by very similar arguments to those used in the proof of Proposition~\ref{Prop:Variance}, we have that $\mathrm{Var}(E_{m,n}) = o(1/m + 1/n)$. Further, we have that}
\begin{align*}
	\mathbb{E} \biggl[ \mathrm{Var} \biggl( \frac{1}{m} \sum_{i=1}^m \bigl\{\phi \bigl(f(X_i), \hat{g}_{(k_Y),i} \bigr) &- \phi_{X_i} \bigr\} \biggm| Y_1, \ldots, Y_n \biggr) \biggr] \\
	& = \frac{1}{m} \mathbb{E} \Bigl[ \mathrm{Var} \Bigl( \phi \bigl(f(X_1), \hat{g}_{(k_Y),1} \bigr) - \phi_{X_1} \Bigm| Y_1, \ldots, Y_n \Bigr) \Bigr] \\
	& \leq \frac{1}{m} \mathbb{E} \Bigl[ \bigl\{ \phi \bigl(f(X_1), \hat{g}_{(k_Y),1} \bigr) - \phi_{X_1} \bigr\}^2 \Bigr] = o(1/m).
\end{align*}
Recalling the definitions of $T_m^{(1)}$ and $T_n^{(2)}$ in~\eqref{Eq:SemiOracle}, we therefore have that
\begin{align*}
	\mathrm{Var}&(  \hat{T}_{m,n} - T_m^{(1)} - T_n^{(2)} ) \leq 2 \mathrm{Var} \biggl( T_n^{(2)} \!-\! \frac{1}{m}\sum_{i=1}^m \bigl\{\phi \bigl(f(X_i), \hat{g}_{(k_Y),i} \bigr) \!-\! \phi_{X_i} \bigr\} \biggr) + 2 \mathrm{Var}(E_{m,n}) \\
	& = 2 \mathrm{Var} \biggl( \mathbb{E} \biggl\{ T_n^{(2)} \!-\! \frac{1}{m}\sum_{i=1}^m \bigl\{\phi \bigl(f(X_i), \hat{g}_{(k_Y),i} \bigr) - \phi_{X_i} \bigr\} \biggm| Y_1,\ldots,Y_n \biggr\} \biggr) + o(1/m + 1/n) \\
	& = o(1/m + 1/n).
\end{align*}
It now follows immediately from Proposition~\ref{Prop:Poisson} that $\mathrm{Var}( \hat{T}_{m,n} - T_m^{(1),\mathrm{p}} - T_n^{(2),\mathrm{p}} ) = o(1/m+1/n)$. Noting that $T_m^{(1),\mathrm{p}}$ depends only on $M,X_1,\ldots,X_m$ and $T_n^{(2),\mathrm{p}}$ depends only on $N,Y_1,\ldots,Y_n$ (so they are independent), we now proceed to establish the asymptotic normality of these two random variables separately, and then the result will follow.

We start with $T_m^{(1),\mathrm{p}}$, and adopt the notation of Proposition~\ref{Prop:Poisson}. Define the events $A_{i,m}:= \{ h_{X_i,f}( \rho_{(k),i,M}) \in \mathcal{I}_{m,X} \}$ for $i=1,\ldots,M$, similarly to in~\eqref{Eq:SetA}, and define
\[
	\mathcal{X}_{m,f}:= \biggl\{x:f(x)M_\beta(x)^{-d} \geq \frac{k_X \log^2 m}{m}\biggr\}.
\]
By separately considering the event that $|M/m-1| \leq 1/k_X$ and its complement we may use similar arguments to those in Proposition~\ref{Prop:Poisson} and Lemma~\ref{Lemma:Acomplement} to see that $\mathbb{P}(A_{1,m}^c) = o(m^{-4})$, and moreover that
\begin{align*}
	\mathbb{E} \biggl[ \mathbbm{1}_{A_{1,m}^c} \phi_{X_1}^g \biggl( \frac{M}{m} \hat{f}_{(k_X),1,M} \biggr)^2 \biggr] = o( m^{-4}).
\end{align*}
Further,
\begin{align*}
	\mathbb{E} & \biggl[ \biggl\{ \frac{1}{m} \sum_{i=1}^M \mathbbm{1}_{A_{i,m}} \mathbbm{1}_{\{X_i \not\in \mathcal{X}_{m,f} \}} \phi_{X_i}^g \biggl( \frac{M}{m} \hat{f}_{(k_X),i,M} \biggr) \biggr\}^2 \biggr] \\
	& =\frac{1}{m^2} \mathbb{E} \biggl[ M(M-1) \mathbbm{1}_{A_{1,m} \cap A_{2,m}} \mathbbm{1}_{\{X_1,X_2 \not\in \mathcal{X}_{m,f} \}} \phi_{X_1}^g \biggl( \frac{M}{m} \hat{f}_{(k_X),1,M} \biggr) \phi_{X_2}^g \biggl( \frac{M}{m} \hat{f}_{(k_X),2,M}\biggr) \biggr] \\
	& \hspace{15pt} +\frac{1}{m^2} \biggl[ M \mathbbm{1}_{A_{1,m}} \mathbbm{1}_{\{X_1 \not\in \mathcal{X}_{m,f} \}} \phi_{X_1}^g \biggl( \frac{M}{m} \hat{f}_{(k_X),1,M} \biggr) \biggr] = o(1/m).
\end{align*}
Writing
\[
	\widetilde{T}_m^{(1),\mathrm{p} }:= \frac{1}{m} \sum_{i=1}^M \mathbbm{1}_{A_{i,m}} \mathbbm{1}_{\{X_i \in \mathcal{X}_{m,f} \}}  \biggl\{ \phi_{X_i}^g \biggl( \frac{M}{m} \hat{f}_{(k_X),i,M} \biggr) - \int_\mathcal{X} f(x) \{ \phi_x + (f \phi_{10})_x \}\,dx \biggr\},
\]
we may now see that $\mathrm{Var}( T_m^{(1),\mathrm{p}} - \widetilde{T}_m^{(1),\mathrm{p}}) = o(1/m)$. Letting $\{C_j : j \in 1,\ldots,V_m\}$ denote a partition of $\mathcal{X}_{m,f}$ as in the statement of Proposition~\ref{Prop:Partition}, and writing $\mathcal{X}_{m,f}^{(j)}:= C_j \cap \mathcal{X}_{m,f}$, for $j=1,\ldots,V_m$ define
\begin{align*}
	W_j:= \frac{1}{m} \sum_{i=1}^M \mathbbm{1}_{A_{i,m}} \mathbbm{1}_{\{X_i \in \mathcal{X}_{m,f}^{(j)} \}}  \biggl\{ \phi_{X_i}^g \biggl( \frac{M}{m} \hat{f}_{(k_X),i,M} \biggr) - \int_\mathcal{X} f(x) \{ \phi_x + (f \phi_{10})_x \}\,dx \biggr\}
\end{align*}
so that $\widetilde{T}_m^{(1),\mathrm{p}} = \sum_{j =1}^{V_m} W_j$.  For $j,j' = 1,\ldots,V_m$, write $j \sim j'$ if $W_j$ and $W_{j'}$ are dependent. Because we are working on the events $A_{i,m}$, the random variable $W_j$ is only a function of those $X_i$ that lie within distance $\sup_{x \in \mathcal{X}_{m,f}^{(j)}} h_{x,f}^{-1} (a_{m,X}^+)$ of the set $C_j$. Hence, by the independence properties of Poisson processes, we can only have $j \sim j'$ if
\[
	\mathrm{dist}(C_j,C_{j'})  \leq \sup_{x \in \mathcal{X}_{m,f}^{(j)} } h_{x,f}^{-1}(a_{m,X}^+) +  \sup_{x' \in \mathcal{X}_{m,f}^{(j')} } h_{x',f}^{-1}(a_{m,X}^+) .
\]
Hence, by Lemma~\ref{Lemma:15over7} and property (i) of the partition and arguing as after~\eqref{Eq:NearCells}, there exists $m_0 = m_0(d,\vartheta)$ such that for $m \geq m_0$, we can only have $j \sim j'$ if
\begin{align*}
	&\mathrm{dist}(C_j,C_{j'}) \leq \sup_{x \in \mathcal{X}_{m,f}^{(j)} } \biggl( \frac{3k_X}{2mV_d f(x)} \biggr)^{1/d} \!+\! \sup_{x' \in \mathcal{X}_{m,f}^{(j')}} \biggl( \frac{3k_X}{2mV_d f(x')} \biggr)^{1/d} \leq 4 \biggl( \frac{k_X}{mV_d f(x_j)} \biggr)^{1/d},
\end{align*}
where $\{x_j : j =1,\ldots,V_m\}$ are the points associated to our partition given in Proposition~\ref{Prop:Partition}. By property (ii) of our partition, then, for each $j =1,\ldots,V_m$, we have $|\{j':j' \sim j\}| \leq 2^{2+4d} \log m$.  For $j = 1,\ldots,V_m$ and $p \in \mathbb{N}$, we write $L_j^{(p)}$ for the number of connected subsets of $\{1,\ldots,V_m\}$ (with edge relations defined by $\sim$) of cardinality at most $p$ containing $j$.  Then
\[
	L_j^{(p)} \leq 2^{(p-1)(2+4d)}\log^{p-1} m
\]
for $p=3,4$. 
Now, by Lemma~\ref{Lemma:15over7} and property (i) of our partition, for any $j =1,\ldots,V_m$ we have
\begin{equation}
\label{Eq:ConstantCj}
	\sup_{x \in C_j} \max \biggl\{ \biggl| \frac{f(x)}{f(x_j)} -1 \biggr|, \biggl| \frac{g(x)}{g(x_j)} -1 \biggr| \biggr\}  \leq \frac{2 \times 3^{1 \wedge \beta}}{(V_d \log^{3/4} m)^{(1 \wedge \beta)/d}}.
\end{equation}
Moreover, by very similar methods to those used in the proof of Proposition~\ref{Prop:Variance}, we may see that
\begin{equation}
\label{Eq:TmVar}
	\mathrm{Var}(\widetilde{T}_m^{(1), \mathrm{p}}) = \mathrm{Var} (T_m^{(1)}) + o(1/m) = \frac{v_1}{m} + o(1/m).
\end{equation}
Hence, using~\eqref{Eq:ConstantCj},~\eqref{Eq:TmVar} and the facts that $v_1 \geq 1/C $ and $p_{m,f,(j)}:=\mathbb{P}(X_1 \in \mathcal{X}_{m,f}^{(j)}) \lesssim 9(k_X/m) \log m$, we have that for $p=3,4$,
\begin{align*}
	&\frac{1}{\mathrm{Var}^{p/2}  \bigl(\widetilde{T}_m^{(1)}\bigr)} \sum_{j=1}^{V_m} L_j^{(p)}  \mathbb{E}\{|W_j - \mathbb{E} W_j|^p\} \\
	&\hspace{20pt} \lesssim m^{-p/2} \log^{p-1} m \sum_{j=1}^{V_m} \mathbb{E} \biggl\{ \biggl[ \sum_{i=1}^m \mathbbm{1}_{A_i^X} \mathbbm{1}_{\{X_i \in \mathcal{X}_{m,f}^{(j)}\}}  \biggl\{ \biggl|\phi_{X_i}^g \biggl( \frac{M}{m} \hat{f}_{(k_X),i,M} \biggr)\biggr| + 1 \biggr\} \biggr]^p \biggr\} \\
	& \hspace{20pt}\lesssim m^{-p/2} \log^{p-1} m \sum_{j=1}^{V_m} f(x_j)^{p\kappa_1} g(x_j)^{p \kappa_2} \bigl\{ m^p p_{m,f,(j)}^p + m p_{m,f,(j)} \bigr\} \\
	& \hspace{20pt} \lesssim \frac{k_X^{p-1} \log^{2p-2} m}{m^{p/2-1}} \int_\mathcal{X} f(x)^{1+p\kappa_1} g(x)^{p\kappa_2} \,dx \rightarrow 0.
\end{align*}
It now follows from Theorem~1 of \citet{Baldi1989} that
\[
	d_\mathrm{K} \biggl(\mathcal{L}\Bigl( \frac{m^{1/2}\{\widetilde{T}_m^{(1)} - \mathbb{E} \widetilde{T}_m^{(1)} \}}{v_1^{1/2}}\Bigr), N(0,1) \biggr) \rightarrow 0.
\]

We now take a similar approach to establish the asymptotic normality of $\widetilde{T}_n^{(2)}$. Letting $\{C_j : j = 1,\ldots,V_n\}$ denote a partition of $\mathcal{X}_{n,g}$ as in the statement of Proposition~\ref{Prop:Partition}, we may write $\rho_{(k_Y),Y}(x):= \|Y_{(k_Y)}(x)-x\|, A_n:=\bigl\{ x: h_{x,g}(\rho_{(k_Y),N}(x)) \in \mathcal{I}_{n,Y} \bigr\}, \mathcal{X}_{n,g}^{(j)}:= C_j \cap \mathcal{X}_{n,g}$, and
\begin{align*}
	W_j:=  \int_{\mathcal{X}_{n,g}^{(j)} \cap A^Y}  f(x)\phi_x^f \Bigl( \frac{k_Y}{n V_d \rho_{(k_Y),N}(x)^d} \Bigr)\,dx - \frac{|\{ i : Y_i \in \mathcal{X}_{n,g}^{(j)}\}|}{n} \int_\mathcal{X} f(x) (g \phi_{01})_x \,dx.
\end{align*}
Writing $\widetilde{T}_n^{(2),\mathrm{p}} := \sum_{j=1}^{V_n} W_j$ and arguing as above, we can see that $\mathrm{Var}( T_n^{(2),\mathrm{p}} - \widetilde{T}_n^{(2),\mathrm{p}} ) = o(1/n)$. By properties (i) and (ii) of our partition we again have that $L_j^{(p)} \lesssim \log^{p-1} n$, as above. 
Recall the definition of the conditional distribution function $F_{n,x,y}^{(2)}$ from the proof of Proposition~\ref{Prop:Variance}.  By similar but simpler arguments to those used in Proposition~\ref{Prop:Variance}, we have that
\begin{align}
\label{Eq:TnVar}
	&\mathrm{Var} (\widetilde{T}_n^{(2),\mathrm{p}}) = \mathrm{Var}(T_n^{(2)}) + o(1/n) = \mathrm{Var} \biggl( \int_{\mathcal{X}} f(x) \phi_x^f \Bigl(\frac{k_Y}{n V_d \rho_{(k_Y),n}(x)^d} \Bigr)  \,dx \biggr) + o(1/n) \nonumber \\
	&= \int_{\mathcal{X}} f(x) f(y) \int_{\mathcal{I}_{n,Y}} \int_{\mathcal{I}_{n,Y}} \phi \bigl( f(x), v_{x,t_1} \bigr)\phi \bigl(f(y), v_{y,t_2} \bigr)  \nonumber \\
	& \hspace{35pt} \times \bigl\{ dF_{n,x,y}^{(2)}(t_1,t_2) - \mathrm{B}_{k_Y,n+1-k_Y}(t_1) \mathrm{B}_{k_Y,n+1-k_Y}(t_2) \,dt_1 \,dt_2 \bigr\} \,dx \,dy +o(1/n) \nonumber \\
	&=\frac{v_2}{n} +o (1/n).
\end{align}
Now, using an analogous statement to that in~\eqref{Eq:ConstantCj}, using~\eqref{Eq:TnVar} and the facts that $\mathbb{P}(Y_1 \in \mathcal{X}_{n,g}^{(j)}) \lesssim (k_Y/ n)\log n$ and that $v_2 \geq 1/C$, we have for $p=3,4$ that
\begin{align*}
	\frac{1}{\mathrm{Var}^{p/2} \bigl(\widetilde{T}_n^{(2)}\bigr)}&  \sum_{j=1}^{V_n} L_j^{(p)} \mathbb{E}\{|W_j - \mathbb{E} W_j|^p\}  \\
	& \lesssim \frac{k_Y^{p-1}\log^{2p-2}n}{n^{p/2-1}} \biggl\{ \int_{\mathcal{X}_{n,g}} f(x)^{p+p\kappa_1} g(x)^{-(p-1)+p \kappa_2} \,dx + 1 \biggr\} \rightarrow 0.
\end{align*}
By Theorem~1 of \citet{Baldi1989} we now have that
\[
	d_\mathrm{K} \biggl( \mathcal{L}\Bigl(\frac{n^{1/2}\{\widetilde{T}_n^{(2)} - \mathbb{E} \widetilde{T}_n^{(2)} \}}{v_2^{1/2}}\Bigr), N(0,1)\biggr) \rightarrow 0.
      \]

      For our weighted estimator $\widehat{T}_{m,n}$, we can define weighted analogues $\widehat{T}_{m}^{(1)}$ and $\widehat{T}_{n}^{(2)}$ of $\widetilde{T}_{m}^{(1)}$ and $\widetilde{T}_{n}^{(2)}$ and deduce that 
      \begin{equation}
       \label{Eq:HatDecomp} 
        \widehat{T}_{m,n} - \mathbb{E}(\widehat{T}_{m,n}) = \widehat{T}_{m,n}^{(1)} - \mathbb{E}(\widehat{T}_{m,n}^{(1)}) + \widehat{T}_{m,n}^{(2)} - \mathbb{E}(\widehat{T}_{m,n}^{(2)}) + o_p(m^{-1/2} + n^{-1/2}),
      \end{equation}
      where
      \begin{align}
        \label{Eq:HatNormality}
        d_\mathrm{K} \biggl( \mathcal{L}\Bigl(\frac{m^{1/2}\{\widehat{T}_m^{(1)} - \mathbb{E} \widehat{T}_m^{(1)} \}}{v_1^{1/2}}\Bigr)&, N(0,1)\biggr) \nonumber \\
	& \hspace{-10pt} + d_\mathrm{K} \biggl( \mathcal{L}\Bigl(\frac{n^{1/2}\{\widehat{T}_n^{(2)} - \mathbb{E} \widehat{T}_n^{(2)} \}}{v_2^{1/2}}\Bigr), N(0,1)\biggr) =o(1).
      \end{align}
      
If $W,X,Y,Z$ are independent random variables it can be seen by simple conditioning arguments that
\begin{equation}
  \label{Eq:dKsums}
	d_\mathrm{K}\bigl( \mathcal{L}(W+X), \mathcal{L}(Y+Z) \bigr) \leq d_\mathrm{K}\bigl( \mathcal{L}(W), \mathcal{L}(Y)\bigr) + d_\mathrm{K}\bigl(\mathcal{L}(X), \mathcal{L}(Z) \bigr).
\end{equation}
Thus, by~\eqref{Eq:HatDecomp},~\eqref{Eq:HatNormality},~\eqref{Eq:dKsums} and Corollary~\ref{Cor:WeightedBias}, we may write
\[
	\widehat{Z}_{m,n}:= \frac{\widehat{T}_{m,n} - T}{\{v_1/m +  v_2/n\}^{1/2}} = Z_{m,n}^* + W_{m,n},
\]
where $d_\mathrm{K}\bigl( \mathcal{L}(Z_{m,n}^*), N(0,1)\bigr) \rightarrow 0$ and $W_{m,n} = o_p(1)$. Thus, for any $\epsilon >0$,
\begin{align*}
	d_\mathrm{K}&\bigl( \widehat{Z}_{m,n}, N(0,1)\bigr) \leq \sup_{x \in \mathbb{R}} \bigl| \mathbb{P} \bigl( \widehat{Z}_{m,n} \leq x, |W_{m,n}| \leq \epsilon \bigr) - \Phi(x) \bigr| + \mathbb{P}( |W_{m,n}| > \epsilon) \\
	& \leq \sup_{x \in \mathbb{R}} \max \bigl\{ \mathbb{P}(Z_{m,n}^* \leq x+\epsilon) - \Phi(x), \Phi(x) - \mathbb{P}(Z_{m,n}^* \leq x-\epsilon) \bigr\} + 2 \mathbb{P}( |W_{m,n}| > \epsilon) \\
	& \leq d_\mathrm{K}\bigl( Z_{m,n}^*, N(0,1)\bigr) + \frac{\epsilon}{(2\pi)^{1/2}} + 2 \mathbb{P}( |W_{m,n}| > \epsilon),
\end{align*}
so the result follows.  
\end{proof}

\begin{proof}[Proof of Theorem~\ref{Thm:ConfidenceIntervals}]

The main task is to establish the consistency of $\hat{V}_{m,n}^{(1)}$ and $\hat{V}_{m,n}^{(2)}$.  For the first of these, we start by noting that
\begin{align}
  \label{Eq:deltaintegrability}
	\mathbb{E} \Bigl[ \bigl\{ \phi_{X_1} +(f \phi_{10})_{X_1} \bigr\}^{4} \Bigr] &\leq 16 L^4 C^{8L+4(|\kappa_1|+|\kappa_2|)} \int_{\mathcal{X}} f(x)^{1+4\kappa_1}g(x)^{4\kappa_2} \,dx \nonumber \\
	&\leq 16 L^4 C^{1+8L+4(|\kappa_1|+|\kappa_2|)}.
\end{align}
Using this and Lemmas~\ref{Prop:FunctionalClasses}(i),~\ref{Lemma:BetaTailBounds} and~\ref{Lemma:GeneralisedHolder}, and writing $\tilde{\phi}(u,v):=\{\phi(u,v)+u\phi_{10}(u,v)\}^2$ and $b_{m,n} := \log m \wedge \log n$, we have that
\begin{align*}
	&\biggl| \mathbb{E} \hat{V}_{m,n}^{(1),1} - \int_\mathcal{X} f(x) \{ \phi_x + (f \phi_{10})_x \}^2 \,dx \biggr| \\
	& \!= \biggl| \int_\mathcal{X} f(x) \int_0^1 \int_0^1 \Bigl[ \min \bigl\{ \tilde{\phi}(u_{x,s},v_{x,t}) , b_{m,n}\bigr\} - \min \bigl\{ \tilde{\phi}_x , b_{m,n}\bigr\} \Bigr]  \\
	& \hspace{50pt} \times  \mathrm{B}_{k_X,m-k_X}(s) \mathrm{B}_{k_Y,n+1-k_Y}(t) \,ds \,dt \,dx \biggr| + O \biggl( \frac{1}{ \log m \wedge \log n} \biggr) \\
	& \!\leq \! \int_{\mathcal{X}_{m,n}} \!\!\!\! f(x) \!\! \int_{\mathcal{I}_{m,X}} \! \int_{\mathcal{I}_{n,Y}} \! \! \! \bigl| \tilde{\phi}(u_{x,s},v_{x,t}) \!-\!  \tilde{\phi}_x \bigr|  \mathrm{B}_{k_X,m-k_X}(s) \mathrm{B}_{k_Y,n+1-k_Y}(t) \,ds \,dt \,dx  \\
	& \hspace{10pt}+ O \biggl( b_{m,n}\max \biggl\{\Bigl( \frac{k_X \log m}{m} \Bigr)^{\lambda_1}, \Bigl( \frac{k_Y \log n}{n} \Bigr)^{\lambda_2}, \frac{1}{m^4} , \frac{1}{n^4}, \frac{1}{b_{m,n}^2} \biggr\} \biggr) \\
	& \!\lesssim \! \int_{\mathcal{X}_{m,n}} \!\!\!\! f(x)^{1+2\kappa_1} g(x)^{2\kappa_2} \! \biggl\{ \frac{\log^{\frac{1}{2}} m}{k_X^{1/2}}\! +\! \frac{\log^{\frac{1}{2}} n}{k_Y^{1/2}} \! + \!\Bigl( \frac{k_X M_\beta(x)^d}{m f(x)} \Bigr)^{\frac{2 \wedge \beta}{d}} \!\!\!\!+\! \Bigl( \frac{k_Y M_\beta(x)^d}{n g(x)} \Bigr)^{\frac{2 \wedge \beta}{d}} \! \biggr\} \,dx  \\
	& \hspace{10pt}+ O \biggl( b_{m,n} \max \biggl\{\Bigl( \frac{k_X \log m}{m} \Bigr)^{\lambda_1}, \Bigl( \frac{k_Y \log n}{n} \Bigr)^{\lambda_2}, \frac{1}{m^4} , \frac{1}{n^4}, \frac{1}{b_{m,n}^2}\biggr\} \biggr) = o(1).
\end{align*}
Now, for $i=1,\ldots,m$, write $\xi_i:= \min\{ \tilde{\phi}(\hat{f}_{(k_X),i}, \hat{g}_{(k_Y),i} ), b_{m,n}\}$, $\xi_i^*:= \min \{ \tilde{\phi}_{X_i}, b_{m,n}\}$ and
\[
	\widetilde{\mathcal{X}}_{m,n}:= \biggl\{ x: \frac{f(x)}{M_\beta(x)^d} \geq \frac{k_X^{1/2}}{m^{1/2}}, \frac{g(x)}{M_\beta(x)^d} \geq \frac{k_Y^{1/2}}{n^{1/2}} \biggr\}.
\]
We now have that
\begin{align*}
	&\mathbb{E}\{ (\xi_1 - \xi_1^*)^2 \} \leq \mathbb{E}\bigl\{ \mathbbm{1}_{A_1^X \cap A_1^Y} \mathbbm{1}_{\{X_1 \in \widetilde{\mathcal{X}}_{m,n}\}}( \xi_1 - \xi_1^*)^2\bigr\} \\
	& \hspace{150pt} + O \biggl( b_{m,n}^2 \max \biggl\{ \frac{1}{m^4}, \frac{1}{n^4}, \Bigl( \frac{k_X}{m} \Bigr)^{\lambda_1/2},\Bigl( \frac{k_Y}{n} \Bigr)^{\lambda_2/2} \biggr\} \biggr) \\
	& \!\lesssim \! \int_{\mathcal{X}_{m,n}} \!\!\!\! f(x)^{1+2\kappa_1} g(x)^{2\kappa_2} \! \biggl\{ \frac{\log m}{k_X}\! +\! \frac{\log n}{k_Y} \! + \!\Bigl( \frac{k_X \! M_\beta(x)^d}{m f(x)} \Bigr)^{\frac{2(2 \wedge \beta)}{d}} \!\!\!\!+\! \Bigl( \frac{k_Y \! M_\beta(x)^d}{n g(x)} \Bigr)^{\frac{2(2 \wedge \beta)}{d}} \! \biggr\} \! \,dx  \\
	& \hspace{50pt} + O \biggl( b_{m,n}^2 \max \biggl\{ \frac{1}{m^4}, \frac{1}{n^4}, \Bigl( \frac{k_X}{m} \Bigr)^{\lambda_1/2},\Bigl( \frac{k_Y}{n} \Bigr)^{\lambda_2/2} \biggr\} \biggr) \\
	& = O \biggl( b_{m,n}^2 \max \biggl\{  \frac{\log m}{k_X}, \frac{\log n}{k_Y} , \Bigl( \frac{k_X}{m} \Bigr)^{\frac{2 \wedge \beta}{d}} , \Bigl( \frac{k_Y}{n} \Bigr)^{\frac{2 \wedge \beta}{d}}, \Bigl( \frac{k_X}{m} \Bigr)^{\frac{\lambda_1}{2}},\Bigl( \frac{k_Y}{n} \Bigr)^{\frac{\lambda_2}{2}}  \biggr\} \biggr).
\end{align*}
It therefore follows by Cauchy--Schwarz that
\begin{align*}
	&\mathrm{Var}(\hat{V}_{m,n}^{(1),1}) = \frac{1}{m}\mathrm{Var}(\xi_1)  + 2\Bigl(1 - \frac{1}{m} \Bigr) \mathrm{Cov}( \xi_1 - \xi_1^*,\xi_2^*) + \Bigl(1 - \frac{1}{m} \Bigr) \mathrm{Cov}( \xi_1 - \xi_1^*, \xi_2 - \xi_2^*) \\
	& \leq \frac{b_{m,n}^2}{m} + 2b_{m,n}[ \mathbb{E}\{ (\xi_1 - \xi_1^*)^2 \}]^{1/2} + \mathbb{E}\bigl\{(\xi_1-\xi_1^*)^2\bigr\} =o(1)
\end{align*}
By very similar arguments to those employed in the proof of Proposition~\ref{Thm:GeneralBias} we have that $\mathbb{E}(\hat{V}_{m,n}^{(1),2}) - \int_{\mathcal{X}} f(x) \{ \phi_x + (f \phi_{10})_x\} \,dx = o(1)$. By Proposition~\ref{Prop:Variance} we have that $\mathrm{Var}(\widetilde{T}_{m,n}) = o(1)$.  Since $\zeta < 1/2$, the summands in $\hat{V}_{m,n}^{(1),2}-\widetilde{T}_{m,n}$ are square integrable and, writing $\xi_i:=\hat{f}_{(k_X),i} \phi_{10} \bigl( \hat{f}_{(k_X),i}, \hat{g}_{(k_Y),i} \bigr)$ and $\xi_i^*:= (f \phi_{10})_{X_i}$, we have by Cauchy--Schwarz again that
\begin{align*}
	\mathrm{Var} \biggl( \frac{1}{m} \sum_{i=1}^m \xi_i \biggr) 
	& \leq \frac{1}{m}\mathrm{Var}(\xi_1) + 2\mathrm{Var}^{1/2}(\xi_2)\mathrm{Var}^{1/2}(\xi_1-\xi_1^*) + \mathrm{Var}(\xi_1-\xi_1^*) = o(1).
\end{align*}
Combining our bounds on expectations and variances we have now established that, for any $\epsilon>0$,
\begin{equation}
\label{Eq:ConsistentVariance1}
	\sup_{\phi \in \Phi(\xi)} \sup_{(f,g) \in \widetilde{\mathcal{F}}_{d,\vartheta}}  \max_{\substack{k_X \in \{k_X^{\mathrm{L}},\ldots, k_X^{\mathrm{U}}\} \\ k_Y \in \{k_Y^{\mathrm{L}},\ldots, k_Y^{\mathrm{U}}\}}} \mathbb{P}( | \hat{V}_{m,n}^{(1)} - v_1| \geq \epsilon) \rightarrow 0.
\end{equation}

Now, we have by Cauchy--Schwarz and Lemma~\ref{Lemma:GeneralisedHolder} that
\begin{align*}
	&\sup_{(f,g) \in \widetilde{\mathcal{F}}_{d,\vartheta}} \int_\mathcal{X} f(x) \{ f(x)^{1+2\kappa_1} g(x)^{-1+2\kappa_2} \}^{3/2} \,dx \\
	& \leq \sup_{(f,g) \in \widetilde{\mathcal{F}}_{d,\vartheta}} \biggl\{ \int_\mathcal{X} g(x)^{1+4(\kappa_2-1)} f(x)^{4(1+\kappa_1)} \,dx \biggr\}^{1/2} \biggl\{ \int_\mathcal{X} f(x)^{1+2\kappa_1} g(x)^{2 \kappa_2} \,dx \biggr\}^{1/2} < \infty.
\end{align*}
Hence, by analogous calculations to those carried out earlier in this proof, we have for any $\epsilon>0$ that
\begin{equation}
\label{Eq:ConsistentVariance2}
	\sup_{\phi \in \Phi(\xi)} \sup_{(f,g) \in \widetilde{\mathcal{F}}_{d,\vartheta}}  \max_{\substack{k_X \in \{k_X^{\mathrm{L}},\ldots, k_X^{\mathrm{U}}\} \\ k_Y \in \{k_Y^{\mathrm{L}},\ldots, k_Y^{\mathrm{U}}\}}} \mathbb{P}( | \hat{V}_{m,n}^{(2)} - v_2| \geq \epsilon) \rightarrow 0.
\end{equation}

To conclude the proof, given $\epsilon>0$, we will consider the event $B_\epsilon:= \bigl\{ \max \bigl(|\hat{V}_{m,n}^{(1)}/ v_1 -1 | , |\hat{V}_{m,n}^{(2)}/ v_2 -1  | \bigr) \leq \epsilon \bigr\}$, and define the shorthand
\[
  \hat{Z} := \frac{\hat{T}_{m,n}-T}{\{\hat{V}_{m,n}^{(1)}/m + \hat{V}_{m,n}^{(2)}/n \}^{1/2} } \quad \text{and} \quad Z^*:= \frac{\hat{T}_{m,n}-T}{\{ v_1/m + v_2/n \}^{1/2} }.
\]
For all $\epsilon \in (0,1/2)$ we have that
\begin{align}
  \label{Eq:FinaldKbound}
	&d_\mathrm{K} \bigl( \mathcal{L}( \hat{Z}), N(0,1) \bigr) \leq \sup_{z \in \mathbb{R}}  \bigl| \mathbb{P}( \hat{Z} \leq z) - \mathbb{P}( Z^* \leq z) \bigr| +  d_\mathrm{K} \bigl( \mathcal{L}( Z^* ), N(0,1) \bigr) \nonumber \\
	& \leq \sup_{z \in \mathbb{R}} \Bigl\{\bigl| \mathbb{P}( Z^* \leq (1+\epsilon)z) - \mathbb{P}( Z^* \leq z)\bigr| \vee \bigl| \mathbb{P}( Z^* \leq z) - \mathbb{P}( Z^* \leq (1-\epsilon)z) \bigr| \Bigr\} \nonumber \\
	& \hspace{50pt} + d_\mathrm{K} \bigl( \mathcal{L}( Z^* ), N(0,1) \bigr) + 2 \mathbb{P}(B_\epsilon^c) \nonumber \\
	&= \sup_{z \in \mathbb{R}} \bigl| \mathbb{P}( Z^* \leq (1+\epsilon)z) \!  -\! \mathbb{P}( Z^* \leq (1-\epsilon)z) \bigr| + d_\mathrm{K} \bigl( \mathcal{L}( Z^* ), N(0,1) \bigr) \! +\! 2 \mathbb{P}(B_\epsilon^c) \nonumber \\ 
	& \leq 2 \epsilon \sup_{z \in \mathbb{R}} \frac{|z| e^{-z^2/8}}{(2 \pi)^{1/2}} +3 d_\mathrm{K} \bigl( \mathcal{L}( Z^* ), N(0,1) \bigr) + 2 \mathbb{P}(B_\epsilon^c).
\end{align}
The first conclusion of Theorem~\ref{Thm:ConfidenceIntervals} now follows from~\eqref{Eq:ConsistentVariance1},~\eqref{Eq:ConsistentVariance2} and~\eqref{Eq:FinaldKbound}.  The second conclusion is an immediate consequence of the first.
\end{proof}

\subsection{Proof of Proposition~\ref{Thm:SuperOracleLAM}}
\label{App:SuperOracle}

\begin{proof}[Proof of Proposition~\ref{Thm:SuperOracleLAM}] Since $f$ vanishes at infinity, there exists $x_0 >0$ such that $h(x) \geq 0 $ for all $x \leq x_0$ and $h(x) \leq 0$ for all $x \geq x_0$. Further, as $x \rightarrow \infty$, we have by Karamata's theorem \citep[][Proposition~1.5.10]{BGT1989} that
\begin{align}
\label{Eq:Karamata1}
	h(x) &\sim -P'(x) \int_0^x e^{(1-\kappa) P(y)} \,dy = - \frac{P'(x)}{1-\kappa} \int_{e^{(1-\kappa)P(0)}}^{e^{(1-\kappa)P(x)}} \frac{1}{P' \bigl( P^{-1} \bigl( \frac{\log u}{1- \kappa} \bigr) \bigr)} \,du \nonumber \\
	& \sim -\frac{P'(x)}{1-\kappa} \frac{e^{(1-\kappa) P(x)}}{P' (x)} = -\frac{f(x)^{-(1-\kappa)}}{1-\kappa}.
\end{align}
\sloppy{In particular, since $h$ is continuous, we can now see that $\sup_{ x \geq 0} h(x) < \infty$ and $\inf_{x \geq 0} f(x)h(x) > - \infty$. Hence, for $t\geq0$ sufficiently small, the function $f_t : [0,\infty) \rightarrow \mathbb{R}$ defined by}
\[
	f_t(x) := \{1-th(x) \} f(x),
\]
is bounded and takes values in $[0,\infty)$. Moreover, by Fubini's theorem,
\begin{align*}
	\int_0^\infty f(x) h(x) \,dx &= \int_0^\infty f'(x) \int_0^x \bigl\{ \psi \bigl( f(y) \bigr) - H(f) \bigr\} \,dy \,dx \\
	&= - \int_0^\infty f(y) \bigl\{\psi \bigl( f(y) \bigr) - H(f) \bigr\} \, dy = 0,
\end{align*}
so there exists $t_0 > 0$, depending only on $\kappa$ and $f$, such that $f_t$ is a density function for $t \in [0,t_0]$.

Observe that the function $h$ defined in~\eqref{Eq:h} solves the differential equation
\begin{equation}
\label{Eq:DE}
	\frac{d}{d x} \biggl( h(x) \frac{f(x)}{f'(x)} \biggr) = \psi \bigl( f(x) \bigr) - H(f) =: g(x).
\end{equation}
We now derive, for $t \in [0,t_0]$, the density function of the non-negative random variable $f_t(X_1)$ when $X_1$ has density function $f_t$ on $[0,\infty)$. As $x \rightarrow 0$, we have that
\begin{align}
\label{Eq:SmallxDeriv}
	f_t'(x) &= f'(x) - t f''(x) \int_0^x g(y) \,dy -t f'(x) g(x) \sim f'(x) - tx g(0) f''(x) - tf'(x) g(0) \nonumber \\
	&= f'(x) \biggl[ 1 - t g(0) \frac{ x \{P''(x) - P'(x)^2 \} + P'(x) }{P'(x)} \biggr].
\end{align}
We can also see that as $x \rightarrow \infty$ we have
\begin{align}
\label{Eq:Karamata2}
	f''(x)& \int_0^x g(y) \,dy + f'(x) g(x) \sim \frac{f''(x) f(x)^{-(1-\kappa)} }{(1-\kappa) P'(x)} + f'(x) f(x)^{-(1-\kappa)} \nonumber \\
	&= f(x)^\kappa \biggl\{ \frac{P'(x)^2 - P''(x)}{(1-\kappa) P'(x)} - P'(x) \biggr\} \sim f(x)^\kappa \frac{\kappa}{1-\kappa} P'(x),
\end{align}
using the fact that $P''(x) \ll P'(x)^2$ as $x \rightarrow \infty$ for strictly increasing polynomials $P$.  Finally, we note that $\sup_{x \in [a,b]} f'(x) <0$ for every $0<a<b<\infty$.  This, together with~\eqref{Eq:SmallxDeriv} and ~\eqref{Eq:Karamata2}, means that by reducing $t_0 = t_0(\kappa,f) > 0$ if necessary, we may assume that $f_t$ is strictly decreasing on $[0,\infty)$ for $t \in [0,t_0]$.  Thus, for $t \in [0,t_0]$, we can define the inverse function $f_t^{-1}$, and since $f_t(0) = f(0)$, see that when $X_1 \sim f_t$, the density of $f_t(X_1)$ at $z \in (0,f(0) )$ is given by
\[
	\lim_{\delta \rightarrow 0} \frac{1}{\delta} \int_0^\infty f_t(x) \bigl( \mathbbm{1}_{\{ f_t(x) \leq z + \delta\}} - \mathbbm{1}_{\{ f_t(x) \leq z \}} \bigr) \,dx = \frac{z}{- f_t'( f_t^{-1}(z))} =:p_t(z).
\]

Our goal now is to show that the family $\{p_t : t \in [0,t_0] \}$ is differentiable in quadratic mean at $t=0$, with score function $g \circ f^{-1}$. For a fixed $z \in (0,f(0))$, let $x=f^{-1}(z)$ and $x_t=f_t^{-1}(z)$. Then we have
\[
	0 = f_t(x_t) - f(x) = f(x_t) - f(x) - t f(x_t) h(x_t)  = (x_t -x) f'(x) - tz h(x) + o(t)
\]
as $t \searrow 0$, and hence $ \partial x_t/ \partial t |_{t=0} = z h(x) / f'(x)$. It now follows from~\eqref{Eq:DE} that
\begin{align*}
	\frac{\partial}{\partial t} & \Bigm|_{t=0} p_t(z) = \frac{\partial}{\partial t} \Bigm|_{t=0} \biggl(  \frac{z}{-f_t'(x_t)} \biggr) = z \frac{ \frac{\partial}{\partial t} |_{t=0} f_t'(x_t) }{f'(x)^2} \\
	&= -\frac{z}{f'(x)} \biggl\{ h(x) + \frac{h'(x) f(x)}{f'(x)} - \frac{h(x) f''(x) f(x)}{f'(x)^2} \biggr\} = p_0(z) g(x) \\
	&= p_0(z) \{ \psi(z) - H(f)\}.
\end{align*}
To prove differentiability in quadratic mean at $t=0$ with score function $g \circ f^{-1}$, i.e.~that
\begin{equation}
\label{Eq:DQM}
	\int_0^{f(0)} \biggl[ \frac{p_t(z)^{1/2} - p_0(z)^{1/2}}{t} - \frac{1}{2} \{ \psi(z) - H(f) \} p_0(z)^{1/2} \biggr]^2 \, dz \rightarrow 0
\end{equation}
as $t \searrow 0$, it now suffices by the dominated convergence theorem to show that $ t^{-2} \{ p_t(z)^{1/2} - p_0(z)^{1/2}\}^2$ can be bounded by an integrable function of $z$ for $t \in [0,t_0]$. Define $b_t:=(3t / (1-\kappa))^{1/(1-\kappa)}$ and $a_t:= f^{-1}( b_t )$. Now, by~\eqref{Eq:Karamata1},~\eqref{Eq:SmallxDeriv} and~\eqref{Eq:Karamata2}, it follows that there exists $C' = C'(\kappa,f) >0$ such that for all $x \leq a_t$ and $t \in [0,t_0]$, we have
\[
	\max\biggl\{ \biggl| \frac{f_t(x)}{f(x)} -1 \biggr| , \biggl| \frac{f_t'(x)}{f'(x)} -1 \biggr| \biggr\} \leq  t \max \biggl\{ C', \frac{3f(x)^{-(1-\kappa)}}{2(1-\kappa)} \biggr\} \leq \frac{1}{2}.
\]
Write $\epsilon_{t,z}:= t \max\{ C', \frac{3 z^{-(1-\kappa)}}{2^\kappa(1-\kappa)} \}$ so that for $z > 2b_t$ we have $\epsilon_{t,z} \leq 1/2$ and
\[
	f_t \biggl( f^{-1} \biggl( \frac{z}{1+\epsilon_{t,z}} \biggr) \biggr) \leq \frac{z}{1+ \epsilon_{t,z} } \biggl[ 1 + t \max \biggl\{ C', \frac{3(z/(1+\epsilon_{t,z}))^{-(1-\kappa)}}{2(1-\kappa)} \biggr\} \biggr] \leq z.
\]
We can similarly establish that $ f_t\bigl( f^{-1} ( z/(1-\epsilon_{t,z}))\bigr) \geq z$. Now, there exists $x_0 \in (0,\infty)$, depending only on $f$, such that $f''(x)=\{P'(x)^2 - P''(x)\} f(x) \geq 0$ for all $x \geq x_0$. We can therefore see that, by the convexity of $P$, for $z > 2b_t$ sufficiently small we have
\begin{align}
\label{Eq:ConvexP}
	|f_t^{-1}(z) &- f^{-1}(z)| \leq f^{-1} \biggl( \frac{z}{1+\epsilon_{t,z}} \biggr) - f^{-1} \biggl( \frac{z}{1-\epsilon_{t,z}} \biggr) \nonumber \\
	& = P^{-1} \biggl( \log \frac{1+\epsilon_{t,z}}{z} \biggr) - P^{-1} \biggl( \log \frac{1-\epsilon_{t,z}}{z} \biggr) \leq \frac{  \log \frac{1+\epsilon_{t,z}}{z} -  \log \frac{1-\epsilon_{t,z}}{z}}{P'\bigl(P^{-1}( \log \frac{1-\epsilon_{t,z}}{z})\bigr)} \nonumber \\
	& = \frac{ \frac{z}{1-\epsilon_{t,z}} \log \frac{1+\epsilon_{t,z}}{1-\epsilon_{t,z}} }{ -f'\bigl( f^{-1}(z/(1-\epsilon_{t,z}))\bigr)} \leq \frac{ \frac{z}{1-\epsilon_{t,z}} \log \frac{1+\epsilon_{t,z}}{1-\epsilon_{t,z}} }{ -f'( f^{-1}(z))} \lesssim_{\kappa,f} \frac{t z^\kappa}{-f'(f^{-1}(z))} \nonumber \\
	& = \frac{t z^{-(1-\kappa)} f^{-1}(z)}{ f^{-1}(z) P'( f^{-1}(z))},  
\end{align}
and $f^{-1}(z)P'\bigl(f^{-1}(z)\bigr) \rightarrow \infty$ as $z \searrow 0$. The derivative $P'(x)$ is bounded away from zero for $x$ bounded away from zero, so for $z$ bounded away from $f(0)$ and $0$, we can also see that $|f_t^{-1}(z)/f^{-1}(z)-1| \lesssim_{\kappa,f} t$. As $x \rightarrow 0$, 
\[
	\frac{f_t(x)}{f(x)}  = 1 + tP'(x) \int_0^x g(y) \,dt = 1 - txP'(x) |g(0)| \{1+o_{\kappa,f}(1)\},
\]
uniformly for $t \in [0,t_0]$.  Thus, similarly to in~\eqref{Eq:ConvexP} and by a Taylor expansion, we can see that for $z$ close to $f(0)$ we have that
\[
	| f_t^{-1}(z) - f^{-1}(z)| \lesssim_{\kappa,f} \frac{t f^{-1}(z) P'(f^{-1}(z))}{- f'(f^{-1}(z))} = \frac{t f^{-1}(z)}{z} \lesssim_{\kappa,f} t f^{-1}(z).
\]
Hence, combining this fact with~\eqref{Eq:ConvexP}, uniformly over all $z \in \bigl(2b_t,f(0)\bigr)$, we now have that
\[
	\biggl| \frac{f'(x_t)}{f'(x)} -1 \biggr| = \biggl| \frac{P'(x_t)}{P'(x)} \frac{1}{1-t h (x_t)} -1 \biggr| \lesssim_{\kappa,f} \biggl| \frac{x_t}{x} -1 \biggr| + t z^{-(1-\kappa)} \lesssim_{\kappa,f} tz^{-(1-\kappa)}.
\]
We deduce that there exists $c = c(\kappa,f) \in (0, \frac{1-\kappa}{3 \times 2^{1-\kappa}})$ such that for $t \in [0,t_0]$, when $ t z^{-(1-\kappa)} \leq c$ we have $z > 2b_t$ and
\begin{equation}
\label{Eq:largez}
	\biggl|  \frac{p_t(z)}{p_0(z)} -1 \biggr| \leq  \biggl| \frac{f'(f^{-1}(z))}{f'(f_t^{-1}(z))} \frac{f'(f_t^{-1}(z))}{f_t'(f_t^{-1}(z))}  -1 \biggr| \leq \frac{ t z^{-(1-\kappa)}}{2c} \leq \frac{1}{2} .
\end{equation}
Now, after reducing $t_0 = t_0(\kappa,f) > 0$ if necessary, for $t \in [0,t_0]$ and $t f(x)^{-(1-\kappa)} > c$, we have by~\eqref{Eq:Karamata1} that
\[
	f_t(x) \leq f(x) + \frac{2t f(x)^\kappa}{1-\kappa} \leq t f(x)^\kappa \biggl( \frac{1}{c} + \frac{2}{1-\kappa} \biggr).
\]
Thus, when $t z^{-(1-\kappa)}>c$, we have $x_t=f_t^{-1}(z) \leq f^{-1} ( ( \frac{z}{t(1/c + 2/(1-\kappa))})^{1/\kappa})$. Moreover, for $z$ bounded away from $f(0)$, we have that $p_0(z) = 1/P'(f^{-1}(z))$ is bounded.  Hence, when $t\in [0,t_0]$ and $tz^{-(1-\kappa)} > c$, using~\eqref{Eq:Karamata2} we can see that
\begin{align}
\label{Eq:smallz}
	p_t(z) = \frac{z}{-f_t'(f_t^{-1}(z))} \leq \frac{z}{ t \bigl\{ f''(x_t) \int_0^{x_t} g(y) \,dy + f'(x_t) g(x_t) \bigr\}} &\leq \frac{2(1-\kappa)z}{ \kappa t f(x_t)^\kappa P'(x_t)} \nonumber \\
	&\leq \frac{4 + \frac{2(1-\kappa)}{c}}{\kappa P'(x_t)},
\end{align}
so $p_t$ is also bounded uniformly for $t \in [0,t_0]$ and $tz^{-(1-\kappa)} > c$. It now follows from~\eqref{Eq:largez} and~\eqref{Eq:smallz} that for $t \in [0,t_0]$,
\begin{align*}
	\frac{ \{p_t(z)^{1/2} - p_0(z)^{1/2} \}^2}{t^2} &\lesssim_{\kappa,f} \mathbbm{1}_{\{ t\leq cz^{1-\kappa}\}} p_0(z) z^{2\kappa -2} + \mathbbm{1}_{\{ t > cz^{1-\kappa}\}} t^{-2} \\
	&\leq p_0(z) z^{2\kappa-2} + c^{-2} z^{2\kappa -2}.
\end{align*}
Since $\kappa > 1/2$, we have
\[
	\int_0^{f(0)} z^{2\kappa-2} p_0(z) \,dz = \int_0^\infty f(x)^{2\kappa -1} \,dx < \infty,
\]
and also the second term is integrable. Finally, then, the differentiability in quadratic mean property~\eqref{Eq:DQM} follows from the dominated convergence theorem.

To complete the proof of the first part of Proposition~\ref{Thm:SuperOracleLAM}, it suffices to study the differentiability properties of the functional $H$ along our path $\{f_t:t \in [0,t_0]\}$.  To this end, integrating by parts and using~\eqref{Eq:Karamata1}, we may see that
\begin{align*}
	\frac{d}{dt} \Bigm|_{t=0} &H(f_t) = \frac{d}{dt} \Bigm|_{t=0} \int_0^\infty f(x)^{\kappa} \{1-th(x)\}^\kappa \,dx = - \kappa \int_0^\infty f(x)^\kappa h(x) \,dx \\
	& = -\kappa \int_0^\infty f(x)^{\kappa -1 } f'(x) \int_0^x g(y) \,dy \,dx = - \int_0^\infty \frac{d}{dx}\{ f(x)^\kappa \} \int_0^x g(y) \,dy \,dx \\
	&= \int_0^\infty f(x)^\kappa g(x) \,dx = \int_0^\infty \bigl\{f(x)^{-(1-\kappa)} - H(f)\bigr\} g(x) f(x) \, dx \\
	&= \int_0^{f(0)} \bigl\{z^{-(1-\kappa)} - H(f)\bigr\} g\bigl(f^{-1}(z)\bigr) p_0(z) \, dz. 
\end{align*}
We therefore conclude that the efficient influence function is given by $z \mapsto \bigl\{z^{-(1-\kappa)} - H(f)\bigr\}$, and our result now follows from \citet[][Theorem~25.21]{vanderVaart1998}.

We now turn to the second claim of Proposition~\ref{Thm:SuperOracleLAM}. First, it is clear that $\|f_t\|_\infty = f(0) < \infty $ for all $t \in [0,t_0]$. As shown by~\eqref{Eq:Karamata1}, we have that $f_t(x) \lesssim f(x)^\kappa$ uniformly for $x \in [0,\infty)$ and $t\in [0,t_0]$, and it follows that, for any $\alpha>0$, we have $\sup_{t \in [0,t_0]} \int_0^\infty x^\alpha f_t(x) \,dx < \infty$. For the smoothness condition, in the interests of brevity, we will restrict attention here to $\beta \in (0,1]$; the arguments extend naturally to any $\beta>0$. For $\beta \in (0,1]$ we claim that $\sup_{t \in [0,t_0]} M_{f_t,\beta}(x) \lesssim_{\kappa,f} \max \{1/x, P'(x) \}$, so that we have
\[
	\sup_{t \in [0,t_0]} \int_0^\infty f_t(x) \biggl\{ \frac{M_{f,\beta}(x)}{f_t(x)} \biggr\}^\lambda \,dx  \lesssim_{\kappa,f} \int_0^\infty f(x)^{(1-\lambda)\kappa} \max\{ x^{-\lambda}, P'(x)^\lambda \} \,dx  < \infty
\]
for any $\lambda \in (0,1)$. To establish this claim, we have $\inf_{t \in [0,t_0]} \inf_{x \in [0,1]} f_t(x) > 0$, and so it follows from the smoothness of $f$ and $h$ that for $t \in [0,t_0]$,
\begin{align*}
	\sup_{x \in (0,1]} \sup_{y,z \in [0,2x],y \neq z} & \frac{|f_t(z) - f_t(y)|}{|z-y|^\beta f_t(x)} \\
	&\lesssim \sup_{x \in (0,1]} \sup_{y,z \in [0,2x], y \neq z} \frac{|f(z)-f(y)| + t_0 | f(z)h(z) - f(y)h(y)|}{|z-y|^\beta} < \infty.
\end{align*}
It follows that for $x \in (0,1]$ we have $M_{f,\beta}(x) \lesssim_{\kappa,f} 1/x$. Writing $ \mathrm{deg}(P)$ for the degree of the strictly increasing polynomial $P$, we have that 
\[
	0 < \inf_{x \in [1,\infty) } \frac{P'(x)}{x^{\mathrm{deg}(P)-1}} \leq \sup_{x \in [1,\infty)} \frac{P'(x)}{x^{\mathrm{deg}(P)-1}}  < \infty.
\]
Now for $x \geq 1$ and $y$ such that $|y-x| \leq x \wedge \{1/P'(x)\}$ we have that
\begin{align}
\label{Eq:P'cont}
	|P(y)-P(x)| &\lesssim_f |y-x| \max(x,y)^{\mathrm{deg}(P)-1} \lesssim_f |y-x| P'(x) \leq 1 \nonumber \\
	|P'(y) - P'(x)| & \lesssim_f |y-x| \max(x,y)^{\mathrm{deg}(P) \vee 2 - 2 } \lesssim_f |y-x| P'(x).
\end{align}
It therefore follows that
\begin{align*}
	\sup_{x \in [1,\infty)} \sup_{\substack{y,z \in B_x(x \wedge \{1/P'(x)\}) \\ y \neq z}}& \frac{|f(z)-f(y)|}{f(x) \{ P'(x) |z-y| \}^\beta} \\
	&= \sup_{x \in [1,\infty)} \sup_{\substack{y,z \in B_x(x \wedge \{1/P'(x)\}) \\ y \neq z}} \frac{e^{P(x) - P(y)}  \bigl| e^{P(y)-P(z)} -1 \bigr|}{\{ P'(x) |z-y| \}^\beta} < \infty.
\end{align*}
We conclude from~\eqref{Eq:Karamata1} both that $\sup_{x \in [1,\infty)} f(x)^{1-\kappa} P'(x) | \int_0^x g(y) \,dy | < \infty$ and that $\inf_{x \in [1,\infty)} f_t(x) / \{ f(x) + t f(x)^\kappa\} > 0$. Using~\eqref{Eq:P'cont} we can now see that for $x \in [1,\infty)$ and $y,z \in B_x( x \wedge \{1/P'(x)\})$ we have for $t \in [0,t_0]$ that
\begin{align*}
	& \frac{|f_t(z)-f_t(y)|}{f_t(x)} \lesssim_{\kappa,f} \frac{|f(z)-f(y)|}{ f(x)} + \frac{| f(z)h(z) - f(y)h(y)|}{f(x)^\kappa} \\
	& \lesssim_{\kappa,f} \{P'(x) |z-y| \}^\beta \\
	&\hspace{1cm}+ \frac{ |f'(z) \int_y^z g(u) \,du | + |\int_0^y g(u) \,du| \{ f(z)|P'(z) - P'(y)| + P'(y)| f(z)-f(y)| \} }{f(x)^\kappa} \\
	& \lesssim_{\kappa,f} \{P'(x) |z-y| \}^\beta + |z-y| |f'(x)| /f(x) + |z-y| + |f(z)/f(y)-1| \\
	&\lesssim_{\kappa,f} \{P'(x) |z-y|\}^\beta.
\end{align*}
This verifies our claim and the result therefore follows.
\end{proof}

\subsection{Proof of Theorem~\ref{Thm:LAM} on the local asymptotic minimax lower bound}

\begin{proof}[Proof of Theorem~\ref{Thm:LAM}]
(i) We check the conditions of, and apply, Theorem~3.11.5 of \citet{vanderVaartWellner1996}, and therefore borrow some of their terminology. Define the Hilbert space $H:= \mathbb{R}^2$ with inner product $ \langle (t_1,t_2),(t_1',t_2') \rangle_H := t_1 t_1' v_1(f,g) + t_2 t_2' v_2(f,g)$. We first claim that our sequence of experiments is asymptotically normal. That is to say, for independent normal random variables $Z_1 \sim N(0,v_1)$ and $Z_2 \sim N(0,v_2)$, if we define the iso-Gaussian process $\{\Delta_t =t_1Z_1 + t_2 Z_2 : t=(t_1,t_2) \in H\}$ we claim that
\[
	\log \frac{dP_{n,t}}{dP_{n,0}} = \Delta_{n,t} - \frac{1}{2} \|t\|_H^2
\]
with $\Delta_{n,t} \overset{d}{\rightarrow} \Delta_{t}$ for each fixed $t \in H$. 
Since $\int_\mathcal{X} f(x) h_1(x)^2 < \infty$, and since $K(0) = K'(0) = K''(0) = 1$, we have by the dominated convergence theorem that
\begin{align*}
	\Bigl| 1&/c_1(t_1) - 1 - \frac{t_1^2}{2} v_1 \Bigr| \\
	& =\biggl| \int_{\mathcal{X}} f(x) \Bigl\{ K(t_1h_1(x)) - 1 - t_1h_1(x) - \frac{t_1^2}{2} h_1(x)^2 \Bigr\} \,dx \biggr| \\
	& \leq \frac{1}{6}\sup_{w \in [-1,1]} |K'''(w)| \int_{|t_1h_1(x)| \leq 1} f(x) |t_1h_1(x)|^3 \,dx \\
	& \hspace{25pt} + \Bigl\{ 2 \sup_{w \in \mathbb{R}} |K(w)| + 1 + \frac{1}{2}\Bigr\} \int_{|t_1h_1(x)|>1} f(x) \{t_1h_1(x)\}^2 \,dx = o(t_1^2)
\end{align*}
as $t_1 \rightarrow 0$, with a similar calculation holding for $1/c_2(t_2)$ since $\int g h_2^2 < \infty$. Therefore, for each fixed $t=(t_1,t_2) \in H$ we have
\begin{align*}
	&\log \frac{dP_{n,t}}{dP_{n,0}} = \sum_{i=1}^m \log \frac{f_{m^{-1/2}t_1}(X_i)}{f(X_i)} + \sum_{j=1}^n \log \frac{g_{n^{-1/2} t_2}(Y_j)}{g(Y_j)} \\
	&= \sum_{i=1}^m \log K \Bigl( \frac{t_1 h_1(X_i)}{m^{1/2}} \Bigr)  \!+\! m \log c_1(m^{-1/2}t_1) +\! \sum_{j=1}^n \log K \Bigl( \frac{t_2 h_2(Y_j)}{n^{1/2}} \Bigr) + n \log c_2(n^{-1/2}t_2) \\
	& = \frac{t_1}{m^{1/2}} \sum_{i=1}^m h_1(X_i) + \frac{t_2}{n^{1/2}} \sum_{j=1}^n h_2(Y_j) - \frac{1}{2} \|t\|_H^2 + o_p(1) \overset{d}{\rightarrow} \Delta_t - \frac{1}{2} \|t\|_H^2,
\end{align*}
as claimed.

\sloppy{We will now show that the sequence of parameters defined by $\kappa_n(t) := T(f_{m^{-1/2}t_1}, g_{n^{-1/2} t_2})$ is \emph{regular}, in that there exists a continuous linear map $\dot{\kappa} : H \rightarrow \mathbb{R}$ and a sequence $(r_n)$ of real numbers such that}
\[
	r_n \{\kappa_n(t) - \kappa_n(0) \} \rightarrow \dot{\kappa} (t) 
\]
for each $t \in H$. Indeed, for any fixed $t=(t_1,t_2) \in H$ we have
\begin{align*}
&\kappa_n(t) - \kappa_n(0) = \int_\mathcal{X} \Bigl\{ f_{m^{-1/2} t_1} (x) \phi \bigl( f_{m^{-1/2} t_1} (x), g_{n^{-1/2} t_2} (x) \bigr) - f(x) \phi_x \Bigr\} \,dx \\
  &= \int_\mathcal{X}  f(x) \biggl\{ K \Bigl( \frac{t_1 h_1(x)}{m^{1/2}} \Bigr) \phi \biggl( K \Bigl( \frac{t_1 h_1(x)}{m^{1/2}} \Bigr) f(x), K \Bigl( \frac{t_2 h_2(x)}{n^{1/2}} \Bigr) g(x) \biggr)  - \phi_x \biggr\} \,dx \\
  &\hspace{7.5cm}+o(m^{-1/2}+n^{-1/2}) \\
	&= \int_\mathcal{X} f(x)  \biggl[ \frac{t_1 h_1(x)}{m^{1/2}} \bigl\{ \phi_x +(f \phi_{10})_x \bigr\} +  \frac{t_2 h_2(x)}{n^{1/2}} (g\phi_{01})_x  \biggr] \,dx +o(m^{-1/2}+n^{-1/2}) \\
	&= \frac{t_1v_1}{m^{1/2}} + \frac{t_2v_2}{n^{1/2}} + o(m^{-1/2}+n^{-1/2}).
\end{align*}
We may therefore take
\[
	r_n = (v_1/m + v_2/n)^{-1/2} \quad \text{and} \quad \dot{\kappa} (t_1,t_2)= \frac{ t_1 v_1 + A^{1/2}t_2v_2}{(v_1 + A v_2)^{1/2}}
\]
to conclude that our sequence of parameters $\kappa_n$ is regular. 

The adjoint $\dot{\kappa}^*: \mathbb{R} \rightarrow H$ of $\dot{\kappa}$ is given by
\[
	\dot{\kappa}^*(b^*) = \Bigl( \frac{b^*}{(v_1+Av_2)^{1/2}}, \frac{A^{1/2}b^*}{(v_1+Av_2)^{1/2}} \Bigr)
\]
as this satisfies $ \langle \dot{\kappa}^* (b^*), t \rangle_H = b^* \dot{\kappa}(t)$ for all $b^* \in \mathbb{R}$ and $t \in H$. Since $\|\dot{\kappa}^*(b^*)\|_H^2 = (b^*)^2$ for all $b^* \in \mathbb{R}$, we may therefore take $G \sim N(0,1)$ and apply Theorem~3.11.5 of \citet{vanderVaartWellner1996} to deduce that for any estimator sequence $T_{m,n}$,
\[
	\sup_{I \in \mathcal{I}} \liminf_{n \rightarrow \infty} \max_{t \in I} \mathbb{E}_{P_{n,t}} \biggl\{ \frac{(T_{m,n} - T)^2}{v_1/m + v_2/n} \biggr\} \geq \mathbb{E}( G^2) = 1.
\]
This concludes the proof of (i).

(ii) Since $k: \mathbb{R} \rightarrow [1/2,3/2]$ we have that $f(x)/3 \leq f_t(x) \leq 3f(x)$ and $g(x)/3 \leq g_t(x) \leq 3g(x)$ for all $t \in \mathbb{R}$ and $x \in \mathbb{R}^d$ and, to establish the result, it remains to show that $\max\{M_{f_t,\tilde{\beta}}(x), M_{g_t,\tilde{\beta}}(x)\} \lesssim M_{\beta}(x)$ for $t \leq 1$, say. For ease of presentation, we first prove this in the case $\beta \in (0,1]$. When $x \in \mathcal{X}$ and $y,z \in B_x\bigl(1/M_\beta(x)\bigr)$, we have that
\begin{align}
  \label{Eq:ft1}
	\frac{|f_t(z)-f_t(y)|}{f_t(x)} &= \frac{|K(th_1(z))f(z) - K(th_1(y))f(y)|}{K(th_1(x))f(x)} \nonumber \\
	& \leq \frac{3}{f(x)}|f(z)-f(y)| + \frac{2 f(y)}{f(x)} | K(th_1(z)) - K(th_1(y))| \nonumber \\
	& \leq 3 \{M_\beta(x) \|z-y\|\}^\beta + 4  | K(th_1(z)) - K(th_1(y))|.
\end{align}
Additionally,
\begin{align}
  \label{Eq:ft2}
	&|h_1(z)-h_1(y)| \leq | \phi_z - \phi_y| + f(z) | (\phi_{10})_z - (\phi_{10})_y| + |(\phi_{10})_y| |f(z)-f(y)| \nonumber \\
	& \leq L \bigl( 1 \vee |\phi_y +(f \phi_{10})_y| \bigr) \Bigl( 1 + \frac{f(z)}{f(y)} \Bigr) \biggl\{ 3 \Bigl| \frac{f(z)}{f(y)} -1 \Bigr|^{(\beta^*-1) \wedge 1} + \Bigl| \frac{g(z)}{g(y)} -1 \Bigr|^{\beta^* \wedge 1} \biggr\} \nonumber \\
	& \lesssim (1 + |h_1(y)|) \{ M_\beta(x) \|z-y\| \}^{\tilde{\beta}}.
\end{align}
In particular, there exists $c = c(d,\vartheta,\xi)$ such that, whenever $\|z-y\| M_\beta(x) \leq c$, we have $|h_1(z)-h_1(y)| \leq \max(1,|h_1(y)| \wedge |h_1(z)|)/2$. Writing $L_{t,y,z}$ for the line segment between $th_1(y)$ and $th_1(z)$, and using the fact that $\sup_{w \in \mathbb{R}} (1+|w|)|K'(w)| < \infty$, we now have for $z,y$ such that $\|z-y\|M_\beta(x) \leq c$ that
\begin{align}
\label{Eq:LineSegment}
	| K(th_1(z)) &- K(th_1(y))|  \leq t |h_1(z)-h_1(y)| \sup_{w \in L_{t,y,z}} |K'(w)| \nonumber \\
	& \lesssim \frac{(1+t|h_1(y)| \wedge |h_1(z)|)}{1+ \inf_{w \in L_{t,y,z}} |w|} \{M_\beta(x) \|z-y\|\}^{\tilde{\beta}} \lesssim \{M_\beta(x) \|z-y\|\}^{\tilde{\beta}}.
\end{align}
From~\eqref{Eq:ft1},~\eqref{Eq:ft2} and~\eqref{Eq:LineSegment}, we deduce that $M_{f_t,\tilde{\beta}}(x) \lesssim M_\beta(x)$.  Moreover, when $y,z \in B_x\bigl(1/M_\beta(x)\bigr)$, we have that
\begin{align*}
	&|h_2(z) - h_2(y)| \leq f(z) | (\phi_{01})_z - (\phi_{01})_y| + |(\phi_{01})_y| |f(z)-f(y)| \\
	& \leq \frac{f(z)}{f(y)} \bigl( 1 \vee f(y) |(\phi_{01})_y| \bigr)\biggl\{ \Bigl| \frac{f(z)}{f(y)} -1 \Bigr|^{\beta^* \wedge 1} \!\!\!\!\! + \Bigl| \frac{g(z)}{g(y)} -1 \Bigr|^{(\beta^*-1) \wedge 1} \biggr\} +|(\phi_{01})_y| |f(z)-f(y)| \\
	& \lesssim ( 1+ |h_2(y)|) \{M_\beta(x) \|z-y\|\}^{\tilde{\beta}}.
\end{align*}
It now follows by very similar arguments to those in~\eqref{Eq:LineSegment} that $M_{g_t,\tilde{\beta}}(x) \lesssim M_\beta(x)$.

We now extend these arguments to cover the $\beta >1 $ case. For a multi-index $\boldsymbol{\alpha} \in \mathbb{N}_0^d$ with $|\boldsymbol{\alpha}| \leq \tilde{\underline{\beta}}:= \lceil \tilde{\beta} \rceil -1$, we have that $\partial^{\boldsymbol{\alpha}} \{ K(t h_1(x)) \}$ can be written as a finite sum of terms of the form
\begin{equation}
\label{Eq:Sum1}
	t^r ( \partial^{\boldsymbol{\alpha}^{(1)}} h_1) \ldots ( \partial^{\boldsymbol{\alpha}^{(r)}} h_1)(x) K^{(r)}\bigl(th_1(x)\bigr)
\end{equation}
where $r \in \mathbb{N}_0$ satisfies $r \leq | \boldsymbol{\alpha}|$, and the multi-indices $\boldsymbol{\alpha}^{(1)},\ldots,\boldsymbol{\alpha}^{(r)} \in \mathbb{N}_0^d$ satisfy $|\boldsymbol{\alpha}^{(1)}| + \ldots + | \boldsymbol{\alpha}^{(r)}| = | \boldsymbol{\alpha}|$. Moreover, for any $j=1,\ldots,r$, we have that $\partial^{\boldsymbol{\alpha}^{(j)}} h_1$ is a finite sum of terms of the form
\begin{equation}
\label{Eq:Sum2}
	( \partial^{\boldsymbol{\beta}^{(1)}} f) \ldots ( \partial^{\boldsymbol{\beta}^{(\ell_1)}} f)( \partial^{\boldsymbol{\gamma}^{(1)}} g) \ldots ( \partial^{\boldsymbol{\gamma}^{(\ell_2)}} g)(x) \phi_{\ell_1 \ell_2} (f(x),g(x)),
\end{equation}
where $\ell_1,\ell_2 \in \mathbb{N}_0$ satisfy $\ell_1+\ell_2\leq | \boldsymbol{\alpha}^{(j)}|+1$, and where moreover the multi-indices $\boldsymbol{\beta}^{(1)},\ldots,\boldsymbol{\beta}^{(\ell_1)}, \boldsymbol{\gamma}^{(1)},\ldots,\boldsymbol{\gamma}^{(\ell_2)} \in \mathbb{N}_0^d$ satisfy $|\boldsymbol{\beta}^{(1)}| + \ldots + | \boldsymbol{\beta}^{(\ell_1)}| + |\boldsymbol{\gamma}^{(1)}| + \ldots + | \boldsymbol{\gamma}^{(\ell_2)}|  = | \boldsymbol{\alpha}^{(j)}|$. Using the fact that $\sup_{w \in \mathbb{R}} (1+|w|^r)|K^{(r)}(w)| < \infty$ for any $r \in \mathbb{N}$ and assumption (i) in the definition of $\tilde{\Phi}$, we therefore have the bounds
\[
	|\partial^{\boldsymbol{\alpha}^{(j)}} h_1(x)| \lesssim M_\beta(x)^{|\boldsymbol{\alpha}^{(j)}|} (1+|h_1(x)|) \! \quad \text{ and } \quad \! |\partial^{\boldsymbol{\alpha}} \{ K(t h_1(x)) \} | \lesssim M_\beta(x)^{|\boldsymbol{\alpha}|}.
\]
It follows that, for any multi-index $\boldsymbol{\alpha}$ with $|\boldsymbol{\alpha}| \leq \tilde{\underline{\beta}}$ we have that
\[
	\Bigl| \frac{\partial^{\boldsymbol{\alpha}} f_t(x)}{f_t (x)} \Bigr| \lesssim M_\beta(x)^{|\boldsymbol{\alpha}|}.
\]
Since, for any multi-index $\boldsymbol{\alpha}$ with $|\boldsymbol{\alpha}| \leq \tilde{\underline{\beta}}$, we have that $\partial^{\boldsymbol{\alpha}} h_2$ is a finite sum of terms of the form
\[
	( \partial^{\boldsymbol{\beta}^{(1)}} f) \ldots ( \partial^{\boldsymbol{\beta}^{(\ell_1+1)}} f)( \partial^{\boldsymbol{\gamma}^{(1)}} g) \ldots ( \partial^{\boldsymbol{\gamma}^{(\ell_2-1)}} g)(x) \phi_{\ell_1 \ell_2} (f(x),g(x)),
\]
we deduce by similar arguments that $| \partial^{\boldsymbol{\alpha}} g_t(x) | \lesssim g_t (x) M_\beta(x)^{|\boldsymbol{\alpha}|}$ for any multi-index $\boldsymbol{\alpha}$ with $|\boldsymbol{\alpha}| \leq \tilde{\underline{\beta}}$. Now we have for any $\ell_1, \ell_2 \in \mathbb{N}_0$ with $\ell_1+\ell_2 \leq \beta^*-1$ and $y,z \in B_x\bigl(1/M_\beta(x) \bigr)$ that
\begin{align*}
	&| \phi_{\ell_1\ell_2}( f(z), g(z)) - \phi_{\ell_1\ell_2}(f(y),g(y)) | \\
	& \lesssim f(y)^{-\ell_1}g(y)^{-\ell_2}( 1 \vee |h_1(y)|) \biggl\{ \Bigl| \frac{f(z)}{f(y)} -1 \Bigr|^{(\beta^*-\ell_1) \wedge 1} + \Bigl| \frac{g(z)}{g(y)} -1 \Bigr|^{(\beta^*-\ell_2) \wedge 1} \biggr\} \\
	& \lesssim f(y)^{-\ell_1}g(y)^{-\ell_2}( 1 \vee |h_1(y)|) \{M_\beta(x) \|z-y\|\}^{\min\{1,\beta^*-\ell_1, \beta^*-\ell_2 \}}.
\end{align*}
It follows from this, together with the representation~\eqref{Eq:Sum2} and Lemma~\ref{Lemma:15over7} that, for any multi-index $\boldsymbol{\alpha}$ with $|\boldsymbol{\alpha}| \leq \tilde{\underline{\beta}}$, we have that
\begin{align*}
	| \partial^{\boldsymbol{\alpha}} h_1(z) - \partial^{\boldsymbol{\alpha}} h_1(y) | &\lesssim M_\beta(x)^{|\boldsymbol{\alpha}|} ( 1 \vee |h_1(y)| ) \{M_\beta(x) \|z-y\|\}^{\min\{1,\beta - \tilde{\underline{\beta}},\beta^*-1-\tilde{\underline{\beta}}\} \}} \\
	& \lesssim M_\beta(x)^{|\boldsymbol{\alpha}|} ( 1 \vee |h_1(y)| ) \{M_\beta(x) \|z-y\|\}^{\tilde{\beta}-\tilde{\underline{\beta}}}.
\end{align*}
By a similar argument to~\eqref{Eq:LineSegment}, and using~\eqref{Eq:Sum1} and the fact that $\sup_{w \in \mathbb{R}} (1+|w|^r)|K^{(r)}(w)| < \infty$ for any $r \in \mathbb{N}$, we can now see that, for any multi-index $\boldsymbol{\alpha}$ with $|\boldsymbol{\alpha}| \leq \tilde{\underline{\beta}}$,
\[
	\bigl| \partial^{\boldsymbol{\alpha}}\bigl\{K(t h_1(z)) \bigr\} - \partial^{\boldsymbol{\alpha}} \bigl\{K(t h_1(y)) \bigr\} \bigr| \lesssim M_\beta(x)^{|\boldsymbol{\alpha}|}\{M_\beta(x) \|z-y\|\}^{\tilde{\beta}-\tilde{\underline{\beta}}}.
\]
Using Lemma~\ref{Lemma:15over7} it then follows that, for any multi-index $\boldsymbol{\alpha}$ with $|\boldsymbol{\alpha}| = \tilde{\underline{\beta}}$, we have
\[
	| \partial^{\boldsymbol{\alpha}} f_t(z) - \partial^{\boldsymbol{\alpha}} f_t(y) | \lesssim f_t(x) M_\beta(x)^{\tilde{\underline{\beta}}} \{M_\beta(x) \|z-y\|\}^{\tilde{\beta}-\tilde{\underline{\beta}}} 
\]
and so $M_{f_t,\tilde{\beta}}(x) \lesssim M_\beta(x)$, as required.  Similarly, $M_{g_t,\tilde{\beta}}(x) \lesssim M_\beta(x)$, and this completes the proof of the first statement in Theorem~\ref{Thm:LAM}(ii).

It remains to prove the local asymptotic minimax result for $\hat{T}_{m,n}$ under the conditions of Theorem~\ref{Thm:Main}, together with $\tilde{\beta} = \beta$.  Observe that 
\begin{align*}
\sup_{I \in \mathcal{I}} & \limsup_{n \rightarrow \infty} \max_{t = (t_1,t_2) \in I} n\mathbb{E}_{P_{n,t}} \Bigl[ \bigl\{\hat{T}_{m,n} - T(f_{m^{-1/2}t_1},g_{n^{-1/2}t_2})\bigr\}^2\Bigr]  \\
&\leq \limsup_{n \rightarrow \infty} \sup_{(\tilde{f},\tilde{g}) \in \mathcal{F}_{d,\tilde{\vartheta}}} \biggl\{n\mathbb{E}_{\tilde{f},\tilde{g}} \bigl[ \bigl\{\hat{T}_{m,n} - T(\tilde{f},\tilde{g})\bigr\}^2\bigr] - \frac{n}{m}v_1(\tilde{f},\tilde{g}) - v_2(\tilde{f},\tilde{g})\biggr\} \\
&\hspace{1cm}+ \sup_{I \in \mathcal{I}} \limsup_{n \rightarrow \infty} \max_{t = (t_1,t_2) \in I} \biggl\{\frac{n}{m}v_1(f_{m^{-1/2}t_1},g_{n^{-1/2}t_2}) + v_2(f_{m^{-1/2}t_1},g_{n^{-1/2}t_2})\biggr\} \\
&\leq \frac{1}{A} v_1(f,g) + v_2(f,g),
\end{align*}
where, in the second inequality, we have applied Theorem~\ref{Thm:Main} to the first term, and used the continuity properties of $v_1$ and $v_2$ for the second term.  The fact that the inequalities in this display are attained follows from Theorem~\ref{Thm:LAM}(i), and this completes the proof.
\end{proof}

\subsection{Auxiliary lemmas}

\begin{lemma}
\label{Prop:FunctionalClasses}
Suppose that $\phi \in \Phi(\xi)$ for some $\xi = (\kappa_1,\kappa_2,\beta^*,L) \in \Xi$. Then
\begin{enumerate}
    \item[(i)] For all $\boldsymbol{\epsilon} = (\epsilon_1, \epsilon_2) \in (-1/2,1/2)^2$ and $\mathbf{z} = (u,v) \in (0,
\infty)^2$ we have
\begin{align*}
	\max\bigl\{ u| \phi_{10}(\mathbf{z} + \boldsymbol{\epsilon} \circ \mathbf{z})-\phi_{10}(\mathbf{z})| , & v| \phi_{01}(\mathbf{z} + \boldsymbol{\epsilon} \circ \mathbf{z})-\phi_{01}(\mathbf{z})| \bigr\} \\
	& \leq 2^{1+|\kappa_1|+|\kappa_2|+2L} L  \| \boldsymbol{\epsilon} \| u_\wedge^{\kappa_1} u_\vee^L v_\wedge^{\kappa_2} v_\vee^L.
\end{align*}
    \item[(ii)] For all $\boldsymbol{\epsilon} = (\epsilon_1, \epsilon_2) \in (-1/2,1/2)^2$ and $\mathbf{z} = (u,v) \in (0,
\infty)^2$ we have
\begin{align*}
    \biggl| \phi(\mathbf{z} + \boldsymbol{\epsilon} \circ \mathbf{z}) - \sum_{\ell_1,\ell_2 =0}^\infty \mathbbm{1}_{\{\ell_1+\ell_2 \leq \beta^*-1\}} &\frac{(u \epsilon_1)^{\ell_1}(v\epsilon_2)^{\ell_2}}{\ell_1! \ell_2!} \phi_{\ell_1\ell_2}(\boldsymbol{z}) \biggr| \\
    &\leq 2^{1+|\kappa_1|+|\kappa_2|+2L} L u_\wedge^{\kappa_1} u_\vee^L v_\wedge^{\kappa_2} v_\vee^L (|\epsilon_1| \vee |\epsilon_2|)^{\beta^*}.
\end{align*}
\end{enumerate}
\end{lemma}
\begin{proof}[Proof of Lemma~\ref{Prop:FunctionalClasses}]
  By the definition of the class $\Phi$, for each $\mathbf{z} \in \mathcal{Z}$, the Hessian matrix
  \[
    H(\mathbf{z}) := \begin{pmatrix} u^2 \phi_{20}(\mathbf{z}) & uv \phi_{11}(\mathbf{z}) \\ uv \phi_{11}(\mathbf{z}) & v^2 \phi_{02}(\mathbf{z}) \end{pmatrix}
  \]
  satisfies $\|H(\mathbf{z}) \|_\mathrm{op} \leq 2L u_\wedge^{\kappa_1} u_\vee^L v_\wedge^{\kappa_2} v_\vee^L$.  Now, fixing $\mathbf{z} \in \mathcal{Z}$, the function $g:[0,1] \rightarrow \mathbb{R}$ given by $g(t):=u\phi_{10}(\mathbf{z}+t\boldsymbol{\epsilon} \circ \mathbf{z} )$ is differentiable with $g'(t)=\{H(\mathbf{z}+t\boldsymbol{\epsilon} \circ \mathbf{z}) (\epsilon_1,\epsilon_2)^T\}_1$. Thus, by the mean value theorem,
\begin{align*}
	&u| \phi_{10}(\mathbf{z} + \boldsymbol{\epsilon} \circ \mathbf{z})-\phi_{10}(\mathbf{z})| = |g(1)-g(0)| \\
	& \leq 2L (1+1/2)^{2L} \max\{ (1/2)^{-\kappa_1^- - \kappa_2^-}, (1+1/2)^{\kappa_1^+ + \kappa_2^+} \}\| \boldsymbol{\epsilon} \| u_\wedge^{\kappa_1} u_\vee^L v_\wedge^{\kappa_2} v_\vee^L.
\end{align*}
A similar calculation with  $\phi_{01}$ completes the proof of part (i).

To prove part (ii) we use the mean value form of the remainder in Taylor's theorem. Fixing $\boldsymbol{z} \in \mathcal{Z}$ and $\boldsymbol{\epsilon} \in (-1/2,1/2)^2$ define $h:[0,1] \rightarrow \mathbb{R}$ by $h(t) = \phi(\mathbf{z}+t\boldsymbol{\epsilon} \circ \mathbf{z} )$. Then we have
\begin{align*}
    \biggl|& \phi(\mathbf{z} + \boldsymbol{\epsilon} \circ \mathbf{z}) - \sum_{\ell_1,\ell_2 =0}^\infty \mathbbm{1}_{\{\ell_1+\ell_2 \leq \beta^*-1\}} \frac{(u \epsilon_1)^{\ell_1}(v\epsilon_2)^{\ell_2}}{\ell_1! \ell_2!} \phi_{\ell_1\ell_2}(\boldsymbol{z}) \biggr| = \biggl| h(1) - \sum_{b=0}^{\beta^*-1} \frac{1}{b!} h^{(b)}(0) \biggr| \\
    & \leq \sup_{t \in [0,1]} \frac{1}{\beta^*!} |h^{(\beta^*)}(t)| = \sup_{t \in [0,1]} \biggl| \sum_{\ell=0}^{\beta^*} \frac{(u \epsilon_1)^\ell(v\epsilon_2)^{\beta^*-\ell}}{\ell! (\beta^*-\ell)!} \phi_{\ell,\beta^*-\ell}(\mathbf{z}+t\boldsymbol{\epsilon} \circ \mathbf{z} ) \biggr| \\
    & \leq L (|\epsilon_1| \vee |\epsilon_2|)^{\beta^*} 2^{|\kappa_1|+|\kappa_2|+2L}  u_\wedge^{\kappa_1} u_\vee^L v_\wedge^{\kappa_2} v_\vee^L \sum_{\ell=0}^{\beta^*} \frac{1}{\ell! (\beta^*-\ell)!} \\
    & \leq 2^{1+|\kappa_1|+|\kappa_2|+2L} L u_\wedge^{\kappa_1} u_\vee^L v_\wedge^{\kappa_2} v_\vee^L(|\epsilon_1| \vee |\epsilon_2|)^{\beta^*},
\end{align*}
as claimed.
\end{proof}

\begin{lemma}
\label{Lemma:hxinvbounds}
Fix $f \in \mathcal{F}_d$ and $\beta \in (0,\infty)$, and let $\mathcal{S}_n \subseteq (0,1), \mathcal{X}_{n} \subseteq \mathbb{R}^d$ be such that
\[
	a_n:= \sup_{s \in \mathcal{S}_n} \sup_{x \in \mathcal{X}_{n}} \frac{s M_{f,\beta}(x)^d}{V_df(x)} \rightarrow 0.
\]
Then there exist $n_*=n_*(d, \beta, (a_n)) \in \mathbb{N}$, coefficients $b_\ell(x)$ and  $A=A(d, \beta, (a_n)) \in (0,\infty)$ such that, for all $n \geq n_*, s \in \mathcal{S}_n$ and $x \in \mathcal{X}_{n}$, we have
\[
	\biggl| V_d f(x) h_{x,f}^{-1}(s)^d - \sum_{\ell=0}^{\lceil \beta/2 \rceil -1} b_\ell(x) s^{1+2\ell/d} \biggr| \leq As \Bigl\{ \frac{s M_{f,\beta}(x)^d}{f(x)} \Bigr\}^{\beta/d}.
\]
Moreover, $b_0(x)=1$ and $|b_\ell(x)| \leq A \{M_{f,\beta}(x)^d/f(x)\}^{2\ell/d}$.
\end{lemma}
\begin{proof}[Proof of Lemma~\ref{Lemma:hxinvbounds}]
By a Taylor expansion, for $r \leq 1/M_{f,\beta}(x)$ we have that
\begin{equation}
\label{Eq:BasicTaylor}
	\biggl| h_{x,f}(r) - V_d r^d f(x) - \sum_{\ell=1}^{\lceil \beta /2 \rceil -1} r^{d+2\ell} c_\ell(x) \biggr| \lesssim_{\beta,d} r^d f(x) \{M_{f,\beta}(x) r \}^{\beta}
\end{equation}
for some coefficients $c_\ell(\cdot)$ satisfying $|c_\ell(x)| \lesssim_{\beta,d} f(x) M_{f,\beta}(x)^{2\ell}$. In particular,
\[
	\biggl| \frac{h_{x,f}(r)}{V_d r^d f(x)} -1 \biggr| \lesssim_{\beta,d} \{M_{f,\beta}(x) r\}^{2 \wedge \beta}.
\]
Thus there exists $C=C(d,\beta) > 0$ such that we have $| \frac{h_{x,f}(r)}{V_dr^d f(x)}-1| \leq 1/2$ whenever $r \leq 1/\{CM_{f,\beta}(x)\}$. Setting $r=\{ \frac{2s}{V_d f(x)}\}^{1/d}$ we have
\[
	rCM_{f,\beta}(x) = 2^{1/d} C \Bigl\{ \frac{sM_{f,\beta}(x)^d}{V_d f(x)} \Bigr\}^{1/d} \leq (2a_n)^{1/d} C \rightarrow 0.
\]
So, for $n$ large enough that $(2a_n)^{1/d}C \leq 1$, we have $h_{x,f}( \{\frac{2s}{V_df(x)} \}^{1/d}) \geq s$, so $h_{x,f}^{-1}(s) \leq \{ \frac{2s}{V_df(x)} \}^{1/d}$ for all $x \in \mathcal{X}_{n}$ and  $s \in \mathcal{S}_n$. Now, since $ M_{f,\beta}(x) h_{x,f}^{-1}(s) \leq \{ \frac{2sM_{f,\beta}(x)^d}{V_d f(x)} \}^{1/d} \leq (2a_n)^{1/d} \rightarrow 0$, we may substitute $r=h_{x,f}^{-1}(s)$ into~\eqref{Eq:BasicTaylor} to see that
\[
	\biggl| \frac{s}{V_d f(x) h_{x,f}^{-1}(s)^d} -1 - \sum_{\ell=1}^{\lceil \beta/2 \rceil - 1} \frac{b_\ell(x)}{V_df(x)} h_{x,f}^{-1}(s)^{2\ell} \biggr| \lesssim_{\beta,d,(a_n)} \Bigl\{ \frac{s M_{f,\beta}(x)^d}{f(x)} \Bigr\}^{\beta/d}.
\]
This expansion can be inverted to yield the desired result by substituting this bound into itself and expanding functions of the form $r \mapsto r^{2\ell/d}$ about $r=1$.
\end{proof}

\begin{lemma}
\label{Lemma:15over7}
Fix $f \in \mathcal{F}_d$ and $\beta \in (0,\infty)$, and suppose that $\max\{\|y-x\|,\|z-x\|\} \leq 1/ \{ (6d)^{1/(\beta-\underline{\beta})}M_{f,\beta}(x) \}$. Then, for multi-indices $t \in \mathbb{N}_0^d$ with $|t| \leq \underline{\beta}$, we have that
\[
	\bigl| (\partial^t f)(z) -(\partial^t f)(y) \bigr| \leq 2 M_{f,\beta}(x)^{\min(\beta, |t|+1)}f(x) \|z-y\|^{\min(1,\beta-|t|)}. 
\]
\end{lemma}
\begin{proof}
First, if $|t|=\underline{\beta}$ then we simply have that
\[
	\bigl| (\partial^t f)(z) -(\partial^t f)(y) \bigr| \leq \|f^{(\underline{\beta})}(z)-f^{(\underline{\beta})}(y)\| \leq M_{f,\beta}(x)^\beta f(x) \|z-y\|^{\beta-\underline{\beta}},
\]
and the claim holds. Henceforth assume that $|t| \leq \underline{\beta}-1$ and $\underline{\beta} \geq 1$. Writing $\vertiii{\cdot}$ here for the largest absolute entry of an array, writing $L_{yz}$ for the line segment between $y$ and $z$, and arguing inductively we have that
\begin{align*}
	\bigl| & \partial^t f(z) -\partial^t f(y) \bigr| \leq \|z-y\| \sup_{w \in L_{yz}} \bigl\| \nabla \partial^t f(w) \bigl\| \\
	& \leq \|z-y\| \|f^{(|t|+1)}(x)\| + d^{1/2} \|z-y\| \biggl\{ \vertiii{ f^{(|t|+1)}(y) - f^{(|t| +1)}(x)} \\
	& \hspace{225pt} + \sup_{w \in L_{yz} } \vertiii{ f^{(|t|+1)}(w) - f^{(|t| +1)}(y)} \biggr\} \\
	& \leq \|z-y\| f(x) \Bigl[ M_{f, \beta}(x)^{|t|+1} \\
	& \hspace{50pt} + 2d^{1/2} M_{f, \beta}(x)^{\min(\beta,|t|+2)} \bigl\{ \|y-x\|^{\min(1,\beta-|t|-1)} + \|z-y\|^{\min(1,\beta-|t|-1)} \bigr\} \Bigr] \\
	& \leq \|z-y\| f(x) \Bigl\{ M_{f, \beta}(x)^{|t|+1} +  M_{f,\beta}(x)^{\min(\beta, |t| +2) - \min(1,\beta-|t|-1)} \Bigr\} \\
	& = 2 M_{f, \beta}(x)^{|t|+1}f(x) \|z-y\|,
\end{align*}
as required.
\end{proof}
The following lemma presents a tail bound for a $\mathrm{Beta}(a,b-a)$ random variable that is convenient to apply in settings where $a > 0$ is large and $a/b$ is small.
\begin{lemma}
\label{Lemma:BetaTailBounds}
Suppose $b>a>0$ and $B \sim \mathrm{Beta}(a,b-a)$. Writing $h(t):=t-\log(1+t)$ we have that
\begin{align*}
	\mathbb{P} \Bigl( \Bigl| B - &\frac{a}{b} \Bigr| \geq \frac{a^{1/2} u}{b} \Bigr)\leq 2 \exp \biggl( -a h \Bigl( \frac{a^{-1/2}b^{1/2} u}{b^{1/2} + a^{1/2} + u} \Bigr) \biggr) \!+\! 2\exp \biggl( -b h \Bigl( \frac{u}{b^{1/2} + a^{1/2}+ u} \Bigr) \biggr)
\end{align*}
for all $u \in [0, \infty)$.
\end{lemma}
\begin{proof}
Our proof relies on concentration inequalities for gamma random variables, which we establish now. For $a>0$, letting $\Gamma_a \sim \Gamma(a,1)$ we have by a Chernoff bound that for $t \geq 0$,
\begin{align*}
	\mathbb{P} \Bigl( \frac{\Gamma_a-a}{a} \geq t \Bigr) \leq \inf_{ \lambda \in (0,a)} e^{-\lambda t - \lambda} \Bigl( 1- \frac{\lambda}{a} \Bigr)^{-a} = e^{-ah(t)}.
\end{align*}
Similarly, for $t \in [0,1)$ we have that
\[
	\mathbb{P} \Bigl( \frac{\Gamma_a-a}{a} \leq -t \Bigr) \leq \inf_{\lambda>0} e^{\lambda -\lambda t} \Bigl( 1+\frac{\lambda}{a} \Bigr)^{-a}  = e^{-ah(-t)} \leq e^{-ah(t)},
\]
and thus, for all $t \geq 0$, we have that $\mathbb{P}( |\Gamma_a -a| \geq at) \leq 2e^{-ah(t)}$. Now, for independent random variables $\Gamma_a \sim \Gamma(a,1)$ and $\Gamma_{b-a} \sim \Gamma(b-a,1)$ we have that $\Gamma_a/(\Gamma_a+\Gamma_{b-a}) \sim \mathrm{Beta}(a,b)$, and so for $t \geq 0$ and $\epsilon \in (0,1)$ we have that
\begin{align*}
	\mathbb{P} \Bigl( \Bigl| B - \frac{a}{b} \Bigr| \geq t \Bigr) &= \mathbb{P} \biggl( \biggl| \frac{\Gamma_a-a}{\Gamma_a+\Gamma_{b-a}} + \frac{a}{b} \Bigl( \frac{b}{\Gamma_a + \Gamma_{b-a}} -1 \Bigr) \biggr| \geq t \biggr) \\
	& \leq \mathbb{P} \biggl( \biggr| \frac{a}{b} \Bigl( \frac{b}{\Gamma_a + \Gamma_{b-a}} -1 \Bigr) \biggr| \geq \epsilon t \biggr) + \mathbb{P} \biggl( \frac{|\Gamma_a -a|}{b} \geq \frac{(1-\epsilon)t}{1+\epsilon tb/a} \biggr) \\
	& \leq \mathbb{P} \biggl( \frac{|\Gamma_a + \Gamma_{b-a} - b|}{b} \geq \frac{\epsilon t b}{a + \epsilon t b} \biggr) + \mathbb{P} \biggl( \frac{|\Gamma_a-a|}{a} \geq \frac{(1-\epsilon)tb}{a+\epsilon t b} \biggr).
\end{align*}
Choosing $\epsilon=a^{1/2}/(a^{1/2}+b^{1/2})$ and writing $t=a^{1/2}u/b$ we may now see that
\begin{align*}
	\mathbb{P} \Bigl( \Bigl| B - &\frac{a}{b} \Bigr| \geq \frac{a^{1/2}u}{b} \Bigr) \\
	&\leq \mathbb{P} \biggl( \frac{|\Gamma_a + \Gamma_{b-a} - b|}{b} \geq \frac{u}{a^{1/2} + b^{1/2} + u} \biggr) \!+ \mathbb{P} \biggl( \frac{|\Gamma_a-a|}{a} \geq \frac{a^{-1/2} b^{1/2} u}{a^{1/2} + b^{1/2} + u} \biggr) \\
	& \leq 2\exp \biggl( -b h \Bigl( \frac{u}{b^{1/2} + a^{1/2}+ u} \Bigr) \biggr) + 2 \exp \biggl( -a h \Bigl( \frac{a^{-1/2}b^{1/2} u}{b^{1/2} + a^{1/2} + u} \Bigr) \biggr),
\end{align*}
as required.
\end{proof}

\begin{lemma}
\label{Lemma:GeneralisedHolder}
	Fix $d \in \mathbb{N}$ and $\vartheta =(\alpha,\beta,\lambda_1,\lambda_2,C) \in \Theta$. Suppose that $a,b,c \in [0,\infty)$ are such that $\frac{a}{\lambda_1} + \frac{b}{\lambda_2} + \frac{c}{\alpha} \leq 1$. Then
\[
	\sup_{(f,g) \in \mathcal{F}_{d,\vartheta}} \int_\mathcal{X} f(x)  \Bigl\{ \frac{M_{\beta}(x)^d}{f(x)} \Bigr\}^a \Bigl\{ \frac{M_{\beta} (x)^d}{g(x)} \Bigr\}^b (1+\|x\|)^c \,dx < \infty.
\]
\end{lemma}
\begin{proof}
	By the generalised H\"older inequality \citep[e.g.][Chapter~6, Exercise~31]{Folland1999}, if $X \sim f$ we have that
\begin{align*}
	\int_\mathcal{X} f(x) & \Bigl\{ \frac{M_{\beta}(x)^d}{f(x)} \Bigr\}^a \Bigl\{ \frac{M_{\beta} (x)^d}{g(x)} \Bigr\}^b (1+\|x\|)^c \,dx \\
	& = \mathbb{E} \biggl[ \Bigl\{ \frac{M_{\beta}(X)^d}{f(X)} \Bigr\}^a \Bigl\{ \frac{M_{\beta} (X)^d}{g(X)} \Bigr\}^b (1+\|X\|)^c \biggr] \\
	& \leq \mathbb{E} \biggl[ \Bigl\{ \frac{M_{\beta}(X)^d}{f(X)} \Bigr\}^{\lambda_1} \biggr]^{\frac{a}{\lambda_1}} \mathbb{E} \biggl[ \Bigl\{ \frac{M_{\beta}(X)^d}{g(X)} \Bigr\}^{\lambda_2} \biggr]^{\frac{b}{\lambda_2}} \mathbb{E} \biggl[ (1+\|X\|)^\frac{c}{1-\frac{a}{\lambda_1}-\frac{b}{\lambda_2}} \biggr]^{1-\frac{a}{\lambda_1}-\frac{b}{\lambda_2}} \\
	& \leq C^{\frac{a}{\lambda_1}+\frac{b}{\lambda_2}} \bigl[\mathbb{E} \bigl\{ (1+\|X\|)^\alpha \bigr\}\bigr]^{1-\frac{a}{\lambda_1}-\frac{b}{\lambda_2}} \leq C^{\frac{a}{\lambda_1}+\frac{b}{\lambda_2}} \{2^{\alpha} ( 1 + C)\}^{1-\frac{a}{\lambda_1}-\frac{b}{\lambda_2}},
\end{align*}
as required. 
\end{proof}

\begin{lemma}
\label{Lemma:hxinvbounds2}
Fix $f \in \mathcal{F}_d$ with $\max\bigl (\|f\|_\infty,\mu_\alpha(f) \bigr) \leq C$ and $\beta \in (0,\infty)$. Then for all $x \in \mathcal{X}$ and $s \in (0,1)$,
\begin{align*}
	\Bigl( \frac{s}{CV_d} \Bigr)^{1/d} &\leq h_{x,f}^{-1}(s) \leq \min \biggl\{ \|x\| + \Bigl( \frac{C}{1-s} \Bigr)^{1/\alpha},  \\
	& \Bigl( \frac{2s}{V_d f(x)} \Bigr)^{1/d} \biggl[ 1 + (6d)^{1/(\beta -\underline{\beta})} M_{f,\beta}(x) \Bigl\{ \|x\| + \Bigl( \frac{C}{1-s} \Bigr)^{1/\alpha} \Bigr\} \biggr] \biggr\}.
\end{align*}
\end{lemma}
\begin{proof}
The lower bound is immediate on noting that
\[
	h_{x,f}(r) \leq CV_d r^d.
\]
For the upper bound, by Lemma~\ref{Lemma:15over7}, if $\|y-x\| \leq 1/\{(6d)^{1/(\beta -\underline{\beta})} M_{f,\beta}(x)\}$, then we have that
\[
	\Bigl| \frac{f(y)}{f(x)} - 1 \Bigr| \leq 2 M_{f,\beta}(x)^{1 \wedge \beta} \|y-x\|^{1 \wedge \beta} \leq \frac{1}{2}.
\]
Thus, whenever $r \leq 1/ \{(6d)^{1/(\beta -\underline{\beta})} M_{f,\beta}(x) \}$ we have that
\[
	\frac{1}{2} V_d r^d f(x) \leq  h_{x,f}(r) \leq \frac{3}{2} V_d r^d f(x).
\]
Now, by the triangle and Markov's inequalities, for every $s \in (0,1)$,
\[
  \mathbb{P}\biggl(\|X_1-x\| > \|x\| + \Bigl(\frac{C}{1-s}\Bigr)^{1/\alpha}\biggr) \leq \mathbb{P}\biggl(\|X_1\| > \Bigl(\frac{C}{1-s}\Bigr)^{1/\alpha}\biggr) \leq 1-s,
\]
so that
\[
	h_{x,f}^{-1}(s) \leq \|x\| + \Bigl( \frac{C}{1-s} \Bigr)^{1/\alpha}.
\]
Hence,
\begin{align*}
	h_{x,f}^{-1}(s) &\leq \Bigl( \frac{2s}{V_d f(x)} \Bigr)^{1/d} \mathbbm{1}_{\{ h_{x,f}^{-1}(s) \leq 1/\{(6d)^{1/(\beta -\underline{\beta})} M_{f,\beta}(x)\} ) \}} \\
	&\hspace{65pt} + \Bigl\{ \|x\| + \Bigl( \frac{C}{1-s} \Bigr)^{1/\alpha} \Bigr\} \mathbbm{1}_{\{ h_{x,f}^{-1}(s) > 1/\{(6d)^{1/(\beta -\underline{\beta})} M_{f,\beta}(x)\} )\}} \\
	& \leq \Bigl( \frac{2s}{V_d f(x)} \Bigr)^{1/d} \biggl[ 1 + (6d)^{1/(\beta -\underline{\beta})} M_{f,\beta}(x) \Bigl\{ \|x\| + \Bigl( \frac{C}{1-s} \Bigr)^{1/\alpha} \Bigr\} \biggr],
\end{align*}
as required.
\end{proof}

The following lemma shows that we may restrict our main attention to the events
\begin{equation}
\label{Eq:SetA}
	A_i^X:=\bigl\{ h_{X_i,f}( \rho_{(k_X),i,X}^d ) \in \mathcal{I}_{m,X} \bigr\}, \quad A_i^Y:= \bigl\{ h_{X_i,g}( \rho_{(k_Y),i,Y}^d ) \in \mathcal{I}_{n,Y} \bigr\},
      \end{equation}
      for $i =1,\ldots,n$.
\begin{lemma}
\label{Lemma:Acomplement}
Fix $d \in \mathbb{N}$, $\vartheta \in \Theta$, $(\kappa_1,\kappa_2) \in \mathbb{R}^2$ and suppose that
\[
	\frac{\kappa_1^-}{\lambda_1} +  \frac{\kappa_2^-}{\lambda_2} + \frac{d(\kappa_1^-+\kappa_2^-)}{\alpha} \leq 1.
\]
Let $k_X^{\mathrm{L}} \leq k_X^{\mathrm{U}},k_Y^{\mathrm{L}} \leq k_Y^{\mathrm{U}}$ be deterministic sequences of positive integers such that $k_X^{\mathrm{L}} / \log m \rightarrow \infty$, $k_Y^{\mathrm{L}}/ \log n \rightarrow \infty$, $k_X^{\mathrm{U}}/m \rightarrow 0$ and $k_Y^{\mathrm{U}}/n \rightarrow 0$.  Then
\begin{align*}
	&\max_{\substack{k_X \in \{k_X^{\mathrm{L}},\ldots,k_X^{\mathrm{U}}\}\\k_Y \in \{k_Y^{\mathrm{L}},\ldots,k_Y^{\mathrm{U}}\}}} \sup_{(f,g) \in \mathcal{F}_{d, \vartheta}} \mathbb{E} \Bigl[ \max \Bigl\{ \hat{f}_{(k_X),1}^{\kappa_1}, \hat{f}_{(k_X),1}^L,  f(X_1)^{\kappa_1} \Bigr\} \\
	& \hspace{80pt} \times  \max \Bigl\{ \hat{g}_{(k_Y),1}^{\kappa_2}, \hat{g}_{(k_Y),1}^L, g(X_1)^{\kappa_2} \Bigr\} \bigl(1-\mathbbm{1}_{A_1^X} \mathbbm{1}_{A_1^Y} \bigr) \Bigr] = o(m^{-4}+n^{-4}) 
\end{align*}
      as $m,n \rightarrow \infty$.  
    \end{lemma}
    \begin{proof}[Proof of Lemma~\ref{Lemma:Acomplement}]
Given $a > - \min(k_X,k_Y), b > -\min(m-k_X,n+1-k_Y)$ define
\begin{align*}
	\Delta_{a,b}^{(1)} := \int_{[0,1] \setminus \mathcal{I}_{m,X}} \mathrm{B}_{k_X+a,m-k_X+b}(s) \,ds, \quad  \Delta_{a,b}^{(2)} := \int_{[0,1] \setminus \mathcal{I}_{n,Y}} \mathrm{B}_{k_Y+a,n+1-k_Y+b}(t) \,dt.
\end{align*}
By Lemma~\ref{Lemma:BetaTailBounds} we have that
\[
	\max_{\substack{k_X \in \{k_X^{\mathrm{L}},\ldots,k_X^{\mathrm{U}}\}\\k_Y \in \{k_Y^{\mathrm{L}},\ldots,k_Y^{\mathrm{U}}\}}} \sup_{a,b \in [-A,A]} \max( \Delta_{a,b}^{(1)}, \Delta_{a,b}^{(2)}) = o(m^{-9(1-\epsilon)/2} + n^{-9(1-\epsilon)/2})
\]
for any fixed $A \geq 0$ and $\epsilon>0$. Now, by Lemma~\ref{Lemma:hxinvbounds2} and writing $\kappa_i^+:= \max( \kappa_i,0)$ for $i=1,2$, we have that
\begin{align*}
	&\mathbb{E} \Bigl[ \max \Bigl\{ \hat{f}_{(k_X),1}^{\kappa_1}, \hat{f}_{(k_X),1}^L, f(X_1)^{\kappa_1} \Bigr\}  \max \Bigl\{\hat{g}_{(k_Y),1}^{\kappa_2}, \hat{g}_{(k_Y),1}^L, g(X_1)^{\kappa_2} \Bigr\} \bigl(1-\mathbbm{1}_{A_1^X} \mathbbm{1}_{A_1^Y} \bigr) \Bigr]  \\
	& = \int_\mathcal{X} f(x) \int_0^1 \int_0^1 \max \bigl(u_{x,s}^{\kappa_1}, u_{x,s}^L, f(x)^{\kappa_1} \bigr) \max \bigl( v_{x,t}^{\kappa_2}, v_{x,t}^L, g(x)^{\kappa_2} \bigr) \\
	& \hspace{50pt} \times  \max \bigl( \mathbbm{1}_{\{s \not\in \mathcal{I}_{m,X}\}}, \mathbbm{1}_{\{t \not\in \mathcal{I}_{n,Y}\}} \bigr) \mathrm{B}_{k_X,m-k_X}(s) \mathrm{B}_{k_Y,n+1-k_Y}(t) \,ds \,dt \,dx \\
	& \lesssim \int_\mathcal{X} f(x) \int_0^1 \int_0^1  \max \biggl\{ \Bigl( \frac{k_X}{ms} \Bigr)^{\kappa_1^+ \vee L}, \Bigl( \frac{ms M_\beta(x)^d (1+\|x\|)^d}{k_X f(x)(1-s)^{d/\alpha}} \Bigr)^{\kappa_1^-}, f(x)^{\kappa_1} \biggr\} \\
	& \hspace{50pt} \times  \max \biggl\{ \Bigl( \frac{k_Y}{nt} \Bigr)^{\kappa_2^+ \vee L}, \Bigl( \frac{ntM_\beta(x)^d (1+\|x\|)^d}{k_Y g(x)(1-t)^{d/\alpha}} \Bigr)^{\kappa_2^-}, g(x)^{\kappa_2} \biggr\} \\
	& \hspace{50pt}\times   \max \bigl( \mathbbm{1}_{\{s \not\in \mathcal{I}_{m,X}\}}, \mathbbm{1}_{\{t \not\in \mathcal{I}_{n,Y}\}} \bigr) \mathrm{B}_{k_X,m-k_X}(s) \mathrm{B}_{k_Y,n+1-k_Y}(t) \,ds \,dt \,dx \\
	& \lesssim \max \bigl(\Delta_{-(\kappa_1^+ \vee L),0}^{(1)},\Delta_{0,0}^{(1)},\Delta_{\kappa_1^-,-d\kappa_1^-/\alpha}^{(1)},\Delta_{-(\kappa_2^+ \vee L),0}^{(2)}, \Delta_{0,0}^{(2)},\Delta_{\kappa_2^-,-d\kappa_2^-/\alpha}^{(2)} \bigr) \\
	& \hspace{80pt} \times \int_\mathcal{X} f(x) \Bigl( \frac{M_\beta(x)^d}{f(x)} \Bigr)^{\kappa_1^-}  \Bigl( \frac{M_\beta(x)^d}{g(x)} \Bigr)^{\kappa_2^-} (1+\|x\|)^{d (\kappa_1^-+ \kappa_2^-)} \,dx \\
	& \lesssim (m^{-17/4} + n^{-17/4})  \int_\mathcal{X} f(x) \Bigl( \frac{M_\beta(x)^d}{f(x)} \Bigr)^{\kappa_1^-}  \Bigl( \frac{M_\beta(x)^d}{g(x)} \Bigr)^{\kappa_2^-} (1+\|x\|)^{d (\kappa_1^-+ \kappa_2^-)} \,dx.
\end{align*}
The conclusion follows immediately on appealing to Lemma~\ref{Lemma:GeneralisedHolder}.
\end{proof}

\begin{lemma}
\label{Lemma:Pinsker}
	Let $a,b,c \in \mathbb{R}$ be any fixed constants, and let $k^{\mathrm{L}} \leq k^{\mathrm{U}}$ be deterministic sequences of positive integers such that $k^{\mathrm{L}} \rightarrow \infty$ and $k^{\mathrm{U}}/n \rightarrow 0$ as $n \rightarrow \infty$. Then
\begin{align*}
  \int_0^1 \! \int_0^1 \Bigl| \mathrm{B}_{j+a,\ell+b,n+c-j-\ell}(s,t) \!-\! \mathrm{B}_{j+a,n-j}(s) \mathrm{B}_{\ell+b,n-\ell}(t) \Bigr| \! \,ds \,dt \leq \! \frac{(j \ell)^{\frac{1}{2}}}{n}\{1\!+\!o(1)\}
\end{align*}
as $n \rightarrow \infty$, uniformly for $j,\ell \in \{k^{\mathrm{L}}, \ldots, k^{\mathrm{U}} \}$.
\end{lemma}
\begin{proof}
In the following bound we make use the standard asymptotic expansions
\begin{align*}
	\log \Gamma(z) &= z \log z - z - \frac{1}{2} \log \Bigl( \frac{z}{2 \pi} \Bigr) + \frac{1}{12z} + O \Bigl( \frac{1}{z^3} \Bigr) \\
	\Psi(z) &= \log z - \frac{1}{2z} - \frac{1}{12z^2} + O \Bigl( \frac{1}{z^4} \Bigr)
\end{align*}
as $z \rightarrow \infty$.  
Using these expansions, by Lemma~\ref{Lemma:BetaTailBounds} and Pinsker's inequality we have that
\begin{align*}
	\int_0^1  \int_0^1 & \Bigl| \mathrm{B}_{j+a,\ell+b,n+c-j-\ell}(s,t) - \mathrm{B}_{j+a,n-j}(s) \mathrm{B}_{\ell+b,n-\ell}(t) \Bigr| \,ds \,dt \\
	& \leq \biggl\{2 \int_0^1 \int_0^{1-t} \mathrm{B}_{j+a,\ell+b,n+c-j-\ell}(s,t) \log \biggl( \frac{\mathrm{B}_{j+a,\ell+b,n+c-j-\ell}(s,t)}{\mathrm{B}_{j+a,n-j}(s) \mathrm{B}_{\ell+b,n-\ell}(t)} \biggr) \,ds \,dt \biggr\}^{1/2} \\ 
	& = 2^{\frac{1}{2}} \biggl[ \log \biggl( \frac{\Gamma(n+a+b+c) \Gamma(n-j) \Gamma(n-\ell)}{\Gamma(n+c-j-\ell) \Gamma(n+a) \Gamma(n+b)} \biggr) \!+\! (n\!-\!c\!-\!1) \Psi(n+a+b+c) \\
	& \hspace{70pt} -(n-j-1) \Psi(n+b+c-j) - (n-\ell-1) \Psi(n+a+c-\ell) \\
           & \hspace{160pt} + (n+c-j-\ell-1) \Psi(n+c-j-\ell) \biggr]^{1/2} \\
	& = \frac{(j \ell)^{1/2}}{n}\{1+o(1)\}
\end{align*}
as $n \rightarrow \infty$, uniformly for $j, \ell \in \{ k^\mathrm{L},\ldots,k^\mathrm{U}\}$.
\end{proof}
The following lemma provides bounds on the normal approximation to relevant multinomial distributions.
\begin{lemma}
\label{Lemma:NormalApproximation}
Fix $f \in \mathcal{F}_d$ and $\beta \in (0,1]$, and let $k^\mathrm{L} \leq k^\mathrm{U}$ be deterministic sequences of positive integers satisfying $k^\mathrm{L}/\log n \rightarrow \infty$ and $ (k^\mathrm{U}/n) \log n \rightarrow 0$. For $k \in \{k^\mathrm{L},\ldots,k^\mathrm{U}\}$ define $\mathcal{X}_{n}:=\{x : f(x)/M_{f,\beta}(x)^{d} \geq (k/n) \log n\}$. For $j,\ell \in \mathbb{N}$ and $z \in \mathbb{R}^d$ define $y\equiv y_{x,z}^{(j)}:=x + (\frac{j}{nV_d f(x)})^{1/d} z$, $\alpha_z(r):=\mu_d (B_0(1) \cap B_z(r))/V_d$, and
\[
	\Sigma  := \begin{pmatrix} 1 & (j/\ell)^{1/2} \alpha_z((\ell/j)^{1/d}) \\ (j/\ell)^{1/2} \alpha_z((\ell/j)^{1/d}) & 1 \end{pmatrix}.
\]
For $s,t \in (0,1), j,\ell \in \mathbb{N}$ and $x,z \in \mathbb{R}^d$ let $p_\cap = \int_{B_x(h_{x,f}^{-1}(s)) \cap B_y(h_{y,f}^{-1}(t))} \! f(w) dw$, define $(N_1,N_2,N_3,N_4) \sim \mathrm{Multi}(n; s-p_\cap, t-p_\cap,p_\cap, 1-s-t+p_\cap)$, let $(M_1,M_2,M_3) \sim \mathrm{Multi}(n; s,t,1-s-t)$, and
\begin{align*}
	F(s,t)&:=F_{n,x,z}^{(j),(\ell)}(s,t):= \mathbb{P}(N_1+N_3 \geq j, N_2 + N_3 \geq \ell) \\
	G(s,t)&:=G_{n}^{(j),(\ell)}(s,t):= \mathbb{P}(M_1 \geq j, M_2 \geq \ell).
\end{align*}
Then, given $c \in (0,1)$ and writing $\Phi_V$ for the distribution function of the bivariate normal distribution with mean zero and covariance matrix $V$, there exists $A=A(d,\beta,c,(k^\mathrm{L}),(k^\mathrm{U}))$ such that
\begin{align*}
	\max \Bigl\{ \Bigl| F( s,t) - \Phi_{\Sigma}\Bigl(&\frac{ns-j}{j^{1/2}},\frac{nt-\ell}{\ell^{1/2}}\Bigr) \Bigr|,\Bigl| G(s,t) - \Phi_{I_2}\Bigl(\frac{ns-j}{j^{1/2}}, \frac{nt-\ell}{\ell^{1/2}}\Bigr) \Bigr| \Bigr\}  \\
	& \hspace{50pt}\leq A \min \biggl\{ 1 , \frac{1}{\|z\|} \biggl( \frac{\log^{1/2}n}{k^{1/2}} + \biggl( \frac{k M_{f,\beta}(x)^d}{n f(x)} \biggr)^{\beta/d} \biggr) \biggr\}
\end{align*}
for all $k \in \{k^\mathrm{L},\ldots,k^\mathrm{U}\}$, for all $j,\ell \in \mathbb{N}$ such that $ck \leq j,\ell \leq k$, for all $x \in \mathcal{X}_{n}$, for all $s,t \in (0,1)$ such that $j^{-1/2}|ns-j| \vee \ell^{-1/2}|nt-\ell| \leq 3 \log^{1/2} n$, and for all $0 < \|z\| \leq (\frac{nV_df(x)}{j})^{1/d}\{h_{x,f}^{-1}(s) + h_{y,f}^{-1}(t)\} $.
\end{lemma}
\begin{proof}
We present here the approximation for $F(s,t)$, the approximation for $G(s,t)$ being similar but much simpler. Let $X_1,\ldots,X_n \stackrel{\mathrm{iid}}{\sim} f$ and for $i=1,\ldots,n$ and $k,j,\ell,x,s,t,z$ in the specified ranges, define $Y_i:=(\mathbbm{1}_{\{\|X_i-x\| \leq h_{x,f}^{-1}(s) \}}, \mathbbm{1}_{\{\|X_i-y\| \leq h_{y,f}^{-1}(t) \}})^T$,
\[
	V:= \mathrm{Cov}(Y_1) = \begin{pmatrix} s(1-s) & p_\cap-st \\ p_\cap-st & t(1-t) \end{pmatrix} \quad \text{and} \quad Z_i:=V^{-1/2}(Y_i - (s,t)^T).
\]
Then by the Berry--Esseen theorem of \citet{Gotze1991} we have
\begin{equation}
  \label{Eq:BerryEsseen}
	\bigl| \mathbb{P}(N_1+N_3 \geq j, N_2+N_3 \geq \ell) - \Phi_{nV} \bigl( ns-j, nt-\ell \bigr) \bigr| \lesssim n^{-1/2} \mathbb{E}(\|Z_1\|^3).
\end{equation}
In order to control the right hand side of this bound, we will require bounds on $p_\cap$. Writing $\alpha_z$ for $\alpha_z((\ell/j)^{1/d})$, we have 
\begin{align}
\label{Eq:pcapbound}
	&\Bigl| \frac{n p_\cap}{j} - \alpha_z \Bigr| \leq \frac{n}{j} \biggl| p_\cap - f(x) \mu_d \bigl( B_x(h_{x,f}^{-1}(s)) \cap B_y(h_{y,f}^{-1}(t)) \bigr) \biggr| \nonumber \\
	& \hspace{100pt} +  \biggl|\frac{n}{j} f(x) \mu_d \bigl( B_x(h_{x,f}^{-1}(s)) \cap B_y(h_{y,f}^{-1}(t)) \bigr) - \alpha_z \biggr| \nonumber  \\
	& \leq \frac{n}{j} M_{f,\beta}(x)^{\beta} f(x) \int_{B_x(h_{x,f}^{-1}(s)) \cap B_y(h_{y,f}^{-1}(t))} \|w-x\|^{\beta} \,dw \nonumber  \\
	& \hspace{20pt} + \biggl|\frac{1}{V_d} \mu_d \biggl( B_0 \biggl( \biggl( \frac{nV_df(x)h_{x,f}^{-1}(s)^d}{j} \biggr)^{\frac{1}{d}} \biggr) \cap B_z  \biggl( \biggl( \frac{nV_df(x)h_{y,f}^{-1}(t)^d}{j} \biggr)^{\frac{1}{d}} \biggr)- \alpha_z \biggr| \nonumber \\
	& \lesssim \frac{n}{k} M_{f,\beta}(x)^{\beta} f(x) h_{x,f}^{-1}(s)^{d+\beta} + \biggl| \frac{nV_df(x)h_{x,f}^{-1}(s)^d}{j} -1 \biggr| + \biggl| \frac{nV_df(x)h_{y,f}^{-1}(t)^d}{l} -1 \biggr| \nonumber \\
	& \lesssim \Bigl\{ \frac{k M_{f,\beta}(x)^d}{nf(x)} \Bigr\}^{\beta/d} + \frac{\log^{1/2}n}{k^{1/2}},
\end{align}
where the final bound follows by Lemma~\ref{Lemma:hxinvbounds} and similar arguments to those in~\eqref{Eq:MConstant1} and~\eqref{Eq:MConstant2} in the bounds on $U_0$ below. We will also need to bound $s+t-2p_\cap$ below. If $h_{y,f}^{-1}(t) \geq h_{x,f}^{-1}(s)$ then, by the mean value theorem and Lemma~\ref{Lemma:hxinvbounds},
\begin{align*}
	\mu_d \bigl( &B_y(h_{y,f}^{-1}(t)) \cap B_x(h_{x,f}^{-1}(s))^c \bigr)\geq \mu_d \bigl(B_y(h_{x,f}^{-1}(s)) \cap B_x(h_{x,f}^{-1}(s))^c \bigr) \\
	&= V_d h_{x,f}^{-1}(s)^d \int_{(1-\frac{\|x-y\|^2}{4h_{x,f}^{-1}(s)^2})_+}^1 \mathrm{B}_{\frac{d+1}{2},\frac{1}{2}}(\xi) \,d\xi \gtrsim h_{x,f}^{-1}(s)^d \frac{\|x-y\| \wedge h_{x,f}^{-1}(s)}{h_{x,f}^{-1}(s)} \gtrsim \frac{k(\|z\| \wedge 1)}{nf(x)}.
\end{align*}
A similar argument applies with $(x,s)$ and $(y,t)$ swapped and so we have
\begin{align*}
	s+t-2p_\cap &= \biggl( \int_{B_y(h_{y,f}^{-1}(t)) \cap B_x(h_{x,f}^{-1}(s))^c} + \int_{B_x(h_{x,f}^{-1}(s)) \cap B_y(h_{y,f}^{-1}(t))^c} \biggr) f(w)\,dw \\
	& \gtrsim f(x) \bigl\{ \mu_d\bigl(B_y(h_{y,f}^{-1}(t)) \cap B_x(h_{x,f}^{-1}(s))^c\bigr) + \mu_d\bigl(B_x(h_{x,f}^{-1}(s)) \cap B_y(h_{y,f}^{-1}(t))^c\bigr) \bigr\} \\
	& \gtrsim \frac{k(\|z\| \wedge 1)}{n}.
\end{align*}
We will also use a lower bound on $|V|:=\mathrm{det}(V)$ when $\|z\| \geq 1$. Note that with $e_1=(1,0,\ldots,0)^T \in \mathbb{R}^d$, when $\|z\| \geq 1$ we have that  $\alpha_z = \alpha_{\|z\|e_1} \leq \alpha_{e_1}$. If $\ell/j \geq (3/2)^d$ then 
\[
	\frac{j \alpha_z^2}{\ell} \leq \frac{j \alpha_{e_1}^2}{\ell} \leq \frac{j}{\ell} \leq \Bigl( \frac{2}{3} \Bigr)^{d/2} <1.
\]
However if $\ell/j <(3/2)^d$ then
\[
	\frac{j \alpha_z^2}{\ell} \leq \alpha^2 < V_d^{-2} \mu_d \bigl( B_0(1) \cap B_{e_1}((3/2)^{1/d}) \bigr)^2 < 1.
\]
Thus there exists $c_d \in (0,1)$ such that $j \alpha_z^2 /\ell \leq c_d$ whenever $\|z\| \geq 1$. Thus, by~\eqref{Eq:pcapbound}, we have that
\[
	|V|=st(1-s)(1-t) -(p_\cap-st)^2 \geq \frac{(1-c_d)j\ell}{n^2} \{1+o(1)\},
\]
uniformly over $\|z\| \geq 1$.  Similar to~(36),~(37) and~(38) in the supplement of \citet{BSY2018}, and splitting up into cases $\|z\|<1$ and $\|z\| \geq 1$ where necessary, we have that
\begin{align*}
	p_\cap \biggl\| V^{-1/2} \begin{pmatrix} 1-s \\ 1-t \end{pmatrix} \biggr\|^3 &\leq p_\cap \min\Bigl\{ \frac{s+t}{|V|} , \frac{1}{p_\cap-st} \Bigr\}^{3/2} \lesssim \frac{n^{1/2}}{k^{1/2}}, \\
	(1-s-t+p_\cap) \biggl\| V^{-1/2} \begin{pmatrix} s \\ t \end{pmatrix} \biggr\|^3 &= \! (1-s-t+p_\cap) \Bigl\{ \frac{st(s+t-2p_\cap)}{|V|} \Bigr\}^{\frac{3}{2}} \lesssim \frac{k^{\frac{3}{2}}}{n^{\frac{3}{2}}}.
\end{align*}
Likewise,
\begin{align*}
	&(s-p_\cap) \biggl\| V^{-1/2} \begin{pmatrix} 1-s \\ -t \end{pmatrix} \biggr\|^3 \leq (s-p_\cap)t^{3/2} |V|^{-3/2} \\
	& = (s-p_\cap)t^{\frac{3}{2}} \Bigl\{ (s+t-2p_\cap) \Bigl(p_\cap-st + \frac{(s-p_\cap)(t-p_\cap)}{s+t-2p_\cap} \Bigr) \Bigr\}^{-\frac{3}{2}} \lesssim \Bigl( \frac{n}{k\|z\|} \Bigr)^{\frac{1}{2}},
\end{align*}
with a similar bound holding for $(t-p_\cap) \biggl\| V^{-1/2} \begin{pmatrix} -s \\ 1-t \end{pmatrix} \biggr\|^3$. Thus
\[
	n^{-1/2} \mathbb{E} \|Z_3\|^3 \lesssim (k\|z\|)^{-1/2},
\]
which in combination with~\eqref{Eq:BerryEsseen} provides a bound on the difference between $F(s,t)$ and $\Phi_{nV}(ns-j,nt-\ell )$. Next, similar to the displayed equation above~(39) in the supplement of \citet{BSY2018}, we have
\begin{align*}
	&\bigl| \Phi_{nV}\bigl(ns-j,nt-\ell \bigr) - \Phi_\Sigma \bigl( j^{-1/2}(ns-j), \ell^{-1/2}(nt-\ell) \bigr) \bigr| \\
	& \leq \min \biggl\{1, 2 \biggl\| \Sigma^{-\frac{1}{2}} \begin{pmatrix} \frac{ns(1-s)}{j} -1 & \frac{n(p_\cap-st)}{j^{1/2}\ell^{1/2}}-j^{1/2}\ell^{-1/2} \alpha_z \\ \frac{n(p_\cap-st)}{j^{1/2}\ell^{1/2}}-j^{1/2}\ell^{-1/2} \alpha_z & \frac{nt(1-t)}{\ell} -1  \end{pmatrix} \Sigma^{-\frac{1}{2}} \biggr\| \biggr\} \\
	& \lesssim \{1/(1- (j/\ell)^{1/2}\alpha_z) + 1/(1+ (j/\ell)^{1/2}\alpha_z)\} \biggl\{ \frac{\log^{1/2}n}{k^{1/2}} + \frac{j^{1/2}}{\ell^{1/2}} \Bigl| \frac{np_\cap}{j} - \alpha_z \Bigr| \biggr\} \\
	& \lesssim \frac{1}{\|z\|} \biggl\{ \frac{\log^{1/2}n}{k^{1/2}} +  \Bigl( \frac{k M_{f,\beta}(x)^d}{nf(x)} \Bigr)^{\beta/d} \biggr\}
\end{align*}
as required.
\end{proof}


\subsection{Bounds on remainder terms in the proof of Proposition~\ref{Prop:Variance}}
\label{Sec:RemainderTerms}

\emph{To bound $S_1$:} Since $\zeta < 1/2$ we may apply Lemma~\ref{Lemma:Acomplement} to see that
\begin{align*}
S_{11}&:=\int_\mathcal{X} f(x) \int_0^1 \int_0^1 \max \bigl( \mathbbm{1}_{\{s \not\in \mathcal{I}_{m,X}\}}, \mathbbm{1}_{\{t \not\in \mathcal{I}_{n,Y} \}} \bigr) \phi ( u_{x,s}, v_{x,t} )^2 \\
	&\hspace{70pt} \times  \mathrm{B}_{k_X,m-k_X}(s) \mathrm{B}_{k_Y,n+1-k_Y}(t) \,ds \,dt \,dx  = o(m^{-4}+n^{-4}).
\end{align*}
By Lemma~\ref{Lemma:hxinvbounds2} we have that for every $\epsilon > 0$,
\begin{align*}
|S_{12}| &:=|S_1 - S_{11}| \\
&\phantom{:}= \biggl| \int_{\mathcal{X}_{m,n}^c}  \!\!\!\! f(x) \! \int_{\mathcal{I}_{m,X}} \! \int_{\mathcal{I}_{n,Y}} \phi \bigl( u_{x,s} , v_{x,t} \bigr)^2 \mathrm{B}_{k_X,m-k_X}(s) \mathrm{B}_{k_Y,n+1-k_Y}(t) \,ds \,dt \,dx \biggr| \\
	&\phantom{:} \lesssim \int_{\mathcal{X}_{m,n}^c} f(x)^{1-2\kappa_1^-} g(x)^{-2\kappa_2^-} M_\beta(x)^{2d( \kappa_1^-+\kappa_2^-)} (1+\|x\|)^{2d (\kappa_1^-+\kappa_2^-)} \,dx \\
	&\phantom{:} = O \biggl( \max \biggl\{ \Bigl( \frac{k_X}{m} \Bigr)^{\lambda_1(1-2 \zeta)- \epsilon}, \Bigl( \frac{k_Y}{n} \Bigr)^{\lambda_2(1-2 \zeta) - \epsilon} \biggr\} \biggr),
\end{align*}
where the final bound holds by Lemma~\ref{Lemma:GeneralisedHolder}, as in the bound on $R_1$.

\medskip

\emph{To bound $S_2$:} Using Lemma~\ref{Lemma:hxinvbounds} we now have that
\begin{align*}
	|S_2|&= \biggl| \int_{\mathcal{X}_{m,n}} f(x) \int_{\mathcal{I}_{m,X}} \int_{\mathcal{I}_{n,Y}} \biggl\{ \phi \bigl( u_{x,s} , v_{x,t} \bigr)^2 \\
	& \hspace{55pt} - \phi \Bigl( \frac{k_X f(x)}{m s} , \frac{k_Y g(x)}{n t} \Bigr)^2 \biggr\} \mathrm{B}_{k_X,m-k_X}(s) \mathrm{B}_{k_Y,n+1-k_Y}(t) \,ds \,dt \,dx \biggr| \\
	& \lesssim \int_{\mathcal{X}_{m,n}} f(x)^{1 + 2\kappa_1} g(x)^{2 \kappa_2} \int_{\mathcal{I}_{m,X}} \int_{\mathcal{I}_{n,Y}}   \mathrm{B}_{k_X,m-k_X}(s) \mathrm{B}_{k_Y,n+1-k_Y}(t)  \\
	& \hspace{80pt} \times \biggl\{ \biggl| \frac{s}{V_d f(x) h_{x,f}^{-1}(s)^d}-1 \biggr| + \biggl| \frac{t}{V_d g(x) h_{x,g}^{-1}(t)^d}-1 \biggr| \biggr\} \,ds \,dt \,dx \\
	& \lesssim \int_{\mathcal{X}_{m,n}} f(x)^{1+2 \kappa_1} g(x)^{2 \kappa_2} \biggl\{ \Bigl( \frac{k_X M_\beta(x)^d}{m f(x)} \Bigr)^{\frac{2 \wedge \beta}{d}} + \Bigl( \frac{k_Y M_\beta(x)^d}{n g(x)} \Bigr)^{\frac{2 \wedge \beta}{d}} \biggr\} \,dx \\
	& = O \biggl( \max \biggl\{ \Bigl( \frac{k_X}{m} \Bigr)^{\frac{2 \wedge \beta}{d}}, \Bigl( \frac{k_X}{m} \Bigr)^{\lambda_1(1-2 \zeta)- \epsilon},   \Bigl( \frac{k_Y}{n} \Bigr)^{\frac{2 \wedge \beta}{d}} , \Bigl( \frac{k_Y}{n} \Bigr)^{\lambda_2(1-2 \zeta) - \epsilon}\biggr\} \biggr)
\end{align*}
for all $\epsilon>0$, where for the final bound we use Lemma~\ref{Lemma:GeneralisedHolder} as in~\eqref{Eq:fbodyerror} and~\eqref{Eq:gbodyerror}.

\medskip

\emph{To bound $S_3$:} Using Lemma~\ref{Lemma:Acomplement} and Lemma~\ref{Lemma:GeneralisedHolder} we may write
\begin{align*}
	&|S_3| = \biggl| \int_{\mathcal{X}_{m,n}} f(x) \int_{\mathcal{I}_{m,X}} \int_{\mathcal{I}_{n,Y}} \phi \Bigl( \frac{k_X f(x)}{m s} , \frac{k_Y g(x)}{n t} \Bigr)^2 \\
	& \hspace{70pt} \times \mathrm{B}_{k_X,m-k_X}(s) \mathrm{B}_{k_Y,n+1-k_Y}(t) \,ds \,dt \,dx  - \int_\mathcal{X} f(x) \phi_x^2 \,dx \biggr| \\
	& \leq \int_{\mathcal{X}_{m,n}} f(x) \int_{\mathcal{I}_{m,X}} \int_{\mathcal{I}_{n,Y}} \biggl| \phi \Bigl(  \frac{k_X f(x)}{m s} , \frac{k_Y g(x)}{n t} \Bigr)^2 - \phi_x^2 \biggr| \\
	& \hspace{5pt} \times \mathrm{B}_{k_X,m-k_X}(s) \mathrm{B}_{k_Y,n+1-k_Y}(t) \,ds \,dt \,dx + \int_{\mathcal{X}_{m,n}^c} f(x) \phi_x^2 \,dx + o(m^{-4}+n^{-4}) \\
	& \lesssim \bigl( k_X^{-\frac{1}{2}}+ k_Y^{-\frac{1}{2}} \bigr) \int_{\mathcal{X}_{m,n}} f(x)^{1+2\kappa_1} g(x)^{2\kappa_2} \,dx + \int_{\mathcal{X}_{m,n}^c} \!f(x)^{1+2\kappa_1}g(x)^{2\kappa_2} \,dx + o(m^{-4}+n^{-4}) \\
	& = O \biggl( \max \biggl\{ k_X^{-1/2}, k_Y^{-1/2}, \Bigl( \frac{k_X}{m} \Bigr)^{\lambda_1(1-2 \zeta)- \epsilon},  \Bigl( \frac{k_Y}{n} \Bigr)^{\lambda_2(1-2 \zeta) - \epsilon}  \biggr\} \biggr),
\end{align*}
for every $\epsilon > 0$.

\medskip

\emph{To bound $T_1$:} We first consider
\[
	T_{11}:= \biggl( \int_{\mathcal{X}^2} - \int_{\mathcal{X}_{m,f}^2} \biggr) \int_{\mathcal{I}_{m,X}^2} \int_{\mathcal{I}_{n,Y}^2} (h \, dH_m^{(1)} \, dG_{n}^{(2)})(s_1,s_2,t_1,t_2) \,dx \,dy.
\]
By symmetry we may write $T_{11}=T_{111}+2T_{112}$, where
\begin{align*}
	T_{111}&:= \int_{\mathcal{X}_{m,f}^c \times \mathcal{X}_{m,f}^c} f(x)f(y) \int_{\mathcal{I}_{m,X}^2} \int_{\mathcal{I}_{n,Y}^2} (h \, dH_m^{(1)} \, dG_{n}^{(2)})(s_1,s_2,t_1,t_2) \,dx \,dy
\end{align*}
and
\begin{align*}
	T_{112}&:= \int_{\mathcal{X}_{m,f} \times \mathcal{X}_{m,f}^c} f(x)f(y) \int_{\mathcal{I}_{m,X}^2} \int_{\mathcal{I}_{n,Y}^2} (h \, dH_m^{(1)} \, dG_{n}^{(2)})(s_1,s_2,t_1,t_2) \,dx \,dy.
\end{align*}
Using Lemma~\ref{Lemma:GeneralisedHolder} and Lemma~\ref{Lemma:hxinvbounds2} as in the bounds on $S_1$, and using Lemma~\ref{Lemma:Pinsker} we have that
\begin{align*}
	|T_{111}| &\lesssim \frac{k_X}{m} \biggl\{ \int_{\mathcal{X}_{m,f}^c} f(x)^{1-\kappa_1^-} g(x)^{-\kappa_2^-} M_\beta(x)^{d (\kappa_1^-+\kappa_2^-)}  (1+\|x\|)^{d( \kappa_1^-+\kappa_2^-)} \,dx \biggr\}^2 \\
	&= O \biggl( \Bigl( \frac{k_X}{m}\Bigr)^{1+2\lambda_1(1-\zeta) - \epsilon}  \biggr)
\end{align*}
for all $\epsilon>0$. We now turn to $T_{112}$, and similarly write
\begin{align*}
	|T_{112}|&= \biggl| \int_{\mathcal{X}_{m,f} \times \mathcal{X}_{m,f}^c} f(x)f(y) \int_{\mathcal{I}_{m,X}^2} \int_{\mathcal{I}_{n,Y}^2}  (h \, dH_m^{(1)} \, dG_{n}^{(2)})(s_1,s_2,t_1,t_2) \,dx \,dy \biggr| \\
	& \lesssim \frac{k_X}{m} \int_{\mathcal{X}_{m,f} \times \mathcal{X}_{m,f}^c} \frac{f(x)^{1+\kappa_1} g(x)^{\kappa_2} f(y)}{f(y)^{\kappa_1^-}g(y)^{\kappa_2^-}}  M_\beta(y)^{d( \kappa_1^-+\kappa_2^-)} (1+\|y\|)^{d( \kappa_1^-+\kappa_2^-)} \,dy \,dx  \\
	& = O \biggl( \Bigl( \frac{k_X}{m} \Bigr)^{1+ \lambda_1(1-\zeta)- \epsilon} \biggr)
\end{align*}
for all $\epsilon > 0$.  Combining our bounds on $T_{111}$ and $T_{112}$ we have that
\[
	T_{11} = O \biggl(  \Bigl( \frac{k_X}{m} \Bigr)^{1+ \lambda_1(1-\zeta)- \epsilon} \biggr)
\]
for all $\epsilon > 0$.  We can develop analogous bounds on 
\begin{align*}
	T_{12}:= \! \biggl( \int_{\mathcal{X}^2} - \int_{\mathcal{X}_{n,g}^2} \biggr) \int_{\mathcal{I}_{m,X}^2} \int_{\mathcal{I}_{n,Y}^2} (h \, d(G_m^{(1)}-H_m^{(1)}) \, dH_n^{(2)})(s_1,s_2,t_1,t_2) \,dx \,dy
\end{align*}
to conclude that
\begin{align*}
	T_1 = T_{11}+T_{12} = O \biggl( \max \biggl\{ \Bigl( \frac{k_X}{m} \Bigr)^{1+ \lambda_1(1-\zeta)- \epsilon},  \Bigl( \frac{k_Y}{n} \Bigr)^{1+ \lambda_2(1-\zeta)- \epsilon}  \biggr\}\biggr)
\end{align*}
for all $\epsilon > 0$.

\emph{To bound $T_2$:} Here we use the notation
\begin{align*}
	L_x^f(s,t)&:= \phi \bigl(f(x), v_{x,t} \bigr) + \Bigl( \frac{k_X}{ms}-1 \Bigr) f(x) \phi_{10} \bigl(f(x), v_{x,t} \bigr) \\
	R_x^f(s,t)&:= \phi \bigl( u_{x,s}, v_{x,t}  \bigr) - L_x^f(s,t)
\end{align*}
for a linearised version of $\phi(u_{x,s},v_{x,t})$ and the linearisation error, so that we have $h^{(1)}(s_1,s_2,t_1,t_2)=L_x^f(s_1,t_1) L_y^f(s_2,t_2)$. Again we write $T_2=T_{21}+T_{22}$, with
\begin{align*}
	T_{21}&:= \int_{\mathcal{X}_{m,f}^2} f(x) \int_{\mathcal{I}_{m,X}^2} \int_{\mathcal{I}_{n,Y}^2}  (\{h - h^{(1)}\} \, dH_m^{(1)} \, dG_{n}^{(2)})(s_1,s_2,t_1,t_2) \,dx \,dy \\
	&=\int_{\mathcal{X}_{m,f}^2} f(x) \int_{\mathcal{I}_{m,X}^2} \int_{\mathcal{I}_{n,Y}^2} \bigl\{ R_x^f(s_1,t_1) R_y^f(s_2,t_2) + 2 L_x^f(s_1,t_1) R_y^f(s_2,t_2) \bigr\} \\
	& \hspace{225pt} \times dH_m^{(1)}(s_1,s_2) \, dG_n^{(2)}(t_1,t_2) \,dx \,dy \\
	&=: T_{211}+T_{212}
\end{align*}
and $T_{22}:=T_2-T_{21}$ having a similar expression.  Now
\begin{align*}
	|T_{211}| &\lesssim \frac{k_X}{m} \biggl[ \int_{\mathcal{X}_{m,f}} \frac{f(x)^{1+ \kappa_1}}{ g(x)^{\kappa_2^-}} M_\beta(x)^{d \kappa_2^-} (1+\|x\|)^{d \kappa_2^-}  \biggl\{ \biggl( \frac{k_X M_\beta(x)^d}{m f(x)} \biggr)^{\frac{2 \wedge \beta}{d}} + \frac{\log m}{k_X} \biggr\} \,dx \biggr]^2 \\
	& = O \biggl( \frac{k_X}{m} \max \biggl\{ \Bigl( \frac{k_X}{m} \Bigr)^{\frac{2(2 \wedge \beta)}{d}}, \Bigl( \frac{k_X}{m} \Bigr)^{2\lambda_1(1-\zeta) - \epsilon}, \frac{\log^2 m}{k_X^2} \biggr\} \biggr)
\end{align*}
for every $\epsilon > 0$.  When bounding $T_{212}$ we first integrate over $s_1$ using the facts that
\begin{align*}
	\int_0^1 \{ \mathrm{B}_{k_X,k_X,m-2k_X-1}(s_1,s_2) -& \mathrm{B}_{k_X,m-k_X}(s_1)  \mathrm{B}_{k_X,m-k_X}(s_2) \} \,ds_1 \\
	&= \frac{m-1}{m-k_X-1}  \mathrm{B}_{k_X,m-k_X-1}(s_2) \Bigl( s_2 - \frac{k_X}{m-1} \Bigr) \nonumber
\end{align*}
and
\begin{align*}
	& \int_0^1 \frac{k_X}{ms_1} \bigl\{ \mathrm{B}_{k_X,k_X,m-2k_X-1}(s_1,s_2) - \mathrm{B}_{k_X,m-k_X}(s_1)  \mathrm{B}_{k_X,m-k_X}(s_2) \bigr\} \,ds_1 \\
	& \hspace{7pt} = \frac{k_X(m-2)}{m(k_X-1)} \mathrm{B}_{k_X,m-k_X-2}(s_2) \biggl\{ 1 - \frac{(m-1)^2}{(m-k_X-1)(m-k_X-2)}(1-s_2)^2 \biggr\}\\
	& \hspace{7pt} =\mathrm{B}_{k_X,m-k_X-2}(s_2) \biggl\{ 2 \Bigl( \frac{k_X}{m-2} - s_2 \Bigr) + O \Bigl( \frac{k_X^2}{m^2} + \frac{1}{m} \Bigr) \biggr\},
\end{align*}
uniformly for $s_2 \in \mathcal{I}_{m,X}$. Using~\eqref{Eq:CentralMoments} and the fact that $k_X^{3/2}/m \rightarrow 0$ we can now see that
\begin{align*}
	|T_{212}|& \lesssim \frac{k_X^{1/2}}{m} \int_{\mathcal{X}_{m,f}} \frac{f(y)^{1+\kappa_1}}{ g(y)^{\kappa_2^-}} M_\beta(y)^{d \kappa_2^-} (1+\|y\|)^{d \kappa_2^-} \biggl\{ \biggl( \frac{k_X M_\beta(y)^d}{m f(y)} \biggr)^{\frac{2 \wedge \beta}{d}} + \frac{\log m}{k_X} \biggr\} \,dy \\
	&= O \biggl(\frac{k_X^{1/2}}{m} \max \biggl\{ \Bigl( \frac{k_X}{m} \Bigr)^{\frac{2 \wedge \beta}{d}}, \Bigl( \frac{k_X}{m} \Bigr)^{\lambda_1(1-\zeta)- \epsilon}, \frac{\log m}{k_X} \biggr\} \biggr)
\end{align*}
for every $\epsilon > 0$.  Combining our bounds on $T_{211}$ and $T_{212}$ we therefore have that
\[
	|T_{21}| = O \biggl( \max \biggl\{ \Bigl( \frac{k_X}{m} \Bigr)^{1+\lambda_1(1-\zeta) - \epsilon}, \frac{ \log m}{m k_X^{1/2}}, \frac{k_X^{1/2}}{m} \Bigl( \frac{k_X}{m} \Bigr)^{\frac{2 \wedge \beta}{d}}, \Bigl( \frac{k_X}{m} \Bigr)^{1+\frac{2(2 \wedge \beta)}{d}} \biggr\} \biggr)
\]
for every $\epsilon > 0$.  By analogous arguments we can show that
\[
	|T_{22}| = O \biggl( \max \biggl\{ \Bigl( \frac{k_Y}{n} \Bigr)^{1+\lambda_2(1-\zeta)- \epsilon}, \frac{ \log n}{n k_Y^{1/2}}, \frac{k_Y^{1/2}}{n} \Bigl( \frac{k_Y}{n} \Bigr)^{\frac{2 \wedge \beta}{d}}, \Bigl( \frac{k_Y}{n} \Bigr)^{1+\frac{2(2 \wedge \beta)}{d}} \biggr\} \biggr),
\]
for every $\epsilon > 0$, and this concludes the bound on $T_2$.

\emph{To bound $T_3$:} Here we integrate out $(s_1,s_2)$ in the $\mathcal{X}_{m,f}$ term and $(t_1,t_2)$ in the $\mathcal{X}_{n,g}$ term.  Now
\begin{align*}
	&\int_0^1 \int_0^{1-s_1} h^{(1)}(s_1,s_2,t_1,t_2) dG_m^{(1)}(s_1,s_2) \\
	& \hspace{100pt} - \int_0^1 \int_0^1 h^{(1)}(s_1,s_2,t_1,t_2) \mathrm{B}_{k_X,m-k_X}(s_1) \mathrm{B}_{k_X,m-k_X}(s_2) \,ds_1 \,ds_2 \\
	& =  f(x) \phi_{10}( f(x), v_{x,t_1})f(y) \phi_{10}(f(y),v_{y,t_2}) \\
	& \hspace{60pt} \times \biggl\{ \int_0^1 \int_0^{1-s_1} \Bigl( \frac{k_X}{ms_1} -1 \Bigr) \Bigl( \frac{k_X}{ms_2} -1 \Bigr) dG_m^{(1)}(s_1,s_2)  \\
	& \hspace{80pt} - \int_0^1 \int_0^1 \Bigl( \frac{k_X}{ms_1} -1 \Bigr) \Bigl( \frac{k_X}{ms_2} -1 \Bigr) \mathrm{B}_{k_X,m-k_X}(s_1) \mathrm{B}_{k_X,m-k_X}(s_2) \,ds_1 \,ds_2 \biggr\} \\
	&  \hspace{50pt} + \bigl\{f(x) \phi_{10}( f(x), \! v_{x,t_1}) \phi(f(y),\!v_{y,t_2}) \! +\! \phi( f(x), \! v_{x,t_1}) f(y) \phi_{10}(f(y),\!v_{y,t_2}) \bigr\} \\
	& \hspace{130pt} \times \int_0^1 \frac{k_X}{ms} \{ \mathrm{B}_{k_X,m-k_X-1}(s) - \mathrm{B}_{k_X,m-k_X}(s)\} \,ds \\ 
	& = - \frac{k_X}{(k_X\!-\!1)m} \biggl[ \biggl\{ \frac{k_X(3m\!-\!5)}{(k_X\!-\!1)m} - 2\biggr\} f(x) \phi_{10}( f(x), \! v_{x,t_1})f(y) \phi_{10}(f(y),\!v_{y,t_2}) \\
	&  \hspace{50pt} + \bigl\{\! f(x) \phi_{10}( f(x), \! v_{x,t_1}) \phi(f(y),\!v_{y,t_2}) \!\! +\! \phi( f(x), \! v_{x,t_1}) f(y) \phi_{10}(f(y),\!v_{y,t_2}) \! \bigr\}\! \biggr].
\end{align*}
The contribution from the $\mathcal{X}_{n,g}$ term is simpler because the marginals of the  $\mathrm{B}_{k_Y,k_Y,n-2k_Y+1}$ density are equal to $\mathrm{B}_{k_Y,n-k_Y+1}$, and we have
\begin{align*}
	&\int_0^1 \int_0^{1-t_1} h^{(2)}(s_1,s_2,t_1,t_2) dG_n^{(2)}(t_1,t_2) \\
	& \hspace{50pt} - \int_0^1 \int_0^1 h^{(2)}(s_1,s_2,t_1,t_2) \mathrm{B}_{k_Y,n-k_Y+1}(t_1) \mathrm{B}_{k_Y,n-k_Y+1}(t_2) \,dt_1 \,dt_2 \\
	& =g(x) \phi_{01}(u_{x,s_1},g(x)) g(y) \phi_{01}(u_{y,s_2},g(y)) \biggl\{ \int_0^1 \int_0^{1-t_1} \!\! \frac{k_Y^2}{n^2 t_1t_2} dG_n^{(2)}(t_1,t_2) \\ 
	& \hspace{90pt} - \int_0^1 \int_0^1 \frac{k_Y^2}{n^2 t_1t_2} \mathrm{B}_{k_Y,n-k_Y+1}(t_1) \mathrm{B}_{k_Y,n-k_Y+1}(t_2) \,dt_1 \,dt_2 \biggr\} \\
	& = -\frac{k_Y^2}{(k_Y-1)^2 n} g(x) \phi_{01}(u_{x,s_1},g(x)) g(y) \phi_{01}(u_{y,s_2},g(y)).
\end{align*}
The error $T_3$ is the error in, for example, $k_Y^2(k_Y-1)^{-2}/n \approx 1/n$, together with the contribution from $(s_1,s_2) \not\in \mathcal{I}_{m,X}^2$ and $(t_1,t_2) \not\in \mathcal{I}_{n,Y}^2$, and we can use Lemma~\ref{Lemma:BetaTailBounds} to see that
\[
	T_3= o(1/m + 1/n ).
\]

\bigskip

\emph{To bound $U_0$:} \sloppy{We write $r_{m,x,y}^{(1)} := h_{x,f}^{-1}(a_{m,X}^+)+h_{y,f}^{-1}(a_{m,X}^+)$ and $r_{n,x,y}^{(2)} := h_{x,g}^{-1}(a_{n,Y}^+)+h_{y,g}^{-1}(a_{n,Y}^+)$ as shorthand.  For $s_1,s_2 \leq a_{m,X}^+$ and $t_1,t_2 \leq a_{n,Y}^+$ we have $F_{m,n,x,y}(s_1,s_2,t_1,t_2)=G_{m,n}(s_1,s_2,t_1,t_2)$ unless we also have $\|y-x\| \leq \max\bigl\{r_{m,x,y}^{(1)},r_{n,x,y}^{(2)}\bigr\}$. Here we will present bounds in the case $\|y-x\| \leq r_{m,x,y}^{(1)}$, but the other case follows using very similar arguments. First, by using Lemma~\ref{Lemma:GeneralisedHolder} and Lemma~\ref{Lemma:hxinvbounds2}, we have that}
\begin{align}
\label{Eq:XnComplement}
	\int_{\mathcal{X}_{m,n}^c} f(x) \sup_{s \in \mathcal{I}_{m,X}, t \in \mathcal{I}_{n,Y}} & | \phi(u_{x,s}, v_{x,t} ) |  \,dx \nonumber \\
	& \lesssim \int_{\mathcal{X}_{m,n}^c} f(x)^{1-\kappa_1^-}g(x)^{-\kappa_2^-} M_\beta(x)^{d (\kappa_1^-+\kappa_2^-)} (1+\|x\|)^{d (\kappa_1^-+\kappa_2^-)} \,dx  \nonumber \\
	&= O \biggl( \max \biggl\{  \Bigl(\frac{k_X}{m} \Bigr)^{ \lambda_1(1-\zeta) - \epsilon}, \Bigl(\frac{k_Y}{n} \Bigr)^{\lambda_2(1-\zeta) - \epsilon} \biggr\} \biggr),
\end{align}
for every $\epsilon > 0$, and we proceed by showing that, since $x$ and $y$ are close, the contribution from $\mathcal{X}_{m,n}^c \times \mathcal{X}$ behaves similarly to the contribution from $\mathcal{X}_{m,n}^c \times \mathcal{X}_{m,n}^c$, which can be bounded by the square of the final bound in~\eqref{Eq:XnComplement}. It suffices to consider $(x,y) \in \mathcal{X}_{m,n}^c \times \mathcal{X}_{m,n}$, as the contribution from $\mathcal{X}_{m,n}^c \times \mathcal{X}_{m,n}^c$ is more straightforward.

By a very similar argument to that used to establish~\eqref{Eq:PackingApprox} in the proof of Proposition~\ref{Prop:Partition}, we have that $\|x-y\| \leq \{M_\beta(y)^d \log^{1/2} m\}^{-1/d}$ for $m$ sufficiently large, and hence, by Lemma~\ref{Lemma:15over7}, that
\[
	|f(x)/f(y)-1| \leq 2 \{M_\beta(y) \|y-x\|\}^{1 \wedge \beta} \leq 1/2,
\]
and in particular $f(x) \geq f(y)/2$. Thus, again using Lemma~\ref{Lemma:15over7}, we have that
\begin{equation}
\label{Eq:MConstant1}
	\max_{t=1,\ldots, \underline{\beta}} \biggl( \frac{\| f^{(t)}(x)\|}{f(x)} \biggr)^{1/t} \leq 4 M_\beta(y).
\end{equation}
for $m$ sufficiently large. In addition,
\begin{align}
\label{Eq:MConstant2}
	\sup_{w,z \in B_x( 1/\{2M_\beta(y)\})}& \frac{\|f^{(\underline{\beta})}(z)-f^{(\underline{\beta})}(w)\|}{\|z-w\|^{\beta-\underline{\beta}} f(w)} \leq \sup_{w,z \in B_y( 1/M_\beta(y))}\! \frac{\|f^{(\underline{\beta})}(z)-f^{(\underline{\beta})}(w)\|}{\|z-w\|^{\beta-\underline{\beta}} f(w)}  \leq M_\beta(y)^{\beta},
\end{align}
and so we have that $M_\beta(x) \leq 4 M_\beta(y)$. Using this fact and the previously established fact that $f(x) \geq f(y)/2$, we may apply Lemma~\ref{Lemma:hxinvbounds} to see that in fact
\[
	\|x-y\| \leq r_{m,x,y}^{(1)} \lesssim \Bigl( \frac{k_X}{m f(y)} \Bigr)^{1/d}.
\]
Using Lemma~\ref{Lemma:15over7} we also have that $g(x) \geq g(y)/2$, and therefore that
\begin{equation}
\label{Eq:Xntilde}
	\max \biggl\{ \frac{f(y)M_\beta(y)^{-d}}{f(x)M_\beta(x)^{-d}}, \frac{g(y)M_\beta(y)^{-d}}{g(x)M_\beta(x)^{-d}} \biggr\} \leq 2^{2d+1}.
\end{equation}
Since $x \in \mathcal{X}_{m,n}^c$, we have now established that
\[
	\min \biggl\{\frac{m f(y) M_\beta(y)^{-d}}{k_X \log m}, \frac{n g(y) M_\beta(y)^{-d}}{k_Y \log n} \biggr\} \leq  2^{2d+1}.
\]
Applying the same bounds as we would for $\mathcal{X}_{m,n}^c$, as in~\eqref{Eq:XnComplement}, we can now see that
\[
	U_0 = O \biggl( \max \biggl\{  \Bigl(\frac{k_X}{m} \Bigr)^{ 2\lambda_1(1-\zeta) - \epsilon}, \Bigl(\frac{k_Y}{n} \Bigr)^{2\lambda_2(1-\zeta) - \epsilon} \biggr\} \biggr),
\]
for every $\epsilon > 0$, as claimed.

\emph{To bound $U_1$:} By Lemma~\ref{Lemma:BetaTailBounds} we have that
\begin{align}
\label{Eq:F2tails}
	\max_{t_1 \in \{a_{n,Y}^-, a_{n,Y}^+\}}  \sup_{t_2 \in [0,1]} |F_{n,x,y}^{(2)}-&G_{n}^{(2)}|(t_1,t_2)  \nonumber \\
	& \bigvee  \max_{t_2 \in \{a_{n,Y}^-, a_{n,Y}^+\}} \sup_{t_1 \in [0,1]} |F_{n,x,y}^{(2)}-G_{n}^{(2)}|(t_1,t_2)  = o(n^{-4}).
\end{align}
In order to use this to bound $U_1$, corresponding to the right-hand side of~\eqref{Eq:IBPformula}, we must first develop bounds on the derivatives of $h$. Writing $S_x(r):=\{y \in \mathbb{R}^d : \|x-y\|=r\}$ and $d \mathrm{Vol}_S$ for the associated volume element we have by Lemma~\ref{Lemma:15over7} that for $r \leq 1/\{(6d)^{1/(\beta-\underline{\beta})}M_\beta(x)\}$,
\begin{align*}
	\Bigl| \frac{h_{x,f}'(r)}{dV_d r^{d-1}f(x)} -1 \Bigr| &= \Bigl| \frac{1}{dV_d r^{d-1} f(x)} \int_{S_x(r)} \{f(y)-f(x)\} \, d \mathrm{Vol}_S(y) \Bigr| \lesssim \{r M_\beta(x) \}^{2 \wedge \beta},
\end{align*}
with a similar bound holding for $h_{x,g}'(r)$. Using Lemma~\ref{Lemma:hxinvbounds}, for $x \in \mathcal{X}_{m,n}$, we have that $\max\{M_\beta(x)^dh_{x,f}^{-1}(a_{m,X}^+)^d , M_\beta(x)^d h_{x,g}^{-1}(a_{n,Y}^+)^d\}  \lesssim 1/\log m \rightarrow 0$ and so we have, by Lemma~\ref{Prop:FunctionalClasses}(i), that
\begin{align}
\label{Eq:hderivatives}
	\Bigl| \frac{\partial}{\partial s} & \phi (u_{x,s}, v_{x,t}) + \frac{k_Xf(x)}{ms^2} \phi_{10} \Bigl( \frac{k_Xf(x)}{ms} , \frac{k_Yg(x)}{nt} \Bigr) \Bigr| \nonumber \\
	&= \Bigl| - \frac{k_X d \, \phi_{10} (u_{x,s}, v_{x,t} ) }{m V_d h_{x,f}^{-1}(s)^{d+1} h_{x,f}'(h_{x,f}^{-1}(s))} + \frac{k_Xf(x)}{ms^2} \phi_{10} \Bigl( \frac{k_Xf(x)}{ms} , \frac{k_Yg(x)}{nt} \Bigr)  \Bigr| \nonumber \\
	& \leq \frac{k_X d}{m V_d h_{x,f}^{-1}(s)^{d+1} h_{x,f}'(h_{x,f}^{-1}(s))} \Bigl| \phi_{10} (u_{x,s}, v_{x,t} ) - \phi_{10} \Bigl( \frac{k_Xf(x)}{ms} , \frac{k_Yg(x)}{nt} \Bigr)  \Bigr| \nonumber \\
	& \hspace{20pt} + \frac{k_Xf(x)}{ms^2} \Bigl| \frac{d V_d h_{x,f}^{-1}(s)^{d-1}f(x)/h_{x,f}'(h_{x,f}^{-1}(s))}{V_d^2 f(x)^2 h_{x,f}^{-1}(s)^{2d}/s^2} -1 \Bigr|  \Bigl|\phi_{10} \Bigl( \frac{k_Xf(x)}{ms} , \frac{k_Yg(x)}{nt} \Bigr) \Bigr|\nonumber  \\
	& \lesssim \frac{k_X}{ms^2} f(x)^{\kappa_1} g(x)^{\kappa_2} \Bigl\{ \Bigl( \frac{s M_\beta(x)^d}{f(x)} \Bigr)^{(2 \wedge \beta)/d} + \Bigl( \frac{tM_\beta(x)^d}{g(x)} \Bigr)^{(2 \wedge \beta)/d} \Bigr\} \nonumber \\
	& \lesssim (1/s) f(x)^{\kappa_1}g(x)^{\kappa_2} \Bigl\{ \Bigl( \frac{k_X M_\beta(x)^d}{mf(x)} \Bigr)^{(2 \wedge \beta)/d} + \Bigl( \frac{k_Y M_\beta(x)^d}{ng(x)} \Bigr)^{(2 \wedge \beta)/d}\Bigr\},
\end{align}
uniformly for $x \in \mathcal{X}_{m,n}$, $s \in \mathcal{I}_{m,X}$ and $t \in \mathcal{I}_{n,Y}$. In particular, we have that
\begin{equation}
  \label{Eq:hderivbounds}
	\Bigl| \frac{\partial}{\partial s} \phi \bigl( u_{x,s} , v_{x,t} \bigr)  \Bigr| \lesssim (1/s) f(x)^{\kappa_1}g(x)^{\kappa_2},
      \end{equation}
      uniformly for $x \in \mathcal{X}_{m,n}$, $s \in \mathcal{I}_{m,X}$ and $t \in \mathcal{I}_{n,Y}$.  Analogous arguments also reveal that $\frac{\partial}{\partial t} \phi(u_{x,s},v_{x,t}) \lesssim (1/t) f(x)^{\kappa_1}g(x)^{\kappa_2}$, uniformly for $x \in \mathcal{X}_{m,n}$, $s \in \mathcal{I}_{m,X}$ and $t \in \mathcal{I}_{n,Y}$.  Moreover, since $x \in \mathcal{X}_{m,n}$ and $\|y-x\| \leq \max\bigl\{ r_{m,x,y}^{(1)},r_{n,x,y}^{(2)}\bigr\}$, we may argue as we did leading up to~\eqref{Eq:Xntilde} to obtain similar bounds on $\frac{\partial}{\partial s} \phi (u_{y,s}, v_{y,t} )$ and $\frac{\partial}{\partial t} \phi (u_{y,s}, v_{y,t})$. Thus, using~\eqref{Eq:IBPformula} and~\eqref{Eq:F2tails}, we find that $U_1= o (n^{-4}).$

\emph{To bound $U_2$:} Again using Lemma~\ref{Lemma:BetaTailBounds}, we have that
\begin{align}
\label{Eq:F1tails}
	\max \Bigl\{ \sup_{s_1 \in [0,1]} |F_{m,x,y}^{(1)}-G_m^{(1)}|(s_1,a_{m,X}^-), \sup_{s_2 \in [0,1]} & |F_{m,x,y}^{(1)}-G_m^{(1)}|(a_{m,X}^-,s_2), \nonumber \\
	&|F_{m,x,y}^{(1)}-G_m^{(1)}|(a_{m,X}^+,a_{m,X}^+) \Bigr\} = o(m^{-4}).
\end{align}
By similar arguments to those used in the bound on $U_1$ we have that
\begin{align}
  \label{Eq:mixedhderivbounds}
	\biggl| \frac{\partial^2}{ \partial s \partial t} \phi(u_{x,s},v_{x,t}) \biggr| & \!= \! \frac{k_X k_Y d^2 |\phi_{11}(u_{x,s},v_{x,t})|}{mnV_d^2 h_{x,f}^{-1}(s)^{d+1} h_{x,f}' \bigl( h_{x,f}^{-1}(s) \bigr) h_{x,g}^{-1}(t)^{d+1} h_{x,g}' \bigl( h_{x,g}^{-1}(t) \bigr)}  \nonumber \\
	& \lesssim \{1/(s t)\} f(x)^{\kappa_1} g(x)^{\kappa_2},
\end{align}
uniformly for $x \in \mathcal{X}_{m,n}$, $s \in \mathcal{I}_{m,X}$ and $t \in \mathcal{I}_{n,Y}$; moreover, the same bound also holds for $\frac{\partial^2}{ \partial s \partial t}\phi(u_{y,s},v_{y,t})$. We may therefore use~\eqref{Eq:IBPformula},~\eqref{Eq:hderivatives} and~\eqref{Eq:F1tails} to conclude that $U_2=o(m^{-4})$.

\emph{To bound $U_3$:} By Lemma~\ref{Lemma:BetaTailBounds}, we have that
\begin{align}
\label{Eq:Marginals}
	F_{m,x,y}^{(1)}(s_1,a_{m,X}^+) - G_m^{(1)}&(s_1,a_{m,X}^+) = \frac{\mathrm{B}_{k_X,m-k_X}(s_1)}{m-1}  \mathbbm{1}_{\{\|x-y\| \leq h_{x,f}^{-1}(s_1) \}} + o(m^{-4}),
\end{align}
uniformly for $x \in \mathcal{X}_{m,n}$, $\|y-x\| \leq r_{m,x,y}^{(1)}$ and $s \in \mathcal{I}_{m,X}$, with an analogous statement holding for $F_{m,x,y}^{(1)}(a_{m,X}^+,s_2) - G_m^{(1)}(a_{m,X}^+,s_2)$. Now, combining this statement with our bounds on the derivatives of $h$ in~\eqref{Eq:hderivbounds} and~\eqref{Eq:mixedhderivbounds}, and applying the bounds $|F^{(1)}(s_1,s_2)| \leq \mathbbm{1}_{\{\|y-x\| \leq r_{m,x,y}^{(1)}\}}$ and $|F^{(2)}(t_1,t_2)| \leq \mathbbm{1}_{\{\|y-x\| \leq r_{n,x,y}^{(2)}\}} $, we may write
\begin{align*}
	|U_3|  & \lesssim \int_{\mathcal{X} \times \mathcal{X}_{m,n}} f(x)^{1+\kappa_1} g(x)^{\kappa_2} f(y)^{1+\kappa_1} g(y)^{\kappa_2} \mathbbm{1}_{\{\|y-x\| \leq  \min\{r_{m,x,y}^{(1)},r_{n,x,y}^{(2)}\}\}} \\
	& \hspace{100pt} \times \biggl( \frac{\log m \log n}{k_X k_Y} + \frac{\log n}{k_X k_Y} + \frac{\log^{1/2} m \log n}{m^{2} k_X^{1/2} k_Y} \biggr) \,dx \,dy \\
	& \lesssim \frac{\log m \log n}{k_X k_Y} \int_{\mathcal{X} \times \mathcal{X}_{m,n}} \! f(x)^{2+2\kappa_1} g(x)^{2 \kappa_2} \mathbbm{1}_{\{\|y-x\| \leq  \min\{ r_{m,x,y}^{(1)},r_{n,x,y}^{(2)}\} \}} \,dx \,dy \\
	& \lesssim \frac{\log m \log n}{k_X k_Y} \int_{\mathcal{X}_{m,n}} f(x)^{2+2\kappa_1}g(x)^{2\kappa_2} \min \Bigl\{ \frac{k_X}{m f(x)}, \frac{k_Y}{n g(x)} \Bigr\} \,dx.
\end{align*}
Since $\min(m,n) \geq 3$, if $m \geq n$, then $ (1/m) \log m \leq (1/n) \log n$ and therefore
\[
	|U_3| \lesssim \frac{\log m \log n}{k_X k_Y} \int_{\mathcal{X}}  \frac{k_X}{m} f(x)^{1+2 \kappa_1}{g(x)^{2\kappa_2}} \,dx \lesssim \frac{\log^2 n}{n k_Y}.
\]
Similarly, if $n \geq m$ then
\[
	|U_3| \lesssim \frac{\log m \log n}{k_X k_Y} \int_\mathcal{X}\frac{k_Y}{n} f(x)^{2+2\kappa_1} g(x)^{2\kappa_2-1} \,dx \lesssim \frac{\log^2m}{m k_X}.
\]
Putting these two statements together,
\[
	U_3=O \biggl( \max \biggl\{ \frac{\log^2m}{m k_X}, \frac{\log^2 n}{n k_Y} \biggr\} \biggr),
\]
which establishes~\eqref{Eq:UBounds1}.

\emph{To bound $U_4$:} Using~\eqref{Eq:IBPformula},~\eqref{Eq:hderivatives} and~\eqref{Eq:F1tails} we have that $U_4 = o(m^{-4})$.

\emph{To bound $U_5$:}  We first bound the contribution to $U_5$ from the discontinuous parts of $F_{m,x,y}^{(1)}$, arising due to the indicator functions in~\eqref{Eq:MultinomialExpression}.  Recalling the definition of the multinomial random vector $(N_1^{(1)},N_2^{(1)},N_3^{(1)},N_4^{(1)})$ in~\eqref{Eq:Multivector}, we have that
\begin{align*}
  0 &\leq F_{m,x,y}^{(1)}(s_1,s_2) - \mathbb{P}\bigl(N_1^{(1)} + N_3^{(1)} \geq k_X,N_2^{(1)} + N_3^{(1)} \geq k_X\bigr) \\
                            &\leq \mathbb{P}\bigl(N_1^{(1)} + N_3^{(1)} = k_X-1\bigr) + \mathbb{P}\bigl(N_2^{(1)} + N_3^{(1)} = k_X-1\bigr) \\
                            &= \binom{m-2}{k_X-1}s^{k_X-1}(1-s)^{m-k_X-1} + \binom{m-2}{k_X-1}t^{k_X-1}(1-t)^{m-k_X-1} \\
  &\leq \frac{2}{(2 \pi k_X)^{1/2}}\{1+o(1)\},
\end{align*}
uniformly for $x \in \mathcal{X}_{m,n}$, $\|y-x\| \leq r_{m,x,y}^{(1)}$ and $(s_1,s_2) \in \mathcal{I}_{m,X}^2$, and we will see is of no larger order than the error in the normal approximation for the continuous part. Now, writing $y=x+\{\frac{k_X}{mV_d f(x)}\}^{1/d}z$, define
\begin{align*}
 U_{51}:= & \int_{\mathcal{X} \times \mathcal{X}_{m,n}}  f(x)f(y) \int_{\mathcal{I}_{n,Y}^2} \, dG_n^{(2)}(t_1,t_2) \\
	& \times \biggl[\int_{\mathcal{I}_{m,X}^2} h_{1100}\biggl\{F^{(1)}(s_1,s_2) - \bigl(\Phi_{\Sigma}-\Phi_{I_2}\bigr)\Bigl(\frac{ms_1-k_X}{k_X^{1/2}}, \frac{ms_2-k_X}{k_X^{1/2}} \Bigr) \biggr\} \,ds_1 \,ds_2 \nonumber \\
	& \hspace{5pt} - \int_{\mathcal{I}_{m,X}} h_{1000} \biggl\{ F^{(1)}(s_1,a_{m,X}^+) - \frac{\mathrm{B}_{k_X,m-k_X}(s_1)}{m-1} \mathbbm{1}_{\{\|z\| \leq 1\}} \biggr\} \, ds_1 \\
	& \hspace{5pt} -  \int_{\mathcal{I}_{m,X}} h_{0100}\biggl\{ F^{(1)}(a_{m,X}^+,s_2) - \frac{\mathrm{B}_{k_X,m-k_X}(s_2)}{m-1} \mathbbm{1}_{\{\|z\| \leq 1\}} \biggr\}  \, ds_2 \biggr]  \,dx \,dy.
\end{align*}
By Lemma~\ref{Lemma:hxinvbounds} we have that
\begin{align*}
	\int_\mathcal{X} \int_{\mathcal{I}_{m,X}} & \frac{1}{ms} \mathrm{B}_{k_X,m-k_X}(s) \Bigl| \mathbbm{1}_{\{\|y-x\| \leq h_{x,f}^{-1}(s) \}} - \mathbbm{1}_{\{\|z\| \leq 1 \}} \Bigr| \,ds \,dy \\
	& \lesssim \frac{1}{k_X} \int_{\mathbb{R}^d} \mathbbm{1}_{\{ ( \frac{k_X}{mV_d f(x)})^{\frac{1}{d}} < \|y-x\| \leq h_{x,f}^{-1}(a_{m,X}^+) \}} \vee  \mathbbm{1}_{\{h_{x,f}^{-1}(a_{m,X}^-)  < \|y-x\| \leq (\frac{k_X}{mV_d f(x)})^{\frac{1}{d}}  \}}  \,dy \\
	& \leq \frac{V_d}{k_X} \bigl\{ h_{x,f}^{-1}(a_{m,X}^+)^d \! - \! h_{x,f}^{-1}(a_{m,X}^-)^d  \bigr\} \lesssim \frac{1}{mf(x)} \biggl\{ \frac{\log^{\frac{1}{2}}m}{k_X^{1/2}} \!+\! \Bigl( \frac{k_X M_\beta(x)^d}{m f(x)} \Bigr)^{\frac{2 \wedge \beta}{d}} \biggr\},
\end{align*}
uniformly for $x \in \mathcal{X}_{m,n}$. Using this bound together with Lemma~\ref{Lemma:NormalApproximation},~\eqref{Eq:hderivatives} and~\eqref{Eq:Marginals} we may say that
\begin{align*}
	|U_{51}|  &\lesssim \int_{\mathcal{X}_{m,n}} f(x)^{1+2\kappa_1} g(x)^{2\kappa_2} \biggl[ \frac{1}{m} \biggl\{ \frac{\log^{1/2}m}{k_X^{1/2}} + \biggl(\frac{k_XM_\beta(x)^d}{mf(x)} \biggr)^{(2 \wedge \beta)/d} \biggr\} \\
	& \hspace{29pt} + \frac{k_X}{m} \log^2 \Bigl( \frac{a_{m,X}^+}{a_{n,Y}^-} \Bigr) \int_{B_0(2)} \!\!\!\! \min\biggl\{ 1, \frac{1}{\|z\|} \biggl( \frac{\log^{\frac{1}{2}}m}{k_X^{1/2}} + \biggl(\frac{k_X M_\beta(x)^d}{mf(x)} \biggr)^{\frac{1 \wedge \beta}{d}} \biggr) \biggr\} \,dz \biggr]  dx \\
	&=O \biggl( \frac{1}{m} \max\biggl\{ \frac{\log^{5/2}m}{k_X^{1/2}} , \log^2m \Bigl( \frac{k_X}{m} \Bigr)^{(1 \wedge \beta)/d}, \Bigl( \frac{k_X}{m} \Bigr)^{ \lambda_1(1-2\zeta) - \epsilon} \biggr\} \biggr),
\end{align*}
for every $\epsilon > 0$.  In bounding $U_5$ it therefore remains to approximate the derivatives of $h$ using~\eqref{Eq:hderivatives} and to bound the contribution from the tails of the $t_1,t_2,s_1,s_2$ integrals. By Lemma~\ref{Lemma:BetaTailBounds} and standard normal tail bounds the error from these tail contributions is $o(m^{-4})$, and so, using~\eqref{Eq:hderivatives},
\begin{align*}
	&|U_{52}|:=|U_5-U_{51}| \lesssim \frac{1}{m} \int_{\mathcal{X}_{m,n}} f(x)^{1+2\kappa_1} g(x)^{2\kappa_2} \biggl\{ \frac{\log^{1/2}m}{k_X^{1/2}} + \frac{\log^{1/2}n}{k_Y^{1/2}} \\
	& \hspace{140pt} + \biggl( \frac{k_XM_\beta(x)^d}{mf(x)} \biggr)^{\frac{2 \wedge \beta}{d}} +\biggl( \frac{k_Y M_\beta(x)^d}{ng(x)} \biggr)^{\frac{2 \wedge \beta}{d}}  \biggr\} \,dx \\
	& =O \biggl( \frac{1}{m} \max\biggl\{ \frac{\log^{1/2}m}{k_X^{1/2}}, \frac{\log^{1/2}n}{k_Y^{1/2}} ,\Bigl( \frac{k_X}{m} \Bigr)^{\frac{2 \wedge \beta}{d}}, \Bigl( \frac{k_Y}{n} \Bigr)^{\frac{2 \wedge \beta}{d}}, \\
	& \hspace{180pt} \Bigl( \frac{k_X}{m} \Bigr)^{ \lambda_1(1-2\zeta) - \epsilon},  \Bigl( \frac{k_Y}{n} \Bigr)^{ \lambda_2(1-2\zeta) - \epsilon} \biggr\} \biggr),
\end{align*}
for every $\epsilon > 0$. 

\emph{To bound $U_6$:} Using Lemma~\ref{Lemma:GeneralisedHolder} we have that
\begin{align*}
	|U_6| \lesssim \frac{1}{m}  \int_{\mathcal{X}_{m,n}^c} &f(x)^{1+2\kappa_1}g(x)^{2\kappa_2} \,dx =\! O \biggl( \frac{1}{m} \max\Bigl\{\Bigl( \frac{k_X}{m} \Bigr)^{ \lambda_1(1-2\zeta) - \epsilon},\Bigl( \frac{k_Y}{n} \Bigr)^{ \lambda_2(1-2\zeta) - \epsilon}\Bigr\}\biggr),
\end{align*}
for every $\epsilon > 0$.  This establishes~\eqref{Eq:UBounds2}.

\emph{To bound $U_7$:} Analogously to our bounds on $U_1$, we may use~\eqref{Eq:IBPformula},~\eqref{Eq:F2tails} and~\eqref{Eq:hderivatives} to show that $U_7=o(n^{-4})$.

\emph{To bound $U_8$:} Using Lemma~\ref{Lemma:NormalApproximation},~\eqref{Eq:IBPformula},~\eqref{Eq:Holderv2},~\eqref{Eq:F2tails} and~\eqref{Eq:hderivatives}, and the change of variables $y=x+(\frac{k_Y}{nV_d g(x)})^{1/d} z$, we have that
\begin{align*}
	|U_{81}| &:= \biggl| \int_{\mathcal{X} \times \mathcal{X}_{m,n}} f(x) f(y) \int_{\mathcal{I}_{m,X}^2} \int_{\mathcal{I}_{n,Y}^2} h_{0011}(s_1,s_2,t_1,t_2) \,dG_m^{(1)}(s_1,s_2) \\
	& \hspace{45pt} \times \Bigl\{ F^{(2)}(t_1,t_2) - (\Phi_{\Sigma}-\Phi_{I_2}) \Bigl( \frac{nt_1-k_Y}{k_Y^{1/2}}, \frac{nt_2-k_Y}{k_Y^{1/2}} \Bigr) \Bigr\} \,dt_1 \,dt_2  \,dx \,dy \biggr| \\
	& \lesssim \frac{k_Y}{n}\log^2 \Bigl( \frac{a_{n,Y}^+}{a_{n,Y}^-} \Bigr)\int_{\mathcal{X}_{m,n}} f(x)^{2+2 \kappa_1}g(x)^{2\kappa_2-1}  \\
	& \hspace{50pt} \times \int_{B_0(2)}  \min \biggl\{ 1,\frac{1}{\|z\|} \biggl( \frac{\log^{1/2}n}{k_Y^{1/2}} + \biggl( \frac{k_Y M_\beta(x)^d}{ng(x)} \biggr)^{(1 \wedge \beta)/d} \biggr) \biggr\} \,dz \,dx \\
	& \lesssim \frac{\log^2 n}{n} \int_{\mathcal{X}_{m,n}} f(x)^{2+2\kappa_1}g(x)^{2\kappa_2-1} \biggl\{\frac{\log^{1/2}n}{k_Y^{1/2}} + \biggl( \frac{k_Y M_\beta(x)^d}{ng(x)} \biggr)^{(1 \wedge \beta)/d}\biggr\} \,dx \\
	& = O \biggl( \frac{\log^2n}{n}\max\biggl\{\frac{\log^{1/2}n}{k_Y^{1/2}} , \Bigl( \frac{k_Y}{n} \Bigr)^{(1 \wedge \beta)/d}, \Bigl( \frac{k_Y}{n} \Bigr)^{ \epsilon_0 }  \biggr\}\biggr),
\end{align*}
As with $U_5$ we now define $U_{82}:=U_8-U_{81}$ and note that to bound $U_{82}$ we need to control the tails of $s_1,s_2,t_1,t_2$ integrals and our approximations to the derivatives of $h$. By~\eqref{Eq:Holderv2},~\eqref{Eq:hderivatives} and Lemma~\ref{Lemma:BetaTailBounds} we have that
\begin{align*}
	&|U_{82}| \lesssim \frac{1}{n} \int_{\mathcal{X}_{m,n}} f(x)^{2+2\kappa_1} g(x)^{2\kappa_2-1} \biggl\{ \frac{\log^{1/2}m}{k_X^{1/2}} + \frac{\log^{1/2}n}{k_Y^{1/2}} + \biggl( \frac{k_X M_\beta(x)^d}{mf(x)} \biggr)^{(2 \wedge \beta)/d} \\
	& \hspace{165pt}  +\biggl( \frac{k_YM_\beta(x)^d }{ng(x)} \biggr)^{(2 \wedge \beta)/d} + m^{-2} \biggr\} \,dx \\
	& \!=\! O \biggl(\frac{1}{n} \max \biggl\{\frac{\log^{1/2}n}{k_Y^{1/2} }, \frac{\log^{1/2}m}{k_X^{1/2}} ,\Bigl( \frac{k_X}{m} \Bigr)^{(2 \wedge \beta)/d}, \Bigl( \frac{k_Y}{n} \Bigr)^{(2 \wedge \beta)/d}, \Bigl( \frac{k_X}{m} \Bigr)^{ \epsilon_0} , \Bigl( \frac{k_Y}{n} \Bigr)^{ \epsilon_0}, m^{-2} \biggr\} \biggr),
\end{align*}
This establishes~\eqref{Eq:UBounds3}, and therefore concludes the proof.

\end{document}